\documentclass[11pt, a4paper]{amsart}
  \pdfoutput=1
  \usepackage{tikz,genyoungtabtikz,enumitem,rotating,latexsym,bm,stmaryrd}
 \usetikzlibrary{positioning,intersections}
 \usepackage{amsmath,amsthm,amsfonts,amssymb,mathrsfs,pb-diagram}
\usepackage[pagebackref=true,bookmarks=true,colorlinks=true,linktoc=page,citecolor=darkgreen,linkcolor=blue,urlcolor=mediumblue]{hyperref}
\usepackage{caption}
\usepackage{lipsum,wasysym}
\usepackage{mathtools}
\usepackage[a4paper,margin=1in]{geometry}
\usepackage{cleveref}

\usepackage{color}
\definecolor{mediumblue}{rgb}{0.0, 0.0, 0.8}
\colorlet{darkgreen}{green!50!black}
\renewcommand*{\backreflastsep}{, }
\renewcommand*{\backreftwosep}{, }
\renewcommand*{\backref}[1]{}
\renewcommand*{\backrefalt}[4]{%
 \ifcase #1 No citations.
 \or [Page #2.]
 \else [Pages #2.]
 \fi%
}
\newcommand\mptn[2]{\mathscr{P}^{#1}_{#2}}
\renewcommand{\geq}{\geqslant}
\renewcommand{\leq}{\leqslant}
\renewcommand{\trianglerighteq}{\trianglerighteqslant}
\renewcommand{\trianglelefteq}{\trianglelefteqslant}

 \tikzset{wei/.style={red,double=pink,thick,double
distance=1.5pt}}
 \tikzset{wei2/.style={red,double=red!15!white,double
distance=0.5pt}}
 
\allowdisplaybreaks
\numberwithin{equation}{section}
\usepackage{scalefnt}

\newtheorem{thm}{Theorem}[section]
\newtheorem{cor}[thm]{Corollary}

\newtheorem{lem}[thm]{Lemma}
\newtheorem{prop}[thm]{Proposition}
\newtheorem*{prop*}{Proposition}
\newtheorem*{thm*}{Theorem}

\newtheorem*{cor*}{Corollary}
\newtheorem*{conj*}{Conjecture}

\theoremstyle{remark}

\newtheorem*{rmk}{Remark}

\newtheorem{rem}[thm]{Remark}

\newtheorem*{Acknowledgements*}{Acknowledgements}

\theoremstyle{definition}
\newtheorem{defn}[thm]{Definition}
\newtheorem{eg}[thm]{Example}

\newcommand{\degr}{\mathrm{deg}}
\newcommand{\TL}{\mathrm{TL}}
\newcommand{\rad}{\mathrm{rad}}
\newcommand{\res}{\mathrm{res}}
\newcommand{\ik}{{k}}
\newcommand{\Std}{\operatorname{Std}}
\newcommand{\SStd}{\operatorname{SStd}}
\newcommand{\epsilonLIRONdontchange}{\epsilon}
\newcommand{\TSStd}{\operatorname{\mathcal{T}}}
\newcommand{\Shape}{\operatorname{Shape}} 
\newcommand{\Path}{\operatorname{Path}}
\newcommand{\la}{\lambda}
\newcommand{\Lead}{\operatorname{Lead}}
 \newcommand{\SSTS}{\mathsf{S}}  
\newcommand{\SSTT}{\mathsf{T}}  
\newcommand{\SSTU}{\mathsf{U}}  
\newcommand{\SSTV}{\mathsf{V}}  
 \newcommand{\SSTQ}{\mathsf{Q}}  
  \newcommand{\SSTR}{\mathsf{R}}  
\newcommand{\sts}{\mathsf{s}}  
\newcommand{\stt}{\mathsf{t}}  
\newcommand{\stu}{\mathsf{u}}  
\newcommand{\stv}{\mathsf{v}}  
\newcommand{\ZZ}{{\mathbb Z}}
\newcommand{\NN}{{\mathbb N}}
\newcommand{\g}{g}
\newcommand{\CC}{{\mathbb C}}
\newcommand{\RR}{{\mathbb R}}
 
\DeclareMathOperator{\Hom}{Hom}
\def\Mod{\textbf{-Mod}}
\let\<=\langle
\let\>=\rangle

\newcommand\Dec[1][A]{\mathbf{D}_{#1}(t)}
\newcommand\Cart[1][A]{\mathbf{C}_{#1}(t)}

\let\gedom=\trianglerighteq
\let\gdom=\vartriangleright

\tikzset{
    ultra thin/.style= {line width=0.05pt},
    very thin/.style=  {line width=0.2pt},
    thin/.style=       {line width=0.1pt},
    semithick/.style=  {line width=0.6pt},
    thick/.style=      {line width=0.8pt},
    very thick/.style= {line width=1.2pt},
    ultra thick/.style={line width=1.6pt}
}

\newcommand\Dim[2][t]{\text{\rm Dim}_{#1}#2}

\renewcommand{\labelitemi}{$\circ $}

\def\Item{\item\abovedisplayskip=0pt\abovedisplayshortskip=5pt~\vspace*{-\baselineskip}}

\begin{document}

\title[Graded decomposition numbers of Cherednik algebras]{On graded decomposition numbers  \\  for rational Cherednik algebras}

\author{C.~Bowman}
 \email{Chris.Bowman.2@city.ac.uk }
\author{A.~G.~Cox}
 \email{a.g.cox@city.ac.uk }
\address{Department of Mathematics,
 City University London,
 Northampton Square,
 London,
 EC1V 0HB,
UK}

  \author{L.~Speyer} 
  \email{l.speyer@qmul.ac.uk}  

\address{Queen Mary University of London, Mile End Road, London E1 4NS, UK}

\begin{abstract}
     We provide an algorithmic description of a family of 
    graded decomposition numbers for rational Cherednik algebras in terms of affine Kazhdan--Lusztig polynomials.
   \end{abstract}
 
   \maketitle

   \section*{Introduction}

   Double affine Hecke algebras, also known as \emph{Cherednik algebras},  were first introduced by Cherednik as a tool for proving the MacDonald constant term conjectures.
   In \cite{GGOR03}, the authors introduce  
   a category $\mathcal{O}$
    of modules for a given Cherednik algebra, 
    which control the
     representation theory 
    of the underlying cyclotomic  Hecke algebra 
 in much the same way as the classical $q$-Schur algebra does in the case of type $G(1,1,n)$.

Fix a \emph{weighting} $\theta \in \mathbb{R}^l$, an \emph{$e$-multicharge} $\kappa\in (\ZZ/e\ZZ)^l$,  $n\in \NN_0$, and  $g\in \RR$.
In \cite{Webster}, Webster defines a finite dimensional (graded) cellular algebra,  $A(n,\theta,\kappa)$, 
whose module category provides a graded  lift 
   of this category $\mathcal{O}$.
 This algebra is defined by a diagrammatic presentation similar to that of   Khovanov and Lauda, \cite{kl09}.  
 The diagrammatic calculus of this algebra captures a great deal of
 representation theoretic information.  
 
 In this paper, we shall study a certain saturated quotient of Webster's algebra
   in the case 
 that the weighting $\theta=(\theta_1,\ldots, \theta_l)$ is such that $0<\theta_j - \theta_i < g $ for all $1\leq i<j\leq l$ (we refer to this as a FLOTW weighting, after \cite{FLOTW99}).
  Here, the set of one-column multipartitions, $\pi$, forms a saturated set;  of principal interest in this paper will be the quotient algebra (which we call the \emph{quiver Temperley--Lieb  algebra of type $G(l,1,n)$}) 
whose module category is 
 the subcategory of   representations 
whose simple constituents 
   lie in this saturated set.

Strikingly,  the elements of Webster's cellular basis in this quotient algebra  may be indexed by orbits of  paths  in a Euclidean space 
   under the action of an affine Weyl group of type $\hat{A}_{l-1}$. 
    Moreover,
    the graded dimensions  of the standard  modules are given by running a cancellation-free version of  Soergel's algorithm along the paths in this alcove geometry.

 Motivated by this example, we introduce the notion of an algebra with a \emph{Soergel-path basis}. In Theorem \ref{main:theorem:general}, we show that (under certain assumptions) the graded decomposition numbers of such an algebra are given by the associated Kazhdan--Lusztig polynomials. Our approach makes use only of elementary linear algebra and is based on Kleshchev and Nash's proof of the LLT algorithm, \cite{KN10}.

In particular, given a Cherednik algebra with FLOTW weighting, 
the graded   decomposition 
   numbers  $d_{\lambda\mu}(t)$ for  $\lambda,\mu\in \pi$ 
  are given by the affine Kazhdan--Lusztig polynomials 
   of type   $\hat{A}_{l-1}$ (see Corollary \ref{mainresultforchered}).    
   It is shown in  \cite{losev,Webster,RSVV}  that the decomposition numbers of $A(n,\theta,\kappa)$ can be equated with coefficients of Uglov's canonical basis of a twisted Fock space. 
Therefore   the decomposition numbers may, in principle, be  calculated by running an analogue of
 the LLT algorithm, see  \cite{Jacon05}, but not in terms of Kazhdan--Lusztig polynomials.

Our alcove-geometric description comes complete with a translation principle; it also allows us to deduce that the decomposition numbers are stable as the rank
 $n$ tends to infinity.  
Thus, we conjecture that the algebras considered here 
   are asymptotically related
  (as the rank  tends to infinity) to   affine Kac--Moody algebras (see \cite{Kashi}) and in finite rank to the 
  \emph{generalised blob algebras} (see \cite{MW03}).  In the level 2 case, the blob algebra first arose in the study of two-dimensional Potts models,   \cite{MS94}, and has subsequently been  related to the Virasoro algebra \cite{GJSV13} in the limit as $n$ tends to infinity.

In order to clarify the above, let's consider an example.
 We will omit technical details and definitions at this stage, 
 and instead concentrate on giving a flavour of the combinatorics that is involved.  
 Let $l=3$, $n=13$, $e=8$, $\kappa=(0,4,6)$.  We shall consider a single block/linkage class of the 
 algebra 
$ \TL_{13}(\kappa)$.
We embed the one-column multipartitions into Euclidean space via 
the embedding $((1^{\lambda_1}),(1^{\lambda_2}),(1^{\lambda_3}))\mapsto \sum_{i=1}^3\lambda_i \varepsilon_i$ and shall label representations of the algebra by these points (rather than the multipartitions).  
 These points belong to the codimension 1 hyperplane 
given by $\varepsilon_1+\varepsilon_2+\varepsilon_3=13$, which is depicted in Figure \ref{2diag}, below.  
The affine Weyl group acts on this space 
 fixing the point $-\rho$ for $\rho=e(1,1,1)-\kappa=(8,4,2)$.

%
%
%

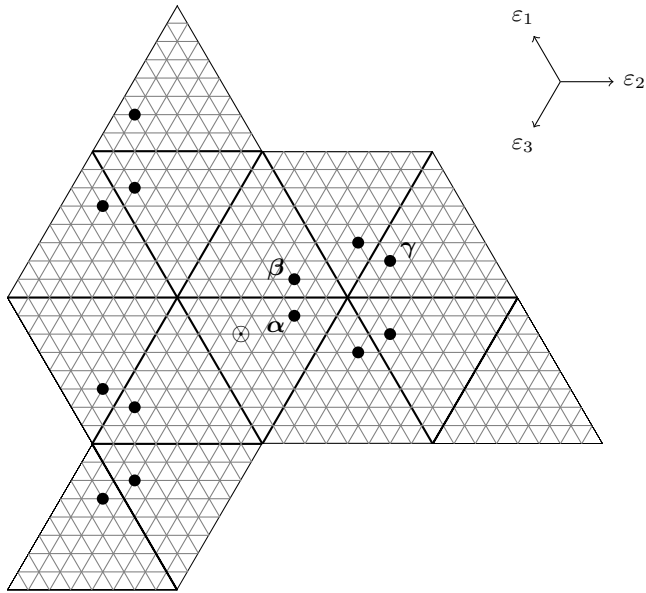
\begin{figure}[ht]\captionsetup{width=0.9\textwidth}\[ 
 \scalefont{0.8}  \begin{tikzpicture}[scale=1.4]       
     \clip (-3.2,1.3) rectangle ++(6.9,5.8);
      \path (120:2.4cm)++(60:2.4cm) coordinate (BB);                    
       \path (0,0) coordinate (origin);
         \path   (origin) ++ (90:6.2 cm)     coordinate (BOOM);   
                  \path   (BOOM) ++(0:2.8 cm)     coordinate (BOOM1);                
  \draw[->](BOOM1) to  ++ (0:0.5 cm);
  \draw[->](BOOM1) to  ++ (120:0.5 cm);
   \draw[->](BOOM1) to  ++ (240:0.5 cm);
 \path(BOOM1)    ++ (0:0.7 cm) coordinate (BOOM2);
 \draw (BOOM2) node {$\varepsilon_2$};
  \path(BOOM1)    ++ (120:0.7 cm) coordinate (BOOM3);
 \draw (BOOM3) node {$\varepsilon_1$};
  \path(BOOM1)    ++ (240:0.7 cm) coordinate (BOOM4);
 \draw (BOOM4) node {$\varepsilon_3$};
   \path[dotted] (origin)++(60:6.4cm) coordinate (aa);
      \path[dotted] (origin)++(60:4.8cm) coordinate (ab);
            \path[dotted] (origin)++(60:4.8cm)++(120:1.6) coordinate (ba);
   \path[dotted] (origin)++(120:6.4cm) coordinate (bb);
   \path[dotted] (origin)++(120:4.8cm) coordinate (bb1);
   \path[dotted] (bb1)++(60:3.2cm) coordinate (bb2);
      \path[dotted] (bb2)++(-60:1.6cm) coordinate (bb3);
      \path[dotted] (origin)++(120:3.2cm) coordinate (cc);
             \path[dotted] (cc)++(-120:1.6cm) coordinate (cc2);
                        \path[dotted] (cc2)++(+0:1.6cm) coordinate (cc3);
            \path[dotted] (origin)++(120:1.6cm)++(60:1.6cm) coordinate (dd);
                        \path[dotted] (origin)++(60:3.2cm)++(0:1.6cm) coordinate (de);
 \clip (ab)--(ba)--(bb3)--(bb2)--(bb1)--(cc)--(cc2)--(cc3)--(dd)--(de)--(ab);

       \foreach \i in {0,1,...,32}
  {
    \path[dotted] (origin)++(60:0.2*\i cm)  coordinate (a\i);
    \path[dotted] (origin)++(120:0.2*\i cm)  coordinate (b\i);
    \path[dotted] (a\i)++(120:4.8cm) coordinate (ca\i);
    \path[dotted] (b\i)++(60:4.8cm) coordinate (cb\i);
   }
      \foreach \i in {33,34,35,36,37,38,39,40,41,42}
     {
    \path[dotted] (origin)++(60:0.2*\i cm)  coordinate (a\i);
    \path[dotted] (origin)++(120:0.2*\i cm)  coordinate (b\i);
    \draw[gray]  (a\i) -- (b\i) ;   
     }
     \foreach \i in {0,8,16,24,32}
{  \draw[thick,black]       (a\i) -- (b\i) ;   
  \draw[thick,black]       (a\i) -- (ca\i);
    \draw[thick,black]       (b\i) -- (cb\i) ; } ;
        \foreach \i 
        in {0,1,2,3,4,5,6,7,9,10,11,12,13,14,15,17,18,19,20,21,22,23,25,26,27,28,29,30,31}
{  \draw[gray]       (a\i) -- (b\i) ;   
  \draw[gray]       (a\i) -- (ca\i);
    \draw[gray]       (b\i) -- (cb\i) ; } ;
  {
    \path[dotted] (a6)++(120:0.2*3 cm)  coordinate (3d);
     \path[dotted] (a10)++(120:0.2*1 cm)  coordinate (2dl);
         \path[dotted] (a15)++(120:0.2*6 cm)  coordinate (1dl);
  \path[dotted] (a17)++(120:0.2*5 cm)  coordinate (2l);
         \path[dotted] (a19)++(120:0.2*7 cm)  coordinate (3l);
  \fill (3d)  circle (1.6pt);
  \fill (2dl)  circle (1.6pt);
   \fill (1dl)  circle (1.6pt);
      \fill (2l)  circle (1.6pt);
         \fill (3l)  circle (1.6pt);
               \draw(3l)++(30:0.2 cm) node {$\boldsymbol{\gamma}$};
      \draw(1dl)++(120:0.8 cm)++(180:0.2 cm)++(150:0.2 cm) node {$\boldsymbol{\beta}$};
            \draw(1dl)++(120:0.4 cm)++(180:0.4 cm)++(210:0.2 cm) node {$\boldsymbol{\alpha}$};
                  \path[dotted] (a11)++(120:23*0.2 cm)  coordinate (XON);  \fill (XON) circle (1.6pt);
    }
 {
      \path[dotted] (b13)++(60:0.2*1 cm)  coordinate (2dr);
         \path[dotted] (b15)++(60:0.2*3 cm)  coordinate (1dr);
  \path[dotted] (b17)++(60:0.2*2 cm)  coordinate (2r);
         \path[dotted] (b22)++(60:0.2*7 cm)  coordinate (3r);
   \fill (2dr)  circle (1.6pt);
   \fill (1dr)  circle (1.6pt);
      \fill (2r)  circle (1.6pt);
         \fill (3r)  circle (1.6pt);
   }
    {
      \path[dotted] (a9)++(120:0.2*21 cm)  coordinate (1u);
         \path[dotted] (a10)++(120:0.2*12 cm)  coordinate (zero);
  \path[dotted] (a13)++(120:0.2*10 cm)  coordinate (2ur);
         \path[dotted] (a14)++(120:0.2*11 cm)  coordinate (2ul);
         \path[dotted] (a18)++(120:0.2*9 cm)  coordinate (2uy);
   \fill (2ur)  circle (1.6pt);
   \fill (2ul)  circle (1.6pt);
      \fill (1u)  circle (1.6pt);
            \fill (2uy)  circle (1.6pt);
         \draw (zero)  node  {$\astrosun$};                 
     }
               \foreach \i in {9,10,11,12,13,14,15,17,18,19,20,21,22,23}
     {
    \path[dotted] (origin)++(60:0.2*\i cm)  coordinate (a\i);
    \path[dotted] (a\i)++(-60:0.2*\i cm)  coordinate (b\i);
    \draw[gray]  (a\i) -- (b\i) ;   
     \path[dotted] (a\i)++(-0:1.6 cm)  coordinate (c\i);
    \draw[gray]  (a\i) -- (c\i) ;         
      \path   (origin)++(60:0.2*16 cm)  coordinate (a1);
          \path  (a1)++(-60:0.2*8 cm)  coordinate (b1);
                    \path  (b1)++(-180:0.2*8 cm)  coordinate (c1);
    \draw  (a1) -- (b1) --(c1) ;   
      \path   (origin)++(60:0.2*24 cm)  coordinate (a2);
          \path  (a2)++(-60:0.2*8 cm)  coordinate (b2);
                    \path  (b2)++(-180:0.2*8 cm)  coordinate (c2);
    \draw  (a2) -- (b2) --(c2) ;   
           \path  (origin)++(60:12 cm)  coordinate (GHY);
    \draw  (origin)--(GHY) ;   
}
 \foreach \i in {1,2,3,4,5,6,7}
{
         \path[dotted] (origin)++(-60:0.2*\i cm)  coordinate (x\i);
     \path[dotted] (x\i)++(60:6 cm)  coordinate (y\i);
     \draw[gray]  (x\i) -- (y\i) ;         
     }
 \path (origin)++(60:1.6 cm)  coordinate (AA);
  \path (origin)++(60:3.2 cm)  coordinate (BB);
    \path (origin)++(60:4.8 cm)  coordinate (CC);
      \path (origin)++(-60:1.6)++(60:4.8 cm)  coordinate (DD);
        \draw   (origin)++(60:1.6 cm)  -- (0:2 cm)  ;
                \foreach \i in {9,10,11,12,13,14,15,17,18,19,20,21,22,23}
     {
    \path[dotted] (origin)++(120:0.2*\i cm)  coordinate (a\i);
    \path[dotted] (a\i)++(-120:0.2*\i cm)  coordinate (b\i);
    \draw[gray]  (a\i) -- (b\i) ;   
     \path[dotted] (a\i)++(180:1.6 cm)  coordinate (c\i);
    \draw[gray]  (a\i) -- (c\i) ;         
      \path   (origin)++(120:0.2*16 cm)  coordinate (a1);
          \path  (a1)++(-120:0.2*8 cm)  coordinate (b1);
                    \path  (b1)++(0:0.2*8 cm)  coordinate (c1);
    \draw  (a1) -- (b1) --(c1) ;   
      \path   (origin)++(120:0.2*24 cm)  coordinate (a2);
          \path  (a2)++(-120:0.2*8 cm)  coordinate (b2);
                    \path  (b2)++(0:0.2*8 cm)  coordinate (c2);
    \draw  (a2) -- (b2) --(c2) ;   
           \path  (origin)++(120:12 cm)  coordinate (GHY);
    \draw  (origin)--(GHY) ;   
}
 \foreach \i in {1,2,3,4,5,6,7}
{
         \path[dotted] (origin)++(-120:0.2*\i cm)  coordinate (x\i);
     \path[dotted] (x\i)++(120:6 cm)  coordinate (y\i);
     \draw[gray]  (x\i) -- (y\i) ;         
     }
 \path (origin)++(120:1.6 cm)  coordinate (AA);
  \path (origin)++(120:3.2 cm)  coordinate (BB);
    \path (origin)++(120:4.8 cm)  coordinate (CC);
      \path (origin)++(-120:1.6)++(120:4.8 cm)  coordinate (DD);
        \draw   (origin)++(120:1.6 cm)  -- (0:2 cm)  ;
                          \path[dotted] (origin)++(120:13*0.2)++(180:0.2 cm)  coordinate (XON);  \fill (XON) circle (1.6pt);             
    \end{tikzpicture} 
\]
  \caption{The black points label the multipartitions of a  block of $\TL_{13}(\kappa)$ via the embedding $(1^{\lambda_1},1^{\lambda_2},1^{\lambda_3})\mapsto \sum_{i=1}^3\lambda_i \varepsilon_i$.     The origin is labelled as $\astrosun$, the points $\alpha=(4,6,3)$, $\beta=(5,6,2)$ and $\gamma=(5,8,0)$ are also marked.  The thick black lines denote the hyperplanes for the $\rho$-shifted action of the Weyl group.  }
\label{2diag}
\end{figure}

  The labels of the representations in this block in which we are most interested are the points 
 $\alpha=(4,6,3)$, $\beta=(5,6,2)$ and $\gamma=(5,8,0)$. 
 The other simple 
 representations  in this block are  labelled by the 
 black points in 
 Figure \ref{2diag}.

Let $\lambda,\mu$ be any elements in our block.  We wish to calculate the graded dimension of the $\mu$-weight space, $\Delta_\mu(\lambda)$, of a 
 cell-module $\Delta(\lambda)$  for $\TL_{13}(\kappa)$.   
  
For a given $\mu$, we   fix 
  a distinguished path, $\omega^\mu$, from the origin to $\mu$, and  
for each $\lambda$ in the above set, we  look at paths which may be obtained
 by folding-up the path $\omega^\mu$ along hyperplanes  so that it terminates at 
$\lambda$ (as illustrated shortly); we denote the set of such paths by $\Path(\lambda,\mu)$.

Each path has an associated degree which can be calculated by running Soergel's (cancellation-free) algorithm along  this path.   
The key to working with the quiver Temperley--Lieb algebras is the following  observation,   
\[\Dim{(\Delta_{\mu}{(\lambda)})} = \; \sum_{\mathclap{\omega \in \Path(\lambda,\mu)}} \; t^{\deg(\omega)}.\] 
 From this, (and the conditions on our distinguished  paths) it is immediate that 
a necessary condition  for $[\Delta(\lambda) :L(\mu)]\neq 0$ is that $\ell(\mu) > \ell(\lambda)$ in the length function associated to our geometry.  
For a fixed $\lambda$, 
we calculate the decomposition numbers $[\Delta(\lambda): L(\mu)]$
 by running Soergel's algorithm not once, but many times: we run the algorithm to each point $\mu$ such that $\ell(\mu)>\ell(\lambda)$.  This is a dual set-up to that usually considered.  

As $n$ tends to infinity, we find that there are infinitely many $\mu$ such that
$\ell(\mu)>\ell(\lambda)$;  the dimension of
 the cell module $\Delta(\lambda)$ 
and the number of  composition factors   of  $\Delta(\lambda)$ also tend  to infinity as $n$ becomes arbitrarily large.  
Fixing a value of $n\in\NN$ truncates the set of weights $\mu$  in our Euclidean space  to a finite set which labels representations of $\TL_n(\kappa)$.   
 We shall see that the decomposition numbers are stable under this limiting behaviour. 



For example, if $\beta=(5,6,2)$, then we take the path 
$\omega^\beta$ given by
\[
 (
 \varepsilon_1,  \varepsilon_2,  \varepsilon_3,
  \varepsilon_1,  \varepsilon_2,  \varepsilon_3,
   \varepsilon_1,  \varepsilon_2,   
   \varepsilon_1,\varepsilon_2,\varepsilon_1,\varepsilon_2,\varepsilon_2
 ).
\]
This path passes through a single hyperplane, namely $x_1-x_3=e$ (this is depicted by the thick black line separating $\alpha$ and $\beta$ in Figure \ref{2diag}).
Reflecting through this hyperplane, we obtain a   path 
\[
 (
 \varepsilon_1,  \varepsilon_2,  \varepsilon_3,
  \varepsilon_1,  \varepsilon_2,  \varepsilon_3,
   \varepsilon_1,  \varepsilon_2,   
\varepsilon_1,\varepsilon_2,\varepsilon_3,\varepsilon_2,\varepsilon_2
 )
\]
of degree 1 which terminates at $\alpha$.  Therefore 
$\Dim{(\Delta_{\beta}(\alpha))}=t^1$.   
We will see that there are no removable subpatterns in following Soergel's procedure in this case, and so (by Theorem \ref{mainresult}) this path labels a graded decomposition number, 
\begin{equation}\label{equationtralalala}\tag{$\dagger$}
[\Delta(\alpha):L(\beta)]=t^1.
\end{equation}


Now let $\gamma= (5,8,0)$; we wish to calculate the dimension of $\Delta_\gamma(\lambda)$ for $\lambda$ in the above set.  The distinguished path, $\omega^\gamma$, in this case is given by
\[
(\varepsilon_1,\varepsilon_2,
\varepsilon_1,\varepsilon_2,
\varepsilon_1,\varepsilon_2,
\varepsilon_1,\varepsilon_2,
\varepsilon_2,\varepsilon_2,\varepsilon_2,\varepsilon_2,\varepsilon_2).
\]
and is pictured in Figure \ref{first path}.
\begin{figure}[ht]\captionsetup{width=0.9\textwidth}
 \[ 
   \begin{tikzpicture}             
      \clip (-0.9,2.4) rectangle ++(3.4,3.3);   
     \path (120:2.4cm)++(60:2.4cm) coordinate (BB);                    
   \path (0,0) coordinate (origin);
 \clip (-0.9,2.4) rectangle ++(3.4,3.3);     \path (0,0) coordinate (origin);
   \path[dotted] (origin)++(60:6.4cm) coordinate (aa);
      \path[dotted] (origin)++(60:4.8cm) coordinate (ab);
            \path[dotted] (origin)++(60:4.8cm)++(120:1.6) coordinate (ba);
   \path[dotted] (origin)++(120:6.4cm) coordinate (bb);
   \path[dotted] (origin)++(120:4.8cm) coordinate (bb1);
   \path[dotted] (bb1)++(60:3.2cm) coordinate (bb2);
      \path[dotted] (bb2)++(-60:1.6cm) coordinate (bb3);
            \path[dotted] (bb3)++(-120:1.6cm) coordinate (bb4);
                       \path[dotted] (bb4)++(-60:1.6cm) coordinate (bb5);
                       \path[dotted] (bb5)++(0:1.6cm) coordinate (bb6);                       
  \clip (ab)--(ba)--(bb3)--(bb4)--(bb5)--(bb6)--(ab);

   \foreach \i in {0,1,...,32}
  {
    \path[dotted] (origin)++(60:0.2*\i cm)  coordinate (a\i);
    \path[dotted] (origin)++(120:0.2*\i cm)  coordinate (b\i);
    \path[dotted] (a\i)++(120:4.8cm) coordinate (ca\i);
    \path[dotted] (b\i)++(60:4.8cm) coordinate (cb\i);
   }
      \foreach \i in {33,34,35,36,37,38,39,40,41,42}
     {
    \path[dotted] (origin)++(60:0.2*\i cm)  coordinate (a\i);
    \path[dotted] (origin)++(120:0.2*\i cm)  coordinate (b\i);
     }
     \foreach \i in {0,8,16,24,32}
{  \draw[thick,black]       (a\i) -- (b\i) ;   
  \draw[thick,black]       (a\i) -- (ca\i);
    \draw[thick,black]       (b\i) -- (cb\i) ; } ;
        \foreach \i 
        in {0,1,2,3,4,5,6,7,9,10,11,12,13,14,15,17,18,19,20,21,22,23,25,26,27,28,29,30,31}
  {
    \path[dotted] (a6)++(120:0.2*3 cm)  coordinate (3d);
     \path[dotted] (a10)++(120:0.2*1 cm)  coordinate (2dl);
         \path[dotted] (a15)++(120:0.2*6 cm)  coordinate (1dl);
  \path[dotted] (a17)++(120:0.2*5 cm)  coordinate (2l);

         \path[dotted] (a19)++(120:0.2*7 cm)  coordinate (3l);
  \fill (3d)  circle (1pt);
  \fill (2dl)  circle (1pt);
   \fill (1dl)  circle (1pt);
      \fill (2l)  circle (1pt);
         \fill (3l)  circle (1pt);
                  \path[dotted] (a11)++(120:23*0.2 cm)  coordinate (XON);  \fill (XON) circle (1pt);
    }
 {
      \path[dotted] (b13)++(60:0.2*1 cm)  coordinate (2dr);
         \path[dotted] (b15)++(60:0.2*3 cm)  coordinate (1dr);
  \path[dotted] (b17)++(60:0.2*2 cm)  coordinate (2r);
         \path[dotted] (b22)++(60:0.2*7 cm)  coordinate (3r);
   \fill (2dr)  circle (1pt);
   \fill (1dr)  circle (1pt);
      \fill (2r)  circle (1pt);
         \fill (3r)  circle (1pt);
   }
    {
      \path[dotted] (a9)++(120:0.2*21 cm)  coordinate (1u);
         \path[dotted] (a10)++(120:0.2*12 cm)  coordinate (zero);
  \path[dotted] (a13)++(120:0.2*10 cm)  coordinate (2ur);
         \path[dotted] (a14)++(120:0.2*11 cm)  coordinate (2ul);
         \path[dotted] (a18)++(120:0.2*9 cm)  coordinate (2uy);
   \fill (2ur)  circle (1pt);
   \fill (2ul)  circle (1pt);
      \fill (1u)  circle (1pt);
            \fill (2uy)  circle (1pt);
         \draw (zero)  circle (1.5pt);                 
     }
     \draw(a10) ++(120:0.2*12 cm)   coordinate (1WALK);
          \draw(1WALK) ++(120:0.2 cm)   coordinate (2WALK);
               \draw(2WALK) ++(0:0.2 cm)   coordinate (3WALK);
\draw(3WALK) ++(120:0.2 cm)   coordinate (4WALK);
               \draw(4WALK) ++(0:0.2 cm)   coordinate (5WALK);
\draw(5WALK) ++(120:0.2 cm)   coordinate (6WALK);
               \draw(6WALK) ++(0:0.2 cm)   coordinate (7WALK);
\draw(7WALK) ++(120:0.2 cm)   coordinate (8WALK);
\draw(8WALK) ++(0:3*0.2 cm)   coordinate (9WALK);
\draw(9WALK) ++(0:2*0.2 cm)   coordinate (10WALK);
\draw(10WALK) ++(0:0.2 cm)   coordinate (11WALK);
\draw(1WALK)--(2WALK)--(3WALK)--(4WALK)--(5WALK)
--(6WALK)--(7WALK)--(8WALK)--(9WALK)--(10WALK)--(11WALK) ;
               \foreach \i in {9,10,11,12,13,14,15,17,18,19,20,21,22,23}
     {
    \path[dotted] (origin)++(60:0.2*\i cm)  coordinate (a\i);
    \path[dotted] (a\i)++(-60:0.2*\i cm)  coordinate (b\i);
     \path[dotted] (a\i)++(-0:1.6 cm)  coordinate (c\i);
      \path   (origin)++(60:0.2*16 cm)  coordinate (a1);
          \path  (a1)++(-60:0.2*8 cm)  coordinate (b1);
                    \path  (b1)++(-180:0.2*8 cm)  coordinate (c1);
    \draw  (a1) -- (b1) --(c1) ;   
      \path   (origin)++(60:0.2*24 cm)  coordinate (a2);
          \path  (a2)++(-60:0.2*8 cm)  coordinate (b2);
                    \path  (b2)++(-180:0.2*8 cm)  coordinate (c2);
    \draw  (a2) -- (b2) --(c2) ;   
           \path  (origin)++(60:12 cm)  coordinate (GHY);
    \draw  (origin)--(GHY) ;   
}

    \end{tikzpicture} 
\]
  \caption{ The distinguished path $\omega^\gamma$ from the origin to  $\gamma=(5,8,0)$.  (The space has been cropped to only include points less than or equal to $\gamma$ in the dominance ordering.) }
\label{first path}
\end{figure}
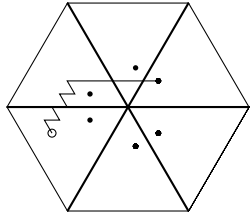

There are a total of $2^3$ distinct paths which may be obtained from this path by a series of reflections (as our path passes through three alcove walls).
For brevity, we truncate our diagrams so as to only consider alcoves between the origin and $\gamma$.  The eight paths are listed in Figures \ref{first path}, \ref{weight a}, \ref{weight b}, and \ref{other weights}.

   \begin{figure}[ht]
\[         \begin{tikzpicture}   \clip (-0.9,2.4) rectangle ++(3.4,3.3);      \path (120:2.4cm)++(60:2.4cm) coordinate (BB);                    
\clip (-0.9,2.4) rectangle ++(3.4,3.3);     \path (0,0) coordinate (origin);
 
\clip (-0.9,2.4) rectangle ++(3.4,3.3);     \path (0,0) coordinate (origin);
   \path[dotted] (origin)++(60:6.4cm) coordinate (aa);
      \path[dotted] (origin)++(60:4.8cm) coordinate (ab);
            \path[dotted] (origin)++(60:4.8cm)++(120:1.6) coordinate (ba);
   \path[dotted] (origin)++(120:6.4cm) coordinate (bb);
   \path[dotted] (origin)++(120:4.8cm) coordinate (bb1);
   \path[dotted] (bb1)++(60:3.2cm) coordinate (bb2);
      \path[dotted] (bb2)++(-60:1.6cm) coordinate (bb3);
            \path[dotted] (bb3)++(-120:1.6cm) coordinate (bb4);
                       \path[dotted] (bb4)++(-60:1.6cm) coordinate (bb5);
                       \path[dotted] (bb5)++(0:1.6cm) coordinate (bb6);                       
  \clip (ab)--(ba)--(bb3)--(bb4)--(bb5)--(bb6)--(ab);   \foreach \i in {0,1,...,32}
  {
    \path[dotted] (origin)++(60:0.2*\i cm)  coordinate (a\i);
    \path[dotted] (origin)++(120:0.2*\i cm)  coordinate (b\i);
    \path[dotted] (a\i)++(120:4.8cm) coordinate (ca\i);
    \path[dotted] (b\i)++(60:4.8cm) coordinate (cb\i);
   }
      \foreach \i in {33,34,35,36,37,38,39,40,41,42}
     {
    \path[dotted] (origin)++(60:0.2*\i cm)  coordinate (a\i);
    \path[dotted] (origin)++(120:0.2*\i cm)  coordinate (b\i);
     }
     \foreach \i in {0,8,16,24,32}
{  \draw[thick,black]       (a\i) -- (b\i) ;   
  \draw[thick,black]       (a\i) -- (ca\i);
    \draw[thick,black]       (b\i) -- (cb\i) ; } ;
        \foreach \i 
        in {0,1,2,3,4,5,6,7,9,10,11,12,13,14,15,17,18,19,20,21,22,23,25,26,27,28,29,30,31}
  {
    \path[dotted] (a6)++(120:0.2*3 cm)  coordinate (3d);
     \path[dotted] (a10)++(120:0.2*1 cm)  coordinate (2dl);
         \path[dotted] (a15)++(120:0.2*6 cm)  coordinate (1dl);
  \path[dotted] (a17)++(120:0.2*5 cm)  coordinate (2l);

         \path[dotted] (a19)++(120:0.2*7 cm)  coordinate (3l);
  \fill (3d)  circle (1pt);
  \fill (2dl)  circle (1pt);
   \fill (1dl)  circle (1pt);
      \fill (2l)  circle (1pt);
         \fill (3l)  circle (1pt);
                  \path[dotted] (a11)++(120:23*0.2 cm)  coordinate (XON);  \fill (XON) circle (1pt);
    }
 {
      \path[dotted] (b13)++(60:0.2*1 cm)  coordinate (2dr);
         \path[dotted] (b15)++(60:0.2*3 cm)  coordinate (1dr);
  \path[dotted] (b17)++(60:0.2*2 cm)  coordinate (2r);
         \path[dotted] (b22)++(60:0.2*7 cm)  coordinate (3r);
   \fill (2dr)  circle (1pt);
   \fill (1dr)  circle (1pt);
      \fill (2r)  circle (1pt);
         \fill (3r)  circle (1pt);
   }
    {
      \path[dotted] (a9)++(120:0.2*21 cm)  coordinate (1u);
         \path[dotted] (a10)++(120:0.2*12 cm)  coordinate (zero);
  \path[dotted] (a13)++(120:0.2*10 cm)  coordinate (2ur);
         \path[dotted] (a14)++(120:0.2*11 cm)  coordinate (2ul);
         \path[dotted] (a18)++(120:0.2*9 cm)  coordinate (2uy);
   \fill (2ur)  circle (1pt);
   \fill (2ul)  circle (1pt);
      \fill (1u)  circle (1pt);
            \fill (2uy)  circle (1pt);
         \draw (zero)  circle (1.5pt);                 
     }
     \draw(a10) ++(120:0.2*12 cm)   coordinate (1WALK);
          \draw(1WALK) ++(120:0.2 cm)   coordinate (2WALK);
               \draw(2WALK) ++(0:0.2 cm)   coordinate (3WALK);
\draw(3WALK) ++(120:0.2 cm)   coordinate (4WALK);
               \draw(4WALK) ++(0:0.2 cm)   coordinate (5WALK);
\draw(5WALK) ++(120:0.2 cm)   coordinate (6WALK);
               \draw(6WALK) ++(0:0.2 cm)   coordinate (7WALK);
\draw(7WALK) ++(120:0.2 cm)   coordinate (8WALK);
\draw(8WALK) ++(0:3*0.2 cm)   coordinate (9WALK);
\draw(9WALK) ++(-120:2*0.2 cm)   coordinate (10WALK);
\draw(10WALK) ++(-120:0.2 cm)   coordinate (11WALK);
\draw(1WALK)--(2WALK)--(3WALK)--(4WALK)--(5WALK)
--(6WALK)--(7WALK)--(8WALK)--(9WALK)--(10WALK)--(11WALK) ;
               \foreach \i in {9,10,11,12,13,14,15,17,18,19,20,21,22,23}
     {
    \path[dotted] (origin)++(60:0.2*\i cm)  coordinate (a\i);
    \path[dotted] (a\i)++(-60:0.2*\i cm)  coordinate (b\i);
     \path[dotted] (a\i)++(-0:1.6 cm)  coordinate (c\i);
      \path   (origin)++(60:0.2*16 cm)  coordinate (a1);
          \path  (a1)++(-60:0.2*8 cm)  coordinate (b1);
                    \path  (b1)++(-180:0.2*8 cm)  coordinate (c1);
    \draw  (a1) -- (b1) --(c1) ;   
      \path   (origin)++(60:0.2*24 cm)  coordinate (a2);
          \path  (a2)++(-60:0.2*8 cm)  coordinate (b2);
                    \path  (b2)++(-180:0.2*8 cm)  coordinate (c2);
    \draw  (a2) -- (b2) --(c2) ;   
           \path  (origin)++(60:12 cm)  coordinate (GHY);
    \draw  (origin)--(GHY) ;   
}

    \end{tikzpicture} 
\quad 
    \quad
          \begin{tikzpicture}   \clip (-0.9,2.4) rectangle ++(3.4,3.3);      \path (120:2.4cm)++(60:2.4cm) coordinate (BB);                    
\clip (-0.9,2.4) rectangle ++(3.4,3.3);     \path (0,0) coordinate (origin);

\clip (-0.9,2.4) rectangle ++(3.4,3.3);     \path (0,0) coordinate (origin);
   \path[dotted] (origin)++(60:6.4cm) coordinate (aa);
      \path[dotted] (origin)++(60:4.8cm) coordinate (ab);
            \path[dotted] (origin)++(60:4.8cm)++(120:1.6) coordinate (ba);
   \path[dotted] (origin)++(120:6.4cm) coordinate (bb);
   \path[dotted] (origin)++(120:4.8cm) coordinate (bb1);
   \path[dotted] (bb1)++(60:3.2cm) coordinate (bb2);
      \path[dotted] (bb2)++(-60:1.6cm) coordinate (bb3);
            \path[dotted] (bb3)++(-120:1.6cm) coordinate (bb4);
                       \path[dotted] (bb4)++(-60:1.6cm) coordinate (bb5);
                       \path[dotted] (bb5)++(0:1.6cm) coordinate (bb6);                       
  \clip (ab)--(ba)--(bb3)--(bb4)--(bb5)--(bb6)--(ab);   \foreach \i in {0,1,...,32}
  {
    \path[dotted] (origin)++(60:0.2*\i cm)  coordinate (a\i);
    \path[dotted] (origin)++(120:0.2*\i cm)  coordinate (b\i);
    \path[dotted] (a\i)++(120:4.8cm) coordinate (ca\i);
    \path[dotted] (b\i)++(60:4.8cm) coordinate (cb\i);
   }
      \foreach \i in {33,34,35,36,37,38,39,40,41,42}
     {
    \path[dotted] (origin)++(60:0.2*\i cm)  coordinate (a\i);
    \path[dotted] (origin)++(120:0.2*\i cm)  coordinate (b\i);
     }
     \foreach \i in {0,8,16,24,32}
{  \draw[thick,black]       (a\i) -- (b\i) ;   
  \draw[thick,black]       (a\i) -- (ca\i);
    \draw[thick,black]       (b\i) -- (cb\i) ; } ;
        \foreach \i 
        in {0,1,2,3,4,5,6,7,9,10,11,12,13,14,15,17,18,19,20,21,22,23,25,26,27,28,29,30,31}
  {
    \path[dotted] (a6)++(120:0.2*3 cm)  coordinate (3d);
     \path[dotted] (a10)++(120:0.2*1 cm)  coordinate (2dl);
         \path[dotted] (a15)++(120:0.2*6 cm)  coordinate (1dl);
  \path[dotted] (a17)++(120:0.2*5 cm)  coordinate (2l);

         \path[dotted] (a19)++(120:0.2*7 cm)  coordinate (3l);
  \fill (3d)  circle (1pt);
  \fill (2dl)  circle (1pt);
   \fill (1dl)  circle (1pt);
      \fill (2l)  circle (1pt);
         \fill (3l)  circle (1pt);
                  \path[dotted] (a11)++(120:23*0.2 cm)  coordinate (XON);  \fill (XON) circle (1pt);
    }
 {
      \path[dotted] (b13)++(60:0.2*1 cm)  coordinate (2dr);
         \path[dotted] (b15)++(60:0.2*3 cm)  coordinate (1dr);
  \path[dotted] (b17)++(60:0.2*2 cm)  coordinate (2r);
         \path[dotted] (b22)++(60:0.2*7 cm)  coordinate (3r);
   \fill (2dr)  circle (1pt);
   \fill (1dr)  circle (1pt);
      \fill (2r)  circle (1pt);
         \fill (3r)  circle (1pt);
   }
    {
      \path[dotted] (a9)++(120:0.2*21 cm)  coordinate (1u);
         \path[dotted] (a10)++(120:0.2*12 cm)  coordinate (zero);
  \path[dotted] (a13)++(120:0.2*10 cm)  coordinate (2ur);
         \path[dotted] (a14)++(120:0.2*11 cm)  coordinate (2ul);
         \path[dotted] (a18)++(120:0.2*9 cm)  coordinate (2uy);
   \fill (2ur)  circle (1pt);
   \fill (2ul)  circle (1pt);
      \fill (1u)  circle (1pt);
            \fill (2uy)  circle (1pt);
         \draw (zero)  circle (1.5pt);                 
     }
     \draw(a10) ++(120:0.2*12 cm)   coordinate (1WALK);
          \draw(1WALK) ++(120:0.2 cm)   coordinate (2WALK);
               \draw(2WALK) ++(0:0.2 cm)   coordinate (3WALK);
\draw(3WALK) ++(120:0.2 cm)   coordinate (4WALK);
               \draw(4WALK) ++(0:0.2 cm)   coordinate (5WALK);
\draw(5WALK) ++(-120:0.2 cm)   coordinate (6WALK);
               \draw(6WALK) ++(0:0.2 cm)   coordinate (7WALK);
\draw(7WALK) ++(-120:0.2 cm)   coordinate (8WALK);
\draw(8WALK) ++(0:3*0.2 cm)   coordinate (9WALK);
\draw(9WALK) ++(120:2*0.2 cm)   coordinate (10WALK);
\draw(10WALK) ++(-120:0.2 cm)   coordinate (11WALK);
\draw(1WALK)--(2WALK)--(3WALK)--(4WALK)--(5WALK)
--(6WALK)--(7WALK)--(8WALK)--(9WALK)--(10WALK)--(11WALK) ;
               \foreach \i in {9,10,11,12,13,14,15,17,18,19,20,21,22,23}
     {
    \path[dotted] (origin)++(60:0.2*\i cm)  coordinate (a\i);
    \path[dotted] (a\i)++(-60:0.2*\i cm)  coordinate (b\i);
     \path[dotted] (a\i)++(-0:1.6 cm)  coordinate (c\i);
      \path   (origin)++(60:0.2*16 cm)  coordinate (a1);
          \path  (a1)++(-60:0.2*8 cm)  coordinate (b1);
                    \path  (b1)++(-180:0.2*8 cm)  coordinate (c1);
    \draw  (a1) -- (b1) --(c1) ;   
      \path   (origin)++(60:0.2*24 cm)  coordinate (a2);
          \path  (a2)++(-60:0.2*8 cm)  coordinate (b2);
                    \path  (b2)++(-180:0.2*8 cm)  coordinate (c2);
    \draw  (a2) -- (b2) --(c2) ;   
           \path  (origin)++(60:12 cm)  coordinate (GHY);
    \draw  (origin)--(GHY) ;   
}

    \end{tikzpicture} 
\]
      \caption{ The  paths in $\Path(\alpha,\gamma)$. These elements are of degrees 1 and 3 respectively. }
\label{weight a}
    \end{figure}
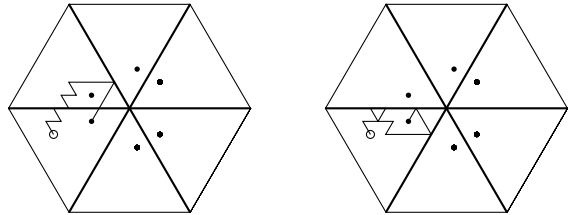

    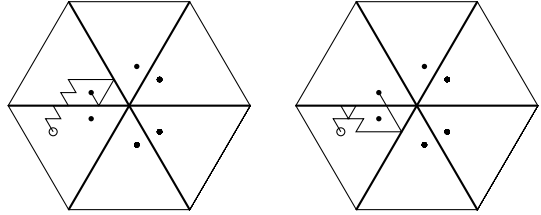
\begin{figure}[ht]
 \[    \begin{tikzpicture}   \clip (-0.9,2.4) rectangle ++(3.4,3.3);      \path (120:2.4cm)++(60:2.4cm) coordinate (BB);                    
\clip (-0.9,2.4) rectangle ++(3.4,3.3);     \path (0,0) coordinate (origin);

\clip (-0.9,2.4) rectangle ++(3.4,3.3);     \path (0,0) coordinate (origin);
   \path[dotted] (origin)++(60:6.4cm) coordinate (aa);
      \path[dotted] (origin)++(60:4.8cm) coordinate (ab);
            \path[dotted] (origin)++(60:4.8cm)++(120:1.6) coordinate (ba);
   \path[dotted] (origin)++(120:6.4cm) coordinate (bb);
   \path[dotted] (origin)++(120:4.8cm) coordinate (bb1);
   \path[dotted] (bb1)++(60:3.2cm) coordinate (bb2);
      \path[dotted] (bb2)++(-60:1.6cm) coordinate (bb3);
            \path[dotted] (bb3)++(-120:1.6cm) coordinate (bb4);
                       \path[dotted] (bb4)++(-60:1.6cm) coordinate (bb5);
                       \path[dotted] (bb5)++(0:1.6cm) coordinate (bb6);                       
  \clip (ab)--(ba)--(bb3)--(bb4)--(bb5)--(bb6)--(ab);   \foreach \i in {0,1,...,32}
  {
    \path[dotted] (origin)++(60:0.2*\i cm)  coordinate (a\i);
    \path[dotted] (origin)++(120:0.2*\i cm)  coordinate (b\i);
    \path[dotted] (a\i)++(120:4.8cm) coordinate (ca\i);
    \path[dotted] (b\i)++(60:4.8cm) coordinate (cb\i);
   }
      \foreach \i in {33,34,35,36,37,38,39,40,41,42}
     {
    \path[dotted] (origin)++(60:0.2*\i cm)  coordinate (a\i);
    \path[dotted] (origin)++(120:0.2*\i cm)  coordinate (b\i);
     }
     \foreach \i in {0,8,16,24,32}
{  \draw[thick,black]       (a\i) -- (b\i) ;   
  \draw[thick,black]       (a\i) -- (ca\i);
    \draw[thick,black]       (b\i) -- (cb\i) ; } ;
        \foreach \i 
        in {0,1,2,3,4,5,6,7,9,10,11,12,13,14,15,17,18,19,20,21,22,23,25,26,27,28,29,30,31}
  {
    \path[dotted] (a6)++(120:0.2*3 cm)  coordinate (3d);
     \path[dotted] (a10)++(120:0.2*1 cm)  coordinate (2dl);
         \path[dotted] (a15)++(120:0.2*6 cm)  coordinate (1dl);
  \path[dotted] (a17)++(120:0.2*5 cm)  coordinate (2l);

         \path[dotted] (a19)++(120:0.2*7 cm)  coordinate (3l);
  \fill (3d)  circle (1pt);
  \fill (2dl)  circle (1pt);
   \fill (1dl)  circle (1pt);
      \fill (2l)  circle (1pt);
         \fill (3l)  circle (1pt);
                  \path[dotted] (a11)++(120:23*0.2 cm)  coordinate (XON);  \fill (XON) circle (1pt);
    }
 {
      \path[dotted] (b13)++(60:0.2*1 cm)  coordinate (2dr);
         \path[dotted] (b15)++(60:0.2*3 cm)  coordinate (1dr);
  \path[dotted] (b17)++(60:0.2*2 cm)  coordinate (2r);
         \path[dotted] (b22)++(60:0.2*7 cm)  coordinate (3r);
   \fill (2dr)  circle (1pt);
   \fill (1dr)  circle (1pt);
      \fill (2r)  circle (1pt);
         \fill (3r)  circle (1pt);
   }
    {
      \path[dotted] (a9)++(120:0.2*21 cm)  coordinate (1u);
         \path[dotted] (a10)++(120:0.2*12 cm)  coordinate (zero);
  \path[dotted] (a13)++(120:0.2*10 cm)  coordinate (2ur);
         \path[dotted] (a14)++(120:0.2*11 cm)  coordinate (2ul);
         \path[dotted] (a18)++(120:0.2*9 cm)  coordinate (2uy);
   \fill (2ur)  circle (1pt);
   \fill (2ul)  circle (1pt);
      \fill (1u)  circle (1pt);
            \fill (2uy)  circle (1pt);
         \draw (zero)  circle (1.5pt);                 
     }
     \draw(a10) ++(120:0.2*12 cm)   coordinate (1WALK);
          \draw(1WALK) ++(120:0.2 cm)   coordinate (2WALK);
               \draw(2WALK) ++(0:0.2 cm)   coordinate (3WALK);
\draw(3WALK) ++(120:0.2 cm)   coordinate (4WALK);
               \draw(4WALK) ++(0:0.2 cm)   coordinate (5WALK);
\draw(5WALK) ++(120:0.2 cm)   coordinate (6WALK);
               \draw(6WALK) ++(0:0.2 cm)   coordinate (7WALK);
\draw(7WALK) ++(120:0.2 cm)   coordinate (8WALK);
\draw(8WALK) ++(0:3*0.2 cm)   coordinate (9WALK);
\draw(9WALK) ++(-120:2*0.2 cm)   coordinate (10WALK);
\draw(10WALK) ++(+120:0.2 cm)   coordinate (11WALK);
\draw(1WALK)--(2WALK)--(3WALK)--(4WALK)--(5WALK)
--(6WALK)--(7WALK)--(8WALK)--(9WALK)--(10WALK)--(11WALK) ;
               \foreach \i in {9,10,11,12,13,14,15,17,18,19,20,21,22,23}
     {
    \path[dotted] (origin)++(60:0.2*\i cm)  coordinate (a\i);
    \path[dotted] (a\i)++(-60:0.2*\i cm)  coordinate (b\i);
     \path[dotted] (a\i)++(-0:1.6 cm)  coordinate (c\i);
      \path   (origin)++(60:0.2*16 cm)  coordinate (a1);
          \path  (a1)++(-60:0.2*8 cm)  coordinate (b1);
                    \path  (b1)++(-180:0.2*8 cm)  coordinate (c1);
    \draw  (a1) -- (b1) --(c1) ;   
      \path   (origin)++(60:0.2*24 cm)  coordinate (a2);
          \path  (a2)++(-60:0.2*8 cm)  coordinate (b2);
                    \path  (b2)++(-180:0.2*8 cm)  coordinate (c2);
    \draw  (a2) -- (b2) --(c2) ;   
           \path  (origin)++(60:12 cm)  coordinate (GHY);
    \draw  (origin)--(GHY) ;   
}

    \end{tikzpicture} \quad 
   \begin{tikzpicture}   \clip (-0.9,2.4) rectangle ++(3.4,3.3);      \path (120:2.4cm)++(60:2.4cm) coordinate (BB);                    
\clip (-0.9,2.4) rectangle ++(3.4,3.3);     \path (0,0) coordinate (origin);
\clip (-0.9,2.4) rectangle ++(3.4,3.3);     \path (0,0) coordinate (origin);
   \path[dotted] (origin)++(60:6.4cm) coordinate (aa);
      \path[dotted] (origin)++(60:4.8cm) coordinate (ab);
            \path[dotted] (origin)++(60:4.8cm)++(120:1.6) coordinate (ba);
   \path[dotted] (origin)++(120:6.4cm) coordinate (bb);
   \path[dotted] (origin)++(120:4.8cm) coordinate (bb1);
   \path[dotted] (bb1)++(60:3.2cm) coordinate (bb2);
      \path[dotted] (bb2)++(-60:1.6cm) coordinate (bb3);
            \path[dotted] (bb3)++(-120:1.6cm) coordinate (bb4);
                       \path[dotted] (bb4)++(-60:1.6cm) coordinate (bb5);
                       \path[dotted] (bb5)++(0:1.6cm) coordinate (bb6);                       
  \clip (ab)--(ba)--(bb3)--(bb4)--(bb5)--(bb6)--(ab);   \foreach \i in {0,1,...,32}
  {
    \path[dotted] (origin)++(60:0.2*\i cm)  coordinate (a\i);
    \path[dotted] (origin)++(120:0.2*\i cm)  coordinate (b\i);
    \path[dotted] (a\i)++(120:4.8cm) coordinate (ca\i);
    \path[dotted] (b\i)++(60:4.8cm) coordinate (cb\i);
   }
      \foreach \i in {33,34,35,36,37,38,39,40,41,42}
     {
    \path[dotted] (origin)++(60:0.2*\i cm)  coordinate (a\i);
    \path[dotted] (origin)++(120:0.2*\i cm)  coordinate (b\i);
     }
     \foreach \i in {0,8,16,24,32}
{  \draw[thick,black]       (a\i) -- (b\i) ;   
  \draw[thick,black]       (a\i) -- (ca\i);
    \draw[thick,black]       (b\i) -- (cb\i) ; } ;
        \foreach \i 
        in {0,1,2,3,4,5,6,7,9,10,11,12,13,14,15,17,18,19,20,21,22,23,25,26,27,28,29,30,31}
  {
    \path[dotted] (a6)++(120:0.2*3 cm)  coordinate (3d);
     \path[dotted] (a10)++(120:0.2*1 cm)  coordinate (2dl);
         \path[dotted] (a15)++(120:0.2*6 cm)  coordinate (1dl);
  \path[dotted] (a17)++(120:0.2*5 cm)  coordinate (2l);

         \path[dotted] (a19)++(120:0.2*7 cm)  coordinate (3l);
  \fill (3d)  circle (1pt);
  \fill (2dl)  circle (1pt);
   \fill (1dl)  circle (1pt);
      \fill (2l)  circle (1pt);
         \fill (3l)  circle (1pt);
                  \path[dotted] (a11)++(120:23*0.2 cm)  coordinate (XON);  \fill (XON) circle (1pt);
    }
 {
      \path[dotted] (b13)++(60:0.2*1 cm)  coordinate (2dr);
         \path[dotted] (b15)++(60:0.2*3 cm)  coordinate (1dr);
  \path[dotted] (b17)++(60:0.2*2 cm)  coordinate (2r);
         \path[dotted] (b22)++(60:0.2*7 cm)  coordinate (3r);
   \fill (2dr)  circle (1pt);
   \fill (1dr)  circle (1pt);
      \fill (2r)  circle (1pt);
         \fill (3r)  circle (1pt);
   }
    {
      \path[dotted] (a9)++(120:0.2*21 cm)  coordinate (1u);
         \path[dotted] (a10)++(120:0.2*12 cm)  coordinate (zero);
  \path[dotted] (a13)++(120:0.2*10 cm)  coordinate (2ur);
         \path[dotted] (a14)++(120:0.2*11 cm)  coordinate (2ul);
         \path[dotted] (a18)++(120:0.2*9 cm)  coordinate (2uy);
   \fill (2ur)  circle (1pt);
   \fill (2ul)  circle (1pt);
      \fill (1u)  circle (1pt);
            \fill (2uy)  circle (1pt);
         \draw (zero)  circle (1.5pt);                 
     }
     \draw(a10) ++(120:0.2*12 cm)   coordinate (1WALK);
          \draw(1WALK) ++(120:0.2 cm)   coordinate (2WALK);
               \draw(2WALK) ++(0:0.2 cm)   coordinate (3WALK);
\draw(3WALK) ++(120:0.2 cm)   coordinate (4WALK);
               \draw(4WALK) ++(0:0.2 cm)   coordinate (5WALK);
\draw(5WALK) ++(-120:0.2 cm)   coordinate (6WALK);
               \draw(6WALK) ++(0:0.2 cm)   coordinate (7WALK);
\draw(7WALK) ++(-120:0.2 cm)   coordinate (8WALK);
\draw(8WALK) ++(0:3*0.2 cm)   coordinate (9WALK);
\draw(9WALK) ++(120:2*0.2 cm)   coordinate (10WALK);
\draw(10WALK) ++(120:0.2 cm)   coordinate (11WALK);
\draw(1WALK)--(2WALK)--(3WALK)--(4WALK)--(5WALK)
--(6WALK)--(7WALK)--(8WALK)--(9WALK)--(10WALK)--(11WALK) ;
               \foreach \i in {9,10,11,12,13,14,15,17,18,19,20,21,22,23}
     {
    \path[dotted] (origin)++(60:0.2*\i cm)  coordinate (a\i);
    \path[dotted] (a\i)++(-60:0.2*\i cm)  coordinate (b\i);
     \path[dotted] (a\i)++(-0:1.6 cm)  coordinate (c\i);
      \path   (origin)++(60:0.2*16 cm)  coordinate (a1);
          \path  (a1)++(-60:0.2*8 cm)  coordinate (b1);
                    \path  (b1)++(-180:0.2*8 cm)  coordinate (c1);
    \draw  (a1) -- (b1) --(c1) ;   
      \path   (origin)++(60:0.2*24 cm)  coordinate (a2);
          \path  (a2)++(-60:0.2*8 cm)  coordinate (b2);
                    \path  (b2)++(-180:0.2*8 cm)  coordinate (c2);
    \draw  (a2) -- (b2) --(c2) ;   
           \path  (origin)++(60:12 cm)  coordinate (GHY);
    \draw  (origin)--(GHY) ;   
}

    \end{tikzpicture} 
\]
 \caption{  The paths in $\Path(\beta,\gamma)$.  These paths are of degrees $0$ and $2$, respectively.}
\label{weight b}
\end{figure}
Those familiar with Soergel's algorithm will  recognise the degrees of the paths listed in the  figures,
 (see Section \ref{pathalgorithm} for more details).
These paths describe the dimension of   $\Delta_{\gamma}(\lambda)$   for any point $\lambda$. 
 In particular, 
  \[\Dim{(\Delta_\gamma(\alpha))}=t^3+t^1\quad \Dim{(\Delta_\gamma(\beta))}=t^2+t^0.\]
The leftmost diagram in Figure \ref{weight b} labels an element of 
$\Path(\beta,\gamma)$ of degree zero. 
This  path of degree zero labels a term that would be removed under Soergel's algorithm. 
 In terms of our basis, this corresponds to  the fact that  this path labels a  vector  in 
  the  basis  of the simple head   $L(\beta)$; in fact
  \[\Dim{(L_\gamma(\beta))}=t^0,\]
  see Section \ref{pathalgorithm} for more details.  The other path, of degree 2, is \emph{not} 
  removed under Soergel's procedure, and  therefore (by  Theorem \ref{mainresult})
   labels a decomposition number
  \[[\Delta(\beta):L{(\gamma)}]=t^2.\]

Having removed the degree zero path in $\Path(\beta,\gamma)$ under Soergel's procedure,
 we also remove all paths in a certain subpattern labelled by this zero (see Section \ref{pathalgorithm}).  In this case, the only other path in this subpattern is the leftmost 
 path pictured in Figure \ref{weight a}.

 The degree 0 path in the subpattern corresponds to $\Dim{(\Delta_\gamma(\beta))} =
 \Dim{(L_\gamma(\beta))}=t^0$.  
 The degree 1 path in the subpattern corresponds to the degree 1 basis element of $\Delta_\gamma(\alpha)$
 which occurs in the  (degree shifted by 1)
 copy of $L_\gamma(\beta)\langle 1 \rangle$ inside $\Delta_\gamma(\alpha)$, (see \ref{equationtralalala}).  
%
%
 This accounts for  the path  in $\Path(\alpha,\gamma)$ of degree $1$.

 That leaves a path of
 degree $3$, which is not removed under Soergel's procedure and so 
    $[\Delta(\alpha):L(\gamma)]=t^3$.
    
In general, we will see that at each stage in our algorithm, the removed subpatterns correspond to the weight spaces of simple modules, and 
the surviving paths correspond to decomposition numbers.


\begin{figure}[ht]
\[
   \begin{tikzpicture}   \clip (-0.9,2.4) rectangle ++(3.4,3.3);      \path (120:2.4cm)++(60:2.4cm) coordinate (BB);                    
\clip (-0.9,2.4) rectangle ++(3.4,3.3);     \path (0,0) coordinate (origin);

\clip (-0.9,2.4) rectangle ++(3.4,3.3);     \path (0,0) coordinate (origin);
   \path[dotted] (origin)++(60:6.4cm) coordinate (aa);
      \path[dotted] (origin)++(60:4.8cm) coordinate (ab);
            \path[dotted] (origin)++(60:4.8cm)++(120:1.6) coordinate (ba);
   \path[dotted] (origin)++(120:6.4cm) coordinate (bb);
   \path[dotted] (origin)++(120:4.8cm) coordinate (bb1);
   \path[dotted] (bb1)++(60:3.2cm) coordinate (bb2);
      \path[dotted] (bb2)++(-60:1.6cm) coordinate (bb3);
            \path[dotted] (bb3)++(-120:1.6cm) coordinate (bb4);
                       \path[dotted] (bb4)++(-60:1.6cm) coordinate (bb5);
                       \path[dotted] (bb5)++(0:1.6cm) coordinate (bb6);                       
  \clip (ab)--(ba)--(bb3)--(bb4)--(bb5)--(bb6)--(ab);   \foreach \i in {0,1,...,32}
  {
    \path[dotted] (origin)++(60:0.2*\i cm)  coordinate (a\i);
    \path[dotted] (origin)++(120:0.2*\i cm)  coordinate (b\i);
    \path[dotted] (a\i)++(120:4.8cm) coordinate (ca\i);
    \path[dotted] (b\i)++(60:4.8cm) coordinate (cb\i);
   }
      \foreach \i in {33,34,35,36,37,38,39,40,41,42}
     {
    \path[dotted] (origin)++(60:0.2*\i cm)  coordinate (a\i);
    \path[dotted] (origin)++(120:0.2*\i cm)  coordinate (b\i);
     }
     \foreach \i in {0,8,16,24,32}
{  \draw[thick,black]       (a\i) -- (b\i) ;   
  \draw[thick,black]       (a\i) -- (ca\i);
    \draw[thick,black]       (b\i) -- (cb\i) ; } ;
        \foreach \i 
        in {0,1,2,3,4,5,6,7,9,10,11,12,13,14,15,17,18,19,20,21,22,23,25,26,27,28,29,30,31}
  {
    \path[dotted] (a6)++(120:0.2*3 cm)  coordinate (3d);
     \path[dotted] (a10)++(120:0.2*1 cm)  coordinate (2dl);
         \path[dotted] (a15)++(120:0.2*6 cm)  coordinate (1dl);
  \path[dotted] (a17)++(120:0.2*5 cm)  coordinate (2l);

         \path[dotted] (a19)++(120:0.2*7 cm)  coordinate (3l);
  \fill (3d)  circle (1.3pt);
  \fill (2dl)  circle (1.3pt);
   \fill (1dl)  circle (1.3pt);
      \fill (2l)  circle (1.3pt);
         \fill (3l)  circle (1.3pt);
                  \path[dotted] (a11)++(120:23*0.2 cm)  coordinate (XON);  \fill (XON) circle (1.3pt);
    }
 {
      \path[dotted] (b13)++(60:0.2*1 cm)  coordinate (2dr);
         \path[dotted] (b15)++(60:0.2*3 cm)  coordinate (1dr);
  \path[dotted] (b17)++(60:0.2*2 cm)  coordinate (2r);
         \path[dotted] (b22)++(60:0.2*7 cm)  coordinate (3r);
   \fill (2dr)  circle (1.3pt);
   \fill (1dr)  circle (1.3pt);
      \fill (2r)  circle (1.3pt);
         \fill (3r)  circle (1.3pt);
   }
    {
      \path[dotted] (a9)++(120:0.2*21 cm)  coordinate (1u);
         \path[dotted] (a10)++(120:0.2*12 cm)  coordinate (zero);
  \path[dotted] (a13)++(120:0.2*10 cm)  coordinate (2ur);
         \path[dotted] (a14)++(120:0.2*11 cm)  coordinate (2ul);
         \path[dotted] (a18)++(120:0.2*9 cm)  coordinate (2uy);
   \fill (2ur)  circle (1.3pt);
   \fill (2ul)  circle (1.3pt);
      \fill (1u)  circle (1.3pt);
            \fill (2uy)  circle (1.3pt);
         \draw (zero)  circle (1.5pt);                 
     }
     \draw(a10) ++(120:0.2*12 cm)   coordinate (1WALK);
          \draw(1WALK) ++(120:0.2 cm)   coordinate (2WALK);
               \draw(2WALK) ++(0:0.2 cm)   coordinate (3WALK);
\draw(3WALK) ++(120:0.2 cm)   coordinate (4WALK);
               \draw(4WALK) ++(0:0.2 cm)   coordinate (5WALK);
\draw(5WALK) ++(-120:0.2 cm)   coordinate (6WALK);
               \draw(6WALK) ++(0:0.2 cm)   coordinate (7WALK);
\draw(7WALK) ++(-120:0.2 cm)   coordinate (8WALK);
\draw(8WALK) ++(0:3*0.2 cm)   coordinate (9WALK);
\draw(9WALK) ++(0:2*0.2 cm)   coordinate (10WALK);
\draw(10WALK) ++(0:0.2 cm)   coordinate (11WALK);
\draw(1WALK)--(2WALK)--(3WALK)--(4WALK)--(5WALK)
--(6WALK)--(7WALK)--(8WALK)--(9WALK)--(10WALK)--(11WALK) ;
               \foreach \i in {9,10,11,12,13,14,15,17,18,19,20,21,22,23}
     {
    \path[dotted] (origin)++(60:0.2*\i cm)  coordinate (a\i);
    \path[dotted] (a\i)++(-60:0.2*\i cm)  coordinate (b\i);
     \path[dotted] (a\i)++(-0:1.6 cm)  coordinate (c\i);
      \path   (origin)++(60:0.2*16 cm)  coordinate (a1);
          \path  (a1)++(-60:0.2*8 cm)  coordinate (b1);
                    \path  (b1)++(-180:0.2*8 cm)  coordinate (c1);
    \draw  (a1) -- (b1) --(c1) ;   
      \path   (origin)++(60:0.2*24 cm)  coordinate (a2);
          \path  (a2)++(-60:0.2*8 cm)  coordinate (b2);
                    \path  (b2)++(-180:0.2*8 cm)  coordinate (c2);
    \draw  (a2) -- (b2) --(c2) ;   
           \path  (origin)++(60:12 cm)  coordinate (GHY);
    \draw  (origin)--(GHY) ;   
}
    \end{tikzpicture} 
    \quad
       \begin{tikzpicture}   \clip (-0.9,2.4) rectangle ++(3.4,3.3);      \path (120:2.4cm)++(60:2.4cm) coordinate (BB);                    
\clip (-0.9,2.4) rectangle ++(3.4,3.3);     \path (0,0) coordinate (origin);

\clip (-0.9,2.4) rectangle ++(3.4,3.3);     \path (0,0) coordinate (origin);
   \path[dotted] (origin)++(60:6.4cm) coordinate (aa);
      \path[dotted] (origin)++(60:4.8cm) coordinate (ab);
            \path[dotted] (origin)++(60:4.8cm)++(120:1.6) coordinate (ba);
   \path[dotted] (origin)++(120:6.4cm) coordinate (bb);
   \path[dotted] (origin)++(120:4.8cm) coordinate (bb1);
   \path[dotted] (bb1)++(60:3.2cm) coordinate (bb2);
      \path[dotted] (bb2)++(-60:1.6cm) coordinate (bb3);
            \path[dotted] (bb3)++(-120:1.6cm) coordinate (bb4);
                       \path[dotted] (bb4)++(-60:1.6cm) coordinate (bb5);
                       \path[dotted] (bb5)++(0:1.6cm) coordinate (bb6);                       
  \clip (ab)--(ba)--(bb3)--(bb4)--(bb5)--(bb6)--(ab);   \foreach \i in {0,1,...,32}
  {
    \path[dotted] (origin)++(60:0.2*\i cm)  coordinate (a\i);
    \path[dotted] (origin)++(120:0.2*\i cm)  coordinate (b\i);
    \path[dotted] (a\i)++(120:4.8cm) coordinate (ca\i);
    \path[dotted] (b\i)++(60:4.8cm) coordinate (cb\i);
   }
      \foreach \i in {33,34,35,36,37,38,39,40,41,42}
     {
    \path[dotted] (origin)++(60:0.2*\i cm)  coordinate (a\i);
    \path[dotted] (origin)++(120:0.2*\i cm)  coordinate (b\i);
     }
     \foreach \i in {0,8,16,24,32}
{  \draw[thick,black]       (a\i) -- (b\i) ;   
  \draw[thick,black]       (a\i) -- (ca\i);
    \draw[thick,black]       (b\i) -- (cb\i) ; } ;
        \foreach \i 
        in {0,1,2,3,4,5,6,7,9,10,11,12,13,14,15,17,18,19,20,21,22,23,25,26,27,28,29,30,31}
  {
    \path[dotted] (a6)++(120:0.2*3 cm)  coordinate (3d);
     \path[dotted] (a10)++(120:0.2*1 cm)  coordinate (2dl);
         \path[dotted] (a15)++(120:0.2*6 cm)  coordinate (1dl);
  \path[dotted] (a17)++(120:0.2*5 cm)  coordinate (2l);

         \path[dotted] (a19)++(120:0.2*7 cm)  coordinate (3l);
  \fill (3d)  circle (1.3pt);
  \fill (2dl)  circle (1.3pt);
   \fill (1dl)  circle (1.3pt);
      \fill (2l)  circle (1.3pt);
         \fill (3l)  circle (1.3pt);
                  \path[dotted] (a11)++(120:23*0.2 cm)  coordinate (XON);  \fill (XON) circle (1.3pt);
    }
 {
      \path[dotted] (b13)++(60:0.2*1 cm)  coordinate (2dr);
         \path[dotted] (b15)++(60:0.2*3 cm)  coordinate (1dr);
  \path[dotted] (b17)++(60:0.2*2 cm)  coordinate (2r);
         \path[dotted] (b22)++(60:0.2*7 cm)  coordinate (3r);
   \fill (2dr)  circle (1.3pt);
   \fill (1dr)  circle (1.3pt);
      \fill (2r)  circle (1.3pt);
         \fill (3r)  circle (1.3pt);
   }
    {
      \path[dotted] (a9)++(120:0.2*21 cm)  coordinate (1u);
         \path[dotted] (a10)++(120:0.2*12 cm)  coordinate (zero);
  \path[dotted] (a13)++(120:0.2*10 cm)  coordinate (2ur);
         \path[dotted] (a14)++(120:0.2*11 cm)  coordinate (2ul);
         \path[dotted] (a18)++(120:0.2*9 cm)  coordinate (2uy);
   \fill (2ur)  circle (1.3pt);
   \fill (2ul)  circle (1.3pt);
      \fill (1u)  circle (1.3pt);
            \fill (2uy)  circle (1.3pt);
         \draw (zero)  circle (1.5pt);                 
     }
     \draw(a10) ++(120:0.2*12 cm)   coordinate (1WALK);
          \draw(1WALK) ++(120:0.2 cm)   coordinate (2WALK);
               \draw(2WALK) ++(0:0.2 cm)   coordinate (3WALK);
\draw(3WALK) ++(120:0.2 cm)   coordinate (4WALK);
               \draw(4WALK) ++(0:0.2 cm)   coordinate (5WALK);
\draw(5WALK) ++(-120:0.2 cm)   coordinate (6WALK);
               \draw(6WALK) ++(0:0.2 cm)   coordinate (7WALK);
\draw(7WALK) ++(-120:0.2 cm)   coordinate (8WALK);
\draw(8WALK) ++(0:3*0.2 cm)   coordinate (9WALK);
\draw(9WALK) ++(0:2*0.2 cm)   coordinate (10WALK);
\draw(10WALK) ++(-120:0.2 cm)   coordinate (11WALK);
\draw(1WALK)--(2WALK)--(3WALK)--(4WALK)--(5WALK)
--(6WALK)--(7WALK)--(8WALK)--(9WALK)--(10WALK)--(11WALK) ;
               \foreach \i in {9,10,11,12,13,14,15,17,18,19,20,21,22,23}
     {
    \path[dotted] (origin)++(60:0.2*\i cm)  coordinate (a\i);
    \path[dotted] (a\i)++(-60:0.2*\i cm)  coordinate (b\i);
     \path[dotted] (a\i)++(-0:1.6 cm)  coordinate (c\i);
      \path   (origin)++(60:0.2*16 cm)  coordinate (a1);
          \path  (a1)++(-60:0.2*8 cm)  coordinate (b1);
                    \path  (b1)++(-180:0.2*8 cm)  coordinate (c1);
    \draw  (a1) -- (b1) --(c1) ;   
      \path   (origin)++(60:0.2*24 cm)  coordinate (a2);
          \path  (a2)++(-60:0.2*8 cm)  coordinate (b2);
                    \path  (b2)++(-180:0.2*8 cm)  coordinate (c2);
    \draw  (a2) -- (b2) --(c2) ;   
           \path  (origin)++(60:12 cm)  coordinate (GHY);
    \draw  (origin)--(GHY) ;   
}

    \end{tikzpicture} \quad
     \begin{tikzpicture}   \clip (-0.9,2.4) rectangle ++(3.4,3.3);      \path (120:2.4cm)++(60:2.4cm) coordinate (BB);                    
\clip (-0.9,2.4) rectangle ++(3.4,3.3);     \path (0,0) coordinate (origin);
 \clip (-0.9,2.4) rectangle ++(3.4,3.3);     \path (0,0) coordinate (origin);
   \path[dotted] (origin)++(60:6.4cm) coordinate (aa);
      \path[dotted] (origin)++(60:4.8cm) coordinate (ab);
            \path[dotted] (origin)++(60:4.8cm)++(120:1.6) coordinate (ba);
   \path[dotted] (origin)++(120:6.4cm) coordinate (bb);
   \path[dotted] (origin)++(120:4.8cm) coordinate (bb1);
   \path[dotted] (bb1)++(60:3.2cm) coordinate (bb2);
      \path[dotted] (bb2)++(-60:1.6cm) coordinate (bb3);
            \path[dotted] (bb3)++(-120:1.6cm) coordinate (bb4);
                       \path[dotted] (bb4)++(-60:1.6cm) coordinate (bb5);
                       \path[dotted] (bb5)++(0:1.6cm) coordinate (bb6);                       
  \clip (ab)--(ba)--(bb3)--(bb4)--(bb5)--(bb6)--(ab);   \foreach \i in {0,1,...,32}
  {
    \path[dotted] (origin)++(60:0.2*\i cm)  coordinate (a\i);
    \path[dotted] (origin)++(120:0.2*\i cm)  coordinate (b\i);
    \path[dotted] (a\i)++(120:4.8cm) coordinate (ca\i);
    \path[dotted] (b\i)++(60:4.8cm) coordinate (cb\i);
   }
      \foreach \i in {33,34,35,36,37,38,39,40,41,42}
     {
    \path[dotted] (origin)++(60:0.2*\i cm)  coordinate (a\i);
    \path[dotted] (origin)++(120:0.2*\i cm)  coordinate (b\i);
     }
     \foreach \i in {0,8,16,24,32}
{  \draw[thick,black]       (a\i) -- (b\i) ;   
  \draw[thick,black]       (a\i) -- (ca\i);
    \draw[thick,black]       (b\i) -- (cb\i) ; } ;
        \foreach \i 
        in {0,1,2,3,4,5,6,7,9,10,11,12,13,14,15,17,18,19,20,21,22,23,25,26,27,28,29,30,31}
  {
    \path[dotted] (a6)++(120:0.2*3 cm)  coordinate (3d);
     \path[dotted] (a10)++(120:0.2*1 cm)  coordinate (2dl);
         \path[dotted] (a15)++(120:0.2*6 cm)  coordinate (1dl);
  \path[dotted] (a17)++(120:0.2*5 cm)  coordinate (2l);

         \path[dotted] (a19)++(120:0.2*7 cm)  coordinate (3l);
  \fill (3d)  circle (1.3pt);
  \fill (2dl)  circle (1.3pt);
   \fill (1dl)  circle (1.3pt);
      \fill (2l)  circle (1.3pt);
         \fill (3l)  circle (1.3pt);
                  \path[dotted] (a11)++(120:23*0.2 cm)  coordinate (XON);  \fill (XON) circle (1.3pt);
    }
 {
      \path[dotted] (b13)++(60:0.2*1 cm)  coordinate (2dr);
         \path[dotted] (b15)++(60:0.2*3 cm)  coordinate (1dr);
  \path[dotted] (b17)++(60:0.2*2 cm)  coordinate (2r);
         \path[dotted] (b22)++(60:0.2*7 cm)  coordinate (3r);
   \fill (2dr)  circle (1.3pt);
   \fill (1dr)  circle (1.3pt);
      \fill (2r)  circle (1.3pt);
         \fill (3r)  circle (1.3pt);
   }
    {
      \path[dotted] (a9)++(120:0.2*21 cm)  coordinate (1u);
         \path[dotted] (a10)++(120:0.2*12 cm)  coordinate (zero);
  \path[dotted] (a13)++(120:0.2*10 cm)  coordinate (2ur);
         \path[dotted] (a14)++(120:0.2*11 cm)  coordinate (2ul);
         \path[dotted] (a18)++(120:0.2*9 cm)  coordinate (2uy);
   \fill (2ur)  circle (1.3pt);
   \fill (2ul)  circle (1.3pt);
      \fill (1u)  circle (1.3pt);
            \fill (2uy)  circle (1.3pt);
         \draw (zero)  circle (1.5pt);                 
     }
     \draw(a10) ++(120:0.2*12 cm)   coordinate (1WALK);
          \draw(1WALK) ++(120:0.2 cm)   coordinate (2WALK);
               \draw(2WALK) ++(0:0.2 cm)   coordinate (3WALK);
\draw(3WALK) ++(120:0.2 cm)   coordinate (4WALK);
               \draw(4WALK) ++(0:0.2 cm)   coordinate (5WALK);
\draw(5WALK) ++(120:0.2 cm)   coordinate (6WALK);
               \draw(6WALK) ++(0:0.2 cm)   coordinate (7WALK);
\draw(7WALK) ++(120:0.2 cm)   coordinate (8WALK);
\draw(8WALK) ++(0:3*0.2 cm)   coordinate (9WALK);
\draw(9WALK) ++(0:2*0.2 cm)   coordinate (10WALK);
\draw(10WALK) ++(120:0.2 cm)   coordinate (11WALK);
\draw(1WALK)--(2WALK)--(3WALK)--(4WALK)--(5WALK)
--(6WALK)--(7WALK)--(8WALK)--(9WALK)--(10WALK)--(11WALK) ;
               \foreach \i in {9,10,11,12,13,14,15,17,18,19,20,21,22,23}
     {
    \path[dotted] (origin)++(60:0.2*\i cm)  coordinate (a\i);
    \path[dotted] (a\i)++(-60:0.2*\i cm)  coordinate (b\i);
     \path[dotted] (a\i)++(-0:1.6 cm)  coordinate (c\i);
      \path   (origin)++(60:0.2*16 cm)  coordinate (a1);
          \path  (a1)++(-60:0.2*8 cm)  coordinate (b1);
                    \path  (b1)++(-180:0.2*8 cm)  coordinate (c1);
    \draw  (a1) -- (b1) --(c1) ;   
      \path   (origin)++(60:0.2*24 cm)  coordinate (a2);
          \path  (a2)++(-60:0.2*8 cm)  coordinate (b2);
                    \path  (b2)++(-180:0.2*8 cm)  coordinate (c2);
    \draw  (a2) -- (b2) --(c2) ;   
           \path  (origin)++(60:12 cm)  coordinate (GHY);
    \draw  (origin)--(GHY) ;   
}
    \end{tikzpicture} 
\]
 \caption{  The remaining paths in $\Path(-,\gamma)$. 
  These paths are of degree $1$, $2$, and $1$ respectively.}
\label{other weights} 
\end{figure}
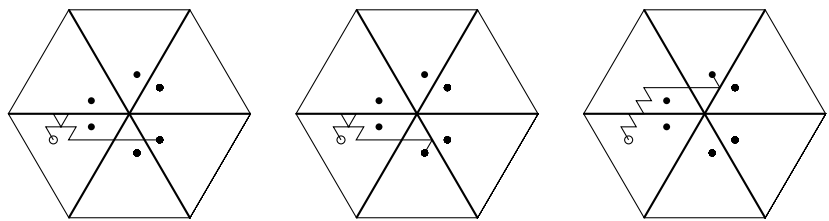

 \begin{Acknowledgements*}
 The authors  would like to thank Joe Chuang, Matt Fayers, Daniel Juteau, and Ben Webster for 
 helpful conversations 
 during the preparation of this manuscript.  

The authors 
  are grateful for the financial support received from 
the {Royal Commission for the Exhibition of 1851},  {EPSRC}  {grant  EP/L01078X/1}, and 
  Queen Mary University of London, respectively.  
 The authors also thank the ICMS in Edinburgh for their hospitality
  during the early stages of this project.  
 \end{Acknowledgements*}
 
\section{Soergel path algebras}

In this section (inspired by  \cite[Section 3]{MW03}), we define an abstract family of algebras whose  bases possess   desirable properties.  The combinatorics of these algebras is controlled by orbits of 
paths in a Euclidean space. 

\subsection{Graded cellular algebras with highest weight theories}

  \begin{defn}\label{defn1}  
  Suppose that $A$ is a $\ZZ$-graded $\CC$-algebra which
  is free of finite rank over $\CC$.
  We say that $A$ is a 
  \emph{graded cellular algebra with a highest weight theory} if the following conditions hold.

 The algebra is equipped with a \emph{cell datum} 
  $(\mptn ln,\TSStd,C,\degr)$, where $(\mptn ln,\trianglerighteq )$ is the \emph{weight poset}.
 For each $\lambda,\mu\in\mptn ln$,  such that $\lambda \trianglerighteq \mu$,  
we have a finite set, denoted $\TSStd(\lambda,\mu)$,  and we let
$\TSStd(\lambda)=\cup_\mu \TSStd(\lambda,\mu)$.
 There exist maps
  \[C:{\coprod_{\lambda\in\mptn ln}\TSStd(\lambda)\times \TSStd(\lambda)}\to A;
      \quad\text{and}\quad
     \degr: {\coprod_{\lambda\in\mptn ln}\TSStd(\lambda)} \to \ZZ \] such that $C$ is injective. We denote $C(\SSTS,\SSTT) = c^\lambda_{\SSTS\SSTT}$  for $\SSTS,\SSTT\in\TSStd(\lambda)$, and
  \begin{enumerate}
    \item Each   element $c^\lambda_{\SSTS\SSTT}$ is homogeneous
	of degree 
	\[\degr
        (c^\lambda_{\SSTS\SSTT})=\degr(\SSTS)+\degr(\SSTT),\] for
        $\lambda\in\mptn ln$ and 
      $\SSTS,\SSTT\in \TSStd(\lambda)$.
    \item The set $\{c^\lambda_{\SSTS\SSTT}\mid\SSTS,\SSTT\in \TSStd(\lambda), \,
      \lambda\in\mptn ln \}$ is a  
      $\CC$-basis of $A$.
    \item  If $\SSTS,\SSTT\in \TSStd(\lambda)$, for some
      $\lambda\in\mptn ln$, and $a\in A$ then 
    there exist scalars $r_{\SSTS\SSTU}(a)$, which do not depend on
    $\SSTT$, such that 
      \[ac^\lambda_{\SSTS\SSTT}  =\sum_{\SSTU\in
      \TSStd(\lambda)}r_{\SSTS\SSTU}(a)c^\lambda_{\SSTU\SSTT}\pmod 
      {A^{\triangleright  \lambda}},\]
      where $A^{\triangleright  \lambda}$ is the $\CC$-submodule of $A$ spanned by
      \[\{c^\mu_{\SSTQ\SSTR}\mid\mu \triangleright  \lambda\text{ and }\SSTQ,\SSTR\in \TSStd(\mu )\}.\]
    \item  The $\mathbb{C}$-linear map $*:A\to A$ determined by
      $(c^\lambda_{\SSTS\SSTT})^*=c^\lambda_{\SSTT\SSTS}$, for all $\lambda\in\mptn ln$ and
      all $\SSTS,\SSTT\in\TSStd(\lambda)$, is an anti-isomorphism of $A$.
\item    The algebra $A$ has an identity element, $1_A$, such that $1_A=\sum_{\lambda\in\mptn ln}1_\lambda$ is an orthogonal idempotent decomposition.  
    \item 
 For $\SSTS\in\TSStd(\lambda,\mu)$, $\SSTT \in\TSStd(\lambda,\nu)$, we have that $1_\mu c^\lambda_{\SSTS\SSTT}1_\nu =  c^\lambda_{\SSTS\SSTT}$.  
 There exists a unique element $\SSTT^\lambda \in \TSStd(\lambda,\lambda)$, and $c_{\SSTT^\lambda\SSTT^\lambda}^\lambda = 1_\lambda$.  
   \end{enumerate}
\end{defn}

\begin{rmk}
Notice that the above satisfies the conditions of a graded cellular algebra \cite{hm10}.  In addition, such an algebra is quasi-hereditary (as every cell-ideal contains an idempotent).   Conditions (5) and (6)   allow us to examine standard modules by 
considering their weight space decompositions.  
\end{rmk}

Unless otherwise stated, all results in this section follow from \cite{hm10}.  
 Suppose that $A$ is a graded cellular algebra with a highest weight theory. Given any $\lambda\in\mptn ln$, the   \emph{graded standard module} $\Delta(\lambda)$ is the graded left $A$-module
    \[\Delta(\lambda)=\bigoplus_{\begin{subarray}c \mu\in\mptn ln \\ z\in\ZZ \end{subarray}}\Delta_\mu(\lambda)_z,\]
    where $\Delta_\mu(\lambda)_z$ is the $\mathbb{C}$ vector-space  with basis
    $\{c^\lambda_{\SSTS } \mid \SSTS\in \TSStd(\lambda,\mu)\text{ and }\deg(\SSTS)=z\}$.
    The action of $A$ on $\Delta(\lambda)$ is given by
    \[a c^\lambda_{ \SSTS  }  =\sum_{ \SSTU \in \TSStd(\lambda)}r_{\SSTS\SSTU}(a) c^\lambda_{\SSTU},\]
    where the scalars $r_{\SSTS\SSTU}(a)$ are the scalars appearing in condition (3) of Definition \ref{defn1}.
 
Suppose that $\lambda \in \mptn ln$. There is a bilinear form $\langle\ ,\ \rangle_\lambda$ on $\Delta(\lambda) $ which
is determined by
\[c^\lambda_{\SSTU\SSTS}c^\lambda_{\SSTT\SSTV}\equiv
  \langle c^\lambda_\SSTS,c^\lambda_\SSTT\rangle_\lambda c^\lambda_{\SSTU\SSTV}\pmod{A^{\triangleright \lambda}},\]
for any $\SSTS,\SSTT, \SSTU,\SSTV\in \TSStd(\lambda  )$.

Let $t$ be an indeterminate over $\NN_0$. If $M=\oplus_{z\in\ZZ}M_z$ is
a free graded $\mathbb{C}$-module,  
then its \emph{graded dimension} is the Laurent  polynomial
\[\Dim{(M)}=\sum_{k\in\ZZ}(\dim_{\mathbb{C}}M_k)t^k.\]

If $M$ is a graded $A$-module and $k\in\ZZ$, define $M\langle k \rangle$ to be the same module with $(M\langle k \rangle)_i = M_{i-k}$ for all $i\in\ZZ$. We call this a \emph{degree shift} by $k$.
If $M$ is a graded
$A$-module and $L$ is a graded simple module let $[M:L\langle k\rangle]$ be the
multiplicity of 
$L\langle k\rangle$ as a graded composition factor
of $M$, for $k\in\ZZ$.

   Suppose that $A$ is a graded cellular algebra with a highest weight theory. 
For every $\lambda \in \mptn ln$,  define  $L(\lambda)$ to be  the quotient of the corresponding standard module $\Delta(\lambda)$ by 
  the radical of the bilinear form $\langle\ ,\ \rangle_\lambda$.  This module is simple, and every simple module is of the form $L(\lambda)\langle k \rangle$ for some $k \in \ZZ$, $\lambda \in \mptn ln$.  We let $L_\mu(\lambda)$ denote the $\mu$-weight space $1_\mu L(\lambda)$.  
  The 
  \emph{graded decomposition matrix} of $A$ is the matrix
  $\Dec=(d_{\lambda\mu}(t))$, where
  \[d_{\lambda\mu}(t)=\sum_{k\in\ZZ} [\Delta(\lambda):L(\mu)\langle k\rangle]\,t^k,\]
  for $\lambda,\mu\in\mptn ln$.
 The following proposition is a key ingredient in our proof of the main result of this paper. 
\begin{prop}[\cite{hm10}, Proposition 2.18]\label{humathasprop}
If $\mu\in\mptn ln$ then $\Dim{(L(\mu))} \in \NN_0[t+t^{-1}]$.  
\end{prop}


Given $\lambda,\mu\in \mptn ln$  such that $\lambda\rhd \mu$, we say that $\lambda$ and $\mu$   
are \emph{tableau-linked} if the set $\TSStd(\lambda,\mu)$ is non-empty. The equivalence classes of the equivalence relation on $\mptn ln$ generated by this tableau-linkage are called the \emph{tableau-blocks} of $A$.
\begin{prop}\label{linakge}[The Linkage Principle]
If $\lambda,\mu\in\mptn ln$ label simple modules in  the same block of $A$, then $\lambda$ and $\mu$ are tableau-linked.  
\end{prop}
\begin{proof}
It is clear that a necessary condition for 
$\Dim{(\Hom(P(\lambda),\Delta(\mu)))} = [\Delta(\lambda):L(\mu)]\neq 0$, is 
that $\TSStd(\lambda,\mu)\neq\emptyset$.  The result then follows from 
\cite[(3.9.8)]{GL98}.  
\end{proof}

 This result inspires the next section, in which we connect tableaux to paths in an alcove geometry.  

\subsection{The alcove geometry}  \label{Soergeldeg}
 
 We shall assume standard facts concerning root systems, see \cite{Bourb}.
 Let $\{\varepsilon_1, \varepsilon_2, \ldots , \varepsilon_r\}$ be a set of formal symbols and set
\[E_r=\bigoplus_{i=1}^{r} \mathbb{R}\varepsilon_i\]
to be the $r$-dimensional real vector space with basis $\varepsilon_1, \varepsilon_2, \ldots ,\varepsilon_r$. 
We have an inner product $\langle \, , \, \rangle$ given by extending linearly the relations
\[\langle \varepsilon_i, \varepsilon_j\rangle = \delta_{i,j}\]
for all $1\leq i, j \leq r$, where $\delta_{i,j}$ is the Kronecker delta. 

 Let $A(\rho,e)$ denote a cellular algebra with a highest weight theory depending on  parameters $\rho\in E_r$ and $e \in \NN\cup \{\infty\}$ and let  
  $\mptn ln$ denote the indexing set of the simple modules.  
   Suppose that there exists an embedding 
   $ \mptn ln \hookrightarrow E_r   
$; we identify an element $\lambda \in \mptn ln$ with its image under this map.  

Let $\Phi$ 
 denote a root system embedded in $E_r$ as in \cite[Plates I to IX]{Bourb} and let $h$ denote the corresponding  Coxeter number. We take $R^+$ to be the set of positive roots.  
  For each $\alpha \in \Phi$ there is a unique coroot $\alpha^\vee$ such that 
 $\langle \alpha , \alpha^\vee \rangle=2$.
For $e \in \NN \cup \{\infty\}$ we let $W^e$ denote the affine reflection group generated by the reflections $s_{\alpha,me}$ (for $\alpha \in \Phi$, $m\in \ZZ$) given by
\[s_{\alpha,me}(x)=x-(\langle x, \alpha^\vee \rangle -me) \alpha\]
for all $x\in E_r$. 
 For $e= \infty$, define $W^\infty $ to be the subgroup generated by the reflections $s_{\alpha,0}$ for $\alpha \in \Phi$.

  Now, given $  \rho \in E_r$,   we shall always consider the shifted action of $W^e$ by $\rho$ given by
\[w \cdot x = w(x+\rho)-\rho\]
for all $w\in W^e$ and $x\in E_r$.  
 We regard $s_{\alpha,me}$ as a reflection with respect to the hyperplane
\[
h_{\alpha,me}= \{\lambda \in E_r\mid \langle \lambda+\rho , \alpha^\vee\rangle =me\}.
\]
The reflection group $W^e$ acting on $E_r$ defines a system of facets.  A \emph{facet} is a non-empty subset of $E_r$ of the form
\begin{align*}
\mathfrak{f}=\{
\lambda \in E_r \mid &
\langle \lambda + \rho, \alpha^\vee\rangle = m_\alpha e \text{ for all }\alpha \in R^0_+(\mathfrak{f}), \\
& (m_\alpha-1)e < \langle \lambda + \rho, \alpha^\vee\rangle < m_\alpha e \text{ for all }\alpha \in R^1_+(\mathfrak{f})  
\},
\end{align*}
 for suitable integers $m_\alpha \in \ZZ$ and a disjoint decomposition $R^+ = R^0_+(\mathfrak{f}) \cup  R^1_+(\mathfrak{f})$.  A facet, $\mathfrak{f}$, is called an \emph{alcove} if $|R^0_+(\mathfrak{f})|=0$ and a \emph{wall} if
  $|R^0_+(\mathfrak{f})|=1$.  
  A  point $x\in\mathfrak{f}$  is called $e$-\emph{regular} if $|R^0_+(\mathfrak{f})|=0$, and 
  is called $e$-\emph{singular} if $|R^0_+(\mathfrak{f})|\geq1$.  
We assume that $\rho_i\neq 0$ modulo $e$ for any $1\leq i \leq r$,  so that     the origin is always contained in an alcove, which we refer to as the
 \emph{fundamental alcove}. 
 The \emph{closure}, $ \overline{\mathfrak{f}}$,  of a facet, $\mathfrak{f}$, is defined as follows
\begin{align*}\overline{\mathfrak{f}} = \{
\lambda \in E_r \mid &
 \langle \lambda + \rho, \alpha^\vee\rangle = m_\alpha e \text{ for all }\alpha \in R^0_+(\mathfrak{f}), \\
& (m_\alpha-1)e \leq \langle \lambda + \rho, \alpha^\vee\rangle \leq m_\alpha e \text{ for all }\alpha \in R^1_+(\mathfrak{f})  
\} .
\end{align*}
We define a length function on the set of alcoves as follows.  
We say that two alcoves, $\mathfrak{a}_i, \mathfrak{a}_j$ are adjacent if $\overline{\mathfrak{a}}_i \cap \overline{\mathfrak{a}}_j$ is non-empty. 
   Given any pair of alcoves $\mathfrak{a}$ and $\mathfrak{b}$, there exists a  chain of adjacent alcoves,
\[\mathfrak{a} =\mathfrak{a}_0, \mathfrak{a}_1, \ldots ,\mathfrak{a}_\ell = \mathfrak{b},\]
  and we define the \emph{length} $\ell( \mathfrak{a}, \mathfrak{b})$ to be the minimal number of alcoves in such a chain.  We extend this notation to points in alcoves in the obvious manner.  

\subsection{Paths in an alcove geometry}\label{pathsinanaclove} 
In this section we fix $e>h$ and consider paths in our alcove geometry.  
  Given $k\in \NN$, we let $\mathbf{k}$ denote the set $\{1,2,\ldots, k\}$.  
Given   a map  $w: \mathbf{n}\to \mathbf{r}$  we 
 define points $\omega(k)\in E_r$ 
 by
  \[
  \omega(k)=\sum_{1\leq i \leq k}\varepsilon_{w(i)},
  \]   
 for $1\leq \ik \leq n$. We define the associated path of length $n$ in our alcove geometry $E_r$ by $\omega=(\omega(0),\omega(1),\omega(2), \ldots, \omega(n))$, where we fix all paths to begin at the origin, so that $\omega(0)=\astrosun \in E_r$.  We let $\omega_{\leq \ik}$ denote the subpath of $\omega$ of length $\ik$ corresponding to   $w_{\leq \ik}: \mathbf{\ik}\to \mathbf{r}$.

 
\begin{defn}\label{Soergeldegree}
 Fix a path $\omega=(\omega(0),\omega(1),\omega(2), \ldots, \omega(n))$ such that $\omega(0)=\astrosun \in E_r$. We define a
 degree function on $\omega$ by induction.
 We set $\deg(\omega(0))=0$ and set 
 \[
 \deg(\omega_{\leq \ik}) = \deg(\omega_{\leq \ik-1})	 + \sum_\alpha d_\alpha(\omega,\ik)
 \]
where $d_\alpha(\omega,\ik)$ is defined as follows.  
Fix  $\alpha \in \Phi$, and consider the hyperplanes $h_{\alpha,me} $ for $m\in \ZZ$.  
If $\omega(\ik)$ and $\omega(\ik+1)$   both lie on 
some $h_{\alpha,me}$ or if neither lie on some $h_{\alpha,me}$ for $m \in \ZZ$, then
$d_\alpha(\omega,\ik)=0$.  Otherwise, exactly one of $\omega(\ik)$ and $\omega(\ik-1)$ lies on 
some hyperplane $h_{\alpha,me}$.  
Removing the hyperplane $h_{\alpha,me}$ leaves two distinct subsets
$E_r^+(\alpha,me)$ and $E_r^-(\alpha,me)$ where $\astrosun \in  E_r^-(\alpha,me)$.  If $\omega(\ik-1)\in E_r^-(\alpha,me)$, or $\omega(\ik)\in E_r^+(\alpha,me)$, then set $d_\alpha(\omega,\ik)=0$. 
 If  $\omega(\ik-1)\in E_r^+(\alpha,me)$, then $d_\alpha(\omega,\ik)=-1$.
 If  $\omega(\ik)\in E_r^-(\alpha,me)$, then $d_\alpha(\omega,\ik)=+1$. 
 \end{defn}
 
 Figure \ref{hyper} illustrates the four subcases outlined above.  
 In each case the diagram depicts a hyperplane, labelled by $h_{\alpha,me}$, with the corresponding subsets  $ E_r^+(\alpha,me)$ and $ E_r^-(\alpha,me)$ labelled.  The incoming/outgoing arrows labels steps onto and off of the hyperplane and the corresponding $d_{\alpha}(\omega,\ik)$.
 
 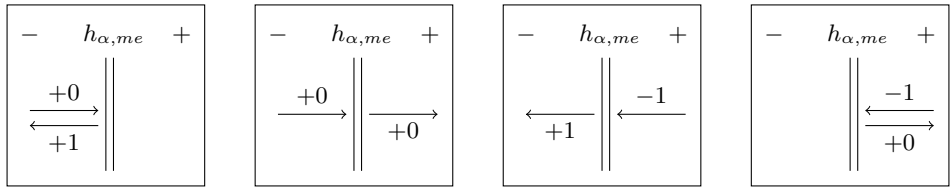
\begin{figure}[ht]\captionsetup{width=0.9\textwidth}
 \[ 
 \scalefont{0.8}  
   \begin{tikzpicture}         
        \draw (-1.3,-0.2) rectangle ++(2.6,2.4); 
\draw(0,0)--(0,1.5);\draw(0.1,0)--(0.1,1.5);
\draw(0.05,1.8) node { \scalefont{1}  $h_{\alpha,me}$};
\draw(-1,1.8) node {$-$};\draw(1,1.8) node {$+$};
\draw[->](-1,.8) to node[auto] {\(+0\)} (-0.1,.8);
\draw[<-](-1,.6) to node[below] {\(+1\)} (-0.1,.6);
       \end{tikzpicture} 
       \quad\quad
    \begin{tikzpicture}         
        \draw (-1.3,-0.2) rectangle ++(2.6,2.4); 
\draw(0,0)--(0,1.5);\draw(0.1,0)--(0.1,1.5);
\draw(0.05,1.8) node { \scalefont{1}  $h_{\alpha,me}$};
\draw(-1,1.8) node {$-$};\draw(1,1.8) node {$+$};
\draw[->](-1,.75) to node[auto] {\(+0\)} (-0.1,.75);
\draw[<-](1.1,.75) to node[auto] {\(+0\)} (0.2,.75);
       \end{tikzpicture} 
\quad\quad
    \begin{tikzpicture}         
        \draw (-1.3,-0.2) rectangle ++(2.6,2.4); 
\draw(0,0)--(0,1.5);\draw(0.1,0)--(0.1,1.5);
\draw(0.05,1.8) node { \scalefont{1}  $h_{\alpha,me}$};
\draw(-1,1.8) node {$-$};\draw(1,1.8) node {$+$};
\draw[<-](-1,.75) to node[below] {\(+1\)} (-0.1,.75);
\draw[->](1.1,.75) to node[above] {\(-1\)} (0.2,.75);
       \end{tikzpicture} 
       \quad\quad
        \begin{tikzpicture}         
        \draw (-1.3,-0.2) rectangle ++(2.6,2.4); 
\draw(0,0)--(0,1.5);\draw(0.1,0)--(0.1,1.5);
\draw(0.05,1.8) node { \scalefont{1}  $h_{\alpha,me}$};
\draw(-1,1.8) node {$-$};\draw(1,1.8) node {$+$};
 \draw[->](1.1,.8) to node[above] {\(-1\)} (0.2,.8);
\draw[<-](1.1,.6) to node[below] {\(+0\)} (0.2,.6);
       \end{tikzpicture} 
\]
\caption{The four subcases for the values of $d_{\alpha}(\omega,\ik)$ as $\omega$ crosses a wall.
The $\pm$ indicate the distinct subsets $E_r^{+}$ and $E_r^{-}$ of $E_r$. In each case the first (respectively second) step has its degree recorded as a superscript (respectively subscript).}
\label{hyper}
\end{figure}

 Let $\omega $ be a path which
  passes through a hyperplane $h_{\alpha,me}$   at point $\omega(\ik)$ (note that $\ik$ is not necessarily unique).  Then,
  let  $\omega'$  be the path 
 obtained from $\omega$  by applying the reflection 
 $s_{\alpha,me}$ to   all the steps in     $\omega $ after the point $\omega(\ik)$.
In other words,    $\omega'(i)=\omega(i)$ for all $1\leq i \leq \ik$ and
  $\omega(i')=s_{\alpha,me} \cdot \omega(i)$ for $\ik \leq i \leq n$. 
  We refer to the path $\omega'$ as the  reflection of 
 $\omega$ in $h_{\alpha,me}$ at point $\omega(\ik)$ and denote this by $s_{\alpha,me}^{\ik}\cdot \omega$.  
  We write $\omega \sim \omega'$ if the path $\omega$ can be obtained from $\omega'$ by a series of reflections in $W^e$.

 Let  $\lambda,\mu  \in E_r$.  We fix a distinguished path $\omega^\mu$ 
  from the origin to $\mu$ such that $\deg(\omega(\ik))=0$ for all $1\leq i \leq n$.
   (It is easy to see that such a path always exists.)
 We let $\Path(\lambda,\mu)$ denote the set  of all  
 paths 
 from the origin to $\lambda$ which may be obtained from $\omega^\mu$ by a series of reflections.

\begin{eg}
Recall the example from the introduction.  Here the geometry is of type $\hat{A}_2$, $n=13$, $e=8$ and $\rho=(8,4,2)$.  

The distinguished path $\omega^\gamma$ is recorded in Figure \ref{first path}.  We clearly have that $d_{\alpha}(\omega,\ik)=0$ at all points $1\leq \ik \leq n$.  
Figure \ref{weight a} contains the two elements of $\Path((4,6,3),\gamma)$. 
Let $\omega$ (respectively  $\omega'$) denote the path in the leftmost (respectively rightmost) case.  We have that
\[d_{\varepsilon_2-\varepsilon_3}(\omega,11)=1, \quad
 d_{\varepsilon_1-\varepsilon_3}(\omega,12)=-1, \quad
 d_{\varepsilon_1-\varepsilon_3}(\omega,13)=1\]  
are the only non-zero values of $d_{\alpha}(\omega,\ik)$ for $1\leq \ik \leq n$,  and therefore $\deg(\omega)=1$.  
We have that
\[d_{\varepsilon_2-\varepsilon_3}(\omega',5)=1, \quad
 d_{\varepsilon_1-\varepsilon_3}(\omega',12)=1, \quad
 d_{\varepsilon_1-\varepsilon_3}(\omega',13)=1\]
 are the only non-zero values of $d_{\alpha}(\omega',\ik)$ for $1\leq \ik \leq n$,  and therefore $\deg(\omega')=3$.  
\end{eg}

   \subsection{Soergel's algorithm for paths}\label{pathalgorithm}
 
Fix $e>h$, we now recall the classical construction of Soergel's algorithm with respect to a walk in the geometry.  The procedure outlined below is somewhat simpler, as all points in our geometry belong to the   \emph{dominant chamber} \cite[Section 4]{Soergel}.  


\begin{defn}\label{admissibledefn}
Let $e>h$, and assume $\mu$ is $e$-regular.  We say that a   path   $\omega$ from $\astrosun$ to $\mu$ of length $n$ is \emph{admissible} if 
 $(i)$ 
  $\deg(\omega(\ik))=0$ for all $1\leq \ik \leq n$, and $(ii)$   whenever 
$\omega (\ik)$ lies on  two hyperplanes $h_{\alpha,m_1e}$ and $h_{\beta,m_2e}$ for some $1\leq \ik\leq n$ this implies that  $\langle \alpha, \beta^\vee\rangle = 0$ (we say that the hyperplanes are orthogonal).
\end{defn}

\begin{rmk}
For $\mu$ an $e$-regular point,  and $\omega$ an admissible path from $\astrosun$ to $\mu$, there exist $2^{\ell(\mu)}$ paths $\omega'$ such that $\omega'\sim \omega$.  
\end{rmk}

 We say that a path, 
  $\omega $, is an \emph{alcove-wall path} if 
$(i)$   $\deg(\omega(\ik))=0$ for all $1\leq \ik \leq n$
and $(ii)$  every 
  step lies either on a wall or in an alcove. 
 It is clear that any alcove-wall path is admissible.

\begin{defn}   For a distinguished admissible path $\omega^\mu$, we define 
  \[ \overline{m}_\mu(\lambda) = \;
  \sum_{
\mathclap{\omega \in \Path(\lambda,\mu)}
  } \; t^{\deg(\omega)}.\]
\end{defn}

 Given $\omega $  an admissible path of length $n$, 
 we let $\mathfrak{f}_k$ denote the facet containing the point $\omega (\ik)$ for $1\leq \ik \leq n$.

\begin{defn}\label{alcovseries}  
We fix an admissible path $\omega$  from $\astrosun$ to $\mu$ of length $n$.
  For $1\leq k \leq n$, we let 
\begin{align*}
A_+(\omega,k) =&
  \{  (\gamma , m_ke)  \mid \omega(k)  \in h_{\gamma,m_{k }e}  \} \setminus \{ (\gamma ,  m_{k-1} e) \mid \omega(k-1)  \in  h_{\gamma,m_{k -1}e}  \},  \\ A_{-}(\omega,k) =&
  \{ (\gamma ,  m_{k-1} e)  \mid \omega(k-1)  \in h_{\gamma,m'_{k-1 }e} \} \setminus
   \{( \gamma,    m_{k} e) \mid \omega(k)  \in  h_{\gamma,m_{k }'e} \}.
   \end{align*}
   The orthogonality condition on the admissible path ensures
 that for $1\leq k \leq n$, the set  
$A_+(\omega,k)$ (respectively $A_{-}(\omega,k)$) either consists of one element, 
denoted
 $\alpha_+(\omega,k)$ (respectively $\alpha_{-}(\omega,k)$) or is empty.  
\end{defn}

\begin{rmk}
The $\alpha_+(\omega,k) $ (respectively $\alpha_{-}(\omega,k) $) record the steps in $\omega$ which are on to (respectively off of) hyperplanes in the geometry.  
\end{rmk}

\begin{defn}
We fix an admissible path $\omega$  from $\astrosun$ to $\mu$ of length $n$.
For $1\leq k \leq n$, we set $A_k$ to be the  
  alcove, minimal in the length ordering, such that 
  $\langle \lambda+\rho , \alpha_+(\omega,i)^\vee \rangle >0 $ for all $\lambda \in A_+(\omega,k)$ and all $0\leq i\leq k$ such that $A_+(\omega,k)\not= \emptyset$.   
We define  the \emph{alcove-series} of $\omega$ to be the ordered set whose elements are given by  the  alcoves   $A_k$ for $0\leq k \leq n$ recorded without repeats and in increasing order. 
\end{defn}



\begin{eg}\label{omegagamma}
Consider a geometry of type $\hat{A}_{2}$ with $\rho=(8,4,2)$ and $n=13$.  
The path $\omega^\gamma$ in Figure \ref{first path} is an alcove-wall path.
We let  $\overline{\omega}^\gamma$  denote the alcove-wall path 
$$
(
\varepsilon_1,\varepsilon_2,
\varepsilon_1,\varepsilon_2,
\varepsilon_1,\varepsilon_2,
\varepsilon_2,
\varepsilon_2,\varepsilon_2,
\varepsilon_2,\varepsilon_2,\varepsilon_2,
\varepsilon_1
).
$$
Both paths pass through (the same)  alcoves of length $0, 1, 2, 3$, which we denote by $\mathfrak{a}(i)$   for $i=0, 1, 2, 3$.  We have that 
$$A^{\omega^\gamma}_{ \ik}=\begin{cases}
 \{ \mathfrak{a}(0) \} & \text{for $\ik=0,1,2$, } \\
 \{    \mathfrak{a}(1)\}  & \text{for $\ik=3,4,5,6,7,8,9$, }\\
  \{   \mathfrak{a}(2)\}  & \text{for $\ik=10,11$, }\\
  \{   \mathfrak{a}(3)\}  & \text{for $\ik=12,13$; } 
 \end{cases}
 \quad
 A^{\overline{\omega}^\gamma}_{ \ik}=\begin{cases}
 \{ \mathfrak{a}(0) \} & \text{for $\ik=0,1,2$, } \\
 \{    \mathfrak{a}(1)\}  & \text{for $\ik= 3,4,5,6,7,8$, }\\
 \{   \mathfrak{a}(2)\}  & \text{for $\ik=9$, }\\
  \{   \mathfrak{a}(3)\}  & \text{for $\ik=10,11,12,13$; } 
 \end{cases}$$
and so the alcove series in both cases is given by $\{  \mathfrak{a}(0), \mathfrak{a}(1),\mathfrak{a}(2),\mathfrak{a}(3)\}$.  
\end{eg}

 We let $ \mathfrak{A}$ denote the set of all alcoves in $E_r$.
We let   $\mathfrak{b}, \mathfrak{c},\mathfrak{d}$ denote   alcoves in our geometry and let $\mathfrak{a}_0,\ldots,\mathfrak{a}_{\ell(\mu)}$ denote the alcove series of an admissible path from $\astrosun$ to $\mu$. 
 We define   maps
$$n_{\mathfrak{a}_i}  : \mathfrak{A} \to \NN_0[t] \quad m_{\mathfrak{a}_i}   :  \mathfrak{A} \to  \NN_0[t] \quad e_{\mathfrak{a}_i}   :  \mathfrak{A} \to \NN_0[t+t^{-1}] ,$$
where $t$ is a formal parameter, as follows. 
  We set
\[n_{\mathfrak{a}_i}  ({\mathfrak{a}_i} ) = 1, \quad
 m_{\mathfrak{a}_i} ({\mathfrak{a}_i} ) = 1, \quad
 e_{\mathfrak{a}_i}  ({\mathfrak{a}_i} ) = 1.\]
 We define 
  \[n_{\mathfrak{a}_{i}}(\mathfrak{b})= 0, \quad
  m_{\mathfrak{a}_{i}}(\mathfrak{b}) = 0, \quad
  e_{\mathfrak{a}_{i}}(\mathfrak{b}) = 0\] 
whenever  $\ell(\mathfrak{b}) \not\leq \ell(\mathfrak{a}_i)$.
     For each adjacent  pair of alcoves $\mathfrak{a}_i$ and  $\mathfrak{a}_{i+1}$,
     we let $s_i$ denote the reflection in the hyperplane passing through
      $\overline{\mathfrak{a}}_i \cap \overline{\mathfrak{a}}_{i+1}$.  
The closure, $\overline{\mathfrak{b}}$, of any alcove $\mathfrak{b}$ has one wall which is in the $W^e$-orbit of $s_i$, and we shall write $s_i\cdot \mathfrak{b}$ for the image of $\mathfrak{b}$ in that
wall.
   Then, with ${m}_{\mathfrak{a}_i} $ known, we set
\begin{equation}\label{oohlalala}
m_{{\mathfrak{a}_{i+1}} }(s_i \cdot \mathfrak{b})
=
\begin{cases}
 m_{\mathfrak{a}_i} (\mathfrak{b}) + t^{-1} m_{\mathfrak{a}_i}(s_i \cdot \mathfrak{b}), &  \ell(s_i \cdot \mathfrak{b}) > \ell(\mathfrak{b}), \\
  m_{\mathfrak{a}_i}(\mathfrak{b}) +  t m_{\mathfrak{a}_i}(s_i \cdot \mathfrak{b}), &  \ell(s_i \cdot \mathfrak{b}) < \ell(\mathfrak{b}).
\end{cases}
\end{equation}
We refer to this procedure as the \emph{cancellation-free Soergel algorithm}.

 \begin{prop}\label{the same}
Given $e>h$, suppose that $\mu$ and $\lambda$ belong to 
alcoves $ \mathfrak{a}$ and $\mathfrak{b}$ respectively, and furthermore that   $\mu \in W^e \cdot \lambda$. 
We let $\mathfrak{a}_0$ denote the fundamental alcove and $\mathfrak{a}_0,\ldots,\mathfrak{a}_{\ell(\mu)}=\mathfrak{a}$ denote the alcove series of
an admissible path $\omega^\mu$.  
  We have that $\overline{m}_{\mu}(\lambda)=m_{\mathfrak{a}}(\mathfrak{b})$.  
 \end{prop}

\begin{proof}  
 For $1\leq i \leq \ell(\mu)$, note that 
 the $i$th  hyperplane $\overline{\mathfrak{a}}_i
  \cap \overline{  \mathfrak{a}}_{i+1}$
 is the hyperplane given by the $i$th non-trivial
  $ {\alpha_+(m,k)}$.  
This gives the required bijection between paths (obtained from $\omega^\mu$ by a series of reflections 
through the $h_{\alpha_+(m,k)}$ for $1\leq k \leq n$) and terms in Soergel's cancellation-free algorithm (given by a sequence of alcoves, which are determined by the alcove walls $\overline{\mathfrak{a}}_i
  \cap \overline{  \mathfrak{a}}_{i+1}$ through which we reflect).  

The   $\alpha_+(\omega^\mu,k)$ and  $\alpha_{-}(\omega^\mu,k')$
for $1\leq k < k' \leq n$ come in pairs (whenever we step on to a   hyperplane,
 we  must step off of it at some later point). 
 For a pair $1\leq k <k' \leq n$, the hyperplanes 
$h_{\alpha_+(\omega^\mu,k'')}$ and 
 $h_{\alpha_{-}(\omega^\mu,k'')}$ for $k < k'' <k'$
  are orthogonal  to 
 $h_{\alpha_+(m,k)}$.

Fix two points $\lambda, \lambda+\varepsilon_i \in E_r$ and suppose that
$\lambda \in h_{\alpha,me}$.   Assume that  $\lambda+\varepsilon_i $ belongs to   $E_r^+(\alpha,me)$ or $E_r^-(\alpha,me)$.  
Let $h_{\beta,m'e}$ denote a hyperplane orthogonal to 
$h_{\alpha,me}$ and
  $s_{\beta,m'e}$ denote the reflection through this hyperplane.  
It is clear that  $s_{\beta,me} \cdot (\lambda+\varepsilon_i)  $ still belongs to either $E_r^+(\alpha,me)$ or $E_r^-(\alpha,me)$, respectively.  
 (Compare this with the definition of the degree of a path, Definition  \ref{Soergeldegree}.)
Note that in general, this would not be true for non-orthogonal hyperplanes.

 Therefore the contribution $d_{\alpha_+(m,k)}(\omega, k')$
  to the  degree given by the step 
 at point $k'-1$, is the same as if it were taken   at point     $ k+1$.  
Thus we can assume that $k'=k+1$, in other words that our path is an alcove wall path.  
 Folding up an alcove wall path, $\omega^\mu$, so that it terminates at $\lambda$ corresponds to tracing one of the terms in the Soergel cancellation-free algorithm, as follows:
    \begin{itemize}
\item[$(i)$]
   When the path steps from  alcove $\mathfrak{b}$ onto the wall $\overline{\mathfrak{b}}\cap \overline{s_i\cdot \mathfrak{b}}$
    and through to the alcove $s_i\cdot \mathfrak{b}$, the degree of the path does not change on alcoves (as $-1+1=0$), as illustrated in the second and third diagrams in Figure \ref{hyper}.  This is equivalent to the first term in each of the two cases of equation \ref{oohlalala}.
    \item [$(ii)$]
 When the path  steps from  alcove $s_i\cdot \mathfrak{b}$ onto the wall $\overline{\mathfrak{b}}\cap \overline{s_i\cdot \mathfrak{b}}$ and then returns to the alcove $s_i\cdot \mathfrak{b}$, the degree either increases or decreases by one, as seen  in the first and fourth  diagrams in Figure \ref{hyper}, respectively.
   This is equivalent to 
  the second term  in the two cases of equation \ref{oohlalala}. 
    \end{itemize}
For ease in the above, we have
tacitly assumed that we never simultaneously step off of a hyperplane and on to  another hyperplane in the same step  
 (as in $\omega^\gamma$ in Example \ref{omegagamma}). 
  In general,  
  this is not the case (as in $\overline{\omega}^\gamma$ in Example \ref{omegagamma}).
 Our ignoring of this   is justified as the Soergel-degree is given by 
     summing over the Soergel-degrees of the 
  steps from passing through these separate facets 
  (note that in Definition \ref{Soergeldegree}, 
  the contributions of the  $d_{\alpha}$ for $\alpha \in R^+$ to the sum are independent).  
\end{proof}

 \begin{rmk}
 Motivated by  the above Proposition, we will omit the bar for 
  $\overline{m}_\mu(\lambda)$
for $\mu,\lambda\in E_r$.    
\end{rmk}

 Similarly, with $n_\mathfrak{a_i} $  known by induction, we set 
\[
n'_{{\mathfrak{a}_{i+1}}}(s_i \cdot \mathfrak{b})
=
\begin{cases}
 n_{\mathfrak{a}_i}(\mathfrak{b}) + t^{-1} n_{\mathfrak{a}_i}(s_i \cdot \mathfrak{b}), &  \ell(s_i \cdot \mathfrak{b}) > \ell(\mathfrak{b}); \\
  n_{\mathfrak{a}_i}(\mathfrak{b}) +   t  n_{\mathfrak{a}_i}(s_i \cdot \mathfrak{b}), &  \ell(s_i \cdot \mathfrak{b}) < \ell(\mathfrak{b}); 
\end{cases}
\]
and 
\[
n_{{\mathfrak{a}_{i+1}}}(\mathfrak{b})=n'_{{\mathfrak{a}_{i+1}}}(\mathfrak{b}) -\; \;
  \sum_{\mathclap{\{\mathfrak{d}\mid \ell(\mathfrak{d}) < 
\ell({\mathfrak{a}_{i+1}})\}}} \; \;
(n'_{ {\mathfrak{a}_{i+1}}}(\mathfrak{d})|_{t=0})n_\mathfrak{d}(\mathfrak{b}).
\]
We refer to this procedure as the \emph{Soergel algorithm}.  Importantly for us,
 it is shown in \cite{Soergel} that this procedure is  independent of the path taken.  
Finally, with $e_{\mathfrak{a}_{i}}$ known by induction, we set 
\begin{align*}
e_{{\mathfrak{a}_{i+1}}}( s_i \cdot \mathfrak{c})
= 
  (t+t^{-1}) e_{ {\mathfrak{a}_i}}( s_i \cdot  \mathfrak{c})	+ 	
 	 e_{{\mathfrak{a}_i}}(   \mathfrak{c})
	 +	 (n'_{{\mathfrak{a}_{i+1}}}(s_i \cdot \mathfrak{c})|_{t=0})
\end{align*}
if   $\ell(s_i \cdot \mathfrak{c}) > \ell(\mathfrak{c})$, 
and $e_{{\mathfrak{a}_{i+1}}}( s_i \cdot \mathfrak{c}) = 0$ otherwise.
 We refer to this procedure   as the \emph{character algorithm}.  
 We extend the   $e$ and $n$ functions to $e$-regular points in a given linkage class in the obvious fashion.    
  
  \begin{eg}
    Let $e=4$ and $\rho = (4,2)$ and consider the root system of type $\hat{A}_1$. The space $E_1$   can be pictured as  $\ZZ$; however, in order to make the steps $+\varepsilon_1$ and $+\varepsilon_2$ in a walk of length $n$ clear,  we draw a graph with $n$ levels, the $i$th level featuring the points $\{-i, -(i-1), \ldots, (i-1), i\}$ and
let   $\omega$ start at level $0$ at point $\astrosun$ and proceed downwards to level $n$, this is made clear in Figure \ref{LEVEL2EXAMPLE1}.  
As pointed out in \cite{PR13,Pla13}, these can be regarded as walks on Pascal's triangle.

    Let $\mu=(0,11)$ and $\lambda=(4,7)$.  
There are two elements $\omega,\omega'\in {\rm Path}_{11}(\lambda,\mu)$, depicted in Figure \ref{LEVEL2EXAMPLE1}.  The former is of degree 2 and the latter of degree 0. 
In the former case, $d_{\varepsilon_1-\varepsilon_2}(\omega,3)=1$, $d_{\varepsilon_1-\varepsilon_2}(\omega,7)=1$.
In the latter case, $d_{\varepsilon_1-\varepsilon_2}(\omega',7)=1$
$d_{\varepsilon_1-\varepsilon_2}(\omega',10)=-1$.   
\begin{figure}[ht]\captionsetup{width=0.9\textwidth}
$$ \scalefont{0.8} \begin{tikzpicture}[scale=0.6]
{ \clip (-4,-4.3) rectangle ++(8,5);
\path [draw,name path=upward line] (-3,-5) -- (3,-5);
  \path 
      (0, 0)  ++ (90:0.5 cm)coordinate (XX);
        \path 
      (XX)  ++ (0:0.1 cm)coordinate (XX1);
              \path 
      (XX)  ++ (180:0.1 cm)coordinate (XX2);
      \draw[->] (XX2) to ++(180:0.7 cm);
  \scalefont{0.8}      \draw  (XX2)  ++(180:1.5 cm) node {$+\varepsilon_2$};
      \draw[->] (XX1) to ++(0:0.7 cm);
        \draw  (XX)  ++(0:1.5 cm) node {$+\varepsilon_1$};
  \path 
      (0, 0)              coordinate (a1)
        (a1)    -- +(-45:0.5) coordinate (a2) 
        (a1)    -- +(-135:0.5) coordinate (a3) 
        (a2)    -- +(-45:0.5) coordinate (b1) 
        (a2)    -- +(-135:0.5) coordinate (b2) 
         (a3)    -- +(-135:0.5) coordinate (b3) 
  (b1)    -- +(-45:0.5) coordinate (c1) 
        (b1)    -- +(-135:0.5) coordinate (c2) 
  (b3)    -- +(-45:0.5) coordinate (c3) 
        (b3)    -- +(-135:0.5) coordinate (c4) 
  (c1)    -- +(-45:0.5) coordinate (d1) 
        (c1)    -- +(-135:0.5) coordinate (d2) 
  (c3)    -- +(-45:0.5) coordinate (d3) 
        (c3)    -- +(-135:0.5) coordinate (d4) 
                (c4)    -- +(-135:0.5) coordinate (d5) 
  (d1)    -- +(-45:0.5) coordinate (e1) 
        (d1)    -- +(-135:0.5) coordinate (e2) 
  (d3)    -- +(-45:0.5) coordinate (e3) 
        (d3)    -- +(-135:0.5) coordinate (e4) 
                (d5)    -- +(-135:0.5) coordinate (e6) 
                                (d5)    -- +(-45:0.5) coordinate (e5) 
  (e1)    -- +(-45:0.5) coordinate (f1) 
        (e1)    -- +(-135:0.5) coordinate (f2) 
  (e3)    -- +(-45:0.5) coordinate (f3) 
        (e3)    -- +(-135:0.5) coordinate (f4) 
                (e5)    -- +(-45:0.5) coordinate (f5) 
                (e5)    -- +(-135:0.5) coordinate (f6) 
                                (e6)    -- +(-135:0.5) coordinate (f7) 
  (f1)    -- +(-45:0.5) coordinate (g1) 
        (f1)    -- +(-135:0.5) coordinate (g2) 
  (f3)    -- +(-45:0.5) coordinate (g3) 
        (f3)    -- +(-135:0.5) coordinate (g4) 
                (f5)    -- +(-45:0.5) coordinate (g5) 
                (f5)    -- +(-135:0.5) coordinate (g6) 
                                (f6)    -- +(-135:0.5) coordinate (g7)          
                                                                (f7)    -- +(-135:0.5) coordinate (g8)        
          (g1)    -- +(-45:0.5) coordinate (h1) 
        (g1)    -- +(-135:0.5) coordinate (h2) 
  (g3)    -- +(-45:0.5) coordinate (h3) 
        (g3)    -- +(-135:0.5) coordinate (h4) 
                (g5)    -- +(-45:0.5) coordinate (h5) 
                (g5)    -- +(-135:0.5) coordinate (h6) 
                                (g6)    -- +(-135:0.5) coordinate (h7)          
                                                                (g7)    -- +(-135:0.5) coordinate (h8)           
                                                                       (h1)    -- +(-45:0.5) coordinate (i1) 
        (h1)    -- +(-135:0.5) coordinate (i2) 
  (h3)    -- +(-45:0.5) coordinate (i3) 
        (h3)    -- +(-135:0.5) coordinate (i4) 
                (h5)    -- +(-45:0.5) coordinate (i5) 
                (h5)    -- +(-135:0.5) coordinate (i6) 
                                (h6)    -- +(-135:0.5) coordinate (i7)          
                                                                (h7)    -- +(-135:0.5) coordinate (i8)           
                            (i1)    -- +(-45:0.5) coordinate (j1) 
        (i1)    -- +(-135:0.5) coordinate (j2) 
  (i3)    -- +(-45:0.5) coordinate (j3) 
        (i3)    -- +(-135:0.5) coordinate (j4) 
                (i5)    -- +(-45:0.5) coordinate (j5) 
                (i5)    -- +(-135:0.5) coordinate (j6) 
                                (i6)    -- +(-135:0.5) coordinate (j7)          
                                                                (i7)    -- +(-135:0.5) coordinate (j8)      
         (j1)    -- +(-45:0.5) coordinate (k1) 
        (j1)    -- +(-135:0.5) coordinate (k2) 
  (j3)    -- +(-45:0.5) coordinate (k3) 
        (j3)    -- +(-135:0.5) coordinate (k4) 
                (j5)    -- +(-45:0.5) coordinate (k5) 
                (j5)    -- +(-135:0.5) coordinate (k6) 
                                (j6)    -- +(-135:0.5) coordinate (k7)          
                                                                (j7)    -- +(-135:0.5) coordinate (k8)                                      
 (g8)    -- +(-135:0.5) coordinate (h9)
   (h9)    -- +(-45:0.5) coordinate (i9)  
   (h9)    -- +(-135:0.5) coordinate (i10)
   (i9)    -- +(-45:0.5) coordinate (j9)  
   (i9)    -- +(-135:0.5) coordinate (j10)
   (i10)    -- +(-135:0.5) coordinate (j11)   
   (j9)    -- +(-45:0.5) coordinate (k9)  
   (j9)    -- +(-135:0.5) coordinate (k10)
   (j10)    -- +(-135:0.5) coordinate (k11)  
      (j11)    -- +(-135:0.5) coordinate (k12)  
   ;
            \fill(a1) circle (1pt);             \fill(a2) circle (1pt); \fill(a3) circle (1pt); 
              \fill(b1) circle (1pt);                 \fill(b2) circle (1pt);                          \fill(b3) circle (1pt);   
               \fill(c1) circle (1pt);                 \fill(c2) circle (1pt);                          \fill(c3) circle (1pt);   \fill(c4) circle (1pt);   
                              \fill(d1) circle (1pt);                 \fill(d2) circle (1pt);                          \fill(d3) circle (1pt);   \fill(d4) circle (1pt);   \fill(d5) circle (1pt);   
                                    \fill(e1) circle (1pt);                 \fill(e2) circle (1pt);                          \fill(e3) circle (1pt);   \fill(e4) circle (1pt);   
\fill(e5) circle (1pt);     \fill(e6) circle (1pt);   
                                                                        \fill(f1) circle (1pt);                 \fill(f2) circle (1pt);                          \fill(f3) circle (1pt);   \fill(f4) circle (1pt);   \fill(f5) circle (1pt);     \fill(f6) circle (1pt);    \fill(f7) circle (1pt);   
    \fill(g1) circle (1pt);                 \fill(g2) circle (1pt);                          \fill(g3) circle (1pt);   \fill(g4) circle (1pt);   \fill(g5) circle (1pt);     \fill(g6) circle (1pt);    \fill(g7) circle (1pt);   
 \fill(g8) circle (1pt);   
  \fill(h1) circle (1pt);                 \fill(h2) circle (1pt);                          \fill(h3) circle (1pt);   \fill(h4) circle (1pt);   \fill(h5) circle (1pt);     \fill(h6) circle (1pt);    \fill(h7) circle (1pt);   
 \fill(h8) circle (1pt);   
  \fill(j1) circle (1pt);                 \fill(j2) circle (1pt);                          \fill(j3) circle (1pt);   \fill(j4) circle (1pt);   \fill(j5) circle (1pt);     \fill(j6) circle (1pt);    \fill(j7) circle (1pt);   
 \fill(j8) circle (1pt);   
  \fill(i1) circle (1pt);                 \fill(i2) circle (1pt);                          \fill(i3) circle (1pt);   \fill(i4) circle (1pt);   \fill(i5) circle (1pt);     \fill(i6) circle (1pt);    \fill(i7) circle (1pt);   
 \fill(i8) circle (1pt);  
   \fill(k1) circle (1pt);                 \fill(k2) circle (1pt);                          \fill(k3) circle (1pt);   \fill(k4) circle (1pt);   \fill(k5) circle (1pt);     \fill(k6) circle (1pt);    \fill(k7) circle (1pt);   
 \fill(k8) circle (1pt);   
        \draw[dotted](b3) -- +(270:10);       \draw[dotted](b3) -- +(90:1);
        \draw[dotted](b1) -- +(270:10);       \draw[dotted](b1) -- +(90:1);
        \draw[dotted](f7) -- +(270:10);        \draw[dotted](f7) -- +(90:1);
         \draw[dotted](f1) -- +(270:10);        \draw[dotted](f1) -- +(90:1);
 \draw[dotted](j1) -- +(270:10);        \draw[dotted](j1) -- +(90:1);
 \draw[dotted](j11) -- +(270:10);        \draw[dotted](j11) -- +(90:1);       ; 
          \fill(i10) circle (1pt);     \fill(i9) circle (1pt);   \fill(h9) circle (1pt);  
             \fill(j9) circle (1pt);   \fill(j10) circle (1pt);  \fill(j11) circle (1pt);  
             \fill(k9) circle (1pt);               \fill(k10) circle (1pt);      \fill(k11) circle (1pt);               \fill(k12) circle (1pt);  
\draw(a1) --  (b3);   \draw(j7) --  (f3); \draw(j7) --  (k8);   \draw(f3) --  (b3);          }
    \end{tikzpicture}
\quad \quad\quad 
 \scalefont{0.8} \begin{tikzpicture}[scale=0.6]
{ \clip (-4,-4.3) rectangle ++(8,5);
\path [draw,name path=upward line] (-3,-5) -- (3,-5);
  \path 
      (0, 0)  ++ (90:0.5 cm)coordinate (XX);
        \path 
      (XX)  ++ (0:0.1 cm)coordinate (XX1);
              \path 
      (XX)  ++ (180:0.1 cm)coordinate (XX2);
      \draw[->] (XX2) to ++(180:0.7 cm);
        \draw  (XX2)  ++(180:1.5 cm) node {$+\varepsilon_2$};
      \draw[->] (XX1) to ++(0:0.7 cm);
        \draw  (XX)  ++(0:1.5 cm) node {$+\varepsilon_1$};
  \path 
      (0, 0)              coordinate (a1)
        (a1)    -- +(-45:0.5) coordinate (a2) 
        (a1)    -- +(-135:0.5) coordinate (a3) 
        (a2)    -- +(-45:0.5) coordinate (b1) 
        (a2)    -- +(-135:0.5) coordinate (b2) 
         (a3)    -- +(-135:0.5) coordinate (b3) 
  (b1)    -- +(-45:0.5) coordinate (c1) 
        (b1)    -- +(-135:0.5) coordinate (c2) 
  (b3)    -- +(-45:0.5) coordinate (c3) 
        (b3)    -- +(-135:0.5) coordinate (c4) 
  (c1)    -- +(-45:0.5) coordinate (d1) 
        (c1)    -- +(-135:0.5) coordinate (d2) 
  (c3)    -- +(-45:0.5) coordinate (d3) 
        (c3)    -- +(-135:0.5) coordinate (d4) 
                (c4)    -- +(-135:0.5) coordinate (d5) 
  (d1)    -- +(-45:0.5) coordinate (e1) 
        (d1)    -- +(-135:0.5) coordinate (e2) 
  (d3)    -- +(-45:0.5) coordinate (e3) 
        (d3)    -- +(-135:0.5) coordinate (e4) 
                (d5)    -- +(-135:0.5) coordinate (e6) 
                                (d5)    -- +(-45:0.5) coordinate (e5) 
  (e1)    -- +(-45:0.5) coordinate (f1) 
        (e1)    -- +(-135:0.5) coordinate (f2) 
  (e3)    -- +(-45:0.5) coordinate (f3) 
        (e3)    -- +(-135:0.5) coordinate (f4) 
                (e5)    -- +(-45:0.5) coordinate (f5) 
                (e5)    -- +(-135:0.5) coordinate (f6) 
                                (e6)    -- +(-135:0.5) coordinate (f7) 
  (f1)    -- +(-45:0.5) coordinate (g1) 
        (f1)    -- +(-135:0.5) coordinate (g2) 
  (f3)    -- +(-45:0.5) coordinate (g3) 
        (f3)    -- +(-135:0.5) coordinate (g4) 
                (f5)    -- +(-45:0.5) coordinate (g5) 
                (f5)    -- +(-135:0.5) coordinate (g6) 
                                (f6)    -- +(-135:0.5) coordinate (g7)          
                                                                (f7)    -- +(-135:0.5) coordinate (g8)        
          (g1)    -- +(-45:0.5) coordinate (h1) 
        (g1)    -- +(-135:0.5) coordinate (h2) 
  (g3)    -- +(-45:0.5) coordinate (h3) 
        (g3)    -- +(-135:0.5) coordinate (h4) 
                (g5)    -- +(-45:0.5) coordinate (h5) 
                (g5)    -- +(-135:0.5) coordinate (h6) 
                                (g6)    -- +(-135:0.5) coordinate (h7)          
                                                                (g7)    -- +(-135:0.5) coordinate (h8)           
                                                                       (h1)    -- +(-45:0.5) coordinate (i1) 
        (h1)    -- +(-135:0.5) coordinate (i2) 
  (h3)    -- +(-45:0.5) coordinate (i3) 
        (h3)    -- +(-135:0.5) coordinate (i4) 
                (h5)    -- +(-45:0.5) coordinate (i5) 
                (h5)    -- +(-135:0.5) coordinate (i6) 
                                (h6)    -- +(-135:0.5) coordinate (i7)          
                                                                (h7)    -- +(-135:0.5) coordinate (i8)           
                            (i1)    -- +(-45:0.5) coordinate (j1) 
        (i1)    -- +(-135:0.5) coordinate (j2) 
  (i3)    -- +(-45:0.5) coordinate (j3) 
        (i3)    -- +(-135:0.5) coordinate (j4) 
                (i5)    -- +(-45:0.5) coordinate (j5) 
                (i5)    -- +(-135:0.5) coordinate (j6) 
                                (i6)    -- +(-135:0.5) coordinate (j7)          
                                                                (i7)    -- +(-135:0.5) coordinate (j8)      
         (j1)    -- +(-45:0.5) coordinate (k1) 
        (j1)    -- +(-135:0.5) coordinate (k2) 
  (j3)    -- +(-45:0.5) coordinate (k3) 
        (j3)    -- +(-135:0.5) coordinate (k4) 
                (j5)    -- +(-45:0.5) coordinate (k5) 
                (j5)    -- +(-135:0.5) coordinate (k6) 
                                (j6)    -- +(-135:0.5) coordinate (k7)          
                                                                (j7)    -- +(-135:0.5) coordinate (k8)                                      
 (g8)    -- +(-135:0.5) coordinate (h9)
   (h9)    -- +(-45:0.5) coordinate (i9)  
   (h9)    -- +(-135:0.5) coordinate (i10)
   (i9)    -- +(-45:0.5) coordinate (j9)  
   (i9)    -- +(-135:0.5) coordinate (j10)
   (i10)    -- +(-135:0.5) coordinate (j11)   
   (j9)    -- +(-45:0.5) coordinate (k9)  
   (j9)    -- +(-135:0.5) coordinate (k10)
   (j10)    -- +(-135:0.5) coordinate (k11)  
      (j11)    -- +(-135:0.5) coordinate (k12)  
   ;
            \fill(a1) circle (1pt);             \fill(a2) circle (1pt); \fill(a3) circle (1pt); 
              \fill(b1) circle (1pt);                 \fill(b2) circle (1pt);                          \fill(b3) circle (1pt);   
               \fill(c1) circle (1pt);                 \fill(c2) circle (1pt);                          \fill(c3) circle (1pt);   \fill(c4) circle (1pt);   
                              \fill(d1) circle (1pt);                 \fill(d2) circle (1pt);                          \fill(d3) circle (1pt);   \fill(d4) circle (1pt);   \fill(d5) circle (1pt);   
                                    \fill(e1) circle (1pt);                 \fill(e2) circle (1pt);                          \fill(e3) circle (1pt);   \fill(e4) circle (1pt);   
\fill(e5) circle (1pt);     \fill(e6) circle (1pt);   
                                                                        \fill(f1) circle (1pt);                 \fill(f2) circle (1pt);                          \fill(f3) circle (1pt);   \fill(f4) circle (1pt);   \fill(f5) circle (1pt);     \fill(f6) circle (1pt);    \fill(f7) circle (1pt);   
    \fill(g1) circle (1pt);                 \fill(g2) circle (1pt);                          \fill(g3) circle (1pt);   \fill(g4) circle (1pt);   \fill(g5) circle (1pt);     \fill(g6) circle (1pt);    \fill(g7) circle (1pt);   
 \fill(g8) circle (1pt);   
  \fill(h1) circle (1pt);                 \fill(h2) circle (1pt);                          \fill(h3) circle (1pt);   \fill(h4) circle (1pt);   \fill(h5) circle (1pt);     \fill(h6) circle (1pt);    \fill(h7) circle (1pt);   
 \fill(h8) circle (1pt);   
  \fill(j1) circle (1pt);                 \fill(j2) circle (1pt);                          \fill(j3) circle (1pt);   \fill(j4) circle (1pt);   \fill(j5) circle (1pt);     \fill(j6) circle (1pt);    \fill(j7) circle (1pt);   
 \fill(j8) circle (1pt);   
  \fill(i1) circle (1pt);                 \fill(i2) circle (1pt);                          \fill(i3) circle (1pt);   \fill(i4) circle (1pt);   \fill(i5) circle (1pt);     \fill(i6) circle (1pt);    \fill(i7) circle (1pt);   
 \fill(i8) circle (1pt);  
   \fill(k1) circle (1pt);                 \fill(k2) circle (1pt);                          \fill(k3) circle (1pt);   \fill(k4) circle (1pt);   \fill(k5) circle (1pt);     \fill(k6) circle (1pt);    \fill(k7) circle (1pt);   
 \fill(k8) circle (1pt);   
        \draw[dotted](b3) -- +(270:10);       \draw[dotted](b3) -- +(90:1);
        \draw[dotted](b1) -- +(270:10);       \draw[dotted](b1) -- +(90:1);
        \draw[dotted](f7) -- +(270:10);        \draw[dotted](f7) -- +(90:1);
         \draw[dotted](f1) -- +(270:10);        \draw[dotted](f1) -- +(90:1);
 \draw[dotted](j1) -- +(270:10);        \draw[dotted](j1) -- +(90:1);
 \draw[dotted](j11) -- +(270:10);        \draw[dotted](j11) -- +(90:1);       ; 
          \fill(i10) circle (1pt);     \fill(i9) circle (1pt);   \fill(h9) circle (1pt);  
             \fill(j9) circle (1pt);   \fill(j10) circle (1pt);  \fill(j11) circle (1pt);  
             \fill(k9) circle (1pt);               \fill(k10) circle (1pt);      \fill(k11) circle (1pt);               \fill(k12) circle (1pt);  
\draw(a1) --    (f7); \draw(f7) --  (j7);           \draw(k8) --  (j7);  }
    \end{tikzpicture}
 $$
 \caption{Two paths $\omega ,\omega' \in  \Path((4,7),(0,11))$ }
 \label{LEVEL2EXAMPLE1}
 \end{figure}
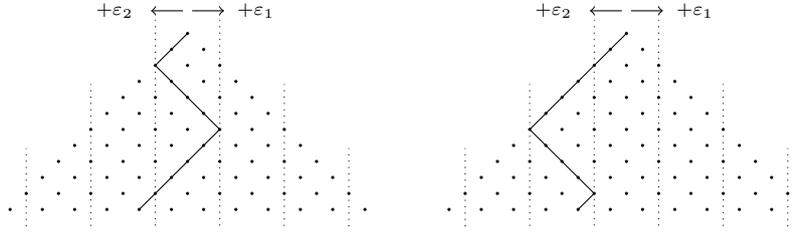
 

  Let $\nu=(5,6)$,
there are two elements of $\omega'',\omega'''\in {\rm Path}_{11}(\nu,\mu)$, depicted in Figure \ref{LEVEL2EXAMPLE2}, of degree 3 and degree 1 respectively.  In the former case, $d_{\varepsilon_1-\varepsilon_2}(\omega'',3)=1$, $d_{\varepsilon_1-\varepsilon_2}(\omega'',7)=1$ and $d_{\varepsilon_1-\varepsilon_2}(\omega'',11)=1$.
In the latter case, $d_{\varepsilon_1-\varepsilon_2}(\omega''',7)=1$.  

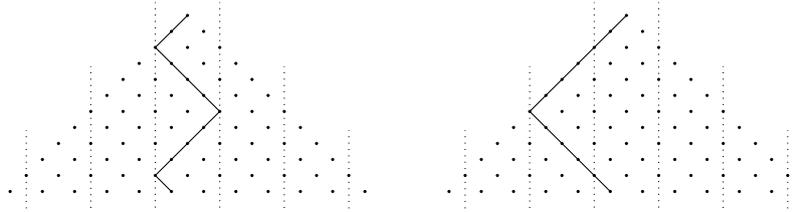
\begin{figure}[ht]\captionsetup{width=0.9\textwidth}  $$  \scalefont{0.8} \begin{tikzpicture}[scale=0.6]
{ \clip (-4,-4.3) rectangle ++(8,5);
\path [draw,name path=upward line] (-3,-5) -- (3,-5);
  \path 
      (0, 0)              coordinate (a1)
        (a1)    -- +(-45:0.5) coordinate (a2) 
        (a1)    -- +(-135:0.5) coordinate (a3) 
        (a2)    -- +(-45:0.5) coordinate (b1) 
        (a2)    -- +(-135:0.5) coordinate (b2) 
         (a3)    -- +(-135:0.5) coordinate (b3) 
  (b1)    -- +(-45:0.5) coordinate (c1) 
        (b1)    -- +(-135:0.5) coordinate (c2) 
  (b3)    -- +(-45:0.5) coordinate (c3) 
        (b3)    -- +(-135:0.5) coordinate (c4) 
  (c1)    -- +(-45:0.5) coordinate (d1) 
        (c1)    -- +(-135:0.5) coordinate (d2) 
  (c3)    -- +(-45:0.5) coordinate (d3) 
        (c3)    -- +(-135:0.5) coordinate (d4) 
                (c4)    -- +(-135:0.5) coordinate (d5) 
  (d1)    -- +(-45:0.5) coordinate (e1) 
        (d1)    -- +(-135:0.5) coordinate (e2) 
  (d3)    -- +(-45:0.5) coordinate (e3) 
        (d3)    -- +(-135:0.5) coordinate (e4) 
                (d5)    -- +(-135:0.5) coordinate (e6) 
                                (d5)    -- +(-45:0.5) coordinate (e5) 
  (e1)    -- +(-45:0.5) coordinate (f1) 
        (e1)    -- +(-135:0.5) coordinate (f2) 
  (e3)    -- +(-45:0.5) coordinate (f3) 
        (e3)    -- +(-135:0.5) coordinate (f4) 
                (e5)    -- +(-45:0.5) coordinate (f5) 
                (e5)    -- +(-135:0.5) coordinate (f6) 
                                (e6)    -- +(-135:0.5) coordinate (f7) 
  (f1)    -- +(-45:0.5) coordinate (g1) 
        (f1)    -- +(-135:0.5) coordinate (g2) 
  (f3)    -- +(-45:0.5) coordinate (g3) 
        (f3)    -- +(-135:0.5) coordinate (g4) 
                (f5)    -- +(-45:0.5) coordinate (g5) 
                (f5)    -- +(-135:0.5) coordinate (g6) 
                                (f6)    -- +(-135:0.5) coordinate (g7)          
                                                                (f7)    -- +(-135:0.5) coordinate (g8)        
          (g1)    -- +(-45:0.5) coordinate (h1) 
        (g1)    -- +(-135:0.5) coordinate (h2) 
  (g3)    -- +(-45:0.5) coordinate (h3) 
        (g3)    -- +(-135:0.5) coordinate (h4) 
                (g5)    -- +(-45:0.5) coordinate (h5) 
                (g5)    -- +(-135:0.5) coordinate (h6) 
                                (g6)    -- +(-135:0.5) coordinate (h7)          
                                                                (g7)    -- +(-135:0.5) coordinate (h8)           
                                                                       (h1)    -- +(-45:0.5) coordinate (i1) 
        (h1)    -- +(-135:0.5) coordinate (i2) 
  (h3)    -- +(-45:0.5) coordinate (i3) 
        (h3)    -- +(-135:0.5) coordinate (i4) 
                (h5)    -- +(-45:0.5) coordinate (i5) 
                (h5)    -- +(-135:0.5) coordinate (i6) 
                                (h6)    -- +(-135:0.5) coordinate (i7)          
                                                                (h7)    -- +(-135:0.5) coordinate (i8)           
                            (i1)    -- +(-45:0.5) coordinate (j1) 
        (i1)    -- +(-135:0.5) coordinate (j2) 
  (i3)    -- +(-45:0.5) coordinate (j3) 
        (i3)    -- +(-135:0.5) coordinate (j4) 
                (i5)    -- +(-45:0.5) coordinate (j5) 
                (i5)    -- +(-135:0.5) coordinate (j6) 
                                (i6)    -- +(-135:0.5) coordinate (j7)          
                                                                (i7)    -- +(-135:0.5) coordinate (j8)      
         (j1)    -- +(-45:0.5) coordinate (k1) 
        (j1)    -- +(-135:0.5) coordinate (k2) 
  (j3)    -- +(-45:0.5) coordinate (k3) 
        (j3)    -- +(-135:0.5) coordinate (k4) 
                (j5)    -- +(-45:0.5) coordinate (k5) 
                (j5)    -- +(-135:0.5) coordinate (k6) 
                                (j6)    -- +(-135:0.5) coordinate (k7)          
                                                                (j7)    -- +(-135:0.5) coordinate (k8)                                      
 (g8)    -- +(-135:0.5) coordinate (h9)
   (h9)    -- +(-45:0.5) coordinate (i9)  
   (h9)    -- +(-135:0.5) coordinate (i10)
   (i9)    -- +(-45:0.5) coordinate (j9)  
   (i9)    -- +(-135:0.5) coordinate (j10)
   (i10)    -- +(-135:0.5) coordinate (j11)   
   (j9)    -- +(-45:0.5) coordinate (k9)  
   (j9)    -- +(-135:0.5) coordinate (k10)
   (j10)    -- +(-135:0.5) coordinate (k11)  
      (j11)    -- +(-135:0.5) coordinate (k12)  
   ;
            \fill(a1) circle (1pt);             \fill(a2) circle (1pt); \fill(a3) circle (1pt); 
              \fill(b1) circle (1pt);                 \fill(b2) circle (1pt);                          \fill(b3) circle (1pt);   
               \fill(c1) circle (1pt);                 \fill(c2) circle (1pt);                          \fill(c3) circle (1pt);   \fill(c4) circle (1pt);   
                              \fill(d1) circle (1pt);                 \fill(d2) circle (1pt);                          \fill(d3) circle (1pt);   \fill(d4) circle (1pt);   \fill(d5) circle (1pt);   
                                    \fill(e1) circle (1pt);                 \fill(e2) circle (1pt);                          \fill(e3) circle (1pt);   \fill(e4) circle (1pt);   
\fill(e5) circle (1pt);     \fill(e6) circle (1pt);   
                                                                        \fill(f1) circle (1pt);                 \fill(f2) circle (1pt);                          \fill(f3) circle (1pt);   \fill(f4) circle (1pt);   \fill(f5) circle (1pt);     \fill(f6) circle (1pt);    \fill(f7) circle (1pt);   
    \fill(g1) circle (1pt);                 \fill(g2) circle (1pt);                          \fill(g3) circle (1pt);   \fill(g4) circle (1pt);   \fill(g5) circle (1pt);     \fill(g6) circle (1pt);    \fill(g7) circle (1pt);   
 \fill(g8) circle (1pt);   
  \fill(h1) circle (1pt);                 \fill(h2) circle (1pt);                          \fill(h3) circle (1pt);   \fill(h4) circle (1pt);   \fill(h5) circle (1pt);     \fill(h6) circle (1pt);    \fill(h7) circle (1pt);   
 \fill(h8) circle (1pt);   
  \fill(j1) circle (1pt);                 \fill(j2) circle (1pt);                          \fill(j3) circle (1pt);   \fill(j4) circle (1pt);   \fill(j5) circle (1pt);     \fill(j6) circle (1pt);    \fill(j7) circle (1pt);   
 \fill(j8) circle (1pt);   
  \fill(i1) circle (1pt);                 \fill(i2) circle (1pt);                          \fill(i3) circle (1pt);   \fill(i4) circle (1pt);   \fill(i5) circle (1pt);     \fill(i6) circle (1pt);    \fill(i7) circle (1pt);   
 \fill(i8) circle (1pt);  
   \fill(k1) circle (1pt);                 \fill(k2) circle (1pt);                          \fill(k3) circle (1pt);   \fill(k4) circle (1pt);   \fill(k5) circle (1pt);     \fill(k6) circle (1pt);    \fill(k7) circle (1pt);   
 \fill(k8) circle (1pt);   
        \draw[dotted](b3) -- +(270:10);       \draw[dotted](b3) -- +(90:1);
        \draw[dotted](b1) -- +(270:10);       \draw[dotted](b1) -- +(90:1);
        \draw[dotted](f7) -- +(270:10);        \draw[dotted](f7) -- +(90:1);
         \draw[dotted](f1) -- +(270:10);        \draw[dotted](f1) -- +(90:1);
 \draw[dotted](j1) -- +(270:10);        \draw[dotted](j1) -- +(90:1);
 \draw[dotted](j11) -- +(270:10);        \draw[dotted](j11) -- +(90:1);       ; 
          \fill(i10) circle (1pt);     \fill(i9) circle (1pt);   \fill(h9) circle (1pt);  
             \fill(j9) circle (1pt);   \fill(j10) circle (1pt);  \fill(j11) circle (1pt);  
             \fill(k9) circle (1pt);               \fill(k10) circle (1pt);      \fill(k11) circle (1pt);               \fill(k12) circle (1pt);  
\draw(a1) --  (b3);   \draw(j7) --  (f3); \draw(j7) --  (k7);   \draw(f3) --  (b3);          }
    \end{tikzpicture}
\quad \quad\quad 
  \begin{tikzpicture}[scale=0.6]
{ \clip (-4,-4.3) rectangle ++(8,5);
\path [draw,name path=upward line] (-3,-5) -- (3,-5);
  \path 
      (0, 0)              coordinate (a1)
        (a1)    -- +(-45:0.5) coordinate (a2) 
        (a1)    -- +(-135:0.5) coordinate (a3) 
        (a2)    -- +(-45:0.5) coordinate (b1) 
        (a2)    -- +(-135:0.5) coordinate (b2) 
         (a3)    -- +(-135:0.5) coordinate (b3) 
  (b1)    -- +(-45:0.5) coordinate (c1) 
        (b1)    -- +(-135:0.5) coordinate (c2) 
  (b3)    -- +(-45:0.5) coordinate (c3) 
        (b3)    -- +(-135:0.5) coordinate (c4) 
  (c1)    -- +(-45:0.5) coordinate (d1) 
        (c1)    -- +(-135:0.5) coordinate (d2) 
  (c3)    -- +(-45:0.5) coordinate (d3) 
        (c3)    -- +(-135:0.5) coordinate (d4) 
                (c4)    -- +(-135:0.5) coordinate (d5) 
  (d1)    -- +(-45:0.5) coordinate (e1) 
        (d1)    -- +(-135:0.5) coordinate (e2) 
  (d3)    -- +(-45:0.5) coordinate (e3) 
        (d3)    -- +(-135:0.5) coordinate (e4) 
                (d5)    -- +(-135:0.5) coordinate (e6) 
                                (d5)    -- +(-45:0.5) coordinate (e5) 
  (e1)    -- +(-45:0.5) coordinate (f1) 
        (e1)    -- +(-135:0.5) coordinate (f2) 
  (e3)    -- +(-45:0.5) coordinate (f3) 
        (e3)    -- +(-135:0.5) coordinate (f4) 
                (e5)    -- +(-45:0.5) coordinate (f5) 
                (e5)    -- +(-135:0.5) coordinate (f6) 
                                (e6)    -- +(-135:0.5) coordinate (f7) 
  (f1)    -- +(-45:0.5) coordinate (g1) 
        (f1)    -- +(-135:0.5) coordinate (g2) 
  (f3)    -- +(-45:0.5) coordinate (g3) 
        (f3)    -- +(-135:0.5) coordinate (g4) 
                (f5)    -- +(-45:0.5) coordinate (g5) 
                (f5)    -- +(-135:0.5) coordinate (g6) 
                                (f6)    -- +(-135:0.5) coordinate (g7)          
                                                                (f7)    -- +(-135:0.5) coordinate (g8)        
          (g1)    -- +(-45:0.5) coordinate (h1) 
        (g1)    -- +(-135:0.5) coordinate (h2) 
  (g3)    -- +(-45:0.5) coordinate (h3) 
        (g3)    -- +(-135:0.5) coordinate (h4) 
                (g5)    -- +(-45:0.5) coordinate (h5) 
                (g5)    -- +(-135:0.5) coordinate (h6) 
                                (g6)    -- +(-135:0.5) coordinate (h7)          
                                                                (g7)    -- +(-135:0.5) coordinate (h8)           
                                                                       (h1)    -- +(-45:0.5) coordinate (i1) 
        (h1)    -- +(-135:0.5) coordinate (i2) 
  (h3)    -- +(-45:0.5) coordinate (i3) 
        (h3)    -- +(-135:0.5) coordinate (i4) 
                (h5)    -- +(-45:0.5) coordinate (i5) 
                (h5)    -- +(-135:0.5) coordinate (i6) 
                                (h6)    -- +(-135:0.5) coordinate (i7)          
                                                                (h7)    -- +(-135:0.5) coordinate (i8)           
                            (i1)    -- +(-45:0.5) coordinate (j1) 
        (i1)    -- +(-135:0.5) coordinate (j2) 
  (i3)    -- +(-45:0.5) coordinate (j3) 
        (i3)    -- +(-135:0.5) coordinate (j4) 
                (i5)    -- +(-45:0.5) coordinate (j5) 
                (i5)    -- +(-135:0.5) coordinate (j6) 
                                (i6)    -- +(-135:0.5) coordinate (j7)          
                                                                (i7)    -- +(-135:0.5) coordinate (j8)      
         (j1)    -- +(-45:0.5) coordinate (k1) 
        (j1)    -- +(-135:0.5) coordinate (k2) 
  (j3)    -- +(-45:0.5) coordinate (k3) 
        (j3)    -- +(-135:0.5) coordinate (k4) 
                (j5)    -- +(-45:0.5) coordinate (k5) 
                (j5)    -- +(-135:0.5) coordinate (k6) 
                                (j6)    -- +(-135:0.5) coordinate (k7)          
                                                                (j7)    -- +(-135:0.5) coordinate (k8)                                      
 (g8)    -- +(-135:0.5) coordinate (h9)
   (h9)    -- +(-45:0.5) coordinate (i9)  
   (h9)    -- +(-135:0.5) coordinate (i10)
   (i9)    -- +(-45:0.5) coordinate (j9)  
   (i9)    -- +(-135:0.5) coordinate (j10)
   (i10)    -- +(-135:0.5) coordinate (j11)   
   (j9)    -- +(-45:0.5) coordinate (k9)  
   (j9)    -- +(-135:0.5) coordinate (k10)
   (j10)    -- +(-135:0.5) coordinate (k11)  
      (j11)    -- +(-135:0.5) coordinate (k12)  
   ;
            \fill(a1) circle (1pt);             \fill(a2) circle (1pt); \fill(a3) circle (1pt); 
              \fill(b1) circle (1pt);                 \fill(b2) circle (1pt);                          \fill(b3) circle (1pt);   
               \fill(c1) circle (1pt);                 \fill(c2) circle (1pt);                          \fill(c3) circle (1pt);   \fill(c4) circle (1pt);   
                              \fill(d1) circle (1pt);                 \fill(d2) circle (1pt);                          \fill(d3) circle (1pt);   \fill(d4) circle (1pt);   \fill(d5) circle (1pt);   
                                    \fill(e1) circle (1pt);                 \fill(e2) circle (1pt);                          \fill(e3) circle (1pt);   \fill(e4) circle (1pt);   
\fill(e5) circle (1pt);     \fill(e6) circle (1pt);   
                                                                        \fill(f1) circle (1pt);                 \fill(f2) circle (1pt);                          \fill(f3) circle (1pt);   \fill(f4) circle (1pt);   \fill(f5) circle (1pt);     \fill(f6) circle (1pt);    \fill(f7) circle (1pt);   
    \fill(g1) circle (1pt);                 \fill(g2) circle (1pt);                          \fill(g3) circle (1pt);   \fill(g4) circle (1pt);   \fill(g5) circle (1pt);     \fill(g6) circle (1pt);    \fill(g7) circle (1pt);   
 \fill(g8) circle (1pt);   
  \fill(h1) circle (1pt);                 \fill(h2) circle (1pt);                          \fill(h3) circle (1pt);   \fill(h4) circle (1pt);   \fill(h5) circle (1pt);     \fill(h6) circle (1pt);    \fill(h7) circle (1pt);   
 \fill(h8) circle (1pt);   
  \fill(j1) circle (1pt);                 \fill(j2) circle (1pt);                          \fill(j3) circle (1pt);   \fill(j4) circle (1pt);   \fill(j5) circle (1pt);     \fill(j6) circle (1pt);    \fill(j7) circle (1pt);   
 \fill(j8) circle (1pt);   
  \fill(i1) circle (1pt);                 \fill(i2) circle (1pt);                          \fill(i3) circle (1pt);   \fill(i4) circle (1pt);   \fill(i5) circle (1pt);     \fill(i6) circle (1pt);    \fill(i7) circle (1pt);   
 \fill(i8) circle (1pt);  
   \fill(k1) circle (1pt);                 \fill(k2) circle (1pt);                          \fill(k3) circle (1pt);   \fill(k4) circle (1pt);   \fill(k5) circle (1pt);     \fill(k6) circle (1pt);    \fill(k7) circle (1pt);   
 \fill(k8) circle (1pt);   
        \draw[dotted](b3) -- +(270:10);       \draw[dotted](b3) -- +(90:1);
        \draw[dotted](b1) -- +(270:10);       \draw[dotted](b1) -- +(90:1);
        \draw[dotted](f7) -- +(270:10);        \draw[dotted](f7) -- +(90:1);
         \draw[dotted](f1) -- +(270:10);        \draw[dotted](f1) -- +(90:1);
 \draw[dotted](j1) -- +(270:10);        \draw[dotted](j1) -- +(90:1);
 \draw[dotted](j11) -- +(270:10);        \draw[dotted](j11) -- +(90:1);       ; 
          \fill(i10) circle (1pt);     \fill(i9) circle (1pt);   \fill(h9) circle (1pt);  
             \fill(j9) circle (1pt);   \fill(j10) circle (1pt);  \fill(j11) circle (1pt);  
             \fill(k9) circle (1pt);               \fill(k10) circle (1pt);      \fill(k11) circle (1pt);               \fill(k12) circle (1pt);  
\draw(a1) --    (f7); \draw(f7) --  (j7);           \draw(k7) --  (j7);  }
    \end{tikzpicture}
 $$
 \caption{Two path $\omega'',\omega''' \in  \Path((6,5),(0,11))$}
 \label{LEVEL2EXAMPLE2}
\end{figure}
\begin{figure}[ht]\captionsetup{width=0.9\textwidth}
$$\begin{tabular}{c|cccccc} \text{ alcove } 
&  \  \ $\mathfrak{a}_{3'}$ \  \  &  \  \ $\mathfrak{a}_{2'}$ \  \ &  \  \ $\mathfrak{a}_{1'}$ \  \  &\ $\mathfrak{a}_{0}$ \  \  &  \  \ $\mathfrak{a}_{1R}$ \  \  & \  \ $\mathfrak{a}_{2R}$  \  \  \\  \hline  &   &  &  & 1 &  &  \\  &   &  & 1 & $t$ &  & \\  &   &1  &$t$  & $t^2$ &$t$  & 
\\  & 1& $t$  &$t^2+1$  &$t^3+t$  & $t^2$ &  $t$    \end{tabular}.$$
\caption{This table records the
result  of running the
 (cancellation-free) Soergel algorithm along the path $\omega^{\mu}$.  The alcoves are labelled by their length and primed (respectively unprimed)  if they correspond to an alcove to the left (respectively right) of the origin in the
  diagrams in Figures \ref{LEVEL2EXAMPLE1} and 
 \ref{LEVEL2EXAMPLE2}}
\label{tabular}\end{figure}
Figure \ref{tabular} records the
result  of running the
 (cancellation-free) Soergel algorithm along the path $\omega^{\mu}$.
Notice that the algorithm produces $m_{\mu}(\lambda) =n'_{\mu}(\lambda)=t^2+1$ and $ n_{\mu}(\lambda) =t^2$;
similarly $m_{\mu}(\nu) =n'_{\mu}(\nu)=t^3+t$ and $ n_{\mu}(\nu) =t^3$.  We have that $e_\mu(\lambda)=1$ and $e_\mu(\nu)=0$.
     \end{eg}

\begin{prop}\label{changeofbasis}
Let $\lambda, \mu$ denote points belonging to alcoves  in $E_r$.
Fix an admissible  path $\omega^\mu$.  
Let   $\nu$ vary over all points such that
$ \Path(\nu,\mu)\neq \emptyset $ and 
  $\Path(\lambda,\nu)\neq \emptyset$.
  We have that,
 $$
m_\mathfrak{\mu}(\mathfrak{\lambda}) = \;
\sum_
{\mathclap{  \begin{subarray}c  
   \Path(\nu,\mu)\neq \emptyset\\
  \Path(\lambda,\nu)\neq \emptyset \end{subarray} }} \;
  n_{\nu}(\lambda)    e_\mu(\nu)  .$$
  
  \end{prop}
\begin{proof}
Let $\lambda, \mu,\nu$ denote  points  
 in alcoves  
$\mathfrak{a}, \mathfrak{b}$, and $\mathfrak{c}$, respectively.
We have fixed a distinguished walk, $\omega^\mu$, and so both 
$e_\mathfrak{b}$ and $m_\mathfrak{b}$ are well-defined.  
By \cite{Soergel},   $n_\mathfrak{c}$ is independent of the choice of 
 path from $\astrosun$ to $\nu$. 
Therefore the expression above is well-defined.  



A subpattern in Soergel's algorithm is removed if $(n'_{ {\mathfrak{a}_{i+1}}}(\mathfrak{d})|_{t=0})\neq 0$ for some alcove $\mathfrak{d}$.  
 The subpatterns removed in the $n$-algorithm are (of course) not removed
 by the $m$-algorithm; the $e$-algorithm will keep track of the leading  terms in these subpatterns. 
The leading  term of the subpattern
 will remain 
constant unless it is reflected through  
a hyperplane through  
 the lower closure of the alcove,  in which case 
we multiply the subpattern by $(t+t^{-1})$.   This is particularly clear from the alcove-wall path definition of Soergel's algorithm 
(see also  the 
singular combinatorics for Soergel's algorithm developed in \cite{ryomtilt}).  
 The result then follows from the definitions.  
\end{proof}

        \begin{eg}
   Let $l=3$, $n=13$, $e=8$, $\rho=(8,4,2)$ and consider the root system of type $\hat{A}_2$.  
Take $\alpha=(4,6,3)$, $\beta=(5,6,2)$ and $\gamma=(5,8,0)$.  
Let $\omega^\gamma$ be the alcove-wall path depicted in Figure \ref{first path} in the introduction.  The set of elements in   $\Path(-,\gamma)$, together with their degrees, is depicted across Figures \ref{weight a}, \ref{weight b}, and \ref{other weights}.  Figures \ref{belowsoergel1} and \ref{belowsoergel2} depict the four steps of running Soregel's algorithm along $\omega^\gamma$.

\begin{figure}[ht]\captionsetup{width=0.9\textwidth} $$ 
\scalefont{0.9} \begin{tikzpicture}[scale=1.0000]
    \clip (-1.2,1) rectangle (60:3.6cm);
     \path (120:2.4cm)++(60:2.4cm) coordinate (BB);
                     \clip (120:1.2cm)  -- (120:2.4cm) -- (BB) -- (60:2.4cm) --(60:1.2cm) ;
  \path (0,0) coordinate (origin);
    \clip   (120:3.6cm) -- (60:3.6cm) -- (origin)    ;
  \path[dotted] (60:1.2cm) coordinate (A1);
  \path[dotted] (60:2.4cm) coordinate (A2);
  \path[dotted] (60:3.6cm) coordinate (A3);
  \path[dotted] (60:4.8cm) coordinate (A4);
    \path[dotted] (60:6cm) coordinate (A5);
  \path[dotted] (120:1.2cm) coordinate (B1);
  \path[dotted] (120:2.4cm) coordinate (B2);
  \path[dotted] (120:3.6cm) coordinate (B3);
  \path[dotted] (120:4.8cm) coordinate (B4);
    \path[dotted] (120:6cm) coordinate (B5);
  \path[dotted] (A1) ++(120:13cm) coordinate (C1);
  \path[dotted] (A2) ++(120:12cm) coordinate (C2);
  \path[dotted] (A3) ++(120:11cm) coordinate (C3);
    \path[dotted] (A4) ++(120:14cm) coordinate (C4);
    \path[dotted] (A5) ++(120:14cm) coordinate (C5);
      \path[dotted] (B1) ++(60:13cm) coordinate (D1);
  \path[dotted] (B2) ++(60:12cm) coordinate (D2);
  \path[dotted] (B3) ++(60:11cm) coordinate (D3);
    \path[dotted] (B4) ++(60:14cm) coordinate (D4);
    \path[dotted] (B5) ++(60:14cm) coordinate (D5);
   \foreach \i in {1,...,60}
  {
    \path[dotted] (origin)++(60:0.2*\i cm)  coordinate (a\i);
    \path[dotted] (origin)++(120:0.2*\i cm)  coordinate (b\i);
    \path[dotted] (a\i)++(120:12cm) coordinate (ca\i);
    \path[dotted] (b\i)++(60:12cm) coordinate (cb\i);
    \draw[thin,white] (a\i) -- (ca\i) (a\i)--(b\i) (b\i) -- (cb\i) ; 
  }
     \draw (b8)++(60:0.4cm)  node  {$\mathbf{0}$};
  \foreach \i in {1,2,3,4,5}
{ \draw[thick]       (A2) -- (B2) (A1) -- (C1) 
(D1) -- (B1);
\draw[thick] (origin) --(A5) (origin)-- (B5)  (A1)--(B1) (A3)--(B3)
(A2) -- (C2) 
(D2) -- (B2);}
\end{tikzpicture}
  \begin{tikzpicture}[scale=1.0000]
    \clip (-1.2,1) rectangle (60:3.6cm);
     \path (120:2.4cm)++(60:2.4cm) coordinate (BB);
                     \clip (120:1.2cm)  -- (120:2.4cm) -- (BB) -- (60:2.4cm) --(60:1.2cm) ;
  \path (0,0) coordinate (origin);
    \clip   (120:3.6cm) -- (60:3.6cm) -- (origin)    ;
  \path[dotted] (60:1.2cm) coordinate (A1);
  \path[dotted] (60:2.4cm) coordinate (A2);
  \path[dotted] (60:3.6cm) coordinate (A3);
  \path[dotted] (60:4.8cm) coordinate (A4);
    \path[dotted] (60:6cm) coordinate (A5);
  \path[dotted] (120:1.2cm) coordinate (B1);
  \path[dotted] (120:2.4cm) coordinate (B2);
  \path[dotted] (120:3.6cm) coordinate (B3);
  \path[dotted] (120:4.8cm) coordinate (B4);
    \path[dotted] (120:6cm) coordinate (B5);
  \path[dotted] (A1) ++(120:13cm) coordinate (C1);
  \path[dotted] (A2) ++(120:12cm) coordinate (C2);
  \path[dotted] (A3) ++(120:11cm) coordinate (C3);
    \path[dotted] (A4) ++(120:14cm) coordinate (C4);
    \path[dotted] (A5) ++(120:14cm) coordinate (C5);
      \path[dotted] (B1) ++(60:13cm) coordinate (D1);
  \path[dotted] (B2) ++(60:12cm) coordinate (D2);
  \path[dotted] (B3) ++(60:11cm) coordinate (D3);
    \path[dotted] (B4) ++(60:14cm) coordinate (D4);
    \path[dotted] (B5) ++(60:14cm) coordinate (D5);
   \foreach \i in {1,...,60}
  {
    \path[dotted] (origin)++(60:0.2*\i cm)  coordinate (a\i);
    \path[dotted] (origin)++(120:0.2*\i cm)  coordinate (b\i);
    \path[dotted] (a\i)++(120:12cm) coordinate (ca\i);
    \path[dotted] (b\i)++(60:12cm) coordinate (cb\i);
    \draw[thin,white] (a\i) -- (ca\i) (a\i)--(b\i) (b\i) -- (cb\i) ; 
  }
 \draw (b10)++(60:0.8cm) node  {$\mathbf{0}$};
     \draw (b8)++(60:0.4cm)  node  {$\mathbf{1}$};
  \foreach \i in {1,2,3,4,5}
{ \draw[thick]       (A2) -- (B2) (A1) -- (C1) 
(D1) -- (B1);
\draw[thick] (origin) --(A5) (origin)-- (B5)  (A1)--(B1) (A3)--(B3)
(A2) -- (C2) 
(D2) -- (B2);}
\end{tikzpicture}
 \begin{tikzpicture}[scale=1.0000]
    \clip (-1.2,1) rectangle (60:3.6cm);
     \path (120:2.4cm)++(60:2.4cm) coordinate (BB);
                     \clip (120:1.2cm)  -- (120:2.4cm) -- (BB) -- (60:2.4cm) --(60:1.2cm) ;
  \path (0,0) coordinate (origin);
    \clip   (120:3.6cm) -- (60:3.6cm) -- (origin)    ;
  \path[dotted] (60:1.2cm) coordinate (A1);
  \path[dotted] (60:2.4cm) coordinate (A2);
  \path[dotted] (60:3.6cm) coordinate (A3);
  \path[dotted] (60:4.8cm) coordinate (A4);
    \path[dotted] (60:6cm) coordinate (A5);
  \path[dotted] (120:1.2cm) coordinate (B1);
  \path[dotted] (120:2.4cm) coordinate (B2);
  \path[dotted] (120:3.6cm) coordinate (B3);
  \path[dotted] (120:4.8cm) coordinate (B4);
    \path[dotted] (120:6cm) coordinate (B5);
  \path[dotted] (A1) ++(120:13cm) coordinate (C1);
  \path[dotted] (A2) ++(120:12cm) coordinate (C2);
  \path[dotted] (A3) ++(120:11cm) coordinate (C3);
    \path[dotted] (A4) ++(120:14cm) coordinate (C4);
    \path[dotted] (A5) ++(120:14cm) coordinate (C5);
      \path[dotted] (B1) ++(60:13cm) coordinate (D1);
  \path[dotted] (B2) ++(60:12cm) coordinate (D2);
  \path[dotted] (B3) ++(60:11cm) coordinate (D3);
    \path[dotted] (B4) ++(60:14cm) coordinate (D4);
    \path[dotted] (B5) ++(60:14cm) coordinate (D5);
   \foreach \i in {1,...,60}
  {
    \path[dotted] (origin)++(60:0.2*\i cm)  coordinate (a\i);
    \path[dotted] (origin)++(120:0.2*\i cm)  coordinate (b\i);
    \path[dotted] (a\i)++(120:12cm) coordinate (ca\i);
    \path[dotted] (b\i)++(60:12cm) coordinate (cb\i);
    \draw[thin,white] (a\i) -- (ca\i) (a\i)--(b\i) (b\i) -- (cb\i) ; 
  }
 \draw (b10)++(60:0.8cm) node  {$\mathbf{1}$};
 \draw (b8)++(60:1.6cm) node  {$\mathbf{0}$};
     \draw (b8)++(60:0.4cm)  node  {$\mathbf{2}$};
        \draw (b4)++(60:0.8cm)   node  {$\mathbf{1}$};
  \foreach \i in {1,2,3,4,5}
{ \draw[thick]       (A2) -- (B2) (A1) -- (C1) 
(D1) -- (B1);
\draw[thick] (origin) --(A5) (origin)-- (B5)  (A1)--(B1) (A3)--(B3)
(A2) -- (C2) 
(D2) -- (B2);}
\end{tikzpicture}
 \begin{tikzpicture}[scale=1.0000]
    \clip (-1.2,1) rectangle (60:3.6cm);
     \path (120:2.4cm)++(60:2.4cm) coordinate (BB);
                     \clip (120:1.2cm)  -- (120:2.4cm) -- (BB) -- (60:2.4cm) --(60:1.2cm) ;
  \path (0,0) coordinate (origin);
    \clip   (120:3.6cm) -- (60:3.6cm) -- (origin)    ;
  \path[dotted] (60:1.2cm) coordinate (A1);
  \path[dotted] (60:2.4cm) coordinate (A2);
  \path[dotted] (60:3.6cm) coordinate (A3);
  \path[dotted] (60:4.8cm) coordinate (A4);
    \path[dotted] (60:6cm) coordinate (A5);
  \path[dotted] (120:1.2cm) coordinate (B1);
  \path[dotted] (120:2.4cm) coordinate (B2);
  \path[dotted] (120:3.6cm) coordinate (B3);
  \path[dotted] (120:4.8cm) coordinate (B4);
    \path[dotted] (120:6cm) coordinate (B5);
  \path[dotted] (A1) ++(120:13cm) coordinate (C1);
  \path[dotted] (A2) ++(120:12cm) coordinate (C2);
  \path[dotted] (A3) ++(120:11cm) coordinate (C3);
    \path[dotted] (A4) ++(120:14cm) coordinate (C4);
    \path[dotted] (A5) ++(120:14cm) coordinate (C5);
      \path[dotted] (B1) ++(60:13cm) coordinate (D1);
  \path[dotted] (B2) ++(60:12cm) coordinate (D2);
  \path[dotted] (B3) ++(60:11cm) coordinate (D3);
    \path[dotted] (B4) ++(60:14cm) coordinate (D4);
    \path[dotted] (B5) ++(60:14cm) coordinate (D5);
   \foreach \i in {1,...,60}
  {
    \path[dotted] (origin)++(60:0.2*\i cm)  coordinate (a\i);
    \path[dotted] (origin)++(120:0.2*\i cm)  coordinate (b\i);
    \path[dotted] (a\i)++(120:12cm) coordinate (ca\i);
    \path[dotted] (b\i)++(60:12cm) coordinate (cb\i);
    \draw[thin,white] (a\i) -- (ca\i) (a\i)--(b\i) (b\i) -- (cb\i) ; 
  }
 \draw (b10)++(60:0.8cm) node  {$\mathbf{2\!\!+\!\!0}$};
 \draw (b8)++(60:1.6cm) node  {$\mathbf{1}$};
        \draw (a10)++(120:0.8cm) node  {$\mathbf{0}$};
     \draw (b8)++(60:0.4cm)  node  {$\mathbf{3\!\!+\!\!1}$};
        \draw (b4)++(60:0.8cm)   node  {$\mathbf{2}$};
          \draw (b2)++(60:1.6cm)   node  {$\mathbf{1}$};
  \foreach \i in {1,2,3,4,5}
{ \draw[thick]       (A2) -- (B2) (A1) -- (C1) 
(D1) -- (B1);
\draw[thick] (origin) --(A5) (origin)-- (B5)  (A1)--(B1) (A3)--(B3)
(A2) -- (C2) 
(D2) -- (B2);}
\end{tikzpicture}
$$
\caption{The first four steps of running Soergel's cancellation-free algorithm along $\omega^\gamma$.
We have recorded the powers of the polynomials only (for example, $2+0$ should be read as $t^2+t^0$).  
}
\label{belowsoergel1}
\end{figure}
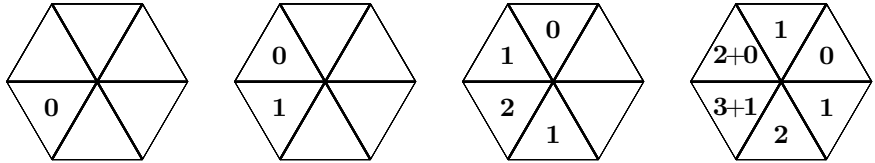
Under the Soergel procedure, we remove the subpattern labelled by the zero in the alcove containing the point $\beta$.   
The `new zero' is recorded by the character algorithm.  
We have that 
$$m_{\gamma}(\lambda)= e_{\gamma}(\gamma)n_{\gamma}(\lambda) + 
 e_{\gamma}(\beta)n_{\beta}(\lambda)$$ for any point $\lambda\in E_r$.   	
 Here 	$e_{\gamma}(\beta)=t^0$ and $e_{\gamma}(\gamma)=t^0$ and
  $e_\gamma(\lambda)=0$ otherwise.  This rewriting process is depicted in Figure \ref{belowsoergel2}.  

    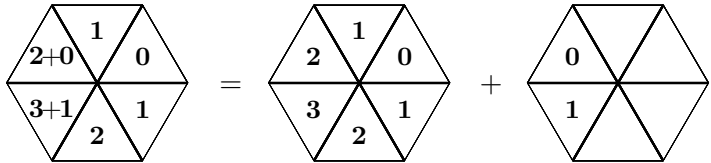
\begin{figure}[ht]\captionsetup{width=0.9\textwidth}
$$   \scalefont{0.9}  \begin{minipage}{27mm} \begin{tikzpicture}[scale=1.0000]
    \clip (-1.2,1) rectangle (60:3.6cm);
     \path (120:2.4cm)++(60:2.4cm) coordinate (BB);
                     \clip (120:1.2cm)  -- (120:2.4cm) -- (BB) -- (60:2.4cm) --(60:1.2cm) ;
  \path (0,0) coordinate (origin);
    \clip   (120:3.6cm) -- (60:3.6cm) -- (origin)    ;
  \path[dotted] (60:1.2cm) coordinate (A1);
  \path[dotted] (60:2.4cm) coordinate (A2);
  \path[dotted] (60:3.6cm) coordinate (A3);
  \path[dotted] (60:4.8cm) coordinate (A4);
    \path[dotted] (60:6cm) coordinate (A5);
  \path[dotted] (120:1.2cm) coordinate (B1);
  \path[dotted] (120:2.4cm) coordinate (B2);
  \path[dotted] (120:3.6cm) coordinate (B3);
  \path[dotted] (120:4.8cm) coordinate (B4);
    \path[dotted] (120:6cm) coordinate (B5);
  \path[dotted] (A1) ++(120:13cm) coordinate (C1);
  \path[dotted] (A2) ++(120:12cm) coordinate (C2);
  \path[dotted] (A3) ++(120:11cm) coordinate (C3);
    \path[dotted] (A4) ++(120:14cm) coordinate (C4);
    \path[dotted] (A5) ++(120:14cm) coordinate (C5);
      \path[dotted] (B1) ++(60:13cm) coordinate (D1);
  \path[dotted] (B2) ++(60:12cm) coordinate (D2);
  \path[dotted] (B3) ++(60:11cm) coordinate (D3);
    \path[dotted] (B4) ++(60:14cm) coordinate (D4);
    \path[dotted] (B5) ++(60:14cm) coordinate (D5);
   \foreach \i in {1,...,60}
  {
    \path[dotted] (origin)++(60:0.2*\i cm)  coordinate (a\i);
    \path[dotted] (origin)++(120:0.2*\i cm)  coordinate (b\i);
    \path[dotted] (a\i)++(120:12cm) coordinate (ca\i);
    \path[dotted] (b\i)++(60:12cm) coordinate (cb\i);
    \draw[thin,white] (a\i) -- (ca\i) (a\i)--(b\i) (b\i) -- (cb\i) ; 
  }
 \draw (b10)++(60:0.8cm) node  {$\mathbf{2\!\!+\!\!0}$};
 \draw (b8)++(60:1.6cm) node  {$\mathbf{1}$};
        \draw (a10)++(120:0.8cm) node  {$\mathbf{0}$};
     \draw (b8)++(60:0.4cm)  node  {$\mathbf{3\!\!+\!\!1}$};
        \draw (b4)++(60:0.8cm)   node  {$\mathbf{2}$};
          \draw (b2)++(60:1.6cm)   node  {$\mathbf{1}$};
  \foreach \i in {1,2,3,4,5}
{ \draw[thick]       (A2) -- (B2) (A1) -- (C1) 
(D1) -- (B1);
\draw[thick] (origin) --(A5) (origin)-- (B5)  (A1)--(B1) (A3)--(B3)
(A2) -- (C2) 
(D2) -- (B2);}
\end{tikzpicture}\end{minipage} = \ \
 \begin{minipage}{27mm}\begin{tikzpicture}[scale=1.0000]
    \clip (-1.2,1) rectangle (60:3.6cm);
     \path (120:2.4cm)++(60:2.4cm) coordinate (BB);
                     \clip (120:1.2cm)  -- (120:2.4cm) -- (BB) -- (60:2.4cm) --(60:1.2cm) ;
  \path (0,0) coordinate (origin);
    \clip   (120:3.6cm) -- (60:3.6cm) -- (origin)    ;
  \path[dotted] (60:1.2cm) coordinate (A1);
  \path[dotted] (60:2.4cm) coordinate (A2);
  \path[dotted] (60:3.6cm) coordinate (A3);
  \path[dotted] (60:4.8cm) coordinate (A4);
    \path[dotted] (60:6cm) coordinate (A5);
  \path[dotted] (120:1.2cm) coordinate (B1);
  \path[dotted] (120:2.4cm) coordinate (B2);
  \path[dotted] (120:3.6cm) coordinate (B3);
  \path[dotted] (120:4.8cm) coordinate (B4);
    \path[dotted] (120:6cm) coordinate (B5);
  \path[dotted] (A1) ++(120:13cm) coordinate (C1);
  \path[dotted] (A2) ++(120:12cm) coordinate (C2);
  \path[dotted] (A3) ++(120:11cm) coordinate (C3);
    \path[dotted] (A4) ++(120:14cm) coordinate (C4);
    \path[dotted] (A5) ++(120:14cm) coordinate (C5);
      \path[dotted] (B1) ++(60:13cm) coordinate (D1);
  \path[dotted] (B2) ++(60:12cm) coordinate (D2);
  \path[dotted] (B3) ++(60:11cm) coordinate (D3);
    \path[dotted] (B4) ++(60:14cm) coordinate (D4);
    \path[dotted] (B5) ++(60:14cm) coordinate (D5);
   \foreach \i in {1,...,60}
  {
    \path[dotted] (origin)++(60:0.2*\i cm)  coordinate (a\i);
    \path[dotted] (origin)++(120:0.2*\i cm)  coordinate (b\i);
    \path[dotted] (a\i)++(120:12cm) coordinate (ca\i);
    \path[dotted] (b\i)++(60:12cm) coordinate (cb\i);
    \draw[thin,white] (a\i) -- (ca\i) (a\i)--(b\i) (b\i) -- (cb\i) ; 
  }
 \draw (b10)++(60:0.8cm) node  {$\mathbf{2 }$};
 \draw (b8)++(60:1.6cm) node  {$\mathbf{1}$};
        \draw (a10)++(120:0.8cm) node  {$\mathbf{0}$};
     \draw (b8)++(60:0.4cm)  node  {$\mathbf{3 }$};
        \draw (b4)++(60:0.8cm)   node  {$\mathbf{2}$};
          \draw (b2)++(60:1.6cm)   node  {$\mathbf{1}$};
  \foreach \i in {1,2,3,4,5}
{ \draw[thick]       (A2) -- (B2) (A1) -- (C1) 
(D1) -- (B1);
\draw[thick] (origin) --(A5) (origin)-- (B5)  (A1)--(B1) (A3)--(B3)
(A2) -- (C2) 
(D2) -- (B2);}
\end{tikzpicture}
\end{minipage}
+ \ \ 
 \begin{minipage}{27mm}\begin{tikzpicture}[scale=1.0000]
    \clip (-1.2,1) rectangle (60:3.6cm);
     \path (120:2.4cm)++(60:2.4cm) coordinate (BB);
                     \clip (120:1.2cm)  -- (120:2.4cm) -- (BB) -- (60:2.4cm) --(60:1.2cm) ;
  \path (0,0) coordinate (origin);
    \clip   (120:3.6cm) -- (60:3.6cm) -- (origin)    ;
  \path[dotted] (60:1.2cm) coordinate (A1);
  \path[dotted] (60:2.4cm) coordinate (A2);
  \path[dotted] (60:3.6cm) coordinate (A3);
  \path[dotted] (60:4.8cm) coordinate (A4);
    \path[dotted] (60:6cm) coordinate (A5);
  \path[dotted] (120:1.2cm) coordinate (B1);
  \path[dotted] (120:2.4cm) coordinate (B2);
  \path[dotted] (120:3.6cm) coordinate (B3);
  \path[dotted] (120:4.8cm) coordinate (B4);
    \path[dotted] (120:6cm) coordinate (B5);
  \path[dotted] (A1) ++(120:13cm) coordinate (C1);
  \path[dotted] (A2) ++(120:12cm) coordinate (C2);
  \path[dotted] (A3) ++(120:11cm) coordinate (C3);
    \path[dotted] (A4) ++(120:14cm) coordinate (C4);
    \path[dotted] (A5) ++(120:14cm) coordinate (C5);
      \path[dotted] (B1) ++(60:13cm) coordinate (D1);
  \path[dotted] (B2) ++(60:12cm) coordinate (D2);
  \path[dotted] (B3) ++(60:11cm) coordinate (D3);
    \path[dotted] (B4) ++(60:14cm) coordinate (D4);
    \path[dotted] (B5) ++(60:14cm) coordinate (D5);
   \foreach \i in {1,...,60}
  {
    \path[dotted] (origin)++(60:0.2*\i cm)  coordinate (a\i);
    \path[dotted] (origin)++(120:0.2*\i cm)  coordinate (b\i);
    \path[dotted] (a\i)++(120:12cm) coordinate (ca\i);
    \path[dotted] (b\i)++(60:12cm) coordinate (cb\i);
    \draw[thin,white] (a\i) -- (ca\i) (a\i)--(b\i) (b\i) -- (cb\i) ; 
  }
 \draw (b10)++(60:0.8cm) node  {$\mathbf{ 0}$};
     \draw (b8)++(60:0.4cm)  node  {$\mathbf{ 1}$};
  \foreach \i in {1,2,3,4,5}
{ \draw[thick]       (A2) -- (B2) (A1) -- (C1) 
(D1) -- (B1);
\draw[thick] (origin) --(A5) (origin)-- (B5)  (A1)--(B1) (A3)--(B3)
(A2) -- (C2) 
(D2) -- (B2);}
\end{tikzpicture}\end{minipage}
$$
\caption{
Rewriting the $m_\gamma(\lambda)$ in terms of $n_{\mu}(\lambda)$ and $e_{\gamma}(\mu)$.  }
\label{belowsoergel2}
    \end{figure}
     \end{eg}
    
        \begin{eg}\label{exampleofnonpositive}
   Let $l=3$, $n=21$, $e=6$, $\rho=(6,4,2)$ 
 and consider the root system of type $\hat{A}_2$.  
 We leave it as an exercise for the reader to show that
 $$
 m_{(4,17,0)}(\lambda)=  n_{(4,17,0)}(\lambda)+ 
 n_{(15,4,2)}(\lambda) + (t+t^{-1}) n_{(6,9,0)}(\lambda)
 $$
for any $\lambda \in E_r$.  This is the smallest example where we find a path in negative degree.  In this case,
 $$e_{(4,17,0)}(6,9,0)=(t+t^{-1}), \quad e_{(4,17,0)}(15,4,2)=t^0, 
 \quad  e_{(4,17,0)}(4,17,0)=t^0.$$  
 \end{eg}
    
%
 
  \subsection{Algebras with Soergel path bases }
We shall now define   a family of algebras  whose representation theory is governed by paths in Euclidean space and show that
 the decomposition numbers of such an algebra are given by Soergel's algorithm.  Our proof is based on Kleshchev and Nash's algorithm for computing decomposition numbers (see \cite{KN10}).

\begin{defn}\label{soergeldefn}
  Let $A(\rho,e)$ denote a  graded cellular algebra with a theory of highest weights with respect to the poset $\mptn ln$.  
  Let $\mptn ln \hookrightarrow E_r$ where $E_r$ is equipped with 
 the action of a  Weyl group associated to a root system $\Phi$.  
  We say that the algebra $A(\rho,e)$ has a \emph{Soergel-path basis} with respect to 
  $\Phi$ if there exists a degree preserving bijective map
$$\omega  :\TSStd(\lambda,\mu) \to {\rm Path}(\lambda,\mu)$$
such that $\omega (\SSTT^\mu)=\omega^\mu$ is  admissible.  
    \end{defn}

    \begin{prop}\label{troll}
  Let $A(\rho,e)$ denote an algebra with  a  {Soergel-path basis} and suppose that $d_{\lambda\mu}(t) \in t\mathbb{N}_{0} (t)$ for all $\lambda\neq \mu \in \mptn ln$.  
   Then the following hold:
\begin{itemize}
\item[$(i)$] we have  $ \Dim{(\Delta_\mu(\lambda))} = m_\mu(\lambda) \in \NN_0[t,t^{-1}]$ and $ \Dim{(L_\mu(\lambda))} \in \NN_0[t+t^{-1}]$;
\item[$(ii)$] if $ \Dim{(\Delta_\mu(\lambda))}=0$, then $d_{\lambda\mu}(t)=0$;
\item[$(iii)$] we have $ \Dim{(\Delta_\mu(\mu))} =  \Dim{(L_\mu(\mu))} =1$;
\item[$(iv)$] if $\Path(\lambda,\mu)=\emptyset$, then $\Dim{(\Delta_\mu(\lambda))}=0$;
\item[$(v)$] if $\Path(\lambda,\mu)=\emptyset$, then $\Dim{(L_\mu(\lambda))}=0$;
\item[$(vi)$] we have that
\[
\Dim{(\Delta_\mu(\lambda)})= \;
\sum_
{  \mathclap{\begin{subarray}c \nu \neq \mu \\ 
   \Path(\nu,\mu)\neq \emptyset\\
  \Path(\lambda,\nu)\neq \emptyset \end{subarray} }} \; d_{\lambda\nu}(t)\Dim{(L_\mu(\nu))}+d_{ \lambda\mu}(t).
  \]
\end{itemize}
   \end{prop}
\begin{proof}
Part $(i)$ is clear by Proposition \ref{humathasprop}, $(iii)$ is a restatement of the condition that $\omega^\mu$ is the only path  in $\Path(\mu,\mu)$.  
A necessary condition for $\Dim{(\Hom (P(\mu),\Delta(\lambda))}\neq 0$ is that $\Delta_\mu(\lambda)\neq 0$, therefore $(ii)$ follows.

Part $(iv)$ is by definition, and part $(v)$ follows from the cellular structure.
Finally, $(vi)$ follows from $(i), (iii), (v)$ and our assumption that $d_{\lambda\mu}(t) \in t\mathbb{N}_0(t)$ for $\lambda \neq \mu$.  
\end{proof}

    \begin{thm}\label{main:theorem:general}
    Let $A(\rho,e)$ denote an algebra with a  {Soergel-path basis} of type $\Phi$.  
 Suppose that $d_{\lambda\mu}(t) \in t\mathbb{N}_{0} (t)$ for all $\lambda,\mu \in \mptn ln$ such that $\lambda \neq \mu$.  
  The graded decomposition numbers of  an $e$-regular block of 
$A(\rho,e)$ are given by the Soergel algorithm  
\begin{align*} 
d_{\lambda\mu}	(t)= n_\mu({\lambda})   \end{align*}
and the characters of the $e$-regular simple modules are given by the character algorithm 
\[
\Dim{(L_\mu(\lambda))} = e_\mu(\lambda).
\]
			      \end{thm}
   
\begin{proof}
By Proposition \ref{troll} $(ii)$, we may assume $\Path(\lambda,\mu)\neq \emptyset$.  We now calculate $d_{ \lambda\mu}(t)$ and $\Dim{(L_\mu(\lambda))}$ by induction on the length ordering on alcoves.  Induction begins when $\ell(\mu,\lambda)=0$, hence $\mu=\lambda$, and we have $d_{\mu\mu}(t)=1$ by Proposition \ref{troll} $(iii)$ and $\Dim{(L_\mu(\mu))}=e_\mu(\mu)=1$.

Let $\ell(\mu,\lambda)\geq 1$.  By induction, we know $d_{\lambda\nu}(t)$   and $\Dim{(L_\mu(\nu))}$ for points $\nu\in E_r$ such that  $\ell(\mu,\nu), \ell(\lambda,\nu)<\ell(\mu,\lambda)$
. By Proposition \ref{troll} $(vi)$ we have
\[
\Dim{(L_\mu(\lambda))} + d_{ \lambda\mu} (t)
 =
\Dim{(\Delta_\mu(\lambda))} 
- \;
\sum_
{  \mathclap{\begin{subarray}c 
\nu \neq \mu, \,
\nu \neq \lambda \\ 
   \Path(\nu,\mu)\neq \emptyset\\
  \Path( \lambda,\nu)\neq \emptyset \end{subarray} }} \;d_{\lambda\nu}(t)\Dim{(L_\mu(\nu))}.
\]
By induction and Proposition \ref{troll} $(i)$, the right-hand side is equal to
\[
m_\mu(\lambda) - \;
\sum_
{  \mathclap{\begin{subarray}c 
\nu \neq \mu, \, 
\nu \neq \lambda \\ 
   \Path(\nu,\mu)\neq \emptyset \\
  \Path(\lambda, \nu)\neq \emptyset \end{subarray} }} \; n_{\nu}(\lambda)e_\mu(\nu).
\]
We know that this final sum is equal to 
$e_\mu(\lambda)
n_\lambda(\lambda)
+n_\mu({\lambda})
e_\mu(\mu)
$ by Proposition \ref{changeofbasis}.  Our base case for induction showed that
$n_\lambda(\lambda)=1=e_\mu(\mu)$, therefore
\[
m_\mu(\lambda) - \;
\sum_
{  \mathclap{\begin{subarray}c 
\nu \neq \mu, \, 
\nu \neq \lambda \\ 
   \Path(\nu,\mu)\neq \emptyset \\
  \Path(\lambda, \nu)\neq \emptyset \end{subarray} }} \; n_{\nu}(\lambda)e_\mu(\nu) = e_\mu(\lambda)+n_\mu({\lambda}).
\]
Recall that 
$\Dim{(L_\mu(\lambda))} \in \mathbb{N}_0(t+t^{-1})$  and $d_{  \lambda\mu}\in t\NN(t)$. Therefore there is a unique  solution to the equality (see \cite[Section 4.1: Basic Algorithm]{KN10} for a general form, or \cite{Soergel} for the interpretation in terms of Kazhdan--Lusztig theory) given by
\[\Dim{(L_\mu(\lambda))} = e_\mu(\lambda), \quad d_{\lambda\mu}(t)=n_\mu({\lambda}).  \qedhere\] 
\end{proof}
\begin{cor}
  Let $A(\rho,e)$ denote an algebra with  a  {Soergel-path basis} and  suppose  $d_{\lambda\mu}(t)\in t\mathbb{N}_{0} (t)$ for $\lambda \neq \mu$.  
  Let    $\lambda,\lambda' \in \mathfrak{a}$  and  $\mu,\mu'\in \mathfrak{b}$ for some  alcoves $\mathfrak{a},\mathfrak{b}$ and 
suppose that $\mu \in W^e \cdot \lambda$ and  $\mu' \in W^e \cdot \lambda'$.  Then 
  $$
  d_{\lambda\mu}(t) =   d_{\lambda'\mu'}(t). 
  $$
\end{cor}   
   \begin{proof}
   This follows as Soergel's algorithm is well-defined on alcoves.  
   \end{proof}

  \section{The diagrammatic Cherednik algebra}

%
%

In this section we recall the definition of  the diagrammatic Cherednik algebras 
(reduced steadied quotients of weighted KLR algebras in Webster's
 terminology) constructed in  \cite{Webster}.

 \subsection{Combinatorial preliminaries}

Fix integers $l,n\in\ZZ_{\geq 0} $, $\g\in \RR_{> 0}$ and $e\in\{3,4,\dots\}\cup\{\infty\}$. 
We define a \emph{weighting} $\theta = (\theta_1,\dots, \theta_l) \in \RR^l$ to be any $l$-tuple such that
$\theta_i-\theta_j$ is not an integer multiple of  $\g$ for $1\leq i< j \leq l$.
Let $\kappa$ denote an $e$-\emph{multicharge} $\kappa = (\kappa_1,\dots,\kappa_l)\in (\ZZ/e\ZZ)^l$.

\begin{rmk}
We say that a weighting $\theta\in\RR^l$ is \emph{well-separated} for $A(n,\theta,\kappa)$ if $|\theta_j-\theta_i| > ng$ for all $1\leq i < j \leq l$.
We say that a weighting $\theta\in\RR^l$ is a \emph{FLOTW} weighting  for $A(n,\theta,\kappa)$ if $0<|\theta_i-\theta_j| < g$ for all $1\leq i < j \leq l$.
\end{rmk}

\begin{defn}
An \emph{$l$-multipartition} $\lambda=(\lambda^{(1)},\dots,\lambda^{(l)})$ of $n$ is an $l$-tuple of partitions such that $|\lambda^{(1)}|+\dots+ |\lambda^{(l)}|=n$. We will denote the set of $l$-multipartitions of $n$ by $\mptn ln$.
\end{defn}

We define the \emph{Russian array} as follows.
For each $1\leq k\leq l$, we place a point on the real line at $\theta_m$ and consider the region bounded by  half-lines at angles $3\pi/4 $ and $\pi/4$.
We tile the resulting quadrant with a lattice of squares, each with diagonal of length $2 \g $.

Let $\lambda=(\lambda^{(1)},\lambda^{(2)},\ldots ,\lambda^{(l)}) \in \mptn ln$.
The \emph{Young diagram} $[\lambda]$ is defined to be the set
\[
\{(r,c,m)\in\mathbb{N}\times\mathbb{N}\times\{1,\dots,l\} \mid c\leq \lambda^{(m)}_r\}.
\]
We refer to elements of $[\la]$ as nodes (of  $[\la]$ or $\la$). 
We define the \emph{residue} of a node $(r,c,m) \in [\lambda]$ to be $\kappa_m+c-r \pmod e$.

For each node of $[\la]$ we draw a box in the plane; we shall draw our Young diagrams in a mirrored-Russian convention. We place the first node of component $m$ at $\theta_m$ on the real line, with rows going northwest from this node, and columns going northeast. The diagram is tilted ever-so-slightly in the clockwise direction so that the top vertex of the box $(r,c,m)$ (that is, the box in the $r$th row and $c$th column of the $m$th component of $ [\lambda] $) has $x$-coordinate $\theta_m + g(r-c) + (r+c)\epsilon$.  

Here the tilt $\epsilon$ is chosen so that $n\epsilon$ is absolutely small with respect to $g$ (so that $\epsilon \ll g/n$) and with respect to the weighting (so that $g$ does not divide any number in the interval  $|\theta_i-\theta_j| + (-n\epsilon,+n\epsilon)$ for $1\leq i < j\leq l$).  With these assumptions firmly  in place, we will drop any mention of $\epsilon$ when speaking of the ghost distance, $g\in \RR_{>0}$, or the weighting, $\theta\in \RR^l$.

We define a \emph{loading}, $\mathbf{i}$, to be an element of $(\mathbb{R} \times (\ZZ/e\ZZ))^n$ such that no real number occurs with multiplicity greater than one.
Given a multipartition $\lambda \in \mptn ln$  we have an associated loading, $\mathbf{i}_\lambda$, given by the projection of the top vertex of each box $(r,c,m)\in[\la]$ to its $x$-coordinate $\mathbf{i}_{(r,c,m)}\in\RR$, and attaching to each point the {residue} $\kappa_m+c-r \pmod e$ of this node.

We let $D_\lambda$ denote the underlying ordered subset of $\RR$ given by the points of the loading. Given $a \in D_\lambda$, we abuse notation and let $a$ denote the corresponding node of $ \lambda $ (that is, the node whose top vertex projects onto $x$-coordinate $a \in \RR$).
The \emph{residue sequence} of $\lambda$ is given by reading the residues of the nodes of $ \lambda $ according to the ordering given by $D_\lambda$.

\begin{eg}
Let $l=2$,  $\g=1$, $\epsilon=1/100$, and $\theta=(0,0.5)$.
The bipartition $((2,1),(1^3))$ has Young diagram and corresponding loading $\mathbf{i}_\lambda$ given in \cref{ALoading}.
The residue sequence of $\lambda$ is $(\kappa_1+1,\kappa_1,\kappa_2,\kappa_1-1,\kappa_2-1,\kappa_2-2)$, and the ordered set $D_\lambda$ is $\{-0.97,0.02, 0.52, 1.03, 1.53 ,2.54 \}$.
The node $x={-0.97}$ in $  \lambda $ can be identified with the node in the first row and second column of the first component of  $\lambda$.

\begin{figure}[ht]\captionsetup{width=0.9\textwidth}
\[
\begin{tikzpicture}[scale=0.6]
 {\clip (-1.7,-2) rectangle ++(5,5);
\path [draw,name path=upward line] (-1.7,-2) -- (2.8,-2);
  \draw 
      (0, 0)              coordinate (a1)
            -- +(40:1) coordinate (a2)  
      (a2)   -- +(130:1) coordinate (a3)
      (a3)   -- +(220:1) coordinate (a4)       
      (a4)   -- +(310:1) coordinate (a5)     
(a3)  -- +(270:10)
      (a2)             
            -- +(40:1) coordinate (b2)  
      (b2)   -- +(130:1) coordinate (b3)
      (b3)   -- +(220:1) coordinate (b4)       
      (b4)   -- +(310:1) coordinate (b5)   
      (a4)   -- +(130:1) coordinate (c1)            
      (c1) -- +(40:1)coordinate (c2)    
            (c2) -- +(-50:1)coordinate (c3)            
      (c2)  -- +(270:10)
      (b3)  -- +(270:10)
        ; 
      \draw[wei2](a1) -- +(270:10);
      \draw[wei2](a1)  circle (2pt);
     \path   (0,0)  -- +(0:0.5) coordinate (d)   
      ;
        \draw[wei2](d) -- +(270:10);

        \draw 
      (d)              
               -- +(40:1) coordinate (d2)  
                (d2)   -- +(130:1) coordinate (d3)
                 (d3)   -- +(220:1) coordinate (d4)       
      (d4)   -- +(310:1) coordinate (d5) 
            (d3) -- +(270:10)
               (d2)              
               -- +(40:1) coordinate (e2)  
                  (e2)   -- +(130:1) coordinate (e3)
                     (e3)   -- +(220:1) coordinate (e4)       
(e3)  -- +(270:10)      
(e2)              
               -- +(40:1) coordinate (f2)  
                  (f2)   -- +(130:1) coordinate (f3)
                    (f3)   -- +(220:1) coordinate (f4)       
 (f3)  -- +(270:10)      
      ;  
        \draw[wei2](d)  circle (2pt); }
    \end{tikzpicture}
    \]
\caption{The diagram and loading of the bipartition $((2,1),(1^3))$ for $l=2$, $g=1$, $\theta=(0,0.5)$.}
\label{ALoading}
\end{figure}
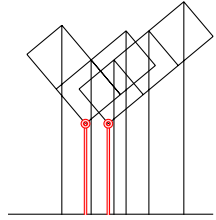
   \end{eg}

    \begin{eg}\label{previous1}
  Let $n=3$, $l=2$, $e=4$, $g=2$,  $\kappa=(0,2)$, and $\theta=(0,1)$.  Consider the block with residue 
  $\{0,1,2\}$.  This block contains 4 multipartitions, $(\emptyset,(1^3)), ((1),(1^2)), ((2),(1)),$ and $((3),\emptyset)$.  
 We record the diagrams corresponding to these partitions in Figure \ref{domorderon012}; in the cases where one of the components  is empty, we record where it would be, for perspective.  

\begin{figure}[ht]\captionsetup{width=0.9\textwidth}
 $$
        \quad    \quad 
 \begin{tikzpicture}[scale=0.6]
{\clip (-.7,-2) rectangle ++(3.8,6);
\path [draw,name path=upward line] (-1.7,-2) -- (2.8,-2);
   \draw 
      (-0.05,1.2)              coordinate (a1)
;
      \draw[wei2](a1) -- +(270:10);
\path   (0,0)  -- +(0:0.5) coordinate (d)   
      ;        \draw[wei2](d) -- +(270:10);
        \draw 
      (d)              
               -- +(30:1) coordinate (d2)  
                (d2)   -- +(120:1) coordinate (d3)
                 (d3)   -- +(210:1) coordinate (d4)       
      (d4)   -- +(300:1) coordinate (d5) 
            (d3) -- +(270:10)
               (d2)              
               -- +(30:1) coordinate (e2)  
                  (e2)   -- +(120:1) coordinate (e3)
                     (e3)   -- +(210:1) coordinate (e4)       
(e3)  -- +(270:10)      
(e2)              
               -- +(30:1) coordinate (f2)  
                  (f2)   -- +(120:1) coordinate (f3)
                    (f3)   -- +(210:1) coordinate (f4)       
 (f3)  -- +(270:10)      
      ;               \draw[wei2,fill](a1) circle (2pt);                               \draw[wei2,fill]  (d)   circle(2pt); 
}
    \end{tikzpicture}
        \quad    \quad 
 \begin{tikzpicture}[scale=0.6]
{\clip (-.7,-2) rectangle ++(3,6);
\path [draw,name path=upward line] (-1.7,-2) -- (2.8,-2);
  \draw 
      (-0.05,1.2)              coordinate (a1)
            -- +(30:1) coordinate (a2)  
      (a2)   -- +(120:1) coordinate (a3)
      (a3)   -- +(210:1) coordinate (a4)       
      (a4)   -- +(300:1) coordinate (a5)     
(a3)  -- +(270:10);
       \draw[wei2](a1) -- +(270:10);
     \path   (0,0)  -- +(0:0.5) coordinate (d)   
      ;
        \draw[wei2](d) -- +(270:10);
        \draw 
      (d)              
               -- +(30:1) coordinate (d2)  
                (d2)   -- +(120:1) coordinate (d3)
                 (d3)   -- +(210:1) coordinate (d4)       
      (d4)   -- +(300:1) coordinate (d5) 
            (d3) -- +(270:10)
               (d2)              
               -- +(30:1) coordinate (e2)  
                  (e2)   -- +(120:1) coordinate (e3)
                     (e3)   -- +(210:1) coordinate (e4)       
(e3)  -- +(270:10)     ;                                    \draw[wei2,fill]  (a1)   circle(2pt);     \draw[wei2,fill]  (d)   circle(2pt); 
}
    \end{tikzpicture}
\quad
 \begin{tikzpicture}[scale=0.6]
 { \clip (-1.1,-2) rectangle ++(2.5,6);
\path [draw,name path=upward line] (-1.7,-2) -- (2.8,-2);
  \draw 
      (-0.05,1.2)              coordinate (a1)
            -- +(30:1) coordinate (a2)  
      (a2)   -- +(120:1) coordinate (a3)
      (a3)   -- +(210:1) coordinate (a4)       
      (a4)   -- +(300:1) coordinate (a5)     
(a3)  -- +(270:10) 
      (a4)   -- +(120:1) coordinate (c1)            
      (c1) -- +(30:1)coordinate (c2)    
             (c2) -- +(-60:1)coordinate (c3)            
      (c2)  -- +(270:10)
;
      \draw[wei2](a1) -- +(270:10);
     \path   (0,0)  -- +(0:0.5) coordinate (d)   
      ;
        \draw[wei2](d) -- +(270:10);
        \draw 
      (d)              
               -- +(30:1) coordinate (d2)  
                (d2)   -- +(120:1) coordinate (d3)
                 (d3)   -- +(210:1) coordinate (d4)       
      (d4)   -- +(300:1) coordinate (d5) 
            (d3) -- +(270:10);                                 \draw[wei2,fill]  (a1)   circle(2pt);     \draw[wei2,fill]  (d)   circle(2pt); }
    \end{tikzpicture}
      \quad    \quad 
 \begin{tikzpicture}[scale=0.6]
 {\clip (-1.7,-2) rectangle ++(3,6.5);
\path [draw,name path=upward line] (-1.7,-2) -- (2.8,-2);
  \draw 
      (-0.05,1.2)              coordinate (a1)
            -- +(30:1) coordinate (a2)  
      (a2)   -- +(120:1) coordinate (a3)
      (a3)   -- +(210:1) coordinate (a4)       
      (a4)   -- +(300:1) coordinate (a5)     
(a3)  -- +(270:10) 
       (a4)   -- +(120:1) coordinate (c1)            
      (c1) -- +(30:1)coordinate (c2)    
             (c2) -- +(-60:1)coordinate (c3)            
      (c2)  -- +(270:10)
      (c1)             
            -- +(30:1) coordinate (b2)  
      (b2)   -- +(120:1) coordinate (b3)
      (b3)   -- +(210:1) coordinate (b4)       
      (b4)   -- +(300:1) coordinate (b5)   
            (b3)  -- +(270:10); 
      \draw[wei2](a1) -- +(270:10);
     \path   (0,0)  -- +(0:0.5) coordinate (d)   
      ;
          \path   (0,0)  -- +(0:0.5) coordinate (d)   
      ;
        \draw[wei2](d) -- +(270:10);
          \draw[wei2,fill]  (d)   circle(2pt);                                  \draw[wei2,fill]  (a1)   circle(2pt);      
}
    \end{tikzpicture}
    $$
  \caption{The loadings of the bipartitions of 3 with residue $\{0,1,2\}$ for $\theta=(0,1)$, $g=2$.  }
   \label{domorderon012}
   \end{figure}
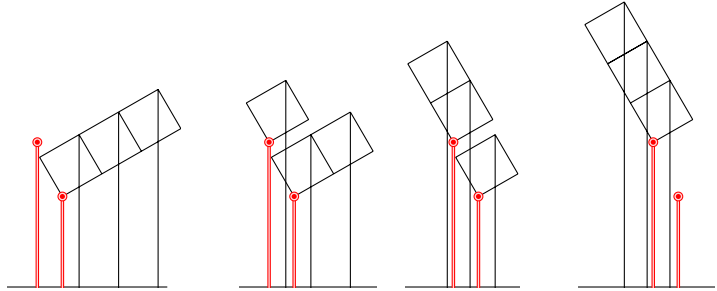
    The respective sets $D_\mu$ for the bipartitions $(\emptyset,(1^3)), ((1),(1^2)), ((2),(1)),$ and $((3),\emptyset)$, 
 are   as follows:
 \[\{1+2\epsilonLIRONdontchange,3+3\epsilonLIRONdontchange,5+4\epsilonLIRONdontchange\}
 , \{0+2\epsilonLIRONdontchange, 1+2\epsilonLIRONdontchange,3+3\epsilonLIRONdontchange\},
       \{-2+3\epsilonLIRONdontchange, 0+2\epsilonLIRONdontchange,   1+2\epsilonLIRONdontchange\}  ,
            \{-4+4\epsilonLIRONdontchange,-2+3\epsilonLIRONdontchange, 0+2\epsilonLIRONdontchange\}.  
   \] 
       \end{eg}
  
%
%
%
%
%
%
%

\begin{defn} Let $\lambda,\mu \in \mptn ln$.  A $\lambda$-tableau of weight $\mu$ is a
  bijective map $\SSTT : [\lambda] \to D_\mu$ which respects residues.  
  In other words,  we fill a given node $(r,c,m)$ of the diagram
  $[\lambda]$ with a real number $d$ from $D_\mu$ (without
  multiplicities) so that the residue attached to the real number $d$ in the loading $\mathbf{i}_\mu$
  is equal to $\kappa_m+c-r \pmod e$.    
    \end{defn}
  
  \begin{defn}
  A $\lambda$-tableau, $\SSTT$, of shape $\lambda$ and weight $\mu$ is
  said to be semistandard if  
  \begin{itemize}
  \item     $\SSTT(1,1,m)>\theta_m$,
  \item    $\SSTT(r,c,m)> \SSTT(r-1,c,m)  +\g$,
  \item 
   $\SSTT(r,c,m)> \SSTT(r,c-1,m) -\g$.
  \end{itemize}
  We  denote the set of all  semistandard tableaux of shape $\lambda$
  and weight $\mu$ by $\SStd(\lambda,\mu)$.   Given $\SSTT \in 
 \SStd(\lambda,\mu)$, we   write $\Shape(\SSTT)=\lambda$.  
  \end{defn}

 \begin{rem}\label{noambiguity}
In this paper, we only consider examples 
 of multipartitions in which each component is a hook.  
 This means that when drawing diagrams in the Russian convention, no two nodes have the same $x$-coordinate for $\epsilonLIRONdontchange=0$, therefore we omit $\epsilonLIRONdontchange$ from our tableaux and weightings without introducing ambiguity.   

 \end{rem}

 \begin{defn} 
 Let $\mathbf{i}$ and $\mathbf{j}$ denote two loadings of size $n$ and let
  $r\in \ZZ/e\ZZ$.  
  We say that  $\mathbf{i}$ dominates $\mathbf{j}$ if 
 for every real number $a \in \RR$  and every $r \in \ZZ/e\ZZ$, we have that 
\[
|\{(x,r ) \in \mathbf{i} \mid  x<a\}| \geq  |\{(x,r ) \in \mathbf{j} \mid x<a\}|.
\]
Given $\lambda,\mu \in \mptn ln$, $\theta\in \RR^l$, we say that $\lambda$
$\theta$-dominates $\mu$ (and write $\mu \trianglelefteq_{\theta}
\lambda$) if $\mathbf{i}_\lambda$ dominates $\mathbf{i}_\mu$. 
 \end{defn}

\begin{eg}
We have the following two important examples of dominance orders.  
 Let $n=3$ and $l=2$ and take $(\theta_1,\theta_2)$ so that 
$(i)$
   $0<\theta_2-\theta_1< g    $
$(ii)$  $\theta_2-\theta_1> ng$.   
We specialise $\kappa$ so that the algebra is non-semisimple. The dominance order on a given block  is  given by intersecting the   posets in Figure \ref{posets} with the set of bipartitions of a given residue class. 
The leftmost poset in Figure \ref{posets} will be of the most interest to us in this paper.  
 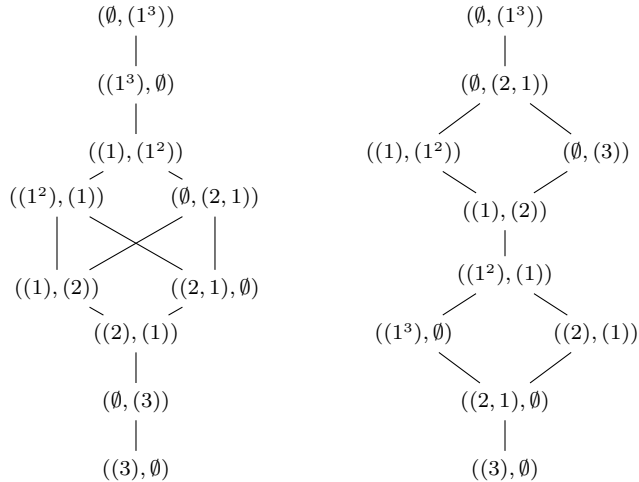
\begin{figure}[ht!]\captionsetup{width=0.9\textwidth}$$  
\scalefont{0.7}\begin{tikzpicture}[scale=0.6]
  \node (higher) at (0,5) {$(\emptyset,(1^3))$};
  \node (high) at (0,3.5) {$((1^3), \emptyset)$};
  \node (one) at (90:2cm) {$((1),(1^2))$};
  \node (b) at (150:2cm) {$((1^2),(1))$};
  \node (a) at (210:2cm) {$((1),(2)) $};
  \node (zero) at (270:2cm) {$((2),(1))$};
  \node (c) at (330:2cm) {$((2,1),\emptyset)$};
  \node (d) at (30:2cm) {$(\emptyset , (2,1))$};
  \draw (zero) -- (a) -- (b) -- (one) -- (d) -- (c) -- (zero);
    \draw (b) -- (c) ;    \draw (a) -- (d) ;
    \draw (one) -- (high) --(higher) ;
      \node (lower) at (0,-5) { $((3), \emptyset)$};
  \node (low) at (0,-3.5) {$(\emptyset,(3))$};
  \draw (zero) -- (low) --(lower) ;
\end{tikzpicture}
\quad \quad \quad \quad 
 \begin{tikzpicture}[scale=0.6]
  \node (higher) at (0,5) {$(\emptyset,(1^3))$};
  \node (high) at (0,3.5) {$( \emptyset, (2,1))$};
   \node (one) at (2,2) {$( \emptyset, (3))$};
    \node (two) at (-2,2) {$( (1), (1^2) )$};
       \node (three) at (-2,-2) {$((1^3),\emptyset)$};
    \node (four) at (2,-2) {$((2),(1))$};
        \node (five) at (0,0.7) {$( (1), (2))$};
        \node (six) at (0,-0.7) {$( (1^2), (1))$};
     \draw (two) -- (high) ;    \draw (one) -- (five) ;\draw (two) -- (five) ;
    \draw (three) -- (six) -- (five) ;
    \draw (four) -- (six) ;    \draw (four) -- (low) -- (three) ;
    \draw (one) -- (high) --(higher) ;
      \node (lower) at (0,-5) {$( (3),\emptyset)$};
  \node (low) at (0,-3.5) {$((2,1), \emptyset)$};
  \draw  (low) --(lower) ;
\end{tikzpicture}
$$
\caption{The Hasse diagrams of the posets corresponding to the FLOTW and well-separated weightings.}
\label{posets}
\end{figure}
\end{eg}

  \begin{eg}
   We continue the example above with  $n=3$, $l=2$, $e=4$, $g=2$,   $\kappa=(0,2)$ and  $\theta=(0,1)$; we consider the block with residue 
  $\{0,1,2\}$.   
 In this case, the dominance order on bipartitions of residue $\{0,1,2\}$
  is given by reading the diagrams in Figure \ref{domorderon012} from left to right in ascending order.  
  In other words
  \[
  (\emptyset,(1^3)) \vartriangleleft_\theta
   ((1),(1^2))
  \vartriangleleft_\theta ((2),(1))
  \vartriangleleft_\theta ((3),\emptyset).
  \]
 Recall the loadings of these bipartitions from Example \ref{previous1}.  
Recall that we let $\epsilonLIRONdontchange\to 0$ for ease of notation.  
   Figure \ref{semistandard} lists all  three semistandard tableaux of shape $\lambda$ and weight $\mu$ (for $\mu\neq \lambda)$ for $\lambda,\mu$ in this block.  
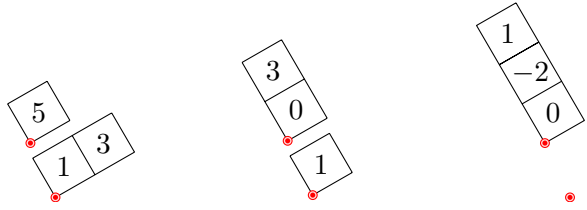
\begin{figure}[ht]\captionsetup{width=0.9\textwidth}
 $$
 \begin{tikzpicture}[scale=0.6]
{\clip (-1.6,-0.1) rectangle ++(5,4.5);
\path [draw,name path=upward line] (-1.7,-2) -- (2.8,-2);
  \draw 
      (-0.05,1.2)              coordinate (a1)
            -- +(30:1) coordinate (a2)  
      (a2)   -- +(120:1) coordinate (a3)
      (a3)   -- +(210:1) coordinate (a4)       
      (a4)   -- +(300:1) coordinate (a5)     
;
     \path   (0,0)  -- +(0:0.5) coordinate (d)   
      ;
        \draw 
      (d)              
               -- +(30:1) coordinate (d2)  
                (d2)   -- +(120:1) coordinate (d3)
                 (d3)   -- +(210:1) coordinate (d4)       
      (d4)   -- +(300:1) coordinate (d5) 
               (d2)              
               -- +(30:1) coordinate (e2)  
                  (e2)   -- +(120:1) coordinate (e3)
                     (e3)   -- +(210:1) coordinate (e4)       
;

  \draw  (a1)
    ++(30:0.5)
   ++(120:0.5) node {${5}$} ; 
      \draw  (d2)  ++(30:0.5)
   ++(120:0.5) node  {$\tiny{3}$}   ; 
            \draw  (d)  ++(30:0.5)
   ++(120:0.5) node  {$\tiny{1}$}   ; 
          \draw[wei2,fill]  (d)   circle(2pt); 
                    \draw[wei2,fill]  (a1)   circle(2pt);  }
    \end{tikzpicture}
\quad
 \begin{tikzpicture}[scale=0.6]
 {\clip (-1.6,-0.15) rectangle ++(5,4.5);
\path [draw,name path=upward line] (-1.7,-2) -- (2.8,-2);
  \draw 
      (-0.05,1.2)              coordinate (a1)
            -- +(30:1) coordinate (a2)  
      (a2)   -- +(120:1) coordinate (a3)
      (a3)   -- +(210:1) coordinate (a4)       
      (a4)   -- +(300:1) coordinate (a5)     
       (a4)   -- +(120:1) coordinate (c1)            
      (c1) -- +(30:1)coordinate (c2)    
             (c2) -- +(-60:1)coordinate (c3)            
 ;
     \path   (0,0)  -- +(0:0.5) coordinate (d)   
      ;                    \draw[wei2,fill]  (a1)   circle(2pt);  
        \draw 
      (d)              
               -- +(30:1) coordinate (d2)  
                (d2)   -- +(120:1) coordinate (d3)
                 (d3)   -- +(210:1) coordinate (d4)       
      (d4)   -- +(300:1) coordinate (d5) ;
   \draw  (a1)  ++(30:0.5)
   ++(120:0.5) node   {$\tiny{0}$}   ;
      \draw  (a4)  ++(30:0.5)
   ++(120:0.5) node  {$\tiny{3}$}   ; 
      \draw  (d)  ++(30:0.5)
   ++(120:0.5) node  {$\tiny{1}$}   ;
                \draw[wei2,fill]  (d)   circle(2pt); }
    \end{tikzpicture}
   \quad 
 \begin{tikzpicture}[scale=0.6]
 {\clip (-1.6,-0.1) rectangle ++(5,4.5);
\path [draw,name path=upward line] (-1.7,-2) -- (2.8,-2);
  \draw 
      (-0.05,1.2)              coordinate (a1)
            -- +(30:1) coordinate (a2)  
      (a2)   -- +(120:1) coordinate (a3)
      (a3)   -- +(210:1) coordinate (a4)       
      (a4)   -- +(300:1) coordinate (a5)     
       (a4)   -- +(120:1) coordinate (c1)            
      (c1) -- +(30:1)coordinate (c2)    
             (c2) -- +(-60:1)coordinate (c3)            
      (c1)             
            -- +(30:1) coordinate (b2)  
      (b2)   -- +(120:1) coordinate (b3)
      (b3)   -- +(210:1) coordinate (b4)       
      (b4)   -- +(300:1) coordinate (b5)   
;                     \draw[wei2,fill]  (a1)   circle(2pt);  
     \path   (0,0)  -- +(0:0.5) coordinate (d)   
      ;
          \path   (0,0)  -- +(0:0.5) coordinate (d)   
      ;
          \draw[wei2,fill]  (d)   circle(2pt); 
   \draw  (a1)  ++(30:0.5)
   ++(120:0.5) node  {$\tiny{0}$}   ;
      \draw  (a4)  ++(30:0.5)
   ++(120:0.5) node  {$\tiny{-2}$}   ; 
      \draw  (c1)  ++(30:0.5)
   ++(120:0.5) node  {$\tiny{1}$}   ;}
    \end{tikzpicture}
    $$
  \caption{ The tableaux in $\SStd(((1),(1^2)), (\emptyset,(1^3)))$, 
  $\SStd(((2),(1)), ((1),(1^2)))$,
  and   $\SStd(	((3),\emptyset),		((2),(1)))$, respectively.  }
   \label{semistandard}
   \end{figure}

%
%
    \end{eg}

 \begin{defn}
 We refer to an unordered  multiset  
  $\mathcal{R}$ of $n$ elements from $ (\ZZ/e\ZZ)$ as a residue set of cardinality $n$.  We let $\mptn ln(\mathcal{R})$ denote the subset of $\mptn ln$ whose residue set is equal to $\mathcal{R}$.  
  \end{defn}

\begin{rmk}
We have that $\mptn ln = \cup_{\mathcal{R} } \mptn ln(\mathcal{R})$ is a disjoint decomposition of  the set $\mptn ln$; notice that all of the above combinatorics respects this decomposition.  
\end{rmk}

 \subsection{The diagrammatic Cherednik algebra}
\label{relationspageofstuff} 

Recall that we have fixed   $l,n\in \ZZ_{>0}$,   $\g\in \RR_{>0}$
and $e\in\{3,4,\dots\}\cup\{\infty\}$. Given any weighting $\theta=(\theta_1,\dots, \theta_l)$ and $\kappa=(\kappa_1,\ldots ,\kappa_l)$ an $e$-multicharge, we will define what we refer to as the diagrammatic Cherednik algebra, $A(n,\theta,\kappa)$.

This is an example of one of many finite dimensional algebras (reduced steadied quotients of weighted KLR algebras in Webster's terminology) constructed in \cite{Webster}, whose module categories are equivalent, over the complex field, to    category $\mathcal{O}$ for rational cyclotomic Cherednik algebra \cite[Theorem 2.3 and 3.9]{Webster}.
  
\begin{defn}
We define a  $\theta$-\emph{diagram} {of type} $G(l,1,n)$ to  be a  \emph{frame}
 $\mathbb{R}\times [0,1]$ with  distinguished  black points on the northern and southern boundaries
 given by the loadings $\mathbf{i}_\mu$ and $\mathbf{i}_\lambda$ for some
  $\lambda,\mu \in \mptn ln(\mathcal{R})$  and a 
 collection of curves each of which starts at a northern point 
 and ends at a southern
  point of the same residue, $i$ say (we refer to this as a \emph{black $i$-strand}). 
   We further require that each curve has a 
 mapping diffeomorphically to $[0,1]$ via the projection to the $y$-axis.  
  Each curve is allowed to carry any number of dots.  We draw
\begin{itemize}
\item  a dashed line $g$ units to the left of each strand, which we call a \emph{ghost $i$-strand} or $i$-\emph{ghost};
\item vertical red lines at $\theta_m \in \RR$ each of which carries a 
residue $\kappa_m$ for $1\leq m\leq l$ which we call a \emph{red $\kappa_m$-strand}. 
\end{itemize}
We now require that there are no triple points or tangencies involving any combination of strands, ghosts or red lines and no dots lie on crossings. We consider these diagrams equivalent if they are related by an isotopy that avoids these tangencies, double points and dots on crossings. 
 \end{defn}
 
\begin{rmk}Note  that our diagrams  do not distinguish between `over' and `under' crossings. 
 \end{rmk}
 \begin{defn}[\cite{Webster}]
 The \emph{diagrammatic Cherednik algebra}, $A(n,\theta,\kappa)$,  is the   span of all $\theta$-diagrams modulo the following local relations (here a local relation means one that 
 can be applied on a small region of the diagram).
  \begin{enumerate}[label=(1.\arabic*)] 
 \item\label{rel1}  Any diagram may be deformed isotopically; that is, by a continuous deformation
  of the diagram which at no point introduces or removes any crossings of 
  strands (black,  ghost, or  red).  
    \item\label{rel2} 
For $i\neq j$ we have that dots pass through crossings. 
    \[   \scalefont{0.8}
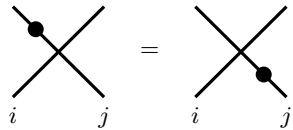
\begin{tikzpicture}[scale=.6,baseline]
      \draw[very thick](-4,0) +(-1,-1) -- +(1,1) node[below,at start]
      {$i$}; \draw[very thick](-4,0) +(1,-1) -- +(-1,1) node[below,at
      start] {$j$}; \fill (-4.5,.5) circle (5pt);
      \node at (-2,0){=}; \draw[very thick](0,0) +(-1,-1) -- +(1,1)
      node[below,at start] {$i$}; \draw[very thick](0,0) +(1,-1) --
      +(-1,1) node[below,at start] {$j$}; \fill (.5,-.5) circle (5pt);
      \node at (4,0){ };
    \end{tikzpicture}\]
\item\label{rel3}  For two like-labelled strands we get an error term.
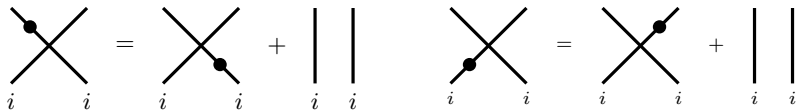
\[
\scalefont{0.8}\begin{tikzpicture}[scale=.5,baseline]
      \draw[very thick](-4,0) +(-1,-1) -- +(1,1) node[below,at start]
      {$i$}; \draw[very thick](-4,0) +(1,-1) -- +(-1,1) node[below,at
      start] {$i$}; \fill (-4.5,.5) circle (5pt);
       \node at (-2,0){=}; \draw[very thick](0,0) +(-1,-1) -- +(1,1)
      node[below,at start] {$i$}; \draw[very thick](0,0) +(1,-1) --
      +(-1,1) node[below,at start] {$i$}; \fill (.5,-.5) circle (5pt);
      \node at (2,0){+}; \draw[very thick](4,0) +(-1,-1) -- +(-1,1)
      node[below,at start] {$i$}; \draw[very thick](4,0) +(0,-1) --
      +(0,1) node[below,at start] {$i$};
    \end{tikzpicture}  \quad \quad \quad
    \scalefont{0.8}\begin{tikzpicture}[scale=.5,baseline]
      \draw[very thick](-4,0) +(-1,-1) -- +(1,1) node[below,at start]
      {$i$}; \draw[very thick](-4,0) +(1,-1) -- +(-1,1) node[below,at
      start] {$i$}; \fill (-4.5,-.5) circle (5pt);
           \node at (-2,0){=}; \draw[very thick](0,0) +(-1,-1) -- +(1,1)
      node[below,at start] {$i$}; \draw[very thick](0,0) +(1,-1) --
      +(-1,1) node[below,at start] {$i$}; \fill (.5,.5) circle (5pt);
      \node at (2,0){+}; \draw[very thick](4,0) +(-1,-1) -- +(-1,1)
      node[below,at start] {$i$}; \draw[very thick](4,0) +(0,-1) --
      +(0,1) node[below,at start] {$i$};
\end{tikzpicture}\]
\item\label{rel4} For double crossings of black strands, we have the following.
\[
\scalefont{0.8}\begin{tikzpicture}[very thick,scale=0.6,baseline]
      \draw (-2.8,0) +(0,-1) .. controls +(1.6,0) ..  +(0,1)
      node[below,at start]{$i$}; \draw (-1.2,0) +(0,-1) .. controls
      +(-1.6,0) ..  +(0,1) node[below,at start]{$i$}; \node at (-.5,0)
      {=}; \node at (0.4,0) {$0$};
    \end{tikzpicture}
\hspace{.7cm}
    \scalefont{0.8}\begin{tikzpicture}[very thick,scale=0.6,baseline]

      \draw (-2.8,0) +(0,-1) .. controls +(1.6,0) ..  +(0,1)
      node[below,at start]{$i$}; \draw (-1.2,0) +(0,-1) .. controls
      +(-1.6,0) ..  +(0,1) node[below,at start]{$j$}; \node at (-.5,0)
      {=};

\draw (1.8,0) +(0,-1) -- +(0,1) node[below,at start]{$j$};
      \draw (1,0) +(0,-1) -- +(0,1) node[below,at start]{$i$}; 
    \end{tikzpicture}\]
  \end{enumerate}
\begin{enumerate}[resume, label=(1.\arabic*)]  
\item\label{rel5} If $j\neq i-1$,  then we can pass ghosts through black strands.
 \[\begin{tikzpicture}[very thick,xscale=1,yscale=0.6,baseline]
 \draw (1,-1) to[in=-90,out=90]  node[below, at start]{$i$} (1.5,0) to[in=-90,out=90] (1,1)
;
  \draw[dashed] (1.5,-1) to[in=-90,out=90] (1,0) to[in=-90,out=90] (1.5,1);
  \draw (2.5,-1) to[in=-90,out=90]  node[below, at start]{$j$} (2,0) to[in=-90,out=90] (2.5,1);
\node at (3,0) {=};
  \draw (3.7,-1) -- (3.7,1) node[below, at start]{$i$}
 ;
  \draw[dashed] (4.2,-1) to (4.2,1);
  \draw (5.2,-1) -- (5.2,1) node[below, at start]{$j$};
\end{tikzpicture} \quad\quad \quad \quad 
 \scalefont{0.8}\begin{tikzpicture}[very thick,xscale=1,yscale=0.6,baseline]
 \draw[dashed] (1,-1) to[in=-90,out=90] (1.5,0) to[in=-90,out=90] (1,1)
;
  \draw (1.5,-1) to[in=-90,out=90] node[below, at start]{$i$} (1,0) to[in=-90,out=90] (1.5,1);
  \draw (2,-1) to[in=-90,out=90]  node[below, at start]{$\tiny j$} (2.5,0) to[in=-90,out=90] (2,1);
\node at (3,0) {=};
  \draw (3.7,-1) -- (3.7,1) node[below, at start]{$i$}
 ;
  \draw[dashed] (4.2,-1) to (4.2,1);
  \draw (5.2,-1) -- (5.2,1) node[below, at start]{$j$};
 \end{tikzpicture}
\]
\end{enumerate} \begin{enumerate}[resume, label=(1.\arabic*)]  
\item\label{rel6} On the other hand, in the case where $j= i-1$, we have the following.
  \[\scalefont{0.8}\begin{tikzpicture}[very thick,xscale=1,yscale=0.6,baseline]
 \draw (1,-1) to[in=-90,out=90]  node[below, at start]{$i$} (1.5,0) to[in=-90,out=90] (1,1)
;
  \draw[dashed] (1.5,-1) to[in=-90,out=90] (1,0) to[in=-90,out=90] (1.5,1);
  \draw (2.5,-1) to[in=-90,out=90]  node[below, at start]{$\tiny i\!-\!1$} (2,0) to[in=-90,out=90] (2.5,1);
\node at (3,0) {=};
  \draw (3.7,-1) -- (3.7,1) node[below, at start]{$i$}
 ;
  \draw[dashed] (4.2,-1) to (4.2,1);
  \draw (5.2,-1) -- (5.2,1) node[below, at start]{$\tiny i\!-\!1$} node[midway,fill,inner sep=2.5pt,circle]{};
\node at (5.75,0) {$-$};

  \draw (6.2,-1) -- (6.2,1) node[below, at start]{$i$} node[midway,fill,inner sep=2.5pt,circle]{};
  \draw[dashed] (6.7,-1)-- (6.7,1);
  \draw (7.7,-1) -- (7.7,1) node[below, at start]{$\tiny i\!-\!1$};
\end{tikzpicture}\]
   
\item\label{rel7} We also have the relation below, obtained by symmetry.  \[
\scalefont{0.8}\begin{tikzpicture}[very thick,xscale=1,yscale=0.6,baseline]
 \draw[dashed] (1,-1) to[in=-90,out=90] (1.5,0) to[in=-90,out=90] (1,1)
;
  \draw (1.5,-1) to[in=-90,out=90] node[below, at start]{$i$} (1,0) to[in=-90,out=90] (1.5,1);
  \draw (2,-1) to[in=-90,out=90]  node[below, at start]{$\tiny i\!-\!1$} (2.5,0) to[in=-90,out=90] (2,1);
\node at (3,0) {=};
  \draw[dashed] (3.7,-1) -- (3.7,1);
  \draw (4.2,-1) -- (4.2,1) node[below, at start]{$i$};
  \draw (5.2,-1) -- (5.2,1) node[below, at start]{$\tiny i\!-\!1$} node[midway,fill,inner sep=2.5pt,circle]{};
\node at (5.75,0) {$-$};

  \draw[dashed] (6.2,-1) -- (6.2,1);
  \draw (6.7,-1)-- (6.7,1) node[midway,fill,inner sep=2.5pt,circle]{} node[below, at start]{$i$};
  \draw (7.7,-1) -- (7.7,1) node[below, at start]{$\tiny i\!-\!1$};
\end{tikzpicture}\]
\end{enumerate}
  \begin{enumerate}[resume, label=(1.\arabic*)] \item\label{rel8} Strands can move through crossings of black strands freely.
\[
\scalefont{0.8}\begin{tikzpicture}[very thick,scale=0.6 ,baseline]
\draw (-3,0) +(1,-1) -- +(-1,1) node[below,at start]{$k$}; \draw
      (-3,0) +(-1,-1) -- +(1,1) node[below,at start]{$i$}; \draw
      (-3,0) +(0,-1) .. controls +(-1,0) ..  +(0,1) node[below,at
      start]{$j$}; \node at (-1,0) {=}; \draw (1,0) +(1,-1) -- +(-1,1)
      node[below,at start]{$k$}; \draw (1,0) +(-1,-1) -- +(1,1)
      node[below,at start]{$i$}; \draw (1,0) +(0,-1) .. controls
      +(1,0) ..  +(0,1) node[below,at start]{$j$};
    \end{tikzpicture}\]
  \end{enumerate}
\indent Similarly, this holds for triple points involving ghosts, 
except for the following relations when $j=i-1$.
\begin{enumerate}[resume, label=(1.\arabic*)]  \Item\label{rel9}
    \[\scalefont{0.8}\begin{tikzpicture}[very thick,xscale=1,yscale=0.6,baseline]
      \draw[dashed] (-3,0) +(.4,-1) -- +(-.4,1);
 \draw[dashed]      (-3,0) +(-.4,-1) -- +(.4,1); 
    \draw (-1.5,0) +(.4,-1) -- +(-.4,1) node[below,at start]{$\tiny j$}; \draw
      (-1.5,0) +(-.4,-1) -- +(.4,1) node[below,at start]{$\tiny j$}; 
 \draw (-3,0) +(0,-1) .. controls +(-.5,0) ..  +(0,1) node[below,at
      start]{$i$};\node at (-.75,0) {=};  \draw[dashed] (0,0) +(.4,-1) -- +(-.4,1);
 \draw[dashed]      (0,0) +(-.4,-1) -- +(.4,1); 
    \draw (1.5,0) +(.4,-1) -- +(-.4,1) node[below,at start]{$\tiny j$}; \draw
      (1.5,0) +(-.4,-1) -- +(.4,1) node[below,at start]{$\tiny j$}; 
 \draw (0,0) +(0,-1) .. controls +(.5,0) ..  +(0,1) node[below,at
      start]{$i$};
\node at (2.25,0)
      {$-$};   
     \draw (4.5,0)
      +(.4,-1) -- +(.4,1) node[below,at start]{$\tiny j$}; \draw (4.5,0)
      +(-.4,-1) -- +(-.4,1) node[below,at start]{$\tiny j$}; 
 \draw[dashed] (3,0)
      +(.4,-1) -- +(.4,1); \draw[dashed] (3,0)
      +(-.4,-1) -- +(-.4,1); 
\draw (3,0)
      +(0,-1) -- +(0,1) node[below,at start]{$i$};
     \end{tikzpicture}\]
  \Item\label{rel10}
    \[\scalefont{0.8}\begin{tikzpicture}[very thick,xscale=1,yscale=0.6,baseline]
\draw[dashed] (-3,0) +(0,-1) .. controls +(-.5,0) ..  +(0,1) ;  
  \draw (-3,0) +(.4,-1) -- +(-.4,1) node[below,at start]{$i$}; \draw
      (-3,0) +(-.4,-1) -- +(.4,1) node[below,at start]{$i$}; 
 \draw (-1.5,0) +(0,-1) .. controls +(-.5,0) ..  +(0,1) node[below,at
      start]{$\tiny j$};\node at (-.75,0) {=};  
    \draw (0,0) +(.4,-1) -- +(-.4,1) node[below,at start]{$i$}; \draw
      (0,0) +(-.4,-1) -- +(.4,1) node[below,at start]{$i$}; 
 \draw[dashed] (0,0) +(0,-1) .. controls +(.5,0) ..  +(0,1);
 \draw (1.5,0) +(0,-1) .. controls +(.5,0) ..  +(0,1) node[below,at
      start]{$\tiny j$};
\node at (2.25,0)
      {$+$};   
     \draw (3,0)
      +(.4,-1) -- +(.4,1) node[below,at start]{$i$}; \draw (3,0)
      +(-.4,-1) -- +(-.4,1) node[below,at start]{$i$}; 
\draw[dashed] (3,0)
      +(0,-1) -- +(0,1);\draw (4.5,0)
      +(0,-1) -- +(0,1) node[below,at start]{$\tiny j$};
     \end{tikzpicture}\]
  \end{enumerate}
 In the diagrams with crossings in \ref{rel9} and \ref{rel10}, we say that the black (respectively ghost) strand bypasses the crossing of ghost strands (respectively black strands).
 The ghost strands may pass through red strands freely.
  For $i\neq j$, the black $i$-strands may pass through red $j$-strands freely. If the red and black strands have the same label, a dot is added to the black strand when straightening.
  \begin{enumerate}[resume, label=(1.\arabic*)] \Item\label{rel11}
  \[\scalefont{0.8}\begin{tikzpicture}[very thick,baseline,scale=0.6]
    \draw (-2.8,0)  +(0,-1) .. controls +(1.6,0) ..  +(0,1) node[below,at start]{$i$};
       \draw[wei] (-1.2,0)  +(0,-1) .. controls +(-1.6,0) ..  +(0,1) node[below,at start]{$i$};
           \node at (-.3,0) {=};
    \draw[wei] (2.8,0)  +(0,-1) -- +(0,1) node[below,at start]{$i$};
       \draw (1.2,0)  +(0,-1) -- +(0,1) node[below,at start]{$i$};
       \fill (1.2,0) circle (3pt);
          \draw[wei] (6.8,0)  +(0,-1) .. controls +(-1.6,0) ..  +(0,1) node[below,at start]{$j$};
  \draw (5.2,0)  +(0,-1) .. controls +(1.6,0) ..  +(0,1) node[below,at start]{$i$};
           \node at (7.7,0) {=};
    \draw (9.2,0)  +(0,-1) -- +(0,1) node[below,at start]{$i$};
       \draw[wei] (10.8,0)  +(0,-1) -- +(0,1) node[below,at start]{$j$};
  \end{tikzpicture}\]
\end{enumerate}
and their mirror images. All black crossings and dots can pass through  red strands, with a
correction term.
\begin{enumerate}[resume, label=(1.\arabic*)]
\Item\label{rel12}
\[
\scalefont{0.8}\begin{tikzpicture}[very thick,baseline,scale=0.6]
      \draw (-3,0)  +(1,-1) -- +(-1,1) node[at start,below]{$i$};
      \draw (-3,0) +(-1,-1) -- +(1,1)node [at start,below]{$j$};
      \draw[wei] (-3,0)  +(0,-1) .. controls +(-1,0) ..  node[at start,below]{$k$}+(0,1);
      \node at (-1,0) {=};
      \draw (1,0)  +(1,-1) -- +(-1,1) node[at start,below]{$i$};
      \draw (1,0) +(-1,-1) -- +(1,1) node [at start,below]{$j$};
      \draw[wei] (1,0) +(0,-1) .. controls +(1,0) .. node[at start,below]{$k$} +(0,1);   
\node at (2.8,0) {$+ $};
      \draw (6.5,0)  +(1,-1) -- +(1,1)  node[at start,below]{$i$};
      \draw (6.5,0) +(-1,-1) -- +(-1,1) node [at start,below]{$j$};
      \draw[wei] (6.5,0) +(0,-1) -- node[at start,below]{$k$} +(0,1);
\node at (3.8,-.2){$\delta_{i,j,k} $}  ;
 \end{tikzpicture}
\]
\Item\label{rel13}
\[
\scalefont{0.8}\begin{tikzpicture}[scale=0.6,very thick,baseline=2cm]
      \draw[wei] (-3,3)  +(1,-1) -- +(-1,1);
      \draw (-3,3)  +(0,-1) .. controls +(-1,0) ..  +(0,1);
      \draw (-3,3) +(-1,-1) -- +(1,1);
      \node at (-1,3) {=};
      \draw[wei] (1,3)  +(1,-1) -- +(-1,1);
  \draw (1,3)  +(0,-1) .. controls +(1,0) ..  +(0,1);
      \draw (1,3) +(-1,-1) -- +(1,1); 
    
    \draw (1,3)  +(5,-1) -- +(3,1);
    \draw (1,3)  +(4,-1) .. controls +(3,0) ..  +(4,1);
    \draw[wei] (1,3) +(3,-1) -- +(5,1);
    \node at (7,3) {=};
    \draw (5,3)  +(5,-1) -- +(3,1);
    \draw (5,3)  +(4,-1) .. controls +(5,0) ..  +(4,1);
    \draw[wei] (5,3) +(3,-1) -- +(5,1);
               \end{tikzpicture}
\]

\Item\label{rel14}
\[
\scalefont{0.8}\begin{tikzpicture}[very thick,baseline,scale=0.6]
  \draw(-3,0) +(-1,-1) -- +(1,1);
  \draw[wei](-3,0) +(1,-1) -- +(-1,1);
\fill (-3.5,-.5) circle (3pt);
\node at (-1,0) {=};
 \draw(1,0) +(-1,-1) -- +(1,1);
  \draw[wei](1,0) +(1,-1) -- +(-1,1);
\fill (1.5,.5) circle (3pt);

\draw[wei](1,0) +(3,-1) -- +(5,1);
  \draw(1,0) +(5,-1) -- +(3,1);
\fill (5.5,-.5) circle (3pt);
\node at (7,0) {=};
 \draw[wei](5,0) +(3,-1) -- +(5,1);
  \draw(5,0) +(5,-1) -- +(3,1);
\fill (8.5,.5) circle (3pt);
    \end{tikzpicture}
    \]
  \end{enumerate}
Finally, we have the following non-local idempotent relation.
\begin{enumerate}[resume, label=(1.\arabic*)]
\item
Any idempotent  where the strands can be broken into two
groups separated by a blank space of size $>g$ (so no ghost from
the right-hand  group can be left of a strand in the left group and {\it vice versa}) with all double-red strands in the right-hand group is referred to as unsteady and set to be equal to zero.
\end{enumerate}
\end{defn}

\subsection{The grading on the diagrammatic Cherednik algebra}\label{grsubsec}
This algebra is graded as follows:
\begin{itemize}
\item
 dots have degree 2;
\item
 the crossing of two strands has degree 0, unless they have the
  same label, in which case it has degree $-2$;
\item the crossing of a black strand with label $i$ and a ghost has degree 1 if the ghost has label $i-1$ and 0 otherwise;
   \item
  the crossing of a black strand with a red strand has degree 0, unless they have the same label, in which case it has degree 1.
\end{itemize}
In other words,
\[
\deg\tikz[baseline,very thick,scale=1.5]{\draw
  (0,.3) -- (0,-.1) node[at end,below,scale=.8]{$i$}
  node[midway,circle,fill=black,inner
  sep=2pt]{};}=
  2 \qquad \deg\tikz[baseline,very thick,scale=1.5]
  {\draw (.2,.3) --
    (-.2,-.1) node[at end,below, scale=.8]{$i$}; \draw
    (.2,-.1) -- (-.2,.3) node[at start,below,scale=.8]{$j$};} =-2\delta_{i,j} \qquad  
  \deg\tikz[baseline,very thick,scale=1.5]{\draw[densely dashed] 
  (-.2,-.1)-- (.2,.3) node[at start,below, scale=.8]{$i$}; \draw
  (.2,-.1) -- (-.2,.3) node[at start,below,scale=.8]{$j$};} =\delta_{j,i+1} \qquad \deg\tikz[baseline,very thick,scale=1.5]{\draw (.2,.3) --
  (-.2,-.1) node[at end,below, scale=.8]{$i$}; \draw [densely dashed]
  (.2,-.1) -- (-.2,.3) node[at start,below,scale=.8]{$j$};} =\delta_{j,i-1}\]
\[
  \deg\tikz[baseline,very thick,scale=1.5]{ \draw[wei]
  (-.2,-.1)-- (.2,.3) node[at start,below, scale=.8]{$i$}; \draw
  (.2,-.1) -- (-.2,.3) node[at start,below,scale=.8]{$j$};} =\delta_{i,j} 
   \qquad \deg\tikz[baseline,very thick,scale=1.5] {\draw (.2,.3) --
  (-.2,-.1) node[at end,below, scale=.8]{$i$}; \draw[wei]   (.2,-.1) -- (-.2,.3) node[at start,below,scale=.8]{$j$};} =\delta_{j,i}.
\]

 \subsection{Representation theory of the diagrammatic Cherednik algebra}
%
%
 Given $\SSTT \in \SStd(\lambda,\mu)$, we have a 
 $\theta$-diagram $B_\SSTT$ consisting of a {\it frame} in which the  
 $n$ black strands each connecting a northern and southern distinguished point are drawn
   so that they trace out the bijection determined by $\SSTT$ in such a way that we use  the minimal number of crossings 
  without creating any bigons between pairs of strands or strands and ghosts. This diagram is not unique up to isotopy (since we have not specified how to resolve triple points), but we can choose one such diagram arbitrarily.

Given a pair of semistandard tableaux of the same shape
$(\SSTS,\SSTT)\in\SStd(\lambda,\mu)\times\SStd(\lambda,\nu)$,  we  have
a diagram  $C_{\SSTS, \SSTT}=B_\SSTS B^\ast_\SSTT$ where $B^\ast_\SSTT$
is the diagram obtained from $B_\SSTT $ by flipping it through the
horizontal axis.   
 Notice that there is a unique element   $\SSTT^\lambda\in
\SStd(\lambda,\lambda)$ and the corresponding basis element
$C_{\SSTT^\lambda,\SSTT^\lambda}$ is the idempotent in which all black strands are
vertical.  A degree function on tableaux is defined in \cite[Defintion 2.13]{Webster};
for our purposes it is enough to note that $\deg(\SSTT)=\deg(B_\SSTT)$ as we shall
always work with the $\theta$-diagrams directly. 


\begin{thm}[{\cite[Section 2.6]{Webster}}]
\label{cellularitybreedscontempt}
The algebra $A(n,\theta,\kappa)$ is a graded cellular algebra with a theory of highest weights.
The cellular basis is given by
\[
 \mathcal{C}=\{C_{\SSTS, \SSTT} \mid \SSTS \in \SStd(\lambda,\mu), \SSTT\in \SStd(\lambda,\nu), 
 \lambda,\mu, \nu \in \mptn ln\}
\]
with respect to the $\theta$-dominance order on the set   $\mptn ln$
and the anti-isomorphism given by flipping a diagram through the horizontal axis.  
  \end{thm}

 \begin{eg}
    We continue the example above with  $n=3$, $l=2$, $e=4$, $g=2$, and $\kappa=(0,2)$ and let $\theta=(0,1)$.  Consider the block with residue 
  $\{0,1,2\}$.   
In this case the  elements $B_\SSTS$ for the semistandard tableaux 
$\SSTS$ of shape $\lambda$ and weight $\mu$ for $\lambda\neq \mu$ are given in Figure \ref{basiselements} below.
   
   \begin{figure}[ht] $$
    \scalefont{0.7}
    \begin{tikzpicture}[scale=0.8] 
   \draw (-2,0) rectangle (10,2);  
  \draw[wei2] (3.6,0)--(3.6,2);
    \draw[wei2] (4.6,0)--(4.6,2);
      \node [wei2,below] at (3.6,0) {\tiny $0$};
            \node [wei2,below] at (4.6,0) {\tiny $2$}; \draw[wei2,fill]  (4.6,0)   circle(1pt);\draw[wei2,fill]  (3.6,0)   circle(1pt);\draw[wei2,fill]  (4.6,2)   circle(1pt);\draw[wei2,fill]  (3.6,2)   circle(1pt);
      \draw (9,2)  to [out=-120,in=60] (4,0) ;    \draw (5,2) -- (5,0);
       \draw (7,2) -- (7,0);
     \draw[dashed] (9-1.8,2)  to [out=-120,in=60] (4-1.8,0) ;    \draw[dashed] (5-1.8,2) -- (5-1.8,0);
       \draw[dashed] (7-1.8,2) -- (7-1.8,0);
       \node [below] at (4,0) {\tiny $0$};
            \node [below] at (5,0) {\tiny $2$};
            \node [below] at (7,0) {\tiny $1$};
 \end{tikzpicture}
    $$
     $$
    \scalefont{0.7}
    \begin{tikzpicture}[scale=0.8] 
  \draw (-2,0) rectangle (10,2);  
  \draw[wei2] (3.6,0)--(3.6,2);
    \draw[wei2] (4.6,0)--(4.6,2);\node [wei2,below] at (3.6,0) {\tiny $0$};
            \node [wei2,below] at (4.6,0) {\tiny $2$}; \draw[wei2,fill]  (4.6,0)   circle(1pt);\draw[wei2,fill]  (3.6,0)   circle(1pt);\draw[wei2,fill]  (4.6,2)   circle(1pt);\draw[wei2,fill]  (3.6,2)   circle(1pt);
      \draw (7,2)  to [out=-120,in=60] (2,0) ;   
       \draw (5,2) -- (5,0);
       \draw (4,2) -- (4,0);
      \draw[dashed] (7-1.8,2)  to [out=-120,in=60] (2-1.8,0) ;   
       \draw[dashed] (5-1.8,2) -- (5-1.8,0);
       \draw[dashed] (4-1.8,2) -- (4-1.8,0);
        \node [below] at (4,0) {\tiny $0$};
            \node [below] at (5,0) {\tiny $2$};
            \node [below] at (2,0) {\tiny $1$};       
  \end{tikzpicture}
    $$
  \[
    \scalefont{0.7}
    \begin{tikzpicture}[scale=0.8] 
  \draw (-2,0) rectangle (10,2);
  \draw[wei2] (3.6,0)--(3.6,2);\node [wei2,below] at (3.6,0) {\tiny $0$};
            \node [wei2,below] at (4.6,0) {\tiny $2$}; 
    \draw[wei2] (4.6,0)--(4.6,2);
      \draw (5,2)  to [out=-120,in=60] (0,0) ;   
       \draw (2,2) -- (2,0);
       \draw (4,2) -- (4,0);
           \draw[dashed] (5-1.8,2)  to [out=-120,in=60] (0-1.8,0) ;   
       \draw[dashed] (2-1.8,2) -- (2-1.8,0);
       \draw[dashed] (4-1.8,2) -- (4-1.8,0);
               \node [below] at (4,0) {\tiny $0$};
            \node [below] at (0,0) {\tiny $2$};
            \node [below] at (2,0) {\tiny $1$};   \draw[wei2,fill]  (4.6,0)   circle(1pt);\draw[wei2,fill]  (3.6,0)   circle(1pt);\draw[wei2,fill]  (4.6,2)   circle(1pt);\draw[wei2,fill]  (3.6,2)   circle(1pt);
  \end{tikzpicture}
    \]
 \caption{The basis elements corresponding to the tableaux in 
  $\SStd(((1),(1^2)), (\emptyset,(1^3)))$, 
  $\SStd(((2),(1)), ((1),(1^2)))$,
  and   $\SStd(	((3),\emptyset),		((2),(1)))$, respectively (see Figure \ref{semistandard}). }
\label{basiselements}
\end{figure}
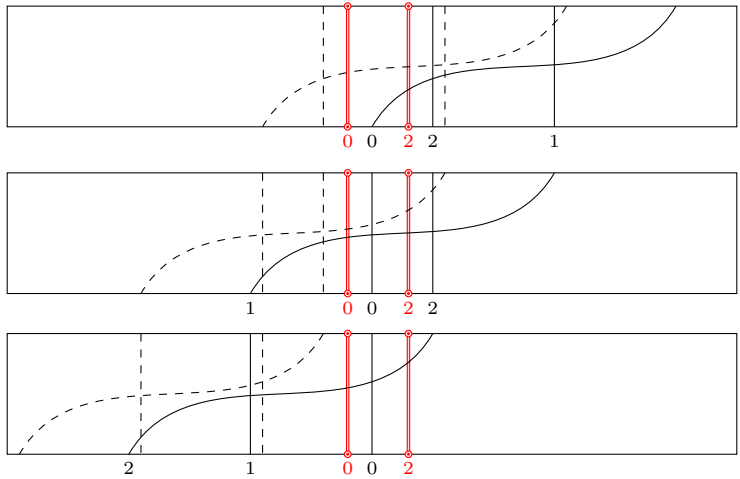

 \end{eg}

 \begin{thm}[\cite{Webster}, Theorem 6.2]\label{step1}
 The (basic algebra of the) diagrammatic  Cherednik algebra  $A(n,\theta,\kappa)$ is Koszul.
 In particular, given $\lambda\neq \mu \in \mptn ln$, the graded decomposition numbers  $d_{\lambda\mu}(t)\in t\NN_0(t)$.
 \end{thm}

\begin{rmk}Notice that  the basis of  $A(n,\theta,\kappa)$ also respects the decomposition of $\mptn ln$ 
by residue sets.  Given a residue set $\mathcal{R}$, we let  $A_\mathcal{R}(n,\theta,\kappa)$ denote the subalgebra of $A(n,\theta,\kappa)$ with basis given by all $\theta$-diagrams indexed by multipartitions $\lambda,\mu,\nu \in \mptn ln(\mathcal{R})$.  
\end{rmk}

  \subsection{An example}
  
  Let $e=2$, $l=1$,  $\g=1$,
   and $n=2$ and $\kappa=(0)$.  In this case
  we shall see that the algebra $A(n,\theta, \kappa)$ is the basic
  algebra of the Schur algebra of the Hecke algebra of type
  $G(1,1,2)$ specialised at $e=2$.     
There are two  partitions $(2)$ and $(1^2)$ with loadings $(-1+2\epsilonLIRONdontchange, \epsilonLIRONdontchange)$ and $(\epsilonLIRONdontchange, 1+2\epsilonLIRONdontchange)$ respectively, depicted in Figure \ref{MORELOADINGS}, 
 .   When discussing the combinatorics of tableaux, we will
 adopt the conventions of Remark \ref{noambiguity}.  
 
 \begin{figure}[ht]\captionsetup{width=0.9\textwidth}   $$
     \begin{tikzpicture}[scale=0.6]
     \clip (-1.7,-2.5) rectangle ++(4.3,4.7);
    \node [wei2,below] at (-0.05,-2) {\tiny $0$};
    \node [below] at (0.2,-2) {\tiny $0$};
        \node [below] at (-0.55,-2) {\tiny $1$};
 {\clip (-1.7,-2) rectangle ++(5,4.3);
\path [draw,name path=upward line] (-1.7,-2) -- (2.8,-2);
  \draw 
      (0, 0)              coordinate (a1)
            -- +(40:1) coordinate (a2)  
      (a2)   -- +(130:1) coordinate (a3)
      (a3)   -- +(220:1) coordinate (a4)       
      (a4)   -- +(310:1) coordinate (a5)     
(a3)  -- +(270:10)
      (a2)             
            -- +(130:1) coordinate (b2)  
      (b2)   -- +(130:1) coordinate (b3)
      (b3)   -- +(220:1) coordinate (b4)       
      (b4)   -- +(310:1) coordinate (b5)   
        ; 
      \draw[wei2](a1) -- +(270:10);
          \draw[wei2,fill]  (a1)   circle(1pt);  
       \draw(b3) -- +(270:10);
     \path  (0,0)  -- +(0:0.5) coordinate (d)   
      ;
}
    \end{tikzpicture}
    \quad    \quad
    \begin{tikzpicture}[scale=0.6]
     \clip (-1.7,-2.5) rectangle ++(4.3,4.7);
   \node [wei2,below] at (-0.05,-2) {\tiny $0$};    \node [below] at (0.2,-2) {\tiny $0$};
        \node [below] at (0.9,-2) {\tiny $1$};
 {\clip (-1.7,-2) rectangle ++(5,4.3);
\path [draw,name path=upward line] (-1.7,-2) -- (2.8,-2);
  \draw 
      (0, 0)              coordinate (a1)
            -- +(40:1) coordinate (a2)  
      (a2)   -- +(130:1) coordinate (a3)
      (a3)   -- +(220:1) coordinate (a4)       
      (a4)   -- +(310:1) coordinate (a5)     
(a3)  -- +(270:10)
      (a2)             
            -- +(40:1) coordinate (b2)  
      (b2)   -- +(130:1) coordinate (b3)
      (b3)   -- +(220:1) coordinate (b4)       
      (b4)   -- +(310:1) coordinate (b5)   
        ; 
      \draw[wei2](a1) -- +(270:10);
       \draw(b3) -- +(270:10);
                 \draw[wei2,fill]  (a1)   circle(1pt);  
     \path  (0,0)  -- +(0:0.5) coordinate (d)   
      ;
}
    \end{tikzpicture}
$$    \caption{Partitions and loadings for $G(1,1,2)$.}
\label{MORELOADINGS}    \end{figure}
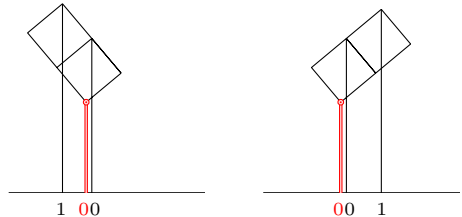 
  There is a unique element   $\SSTU\in \SStd((1^2),(1^2))$ given in Figure \ref{a few choice tableau} below, and the corresponding cell module is 1-dimensional.  The two
        elements $\SSTS, \SSTT \in \SStd((2), -)$ are also given in Figure \ref{a few choice tableau} and are of weight $(2)$ and $(1^2)$ respectively.  The left cell
 module has basis given by $B_\SSTS$ and $B_\SSTT$, depicted in Figure \ref{basisofcell}; the full 5-dimensional algebra is given by 
taking the pairs of flipped elements $C_{\SSTU\SSTU},
C_{\SSTS\SSTS},C_{\SSTS\SSTT},C_{\SSTT\SSTS},C_{\SSTT\SSTT}$.   

  \begin{figure}[ht]
   $$
        \begin{tikzpicture}[scale=1.3,rotate=45]               \YRussian\tyoung(1cm,3cm,<0.1><1.2>)  
    \end{tikzpicture}
\quad \quad 
 \begin{tikzpicture}[scale=1.3,rotate=45]
             \YRussian\tyoung(1cm,3cm,<0>,<{-1}>)       \end{tikzpicture}
 \quad     \quad      \begin{tikzpicture}[scale=1.3,rotate=45]               \YRussian\tyoung(1cm,3cm,<0>,<1>)  
    \end{tikzpicture} $$
    \caption{Semistandard tableaux for $n=2$ and $l=1$.}
        \label{a few choice tableau}
    \end{figure}
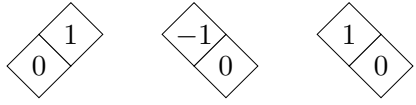

  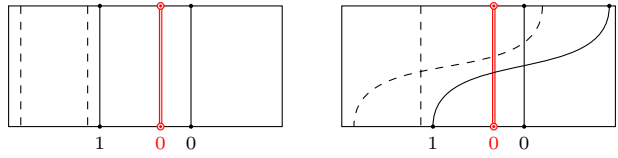
\begin{figure}[ht]\captionsetup{width=0.9\textwidth}$$
\begin{tikzpicture}[scale=0.8]
  \draw (-0.5,0) rectangle (4,2);
  \foreach \x in {1,2.5}
    {\fill (\x,2) circle (1pt);
     \fill (\x,0) circle (1pt);}
 \draw[wei2] (2,0)--(2,2);
 \draw[wei2,fill] (2,0) circle (1pt)(2,2) circle (1pt); 
     \draw (1,2) -- (1,0) ;    \draw (2.5,2) -- (2.5,0);
       \node[wei2] [below] at (2,0) {\tiny $0$};
              \node  [below] at (1,0) {\tiny $1$};
                     \node [below] at (2.5,0) {\tiny $0$};
          \draw[dashed] (0.8,2) -- (0.8,0)  (-0.3,2) -- (-0.3,0);
 \end{tikzpicture}
\quad\quad
\begin{tikzpicture}[scale=0.8]
  \draw (-0.5,0) rectangle (4,2);
  \foreach \x in {3.9,2.5}
    {\fill (\x,2) circle (1pt);}
     \foreach \x in {1,2.5}
{     \fill (\x,0) circle (1pt);}
 
 \draw[wei2] (2,0)--(2,2);
 \draw[wei2,fill] (2,0) circle (1pt)  (2,2) circle (1pt); 
     \draw (3.9,2)  to [out=-90,in=90] (1,0) ;    \draw (2.5,2) -- (2.5,0);
          \draw[dashed] (0.8,2) -- (0.8,0)  (2.8,2)  to [out=-90,in=90] (-0.3,0);
            \node[wei2] [below] at (2,0) {\tiny $0$};
              \node  [below] at (1,0) {\tiny $1$};
                     \node [below] at (2.5,0) {\tiny $0$};
 \end{tikzpicture}
$$
\caption{Basis of the cell module $\Delta(2)$.}
\label{basisofcell}
\end{figure}

 In  $B_\SSTT$, the crossing  
of the ghost strand of  residue 1 with the black strand of residue 0 
has degree 1.  All other crossings in  $B_\SSTT$ have degree 0 and therefore $\deg(B_{\SSTT})=1$.  
We shall now show that $[\Delta(2):L(1^2)]=t$.  
To see this, it is
enough to check that $ B_\SSTT $ is in the radical of  $ \Delta(2)$,
in other words the products $B_\SSTS^\ast B_\SSTT$ and
$B_\SSTT^\ast B_\SSTT$ are zero in the cell module. 

The first case is obvious as we  are multiplying two non-equal
idempotents, in the second case we apply relations $(2.6)$ and (2.11) followed by (2.7) to the product $B_\SSTT^\ast B_\SSTT$ as in Figure \ref{multiplystuff}.
This product   is equal to zero, as the centre of the final diagram in Figure \ref{multiplystuff} is an
unsteady idempotent.
  
  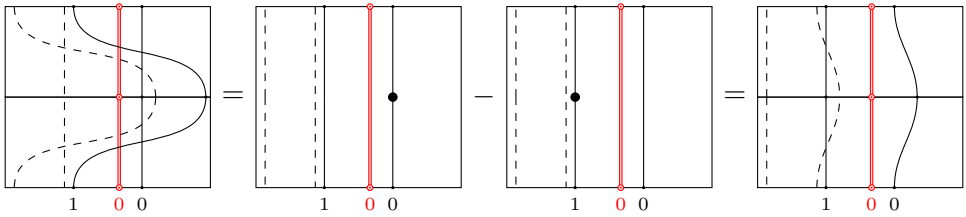
\begin{figure}[ht]\captionsetup{width=0.9\textwidth}
  \[
  \begin{tikzpicture}[scale=0.6]
  \draw (-0.5,0) rectangle (4,2);
   \draw (-0.5,2) rectangle (4,4);
  \foreach \x in {3.9,2.5}
    {\fill (\x,2) circle (1pt);}
     \foreach \x in {1,2.5}
{     \fill (\x,0) circle (1pt);\fill (\x,4) circle (1pt);}
 
 \draw[wei2] (2,0)--(2,4);
 \draw[wei2,fill] (2,0) circle (1pt)  (2,2) circle (1pt) (2,4) circle (1pt); 
     \draw (3.9,2)  to [out=-90,in=90] (1,0) ; 
      \draw (3.9,2)  to [out=90,in=-90] (1,4) ; 
        \draw (2.5,4) -- (2.5,0);
          \draw[dashed]   (2.8,2)  to [out=-90,in=90] (-0.3,0);
  \draw[dashed] (0.8,4) -- (0.8,0)  (2.8,2)  to [out=90,in=-90] (-0.3,4);
            \node[wei2] [below] at (2,0) {\tiny $0$};
              \node  [below] at (1,0) {\tiny $1$};
                     \node [below] at (2.5,0) {\tiny $0$};
\node at (4.5,2){=}; 
  \draw (5.5+-0.5,0) rectangle (5.5+4,4);
     \foreach \x in {5.5+1,5.5+2.5}
{     \fill (\x,0) circle (1pt);\fill (\x,4) circle (1pt);}
      \fill (5.5+2.5,2) circle (3pt);
 \draw[wei2] (5.5+2,0)--(5.5+2,4);
 \draw[wei2,fill] (5.5+2,0) circle (1pt)   (5.5+2,4) circle (1pt); 
     \draw (5.5+1,2)  to [out=-90,in=90] (5.5+1,0) ; 
      \draw (5.5+1,2)  to [out=90,in=-90] (5.5+1,4) ; 
        \draw (5.5+2.5,4) -- (5.5+2.5,0);
          \draw[dashed]   (5.5+-0.3,2)  to [out=-90,in=90] (5.5+-0.3,0);
  \draw[dashed] (5.5+0.8,4) -- (5.5+0.8,0)  (5.5+-0.3,2)  to [out=90,in=-90] (5.5+-0.3,4);
            \node[wei2] [below] at (5.5+2,0) {\tiny $0$};
              \node  [below] at (5.5+1,0) {\tiny $1$};
                     \node [below] at (5.5+2.5,0) {\tiny $0$};
  \draw (11+-0.5,0) rectangle (11+4,4);
     \foreach \x in {11+1,11+2.5}
{     \fill (\x,0) circle (1pt);\fill (\x,4) circle (1pt);}
      \fill (11+1,2) circle (3pt);
 \draw[wei2] (11+2,0)--(11+2,4);
 \draw[wei2,fill] (11+2,0) circle (1pt)   (11+2,4) circle (1pt); 
     \draw (11+1,2)  to [out=-90,in=90] (11+1,0) ; 
      \draw (11+1,2)  to [out=90,in=-90] (11+1,4) ; 
        \draw (11+2.5,4) -- (11+2.5,0);
          \draw[dashed]   (11+-0.3,2)  to [out=-90,in=90] (11+-0.3,0);
  \draw[dashed] (11+0.8,4) -- (11+0.8,0)  (11+-0.3,2)  to [out=90,in=-90] (11+-0.3,4);
            \node[wei2] [below] at (11+2,0) {\tiny $0$};
              \node  [below] at (11+1,0) {\tiny $1$};
                     \node [below] at (11+2.5,0) {\tiny $0$};
                                           \node at (15.5,2){$=$};
   \draw (16.5+-0.5,0) rectangle (16.5+4,2);
   \draw (16.5+-0.5,2) rectangle (16.5+4,4);
  \foreach \x in {16.5+1,16.5+3}
    {\fill (\x,2) circle (1pt);}
     \foreach \x in {16.5+1,16.5+2.5}
{     \fill (\x,0) circle (1pt);\fill (\x,4) circle (1pt);}
 
 \draw[wei2] (16.5+2,0)--(16.5+2,4);
 \draw[wei2,fill] (16.5+2,0) circle (1pt)  (16.5+2,2) circle (1pt) (16.5+2,4) circle (1pt); 
        \draw (16.5+1,4) -- (16.5+1,0);
       \draw (16.5+2.5+0.5,2)  to [out=-90,in=90] (16.5+2.5,0) ; 
      \draw (16.5+2.5+0.5,2)  to [out=90,in=-90] (16.5+2.5,4) ; 
       \draw[dashed] (16.5+0.8,0)  to [out=90,in=-90]  (16.5+0.8+0.5,2)  ; 
      \draw[dashed] (16.5+0.8,4)   to [out=-90,in=90] (16.5+0.8+0.5,2)  ; 
          \draw[dashed]   (16.5-0.3,2)  to [out=-90,in=90] (16.5+-0.3,0);
  \draw[dashed]  
   (16.5-0.3,2)  to [out=90,in=-90] (16.5+-0.3,4);
            \node[wei2] [below] at (16.5+2,0) {\tiny $0$};
              \node  [below] at (16.5+1,0) {\tiny $1$};
                     \node [below] at (16.5+2.5,0) {\tiny $0$};
                     \node at (10,2){$-$};
 \end{tikzpicture}
  \]
  \caption{The product $(B_\SSTT)^\ast(B_\SSTT)$.  
  The first equality follows from relations (2.6) and (2.11); the second  follows from relation (2.7).  The final diagram is zero as the centre is an unsteady idempotent.}
  \label{multiplystuff}
\end{figure}
  
\section{{The quiver Temperley--Lieb algebra of type $G(l,1,n)$  }}
   In this section we define the quiver Temperley--Lieb algebra of type $G(l,1,n)$, which we denote by
  $\TL_n(\kappa)$ for $\kappa \in (\ZZ/e\ZZ)^l$ and $n\in \mathbb{N}$.   
 Given $l \in \NN$,   and  $\g\in\RR$, 
 we   take $\theta \in \RR^l$  such that $0<\theta_j-\theta_i < g $
   for all $1\leq i < j\leq l$.

  We let $\pi $ denote the set of all multipartitions  all of whose components have at most one column.   We refer to this as the set of one-column multipartitions.
 We
   let \[e_{\pi}=\sum_{\lambda\not\in\pi}C_{\SSTT^\lambda\SSTT^\lambda}.\]   
We shall see in the proof of Proposition \ref{step2} below that any such choice of
 $\theta$ guarantees that  $\pi$ is a saturated subset of $\mptn ln$.  
 We shall fix a choice of $\theta$ below; however, given any such choice the resulting algebras are isomorphic (see Proof of Proposition \ref{step2}, below).  
\begin{defn}
Fix $e\in\{3,4,\dots\}\cup\{\infty\}$ and integers $l \leq e/2 $ and  $n\geq 1$.  
  Fix a multicharge $ \kappa\in (\ZZ/e\ZZ)^l$ such  that 
  $\kappa_i\not\in \{
  \kappa_j ,  \kappa_j +1  \}$
   for $i\neq j$ and $\theta$ as above.      
   We define the \emph{quiver Temperley--Lieb algebra of type $G(l,1,n)$}, denoted $\TL_n(\kappa)$,
to be the  algebra   
\[\TL_n(\kappa)=  A(n,\theta,\kappa)/(A(n,\theta,\kappa)  e_\pi A(n,\theta,\kappa)).\]
\end{defn}

\noindent \textbf{Notation:} For the remainder of the paper, given $\TL_n(\kappa)$ with $\kappa \in (\ZZ/e\ZZ)^l$, we shall
fix  $\g=l$ and
 $\theta=(0,1,2,\ldots, l-1)$ (this choice is easily seen to satisfy the conditions above), and 
 adopt the conventions of Remark \ref{noambiguity}.
 With this choice made, the loadings  of multipartitions in $\pi$ have a very simple form.  
 Namely,  any $l$-partition of $n$ is of the form
 \[ {\lambda}=(1^{\lambda_1}, 1^{\lambda_2}, \ldots , 1^{\lambda_l})\in \pi\]
  (with $\sum_{i=1}^l \lambda_i = n$) and has loading 
 \[
 \{(i-1) + jl \mid \lambda_i\neq 0 \text{ and }1\leq j\leq \lambda_i  \}.\]
 Given $\theta$ as above, we refer to the $\theta$-dominance order on multipartitions as the FLOTW dominance order.

\begin{prop}\label{step2}
The quiver Temperley--Lieb algebra of type $G(l,1,n)$ is a   graded cellular algebra with a theory of highest weights.  
The cellular basis is given by
 \[
 \{C_{\SSTS\SSTT}\mid \SSTS \in \SStd(\lambda,\mu), \, \SSTT\in \SStd(\lambda,\nu), \, \lambda,\mu, \nu \in \pi\},
 \]
 with respect to the FLOTW dominance order on the set of one-column multipartitions, $\pi$.  We have that $d_{\lambda\mu}(t)\in t \mathbb{N}_0[t]$ for $\lambda \neq \mu$ elements of   $ \pi$.  
\end{prop}
\begin{proof} Fix  $\theta$ such that  $0<\theta_j-\theta_i < g $ for $1 \leq i < j\leq l$.  
We shall show that the set $\pi$ is saturated in the $\theta$-dominance order.  In other words given any $\lambda\in \pi$ and 
$\mu \lhd_\theta \lambda$, we have that $\mu \in \pi$.  
This will imply that 
 \[
 \langle C_{\SSTS\SSTT} \mid \SSTS \in \SStd(\lambda,\mu), \, \SSTT\in \SStd(\lambda,\nu), \, \lambda,\mu, \nu \not\in \pi\rangle_{\CC}
 \]
is an ideal of $A(n,\theta,\kappa)$ (the ideal generated by $e_\pi$, in fact)
and the resulting quotient has the desired basis (by conditions (2) and (3)  of  
Definition \ref{defn1} and Theorem  \ref{cellularitybreedscontempt}).  
The graded decomposition numbers (as well as dimensions of higher extension groups) are preserved under this quotient, 
see for example \cite[Appendix]{Donkin} for the ungraded case. Applying Theorem \ref{cellularitybreedscontempt} will thus prove  the claim about graded decomposition numbers.

Given any two choices  $\theta^{(1)}$ and  $\theta^{(2)}$ satisfying the above, 
the combinatorics of tableaux are identical.  
This results in a bijection between the cell-bases  of the algebras $A(n,\theta^{(1)},\kappa)$ and 
$A(n,\theta^{(2)},\kappa)$.   
 The basis elements identified under this bijection may be obtained from one another by isotopy. 
Therefore this is an isomorphism of algebras, via relation (2.1) of  Section \ref{relationspageofstuff}.

Our choice of $\theta$ implies that if we add a box to the second
column of any component $\lambda^{(m)}$ (that is, add a node $(1,2,m)$ for some $m$), this box  has $x$-coordinate
strictly less than all boxes in the first column of all components, and thus the resulting multipartition is more dominant.
Therefore the set of one-column  multipartitions is saturated. 
 \end{proof}

\begin{defn}
Let $\lambda$ be a one-column multipartition $(1^{\lambda_1},1^{\lambda_2}, \ldots, 1^{\lambda_l})$.   
A node  of $\lambda$ is \emph{removable} if it can be removed from the diagram of $\lambda$ to leave the  diagram of a (one-column) multipartition, while a node not in the diagram of $\lambda$ is an \emph{addable} node of $\lambda$ if it can be added to the diagram of $\lambda$ to give the   diagram of a one-column multipartition. 

 If the node has residue $r \in \ZZ/e\ZZ$, we say that the node is $r$-\emph{removable} or $r$-\emph{addable}.
Given $\lambda \in \pi$ and $r\in \ZZ/e\ZZ$, we let ${\rm Add} (\lambda,r)$ denote the set of $1\leq j\leq l$ such that there is an $r$-addable node in the $j$th component of  $\lambda$.  
 
\end{defn}

 In the previous section, we refrained from defining the degree of a general tableau.  This was because  of the technicalities in defining \emph{addable} and \emph{removable} nodes for such tableaux (see \cite[Section 2.2]{Webster}).  These difficulties do not appear for tableaux corresponding to one-column multipartitions.

\begin{defn} Suppose $\lambda\in \pi $
 and $\square$ is a removable $A$-node of $[\lambda]$. Set  
\begin{align*}
 d_\lambda(\square)=&
  \big|\{\text{addable $A$-nodes of $\lambda$ to the right of }\square\}\big| 
\\ &  -\big|\{\text{removable $A$-nodes of $\lambda$ to the right of }\square\}\big|.
\end{align*}
Given $1\leq k \leq n$ and $\SSTT \in\SStd(\lambda,\mu)$,
 we let $\SSTT_k$ denote the node of 
$[\lambda]$ containing the entry $D_\mu(k)$ and we let 
$\SSTT_{\leq k}$ denote the tableau consisting of the
nodes with entries in $D_\mu\{1,2,\ldots,k\}$.  
%

For $\SSTT\in\SStd(\lambda,-)$ we define the \emph{degree} of $\SSTT$ recursively, setting $\deg(\SSTT) =0$ when $\SSTT$ is the unique $\emptyset$-tableau.  
 We set
 \[
\deg(\SSTT) = d_\lambda(\SSTT_n)  +\deg(\SSTT_{<n}).
\]
\end{defn}

  \begin{eg}
    Let  $e=4$, $l=2$, $n=7$,  and $\kappa = (0,2)$. 
 By tableau-linkage, it is clear that any residue class
decomposes as a sum of  blocks of $\TL_n(\kappa)$.  
Fix the residue class to be $\{0,0,1,2,2,3,3\}$.  The one-line multipartitions with these residues for our given value $e$-multicharge are 
$$
\{((1^7),(0)),
((1^6),(1)),
((1^3),(1^4)),
((1^2),(1^5))
 \}.
$$
Formally, the loading  of the multipartition $\lambda=((1^7),(0))$ is  
$$D_\lambda=\{0+\epsilonLIRONdontchange,2+2\epsilonLIRONdontchange,4+3\epsilonLIRONdontchange,6+4\epsilonLIRONdontchange,8+5 \epsilonLIRONdontchange,10+6\epsilonLIRONdontchange,12+7\epsilonLIRONdontchange\}.$$  
With the  conventions of Remark \ref{noambiguity} in place, our loadings are
$$
(0,2,4,6,8,10,12), \quad 
(0,1,2,4,6,8,10),\quad
(0,1,2,3,4,5,7),\quad
(0,1,2,3,5,7,9).
$$
The semistandard tableaux of shape $(1^3,1^4)$ are 
given in Figure \ref{ASHOWLBUNCH OF STUFF}, along with their degrees.  For example, the nodes in the rightmost diagram are of degree 0  
except for  
 those containing the integers
$4$ and $12$, which are of degree 1.  Therefore the rightmost tableau has degree 2.

\begin{figure}[ht]
$$
\begin{tikzpicture}[scale=0.5]
{\clip (-.7,-0.2) rectangle ++(4.8,3.8);
\path [draw,name path=upward line] (-1.7,-2) -- (2.8,-2);
  \draw 
      (-0.05,1.2)              coordinate (a1)
            -- +(30:1) coordinate (a2)  
      (a2)   -- +(120:1) coordinate (a3)
      (a3)   -- +(210:1) coordinate (a4)       
      (a4)   -- +(300:1) coordinate (a5)  
;
\draw      (a2)             
            -- +(30:1) coordinate (b2)  
      (b2)   -- +(120:1) coordinate (b3)
      (b3)   -- +(210:1) coordinate (b4)       
      (b4)   -- +(300:1) coordinate (b5)   
       ;        
       \draw      (b2)             
            -- +(30:1) coordinate (c2)  
      (c2)   -- +(120:1) coordinate (c3)
      (c3)   -- +(210:1) coordinate (c4)       
      (c4)   -- +(300:1) coordinate (c5)   
       ;        
     \path   (0,0)  -- +(0:0.5) coordinate (d)   
      ;
        \draw 
      (d)              
               -- +(30:1) coordinate (d2)  
                (d2)   -- +(120:1) coordinate (d3)
                 (d3)   -- +(210:1) coordinate (d4)       
      (d4)   -- +(300:1) coordinate (d5) 
               (d2)              
               -- +(30:1) coordinate (e2)  
                  (e2)   -- +(120:1) coordinate (e3)
                     (e3)   -- +(210:1) coordinate (e4)       
; 
 \draw (e2)              
               -- +(30:1) coordinate (f2)  
                  (f2)   -- +(120:1) coordinate (f3)
                    (f3)   -- +(210:1) coordinate (f4) ;
                    \draw (f2)              
               -- +(30:1) coordinate (g2)  
                  (g2)   -- +(120:1) coordinate (g3)
                    (g3)   -- +(210:1) coordinate (g4)       
      ;   
      \draw  (a1)  ++(30:0.5)
   ++(120:0.5) node  {$0$}   ;
  \draw  (a2)  ++(30:0.5)
   ++(120:0.5) node  {$2$}   ;        
        \draw  (b2)  ++(30:0.5)
   ++(120:0.5) node  {$4$}   ;
                \draw  (d)  ++(30:0.5)
   ++(120:0.5) node  {$1$}   ;
        \draw  (d2)  ++(30:0.5)
   ++(120:0.5) node  {$3$}   ;
        \draw  (e2)  ++(30:0.5)
   ++(120:0.5) node  {$5$}   ;
                \draw  (f2)  ++(30:0.5)
   ++(120:0.5) node  {$7$}   ;
        }
    \end{tikzpicture}
\quad
\begin{tikzpicture}[scale=0.5]
{\clip (-.7,-0.2) rectangle ++(4.8,3.8);
\path [draw,name path=upward line] (-1.7,-2) -- (2.8,-2);
  \draw 
      (-0.05,1.2)              coordinate (a1)
            -- +(30:1) coordinate (a2)  
      (a2)   -- +(120:1) coordinate (a3)
      (a3)   -- +(210:1) coordinate (a4)       
      (a4)   -- +(300:1) coordinate (a5)  
;
\draw      (a2)             
            -- +(30:1) coordinate (b2)  
      (b2)   -- +(120:1) coordinate (b3)
      (b3)   -- +(210:1) coordinate (b4)       
      (b4)   -- +(300:1) coordinate (b5)   
       ;        
       \draw      (b2)             
            -- +(30:1) coordinate (c2)  
      (c2)   -- +(120:1) coordinate (c3)
      (c3)   -- +(210:1) coordinate (c4)       
      (c4)   -- +(300:1) coordinate (c5)   
       ;        
     \path   (0,0)  -- +(0:0.5) coordinate (d)   
      ;
        \draw 
      (d)              
               -- +(30:1) coordinate (d2)  
                (d2)   -- +(120:1) coordinate (d3)
                 (d3)   -- +(210:1) coordinate (d4)       
      (d4)   -- +(300:1) coordinate (d5) 
               (d2)              
               -- +(30:1) coordinate (e2)  
                  (e2)   -- +(120:1) coordinate (e3)
                     (e3)   -- +(210:1) coordinate (e4)       
; 
 \draw (e2)              
               -- +(30:1) coordinate (f2)  
                  (f2)   -- +(120:1) coordinate (f3)
                    (f3)   -- +(210:1) coordinate (f4) ;
                    \draw (f2)              
               -- +(30:1) coordinate (g2)  
                  (g2)   -- +(120:1) coordinate (g3)
                    (g3)   -- +(210:1) coordinate (g4)       
      ;   
      \draw  (a1)  ++(30:0.5)
   ++(120:0.5) node  {$0$}   ;
  \draw  (a2)  ++(30:0.5)
   ++(120:0.5) node  {$2$}   ;        
        \draw  (b2)  ++(30:0.5)
   ++(120:0.5) node  {$9$}   ;
                \draw  (d)  ++(30:0.5)
   ++(120:0.5) node  {$1$}   ;
        \draw  (d2)  ++(30:0.5)
   ++(120:0.5) node  {$3$}   ;
        \draw  (e2)  ++(30:0.5)
   ++(120:0.5) node  {$5$}   ;
                \draw  (f2)  ++(30:0.5)
   ++(120:0.5) node  {$7$}   ;
        }
    \end{tikzpicture}
\quad
\begin{tikzpicture}[scale=0.5]
{\clip (-.7,-0.2) rectangle ++(4.8,3.8);
\path [draw,name path=upward line] (-1.7,-2) -- (2.8,-2);
  \draw 
      (-0.05,1.2)              coordinate (a1)
            -- +(30:1) coordinate (a2)  
      (a2)   -- +(120:1) coordinate (a3)
      (a3)   -- +(210:1) coordinate (a4)       
      (a4)   -- +(300:1) coordinate (a5)  
;
\draw      (a2)             
            -- +(30:1) coordinate (b2)  
      (b2)   -- +(120:1) coordinate (b3)
      (b3)   -- +(210:1) coordinate (b4)       
      (b4)   -- +(300:1) coordinate (b5)   
       ;        
       \draw      (b2)             
            -- +(30:1) coordinate (c2)  
      (c2)   -- +(120:1) coordinate (c3)
      (c3)   -- +(210:1) coordinate (c4)       
      (c4)   -- +(300:1) coordinate (c5)   
       ;        
     \path   (0,0)  -- +(0:0.5) coordinate (d)   
      ;
        \draw 
      (d)              
               -- +(30:1) coordinate (d2)  
                (d2)   -- +(120:1) coordinate (d3)
                 (d3)   -- +(210:1) coordinate (d4)       
      (d4)   -- +(300:1) coordinate (d5) 
               (d2)              
               -- +(30:1) coordinate (e2)  
                  (e2)   -- +(120:1) coordinate (e3)
                     (e3)   -- +(210:1) coordinate (e4)       
; 
 \draw (e2)              
               -- +(30:1) coordinate (f2)  
                  (f2)   -- +(120:1) coordinate (f3)
                    (f3)   -- +(210:1) coordinate (f4) ;
                    \draw (f2)              
               -- +(30:1) coordinate (g2)  
                  (g2)   -- +(120:1) coordinate (g3)
                    (g3)   -- +(210:1) coordinate (g4)       
      ;   
      \draw  (a1)  ++(30:0.5)
   ++(120:0.5) node  {$0$}   ;
  \draw  (a2)  ++(30:0.5)
   ++(120:0.5) node  {$2$}   ;        
        \draw  (b2)  ++(30:0.5)
   ++(120:0.5) node  {$4$}   ;
                \draw  (d)  ++(30:0.5)
   ++(120:0.5) node  {$1$}   ;
        \draw  (d2)  ++(30:0.5)
   ++(120:0.5) node  {$6$}   ;
        \draw  (e2)  ++(30:0.5)
   ++(120:0.5) node  {$8$}   ;
                \draw  (f2)  ++(30:0.5)
   ++(120:0.5) node  {$10$}   ;
        }
    \end{tikzpicture}
\quad
\begin{tikzpicture}[scale=0.5]
{\clip (-.7,-0.2) rectangle ++(4.8,3.8);
\path [draw,name path=upward line] (-1.7,-2) -- (2.8,-2);
  \draw 
      (-0.05,1.2)              coordinate (a1)
            -- +(30:1) coordinate (a2)  
      (a2)   -- +(120:1) coordinate (a3)
      (a3)   -- +(210:1) coordinate (a4)       
      (a4)   -- +(300:1) coordinate (a5)  
;
\draw      (a2)             
            -- +(30:1) coordinate (b2)  
      (b2)   -- +(120:1) coordinate (b3)
      (b3)   -- +(210:1) coordinate (b4)       
      (b4)   -- +(300:1) coordinate (b5)   
       ;        
       \draw      (b2)             
            -- +(30:1) coordinate (c2)  
      (c2)   -- +(120:1) coordinate (c3)
      (c3)   -- +(210:1) coordinate (c4)       
      (c4)   -- +(300:1) coordinate (c5)   
       ;        
     \path   (0,0)  -- +(0:0.5) coordinate (d)   
      ;
        \draw 
      (d)              
               -- +(30:1) coordinate (d2)  
                (d2)   -- +(120:1) coordinate (d3)
                 (d3)   -- +(210:1) coordinate (d4)       
      (d4)   -- +(300:1) coordinate (d5) 
               (d2)              
               -- +(30:1) coordinate (e2)  
                  (e2)   -- +(120:1) coordinate (e3)
                     (e3)   -- +(210:1) coordinate (e4)       
; 
 \draw (e2)              
               -- +(30:1) coordinate (f2)  
                  (f2)   -- +(120:1) coordinate (f3)
                    (f3)   -- +(210:1) coordinate (f4) ;
                    \draw (f2)              
               -- +(30:1) coordinate (g2)  
                  (g2)   -- +(120:1) coordinate (g3)
                    (g3)   -- +(210:1) coordinate (g4)       
      ;   
      \draw  (a1)  ++(30:0.5)
   ++(120:0.5) node  {$0$}   ;
  \draw  (a2)  ++(30:0.5)
   ++(120:0.5) node  {$2$}   ;        
        \draw  (b2)  ++(30:0.5)
   ++(120:0.5) node  {$12$}   ;
                \draw  (d)  ++(30:0.5)
   ++(120:0.5) node  {$4$}   ;
        \draw  (d2)  ++(30:0.5)
   ++(120:0.5) node  {$6$}   ;
        \draw  (e2)  ++(30:0.5)
   ++(120:0.5) node  {$8$}   ;
                \draw  (f2)  ++(30:0.5)
   ++(120:0.5) node  {$10$}   ;
        }
    \end{tikzpicture}
$$  
\caption{These semistandard tableaux are of weights $((1^3),(1^4))$, $((1^2),(1^5))$, $((1^6),(1))$ and $((1^7),\emptyset)$ respectively.  
The tableaux are of degrees 0, 1, 1, and 2 respectively.
} 
\label{ASHOWLBUNCH OF STUFF}
\end{figure}
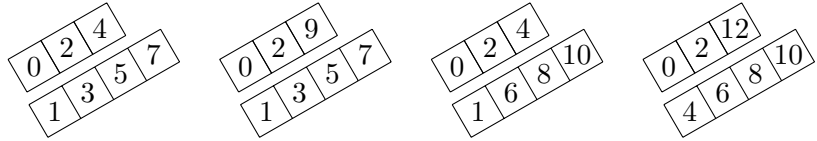
\end{eg}

\subsection{The geometry} Fix   integers $n,l,e\in\NN$ and $\kappa \in (\ZZ/e\ZZ)^l$.  
Let $\Phi_{l-1}$ be a root system of type $A_{l-1}$ with simple roots
$$\{\varepsilon_i-\varepsilon_j: 1\leq i < j \leq l\},$$ and let $W_{l}^e$ denote the corresponding affine Weyl group, generated by the affine reflections 
$s_{i,j,me}$ with $1\leq i<j \leq l$ and $m\in\ZZ$ and which acts on
$E_l$ via 
$$s_{i,j,me}(x)=x-(\langle
x,\varepsilon_i-\varepsilon_j\rangle-me)(\varepsilon_i-\varepsilon_j).$$
  Given a multicharge $\kappa=(\kappa_1,\ldots,\kappa_l)$
  we let $\rho=(e-\kappa_1,\ldots, e-\kappa_l)$.  Given an 
element $w\in W_l^e$ we set 
$$w\cdot_{\rho}x=w(x+\rho)-\rho. $$
 We identify   $\lambda$ an $l$-partition of $n$ with a point in the hyperplane $V$ 
 of $  E_l$ consisting of all points the sum of which is $n$. 
 This is done via the map 
 $ (1^{\lambda_1},\ldots,1^{\lambda_l})\mapsto
 \sum_i \lambda_i \varepsilon_i$.  

\begin{lem}\label{useful-lemma}Given $\lambda\in \pi$, we have that 
 $$ 
 \langle \lambda +\rho , \varepsilon_i-\varepsilon_j \rangle = me
 $$ 
 for some $m\in\ZZ$, if and only if the addable nodes in the $i$th and $j$th components  of the multipartition $\lambda$ have the same residue.  
\end{lem}
\begin{proof}
 To see this, note that both statements are equivalent to
 \[
 (\lambda_i+e-\kappa_i) \equiv (\lambda_j+e-\kappa_j) \pmod e. \qedhere
 \]
 \end{proof}
 
\begin{defn} Given $\SSTT\in\SStd(\lambda,\mu)$, we define the \emph{component word} $R(\SSTT)$ of 
$\SSTT$ to be given by reading the entries of the tableau in numerical order
  and recording the components in which they appear.
 We define the path  $\omega(\SSTT)$ to be the associated path in the alcove geometry. 
\end{defn}

\begin{eg}
    Let  $e=4$, $l=2$, $n=7$,  and $\kappa = (0,2)$. Let $\SSTT\in\SStd(((1^3),(1^4)),((1^7),\emptyset))$ be the following tableau.
$$
 \begin{tikzpicture}[scale=0.5]
{
\clip (-.7,-0.2) rectangle ++(4.8,3.8);
\path [draw,name path=upward line] (-1.7,-2) -- (2.8,-2);
  \draw 
      (-0.05,1.2)              coordinate (a1)
            -- +(30:1) coordinate (a2)  
      (a2)   -- +(120:1) coordinate (a3)
      (a3)   -- +(210:1) coordinate (a4)       
      (a4)   -- +(300:1) coordinate (a5)  
;
\draw      (a2)             
            -- +(30:1) coordinate (b2)  
      (b2)   -- +(120:1) coordinate (b3)
      (b3)   -- +(210:1) coordinate (b4)       
      (b4)   -- +(300:1) coordinate (b5)   
       ;        
       \draw      (b2)             
            -- +(30:1) coordinate (c2)  
      (c2)   -- +(120:1) coordinate (c3)
      (c3)   -- +(210:1) coordinate (c4)       
      (c4)   -- +(300:1) coordinate (c5)   
       ;        
     \path   (0,0)  -- +(0:0.5) coordinate (d)   
      ;
        \draw 
      (d)              
               -- +(30:1) coordinate (d2)  
                (d2)   -- +(120:1) coordinate (d3)
                 (d3)   -- +(210:1) coordinate (d4)       
      (d4)   -- +(300:1) coordinate (d5) 
               (d2)              
               -- +(30:1) coordinate (e2)  
                  (e2)   -- +(120:1) coordinate (e3)
                     (e3)   -- +(210:1) coordinate (e4)       
; 
 \draw (e2)              
               -- +(30:1) coordinate (f2)  
                  (f2)   -- +(120:1) coordinate (f3)
                    (f3)   -- +(210:1) coordinate (f4) ;
                    \draw (f2)              
               -- +(30:1) coordinate (g2)  
                  (g2)   -- +(120:1) coordinate (g3)
                    (g3)   -- +(210:1) coordinate (g4)       
      ;   
      \draw  (a1)  ++(30:0.5)
   ++(120:0.5) node  {$0$}   ;
  \draw  (a2)  ++(30:0.5)
   ++(120:0.5) node  {$2$}   ;        
        \draw  (b2)  ++(30:0.5)
   ++(120:0.5) node  {$12$}   ;
                \draw  (d)  ++(30:0.5)
   ++(120:0.5) node  {$4$}   ;
        \draw  (d2)  ++(30:0.5)
   ++(120:0.5) node  {$6$}   ;
        \draw  (e2)  ++(30:0.5)
   ++(120:0.5) node  {$8$}   ;
                \draw  (f2)  ++(30:0.5)
   ++(120:0.5) node  {$10$}   ;
        }
    \end{tikzpicture}  $$
    The component  word, $R(\SSTT)$, is 
    $(1,1,2,2,2,2,1)$.
  The path  $\omega(\SSTT)$ is pictured in Figure \ref{another path}.
  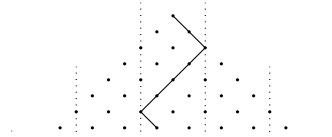
\begin{figure}[ht]\captionsetup{width=0.9\textwidth}
 $$  \begin{tikzpicture}[scale=0.6]
{ \clip (-4,-2.6) rectangle ++(8,3);
\path [draw,name path=upward line] (-3,-5) -- (3,-5);
  \path 
      (0, 0)              coordinate (a1)
        (a1)    -- +(-45:0.5) coordinate (a2) 
        (a1)    -- +(-135:0.5) coordinate (a3) 
        (a2)    -- +(-45:0.5) coordinate (b1) 
        (a2)    -- +(-135:0.5) coordinate (b2) 
         (a3)    -- +(-135:0.5) coordinate (b3) 
  (b1)    -- +(-45:0.5) coordinate (c1) 
        (b1)    -- +(-135:0.5) coordinate (c2) 
  (b3)    -- +(-45:0.5) coordinate (c3) 
        (b3)    -- +(-135:0.5) coordinate (c4) 
  (c1)    -- +(-45:0.5) coordinate (d1) 
        (c1)    -- +(-135:0.5) coordinate (d2) 
  (c3)    -- +(-45:0.5) coordinate (d3) 
        (c3)    -- +(-135:0.5) coordinate (d4) 
                (c4)    -- +(-135:0.5) coordinate (d5) 
  (d1)    -- +(-45:0.5) coordinate (e1) 
        (d1)    -- +(-135:0.5) coordinate (e2) 
  (d3)    -- +(-45:0.5) coordinate (e3) 
        (d3)    -- +(-135:0.5) coordinate (e4) 
                (d5)    -- +(-135:0.5) coordinate (e6) 
                                (d5)    -- +(-45:0.5) coordinate (e5) 
  (e1)    -- +(-45:0.5) coordinate (f1) 
        (e1)    -- +(-135:0.5) coordinate (f2) 
  (e3)    -- +(-45:0.5) coordinate (f3) 
        (e3)    -- +(-135:0.5) coordinate (f4) 
                (e5)    -- +(-45:0.5) coordinate (f5) 
                (e5)    -- +(-135:0.5) coordinate (f6) 
                                (e6)    -- +(-135:0.5) coordinate (f7) 
  (f1)    -- +(-45:0.5) coordinate (g1) 
        (f1)    -- +(-135:0.5) coordinate (g2) 
  (f3)    -- +(-45:0.5) coordinate (g3) 
        (f3)    -- +(-135:0.5) coordinate (g4) 
                (f5)    -- +(-45:0.5) coordinate (g5) 
                (f5)    -- +(-135:0.5) coordinate (g6) 
                                (f6)    -- +(-135:0.5) coordinate (g7)          
                                                                (f7)    -- +(-135:0.5) coordinate (g8)        
          (g1)    -- +(-45:0.5) coordinate (h1) 
        (g1)    -- +(-135:0.5) coordinate (h2) 
  (g3)    -- +(-45:0.5) coordinate (h3) 
        (g3)    -- +(-135:0.5) coordinate (h4) 
                (g5)    -- +(-45:0.5) coordinate (h5) 
                (g5)    -- +(-135:0.5) coordinate (h6) 
                                (g6)    -- +(-135:0.5) coordinate (h7)          
                                                                (g7)    -- +(-135:0.5) coordinate (h8)           
   ;
            \fill(a1) circle (1pt);             \fill(a2) circle (1pt); \fill(a3) circle (1pt); 
              \fill(b1) circle (1pt);                 \fill(b2) circle (1pt);                          \fill(b3) circle (1pt);   
               \fill(c1) circle (1pt);                 \fill(c2) circle (1pt);                          \fill(c3) circle (1pt);   \fill(c4) circle (1pt);   
                              \fill(d1) circle (1pt);                 \fill(d2) circle (1pt);                          \fill(d3) circle (1pt);   \fill(d4) circle (1pt);   \fill(d5) circle (1pt);   
                                    \fill(e1) circle (1pt);                 \fill(e2) circle (1pt);                          \fill(e3) circle (1pt);   \fill(e4) circle (1pt);   
\fill(e5) circle (1pt);     \fill(e6) circle (1pt);   
                                                                        \fill(f1) circle (1pt);                 \fill(f2) circle (1pt);                          \fill(f3) circle (1pt);   \fill(f4) circle (1pt);   \fill(f5) circle (1pt);     \fill(f6) circle (1pt);    \fill(f7) circle (1pt);   
    \fill(g1) circle (1pt);                 \fill(g2) circle (1pt);                          \fill(g3) circle (1pt);   \fill(g4) circle (1pt);   \fill(g5) circle (1pt);     \fill(g6) circle (1pt);    \fill(g7) circle (1pt);   
 \fill(g8) circle (1pt);   
         \draw[dotted](b3) -- +(270:10);       \draw[dotted](b3) -- +(90:1);
        \draw[dotted](b1) -- +(270:10);       \draw[dotted](b1) -- +(90:1);
        \draw[dotted](f7) -- +(270:10);        \draw[dotted](f7) -- +(90:1);
         \draw[dotted](f1) -- +(270:10);        \draw[dotted](f1) -- +(90:1);
 \draw[dotted](j1) -- +(270:10);        \draw[dotted](j1) -- +(90:1);
 \draw[dotted](j11) -- +(270:10);        \draw[dotted](j11) -- +(90:1);       ; 
          \fill(i10) circle (1pt);     \fill(i9) circle (1pt);   \fill(h9) circle (1pt);  
             \fill(j9) circle (1pt);   \fill(j10) circle (1pt);  \fill(j11) circle (1pt);  
             \fill(k9) circle (1pt);               \fill(k10) circle (1pt);      \fill(k11) circle (1pt);               \fill(k12) circle (1pt);  
\draw(a1) --  (b1);   \draw(b1) --  (f5); \draw(g5) --  (f5);
   }
    \end{tikzpicture}$$
    \caption{The path $\omega(\SSTT)\in\Path((3,4),(7,0))$.}
    \label{another path}
\end{figure}\end{eg}

Given the unique  $\SSTT^\mu \in \SStd(\mu,\mu)$, it is clear that 
 $\omega(\SSTT^\mu)=\omega^\mu$ is the path corresponding to the word $w:\{1,\ldots,n\}\to \{1,\ldots,l\}$  
given by 
 $$ 
 w(1)=\min\{i\mid\mu_i \neq0\}
 $$
 and for $i>1$, $$w(i)=(w({i-1}) + j)  \text{ modulo $l$ }$$
 where $  j\geq 1$ is   minimal such that 
 $ \langle \omega^\mu({i-1}) +\rho , \varepsilon_{w({i-1})+j}	\rangle < \mu_{w({i-1})+j}$ where the subscripts are also read modulo $l$.

%
 \begin{eg}
As in the introduction,   let $l=3$, $n=13$, $e=8$, $\kappa=(0,4,6)$. 
 For $\mu=(5,6,2)$ the component word of $\SSTT^\mu$ is  
 $$
 (
 1,  2,  3,
  1,  2,  3,
   1,  2,   
   1,2,1,2,2
 ).
$$
 \end{eg}

\begin{prop}\label{step3}
Given $\mu \in E_l$, we fix the distinguished path $\omega(\SSTT^\mu)$ as above.  We have that  $\omega$ defines a 
 bijective map
 $$
\omega :\SStd(\lambda,\mu) \to \Path(\lambda,\mu).
 $$ 
\end{prop}

\begin{proof}
The map $\omega$ is clearly an injective map, it remains to show 
  that both sets have the same size. 
The sets $\SStd(\mu,\mu)$ and  $\Path(\mu,\mu)$
 each possess a unique element   $\SSTT^\mu$, respectively $\omega^\mu$.  
For $1\leq k \leq n$, let $r(k)$ denote the residue of the node $\SSTT^\mu_k$ and 
let $t(k)$ denote the component of the $l$-partition in which this node is added. For $1\leq k \leq n$,  it follows by Lemma \ref{useful-lemma} that   
 $$
{\rm Add}(\Shape(\SSTT^\mu_{\leq k-1}),r(k))=
 \{ i \mid   \omega^\mu(k) 
 \in h_{\varepsilon_i-\varepsilon_{t(k)}, m_{ik}e}				\text{ for some }  m_{ik}\in\ZZ	\}.
$$
We let $d_k$ denote the cardinality of this set.  

We construct both $\SSTT\in\SStd(-,\mu)$ and $\omega\in \Path(-,\mu)$ step-by-step; in
the former case, by adding one node at a time to the tableau
 and in the latter case
 by taking one step at a time in the geometry.

The number of choices to be made at the $k$th point in the tableau 
  is equal to $d_k$, for $1\leq k \leq n$.  Therefore the number of tableaux of weight $\mu$ is 
  equal to $d_1d_2\ldots d_n$.  
On the other hand, in the notation of Section \ref{pathsinanaclove}, any path $\omega \in \Path(-,\mu)$ may be written as
\[
\omega =
s_{\varepsilon_{i(1)} - \varepsilon_{t(1)},m_{i1}e}^1
\ldots 
s_{\varepsilon_{i(n)} - \varepsilon_{t(n)},m_{in}e}^n
\omega^\mu
\]  for $i(k)\in {\rm Add}(\Shape(\SSTT^\mu_{\leq k-1}),r(k))$ and $m_{ik}\in\ZZ$
(of course, if $d_k=1$, the reflection is necessarily trivial).  
The number of such paths is equal to the number of  distinct possible series of reflections, 
$d_1\ldots d_n$.  
\end{proof}

\begin{cor}
If  $\lambda,\mu \in \pi$, label simple modules in the same $\TL_n(\kappa)$-block, this implies that their images in $E_n$ are in the same $W_{l-1}^e$ orbit under the $\rho$-shifted action. 
\end{cor}
\begin{proof}
This follows from Proposition \ref{linakge},  as it is easy to see that the equivalence classes of the relation generated by $\lambda \sim \mu$ if  $\Path(\lambda,\mu)\neq \emptyset$ are the same as the $W_{l-1}^e$-orbits. 
\end{proof}

\begin{lem}\label{useful-lemma2} 
Let $\lambda=(1^{\lambda_1},\ldots,1^{\lambda_l}) \in \pi$ be   such that 
$ {\lambda_i} > {\lambda_j} $ for some $1\leq i,j \leq l$ and suppose that
the residues of the addable nodes in  $i$th and $j$th components   of $\lambda$ are equal to
 $r\in(\ZZ/e\ZZ)$.  

Then $\lambda\in E_l$ lies on a hyperplane of the form $x_{i}-x_{j}=m_{ij}e$ for some $m_{ij} \in \ZZ$. 
We have that $(\lambda+\varepsilon_i) \in E_l^+(\varepsilon_i-\varepsilon_j,m_{ij}e)$ and 
 $(\lambda+\varepsilon_j) \in E_l^-(\varepsilon_i-\varepsilon_j,m_{ij}e)$.  
\end{lem}
\begin{proof}
We have seen that $\lambda$ lies on a hyperplane by Lemma \ref{useful-lemma}. We have assumed that  $ \lambda_i > \lambda_j $, and so 
 $$ 
 \langle \lambda+\rho +\varepsilon_i   , \varepsilon_i-\varepsilon_j \rangle 
 > \langle \lambda +\rho   , \varepsilon_i-\varepsilon_j\rangle 
 $$
  $$ 
 \langle \lambda+\rho +\varepsilon_j   , \varepsilon_i-\varepsilon_j \rangle 
 < \langle \lambda+\rho   , \varepsilon_i-\varepsilon_j\rangle 
 $$
 as required. 
 \end{proof}

\begin{prop}\label{step3.5}
The map $\omega:\SStd(\lambda,\mu) \to \Path(\lambda,\mu)$ is degree preserving.
\end{prop}

\begin{proof}
We fix a tableau $\SSTT \in \SStd(\lambda,\mu)$ and let $\omega:=\omega(\SSTT)$ denote the corresponding element of $\Path(\lambda,\mu)$. 
For $1\leq k \leq n$, we truncate to consider the path of length $k-1$ (respectively tableau with $k-1$ nodes),
  $\omega_{\leq \ik-1}$ (respectively $\SSTT_{\leq k-1}$) and identify this with
  the multipartition $\Shape(\SSTT_{\leq k-1}) \in \pi$.  

Let $r_k$ denote the residue of the addable node $\SSTT_k$ and let $t(k)$ denote the
component in which this node is added.   
By the definition of the Soergel-degree, we are interested in the cases
 where $1\leq i \leq l$ is such that
\begin{itemize}
\item[$(i)$] 
  $\omega(k-1)\in h_{\varepsilon_i-\varepsilon_{t(k)},m_{ik}e}$
 	and $\omega(k) \in E^-_l(\varepsilon_i-\varepsilon_{t(k)},m_{ik}e)$ for some $m_{ik}\in \ZZ$,  
\item[$(ii)$] 
   $\omega(k-1)\in E^+_l(\varepsilon_i-\varepsilon_{t(k)},m_{ik}e)$
 	and $\omega(k) \in  h_{\varepsilon_i-\varepsilon_{t(k)}
	,m_{ik}e}$ for some $m_{ik}\in \ZZ$.  
\end{itemize}
By Lemma \ref{useful-lemma2}, the 
$1\leq i \leq l$ above  label the components of  
\begin{itemize}
\item[$(i)$]   the  $r_k$-addable nodes of $\SSTT_{\leq k-1}$
		 to the right of  $\SSTT_k$, 
\item[$(ii)$]  the  $(r_k-1)$-addable nodes of $\SSTT_{\leq k-1}$
		 to the right of  $\SSTT_k $.
		 \end{itemize}
We observe that, because of the condition $\kappa_i\not\in \{
  \kappa_j, \kappa_j +1 \}$ for $i\neq j$, the set of $1\leq i \leq l$ which label  $(r_k-1)$-addable nodes of $\SSTT_{\leq k}$
		 to the right of  $\SSTT_k$ is equal to the set of 
		 $r_k$-removable nodes of $\SSTT_{\leq k-1}$  to the right of  $\SSTT_{k}$.
Therefore the result follows.  
\end{proof}

\begin{prop}\label{step4}
Given an $e$-regular $\mu\in\mptn ln$, the path $\omega^\mu$ is admissible.
\end{prop}

\begin{proof}
 It is clear that $\deg(\omega_{\leq \ik})=0$ for $1\leq k \leq n$.  
Now assume that
 $\omega^\mu(k)$ lies on two
 (or more) hyperplanes $x_i-x_j=m_1e$ and  $x_{i'}-x_{j'}=m_2e$ for some $1 \leq k\leq n$ and $m_1,m_2\in\ZZ$. We will show that 
$i,j,i',j'$ are necessarily distinct.  
 

To prove the claim, we recall our description of $\omega^\mu$.  Let $r_k$ denote the residue of the addable node $\SSTT_k$ and let $t(k)$ denote the
component in which this node is added.
It is clear that the result holds for $k=0$, we proceed by induction.  
For $1\leq k \leq n$, assume $\omega^\mu(k)$ lies on the hyperplane $h_{\varepsilon_i - \varepsilon_{t(k)},m_{ik}e} $ for some $m_{ik}\in \ZZ$.  
Our assumption on $\kappa$ ensures that $\kappa_{t(k)} \neq \kappa_j,\kappa_j\pm 1$
for any $1\leq j \leq l$.
 This implies that if $\langle \omega^\mu(k)+\rho  , \varepsilon_{t(k)}-\varepsilon_j \rangle \equiv 0 \pmod e$
 for any $1\leq j \leq l$, then $\langle \omega^\mu(k) +\rho, \varepsilon_j \rangle =
 \langle \omega^\mu+\rho, \varepsilon_j \rangle$.  
Our assumption that $\mu$ is $e$-regular implies that there is a maximum of one such value of $1\leq j \leq l$.  The result follows. 
 \end{proof}
  
  \begin{thm}\label{mainresult} 
The   algebra $\TL_n(\kappa)$  for $\kappa \in (\ZZ/e\ZZ)^l$ 
has  a  {Soergel-path basis} of type $\hat{A}_{l-1}$.  
 The  graded decomposition numbers   of an $e$-regular  block 
  are given by   Soergel's algorithm
\begin{align*} 
d_{\lambda\mu}	(t)= n_\mu({\lambda}),  \end{align*}
and the characters of the simple modules are given by the character algorithm
\[
\Dim{(L_\mu(\lambda))} = e_\mu(\lambda).
\]
  \end{thm}
  \begin{proof}
  This follows from Theorems \ref{main:theorem:general} and \ref{step1} and Propositions \ref{step2}, \ref{step3},  \ref{step3.5}, and \ref{step4}.
  \end{proof}
  
  We also observe the following stability in the decomposition numbers as $n$ tends to infinity.  
Fix $n,l \in \NN$.  
 Given $\lambda$ a one-column multipartition of $n$ and $i\geq 0$, we let 
 $\lambda+{(1^i,\ldots, 1^i)}$ denote the one-column multipartition of $n+il$ obtained by adding $i$ boxes to every component of $\lambda$.  This defines an injective map from  multipartitions of $n$ to
    multipartitions of $n'=n+il$.  
These points may be identified with points    in the hyperplanes 
$ \varepsilon_1+\dots +\varepsilon_l=n$ 
and 
$ \varepsilon_1+\dots +\varepsilon_l=n+il$ 
of   $E_{l}$, respectively.  
We identify points in these two hyperplanes via the projection in the direction $ \varepsilon_1+\dots +\varepsilon_l$.  
%
%
%
%

  \begin{thm}\label{limiting}
 The decomposition numbers of  
 $\TL_n(\kappa)$  for $\kappa \in (\ZZ/e\ZZ)^l$ are stable as $n$ tends to infinity.  
 To be more precise,
 $$d_{\lambda\mu}(t)=d_{\lambda+(1^i,\ldots, 1^i), \mu+(1^i,\ldots, 1^i)}(t)$$
 for $i\geq 0$.
  \end{thm}
  \begin{proof}
  Given $\omega\in \Path(\lambda,\mu)$ we let 
 $\omega' \in \Path(\lambda+(1^i,\ldots, 1^i),\mu+(1^i,\ldots, 1^i))$
 denote the concatenated path 
 $$(\varepsilon_1,\varepsilon_2,\ldots ,\varepsilon_l)^i\circ \omega.$$
It is clear that this map is a degree preserving bijection.  The result follows.  
  \end{proof}

    \begin{cor}\label{mainresultforchered} 
  Fix a multicharge $ \kappa\in (\ZZ/e\ZZ)^l$ such  that 
  $\kappa_i\not\in \{
  \kappa_j ,  \kappa_j +1  \}$
   for all $i\neq j$.          Let $\theta\in\RR^l$ denote a FLOTW weighting.  
 Let $\lambda,\mu$ denote a pair of $e$-regular one-column multipartitions. 
 The  graded decomposition numbers  for $A(n,\theta,\kappa)$
  are  
\begin{align*} 
d_{\lambda\mu}	(t)= n_\mu({\lambda}),
 \end{align*}
where $n_\mu({\lambda})$ is the associated 
   affine Kazhdan--Lusztig polynomial of type $\hat{A}_{l-1}$.  
   These decomposition numbers are stable in the limit as $n$ tends to infinity as in Theorem \ref{limiting} above.
  \end{cor}

 \begin{rem}\label{remakonblob}
This Soergel path basis  contains a vast amount of information
concerning  the representation theory of the    quiver Temperley--Lieb algebras of type $G(l,1,n)$.
 We have already seen that it provides a new interpretation for 
 Soergel's algorithm for computing the decomposition numbers of $\TL_n(\kappa)$.
 In the   next section we shall consider the $l=2$ case, 
   calculate the full submodule structure of the standard modules
 of $\TL_n(\kappa)$ for $\kappa\in (\ZZ/e\ZZ)^2$, and  show that the  algebra is positively graded.

 We have already remarked that our approach to the algebras
  $\TL_n(\kappa)$ is  heavily inspired by the combinatorics of 
\cite{MW03}. 
 In \cite{MW03} it is    conjectured that the   decomposition numbers
  of  the generalised blob algebras are given by 
 the  same  Kazhdan--Lusztig polynomials as those considered here.  
  Our algebra is a quotient of the diagrammatic Cherednik algebra, whereas the generalised blob algebra is 
 the corresponding quotient of the Ariki--Koike algebra.  
For a fixed  weighting $\theta$, the standard/Specht modules of these algebras have the same
  labelling set; however,  there is no known
 cellular basis for  the Ariki--Koike algebra with respect to the $\theta$-dominance order (except when $\theta$ is well-separated,  see \cite{Jacon}) and hence no way to relate the representation theories of the generalised blob
   and   Ariki--Koike algebras via an analogue of Proposition \ref{step2}.  
   Moreover, the resulting quotient algebra would not be amenable to our methods
   as it does not possess a Soergel path basis (for example,  
   for $l=2$ the blob algebra is not positively graded, \cite{Pla13}).  
   However we do believe that the generalised blob algebras are (graded) Morita equivalent to the corresponding quiver Temperley--Lieb algebras.  
 \end{rem}
 
  \subsection{The level two case} For $l=2$, the structure of the standard modules for $\TL_n(\kappa)$ labelled by $e$-regular points is particularly simple.   The proofs in this section are lightly sketched, but augmented with illustrative examples.  
 
 We remark that the submodule lattices obtained here are identical to those computed for the blob algebra in \cite{MW00}. This provides further evidence that the quiver Temperley--Lieb algebras are (graded) Morita equivalent to the generalised blob algebras.
 
Let $\mathfrak{a}_i$ denote the alcove of length $i=\ell( \mathfrak{a}_i) $ to the right of the origin
and 
$\mathfrak{a}_{i'}$ denote the alcove of length $i=\ell( \mathfrak{a}_{i'}) $ to the left of the origin, as depicted in the examples below.  
Fix a point $\lambda_0$ in the alcove containing the origin.  We let 
$\lambda_i$ and $\lambda_{i^{(\prime)}}$ denote the points in alcoves $\mathfrak{a}_i$ and
$\mathfrak{a}_{i'}$ which are in the same orbit as $\lambda_0$.
 For ease of notation, we often identify $\lambda_{i^{(\prime)}}$ with the 
 subscript ${i}^{(\prime)}$.

\begin{prop}
For $\kappa \in (\ZZ/e\ZZ)^2$, the algebra $\TL_n(\kappa)$ is positively graded. 
 For $\lambda_{i^{(\prime)}}$ and $\lambda_{j^{(\prime)}}$ in $E_2$, we have that
$$
d_{ {i^{(\prime)}} {j^{(\prime)}}}(t) = n_{j^{(\prime)}}(i^{(\prime)})=
\begin{cases}
  t^{j-i} & \text{ for $i<j$}, \\
0					& \text{otherwise}.  
\end{cases}  $$
\end{prop}

\begin{proof}
 Positivity   follows as our paths start at $\astrosun$ and 
the root system is of rank 1.  
The closed form for $n_{j^{(\prime)}}(i^{(\prime)})$ is well-known (see for example the introduction to \cite{MW03}).
 \end{proof}

%
%
%
%

%

\begin{rmk}
The algebra $\TL_n(\kappa)$  is not positively graded for  $l\geq 3$, as seen in Example \ref{exampleofnonpositive}.  
\end{rmk}
 
Given any pair $
\lambda_{i^{(\prime)}}, \lambda_{j^{(\prime)}}$ with $i<j$, there exists a unique element 
of $\Path(\lambda_{i^{(\prime)}}, \lambda_{j^{(\prime)}})$ of maximal degree equal to 
$j-i$.  
This is the unique path terminating at $\lambda_{i^{(\prime)}}$ which  may be  obtained from the distinguished path from $\astrosun$ to ${\lambda_{j^{(\prime)}}}$
 using the maximum number of   reflections in the hyperplanes $\overline{\mathfrak{a}}_0 \cap \overline{\mathfrak{a}}_1$
 and $\overline{\mathfrak{a}}_0 \cap \overline{\mathfrak{a}}_{1'}$.

\begin{eg} Let $n=11$, $e=4$ and $\kappa=(0,2)$.  Some of  the maximal paths in $\Path((5,6),-)$ are given in Figures
  \ref{22law} and \ref{11law}.  

\begin{figure}[ht]\captionsetup{width=0.9\textwidth}
$$\scalefont{0.8}  \begin{tikzpicture}[scale=0.6]
{ \clip (-4.5,-4.8) rectangle ++(9,5);
\path [draw,name path=upward line] (-3,-5) -- (3,-5);
  \path 
      (0, 0)              coordinate (a1)
        (a1)    -- +(-45:0.5) coordinate (a2) 
        (a1)    -- +(-135:0.5) coordinate (a3) 
        (a2)    -- +(-45:0.5) coordinate (b1) 
        (a2)    -- +(-135:0.5) coordinate (b2) 
         (a3)    -- +(-135:0.5) coordinate (b3) 
  (b1)    -- +(-45:0.5) coordinate (c1) 
        (b1)    -- +(-135:0.5) coordinate (c2) 
  (b3)    -- +(-45:0.5) coordinate (c3) 
        (b3)    -- +(-135:0.5) coordinate (c4) 
  (c1)    -- +(-45:0.5) coordinate (d1) 
        (c1)    -- +(-135:0.5) coordinate (d2) 
  (c3)    -- +(-45:0.5) coordinate (d3) 
        (c3)    -- +(-135:0.5) coordinate (d4) 
                (c4)    -- +(-135:0.5) coordinate (d5) 
  (d1)    -- +(-45:0.5) coordinate (e1) 
        (d1)    -- +(-135:0.5) coordinate (e2) 
  (d3)    -- +(-45:0.5) coordinate (e3) 
        (d3)    -- +(-135:0.5) coordinate (e4) 
                (d5)    -- +(-135:0.5) coordinate (e6) 
                                (d5)    -- +(-45:0.5) coordinate (e5) 
  (e1)    -- +(-45:0.5) coordinate (f1) 
        (e1)    -- +(-135:0.5) coordinate (f2) 
  (e3)    -- +(-45:0.5) coordinate (f3) 
        (e3)    -- +(-135:0.5) coordinate (f4) 
                (e5)    -- +(-45:0.5) coordinate (f5) 
                (e5)    -- +(-135:0.5) coordinate (f6) 
                                (e6)    -- +(-135:0.5) coordinate (f7) 
  (f1)    -- +(-45:0.5) coordinate (g1) 
        (f1)    -- +(-135:0.5) coordinate (g2) 
  (f3)    -- +(-45:0.5) coordinate (g3) 
        (f3)    -- +(-135:0.5) coordinate (g4) 
                (f5)    -- +(-45:0.5) coordinate (g5) 
                (f5)    -- +(-135:0.5) coordinate (g6) 
                                (f6)    -- +(-135:0.5) coordinate (g7)          
                                                                (f7)    -- +(-135:0.5) coordinate (g8)        
          (g1)    -- +(-45:0.5) coordinate (h1) 
        (g1)    -- +(-135:0.5) coordinate (h2) 
  (g3)    -- +(-45:0.5) coordinate (h3) 
        (g3)    -- +(-135:0.5) coordinate (h4) 
                (g5)    -- +(-45:0.5) coordinate (h5) 
                (g5)    -- +(-135:0.5) coordinate (h6) 
                                (g6)    -- +(-135:0.5) coordinate (h7)          
                                                                (g7)    -- +(-135:0.5) coordinate (h8)           
                                                                       (h1)    -- +(-45:0.5) coordinate (i1) 
        (h1)    -- +(-135:0.5) coordinate (i2) 
  (h3)    -- +(-45:0.5) coordinate (i3) 
        (h3)    -- +(-135:0.5) coordinate (i4) 
                (h5)    -- +(-45:0.5) coordinate (i5) 
                (h5)    -- +(-135:0.5) coordinate (i6) 
                                (h6)    -- +(-135:0.5) coordinate (i7)          
                                                                (h7)    -- +(-135:0.5) coordinate (i8)           
                            (i1)    -- +(-45:0.5) coordinate (j1) 
        (i1)    -- +(-135:0.5) coordinate (j2) 
  (i3)    -- +(-45:0.5) coordinate (j3) 
        (i3)    -- +(-135:0.5) coordinate (j4) 
                (i5)    -- +(-45:0.5) coordinate (j5) 
                (i5)    -- +(-135:0.5) coordinate (j6) 
                                (i6)    -- +(-135:0.5) coordinate (j7)          
                                                                (i7)    -- +(-135:0.5) coordinate (j8)      
         (j1)    -- +(-45:0.5) coordinate (k1) 
        (j1)    -- +(-135:0.5) coordinate (k2) 
  (j3)    -- +(-45:0.5) coordinate (k3) 
        (j3)    -- +(-135:0.5) coordinate (k4) 
                (j5)    -- +(-45:0.5) coordinate (k5) 
                (j5)    -- +(-135:0.5) coordinate (k6) 
                                (j6)    -- +(-135:0.5) coordinate (k7)          
                                                                (j7)    -- +(-135:0.5) coordinate (k8)                                      
 (g8)    -- +(-135:0.5) coordinate (h9)
   (h9)    -- +(-45:0.5) coordinate (i9)  
   (h9)    -- +(-135:0.5) coordinate (i10)
   (i9)    -- +(-45:0.5) coordinate (j9)  
   (i9)    -- +(-135:0.5) coordinate (j10)
   (i10)    -- +(-135:0.5) coordinate (j11)   
   (j9)    -- +(-45:0.5) coordinate (k9)  
   (j9)    -- +(-135:0.5) coordinate (k10)
   (j10)    -- +(-135:0.5) coordinate (k11)  
      (j11)    -- +(-135:0.5) coordinate (k12)  
   ;
            \fill(a1) circle (1pt);             \fill(a2) circle (1pt); \fill(a3) circle (1pt); 
              \fill(b1) circle (1pt);                 \fill(b2) circle (1pt);                          \fill(b3) circle (1pt);   
               \fill(c1) circle (1pt);                 \fill(c2) circle (1pt);                          \fill(c3) circle (1pt);   \fill(c4) circle (1pt);   
                              \fill(d1) circle (1pt);                 \fill(d2) circle (1pt);                          \fill(d3) circle (1pt);   \fill(d4) circle (1pt);   \fill(d5) circle (1pt);   
                                    \fill(e1) circle (1pt);                 \fill(e2) circle (1pt);                          \fill(e3) circle (1pt);   \fill(e4) circle (1pt);   
\fill(e5) circle (1pt);     \fill(e6) circle (1pt);   
                                                                        \fill(f1) circle (1pt);                 \fill(f2) circle (1pt);                          \fill(f3) circle (1pt);   \fill(f4) circle (1pt);   \fill(f5) circle (1pt);     \fill(f6) circle (1pt);    \fill(f7) circle (1pt);   
    \fill(g1) circle (1pt);                 \fill(g2) circle (1pt);                          \fill(g3) circle (1pt);   \fill(g4) circle (1pt);   \fill(g5) circle (1pt);     \fill(g6) circle (1pt);    \fill(g7) circle (1pt);   
 \fill(g8) circle (1pt);   
  \fill(h1) circle (1pt);                 \fill(h2) circle (1pt);                          \fill(h3) circle (1pt);   \fill(h4) circle (1pt);   \fill(h5) circle (1pt);     \fill(h6) circle (1pt);    \fill(h7) circle (1pt);   
 \fill(h8) circle (1pt);   
  \fill(j1) circle (1pt);                 \fill(j2) circle (1pt);                          \fill(j3) circle (1pt);   \fill(j4) circle (1pt);   \fill(j5) circle (1pt);     \fill(j6) circle (1pt);    \fill(j7) circle (1pt);   
 \fill(j8) circle (1pt);   
  \fill(i1) circle (1pt);                 \fill(i2) circle (1pt);                          \fill(i3) circle (1pt);   \fill(i4) circle (1pt);   \fill(i5) circle (1pt);     \fill(i6) circle (1pt);    \fill(i7) circle (1pt);   
 \fill(i8) circle (1pt);  
   \fill(k1) circle (1pt);                 \fill(k2) circle (1pt);                          \fill(k3) circle (1pt);   \fill(k4) circle (1pt);   \fill(k5) circle (1pt);     \fill(k6) circle (1pt);    \fill(k7) circle (1pt);   
 \fill(k8) circle (1pt);   
        \draw[dotted](b3) -- +(270:10);       \draw[dotted](b3) -- +(90:1);
        \draw[dotted](b1) -- +(270:10);       \draw[dotted](b1) -- +(90:1);
        \draw[dotted](f7) -- +(270:10);        \draw[dotted](f7) -- +(90:1);
         \draw[dotted](f1) -- +(270:10);        \draw[dotted](f1) -- +(90:1);
 \draw[dotted](j1) -- +(270:10);        \draw[dotted](j1) -- +(90:1);
 \draw[dotted](j11) -- +(270:10);        \draw[dotted](j11) -- +(90:1);       ; 
          \fill(i10) circle (1pt);     \fill(i9) circle (1pt);   \fill(h9) circle (1pt);  
             \fill(j9) circle (1pt);   \fill(j10) circle (1pt);  \fill(j11) circle (1pt);  
             \fill(k9) circle (1pt);               \fill(k10) circle (1pt);      \fill(k11) circle (1pt);               \fill(k12) circle (1pt);  
\draw(a1) --  (a2); \draw(a2) --  (b2); 
\draw(b2) --  (c2);\draw(c2) --  (d3); \draw(d3) --  (f3); 
\draw(f3) --  (j7); \draw(k7) --  (j7); 
\path (0,0) + (-90:4.4 cm) coordinate (origin); 
\path (origin) ++ (120:1.25)++(-120:1.25) coordinate (1'); 
\path (1') ++ (120:1.5)++(-120:1.5) coordinate (2'); 
\path (2') ++ (120:1.25)++(-120:1.25) coordinate (3'); 
\path (origin) ++ (60:1.5)++(-60:1.5) coordinate (1); 
\path (1) ++ (60:1.5)++(-60:1.5) coordinate (2); 
\path (2) ++ (60:1.25)++(-60:1.25) coordinate (3); 
\draw(origin) node {$\mathfrak{a}_0$};
\draw(1') node {$\mathfrak{a}_{1'}$};
\draw(2') node {$\mathfrak{a}_{2'}$};
\draw(3') node {$\mathfrak{a}_{3'}$};
\draw(1) node {$\mathfrak{a}_{1}$};
\draw(2) node {$\mathfrak{a}_{2}$};
\draw(3) node {$\mathfrak{a}_{3}$};
  }
    \end{tikzpicture}
    \quad\quad
    \begin{tikzpicture}[scale=0.6]
{ \clip (-4.5,-4.8) rectangle ++(9,5);
\path [draw,name path=upward line] (-3,-5) -- (3,-5);
  \path 
      (0, 0)              coordinate (a1)
        (a1)    -- +(-45:0.5) coordinate (a2) 
        (a1)    -- +(-135:0.5) coordinate (a3) 
        (a2)    -- +(-45:0.5) coordinate (b1) 
        (a2)    -- +(-135:0.5) coordinate (b2) 
         (a3)    -- +(-135:0.5) coordinate (b3) 
  (b1)    -- +(-45:0.5) coordinate (c1) 
        (b1)    -- +(-135:0.5) coordinate (c2) 
  (b3)    -- +(-45:0.5) coordinate (c3) 
        (b3)    -- +(-135:0.5) coordinate (c4) 
  (c1)    -- +(-45:0.5) coordinate (d1) 
        (c1)    -- +(-135:0.5) coordinate (d2) 
  (c3)    -- +(-45:0.5) coordinate (d3) 
        (c3)    -- +(-135:0.5) coordinate (d4) 
                (c4)    -- +(-135:0.5) coordinate (d5) 
  (d1)    -- +(-45:0.5) coordinate (e1) 
        (d1)    -- +(-135:0.5) coordinate (e2) 
  (d3)    -- +(-45:0.5) coordinate (e3) 
        (d3)    -- +(-135:0.5) coordinate (e4) 
                (d5)    -- +(-135:0.5) coordinate (e6) 
                                (d5)    -- +(-45:0.5) coordinate (e5) 
  (e1)    -- +(-45:0.5) coordinate (f1) 
        (e1)    -- +(-135:0.5) coordinate (f2) 
  (e3)    -- +(-45:0.5) coordinate (f3) 
        (e3)    -- +(-135:0.5) coordinate (f4) 
                (e5)    -- +(-45:0.5) coordinate (f5) 
                (e5)    -- +(-135:0.5) coordinate (f6) 
                                (e6)    -- +(-135:0.5) coordinate (f7) 
  (f1)    -- +(-45:0.5) coordinate (g1) 
        (f1)    -- +(-135:0.5) coordinate (g2) 
  (f3)    -- +(-45:0.5) coordinate (g3) 
        (f3)    -- +(-135:0.5) coordinate (g4) 
                (f5)    -- +(-45:0.5) coordinate (g5) 
                (f5)    -- +(-135:0.5) coordinate (g6) 
                                (f6)    -- +(-135:0.5) coordinate (g7)          
                                                                (f7)    -- +(-135:0.5) coordinate (g8)        
          (g1)    -- +(-45:0.5) coordinate (h1) 
        (g1)    -- +(-135:0.5) coordinate (h2) 
  (g3)    -- +(-45:0.5) coordinate (h3) 
        (g3)    -- +(-135:0.5) coordinate (h4) 
                (g5)    -- +(-45:0.5) coordinate (h5) 
                (g5)    -- +(-135:0.5) coordinate (h6) 
                                (g6)    -- +(-135:0.5) coordinate (h7)          
                                                                (g7)    -- +(-135:0.5) coordinate (h8)           
                                                                       (h1)    -- +(-45:0.5) coordinate (i1) 
        (h1)    -- +(-135:0.5) coordinate (i2) 
  (h3)    -- +(-45:0.5) coordinate (i3) 
        (h3)    -- +(-135:0.5) coordinate (i4) 
                (h5)    -- +(-45:0.5) coordinate (i5) 
                (h5)    -- +(-135:0.5) coordinate (i6) 
                                (h6)    -- +(-135:0.5) coordinate (i7)          
                                                                (h7)    -- +(-135:0.5) coordinate (i8)           
                            (i1)    -- +(-45:0.5) coordinate (j1) 
        (i1)    -- +(-135:0.5) coordinate (j2) 
  (i3)    -- +(-45:0.5) coordinate (j3) 
        (i3)    -- +(-135:0.5) coordinate (j4) 
                (i5)    -- +(-45:0.5) coordinate (j5) 
                (i5)    -- +(-135:0.5) coordinate (j6) 
                                (i6)    -- +(-135:0.5) coordinate (j7)          
                                                                (i7)    -- +(-135:0.5) coordinate (j8)      
         (j1)    -- +(-45:0.5) coordinate (k1) 
        (j1)    -- +(-135:0.5) coordinate (k2) 
  (j3)    -- +(-45:0.5) coordinate (k3) 
        (j3)    -- +(-135:0.5) coordinate (k4) 
                (j5)    -- +(-45:0.5) coordinate (k5) 
                (j5)    -- +(-135:0.5) coordinate (k6) 
                                (j6)    -- +(-135:0.5) coordinate (k7)          
                                                                (j7)    -- +(-135:0.5) coordinate (k8)                                      
 (g8)    -- +(-135:0.5) coordinate (h9)
   (h9)    -- +(-45:0.5) coordinate (i9)  
   (h9)    -- +(-135:0.5) coordinate (i10)
   (i9)    -- +(-45:0.5) coordinate (j9)  
   (i9)    -- +(-135:0.5) coordinate (j10)
   (i10)    -- +(-135:0.5) coordinate (j11)   
   (j9)    -- +(-45:0.5) coordinate (k9)  
   (j9)    -- +(-135:0.5) coordinate (k10)
   (j10)    -- +(-135:0.5) coordinate (k11)  
      (j11)    -- +(-135:0.5) coordinate (k12)  
   ;
            \fill(a1) circle (1pt);             \fill(a2) circle (1pt); \fill(a3) circle (1pt); 
              \fill(b1) circle (1pt);                 \fill(b2) circle (1pt);                          \fill(b3) circle (1pt);   
               \fill(c1) circle (1pt);                 \fill(c2) circle (1pt);                          \fill(c3) circle (1pt);   \fill(c4) circle (1pt);   
                              \fill(d1) circle (1pt);                 \fill(d2) circle (1pt);                          \fill(d3) circle (1pt);   \fill(d4) circle (1pt);   \fill(d5) circle (1pt);   
                                    \fill(e1) circle (1pt);                 \fill(e2) circle (1pt);                          \fill(e3) circle (1pt);   \fill(e4) circle (1pt);   
\fill(e5) circle (1pt);     \fill(e6) circle (1pt);   
                                                                        \fill(f1) circle (1pt);                 \fill(f2) circle (1pt);                          \fill(f3) circle (1pt);   \fill(f4) circle (1pt);   \fill(f5) circle (1pt);     \fill(f6) circle (1pt);    \fill(f7) circle (1pt);   
    \fill(g1) circle (1pt);                 \fill(g2) circle (1pt);                          \fill(g3) circle (1pt);   \fill(g4) circle (1pt);   \fill(g5) circle (1pt);     \fill(g6) circle (1pt);    \fill(g7) circle (1pt);   
 \fill(g8) circle (1pt);   
  \fill(h1) circle (1pt);                 \fill(h2) circle (1pt);                          \fill(h3) circle (1pt);   \fill(h4) circle (1pt);   \fill(h5) circle (1pt);     \fill(h6) circle (1pt);    \fill(h7) circle (1pt);   
 \fill(h8) circle (1pt);   
  \fill(j1) circle (1pt);                 \fill(j2) circle (1pt);                          \fill(j3) circle (1pt);   \fill(j4) circle (1pt);   \fill(j5) circle (1pt);     \fill(j6) circle (1pt);    \fill(j7) circle (1pt);   
 \fill(j8) circle (1pt);   
  \fill(i1) circle (1pt);                 \fill(i2) circle (1pt);                          \fill(i3) circle (1pt);   \fill(i4) circle (1pt);   \fill(i5) circle (1pt);     \fill(i6) circle (1pt);    \fill(i7) circle (1pt);   
 \fill(i8) circle (1pt);  
   \fill(k1) circle (1pt);                 \fill(k2) circle (1pt);                          \fill(k3) circle (1pt);   \fill(k4) circle (1pt);   \fill(k5) circle (1pt);     \fill(k6) circle (1pt);    \fill(k7) circle (1pt);   
 \fill(k8) circle (1pt);   
        \draw[dotted](b3) -- +(270:10);       \draw[dotted](b3) -- +(90:1);
        \draw[dotted](b1) -- +(270:10);       \draw[dotted](b1) -- +(90:1);
        \draw[dotted](f7) -- +(270:10);        \draw[dotted](f7) -- +(90:1);
         \draw[dotted](f1) -- +(270:10);        \draw[dotted](f1) -- +(90:1);
 \draw[dotted](j1) -- +(270:10);        \draw[dotted](j1) -- +(90:1);
 \draw[dotted](j11) -- +(270:10);        \draw[dotted](j11) -- +(90:1);       ; 
          \fill(i10) circle (1pt);     \fill(i9) circle (1pt);   \fill(h9) circle (1pt);  
             \fill(j9) circle (1pt);   \fill(j10) circle (1pt);  \fill(j11) circle (1pt);  
             \fill(k9) circle (1pt);               \fill(k10) circle (1pt);      \fill(k11) circle (1pt);               \fill(k12) circle (1pt);  
\draw(a1) --  (a2); \draw(a2) --  (b2); 
\draw(b2) --  (d4); \draw(d4) --  (h4); 
\draw(h4) --  (k7); 
\path (0,0) + (-90:4.4 cm) coordinate (origin); 
\path (origin) ++ (120:1.25)++(-120:1.25) coordinate (1'); 
\path (1') ++ (120:1.5)++(-120:1.5) coordinate (2'); 
\path (2') ++ (120:1.25)++(-120:1.25) coordinate (3'); 
\path (origin) ++ (60:1.5)++(-60:1.5) coordinate (1); 
\path (1) ++ (60:1.5)++(-60:1.5) coordinate (2); 
\path (2) ++ (60:1.25)++(-60:1.25) coordinate (3); 
\draw(origin) node {$\mathfrak{a}_0$};
\draw(1') node {$\mathfrak{a}_{1'}$};
\draw(2') node {$\mathfrak{a}_{2'}$};
\draw(3') node {$\mathfrak{a}_{3'}$};
\draw(1) node {$\mathfrak{a}_{1}$};
\draw(2) node {$\mathfrak{a}_{2}$};
\draw(3) node {$\mathfrak{a}_{3}$};
  }
    \end{tikzpicture}
   $$ 
   \caption{Maximal paths 
    in $\Path(\lambda_0,\lambda_{2'})$
     and $\Path(\lambda_0,\lambda_{2})$, respectively. 
   Both paths  have degree 2.}  
      \label{22law}  \end{figure}
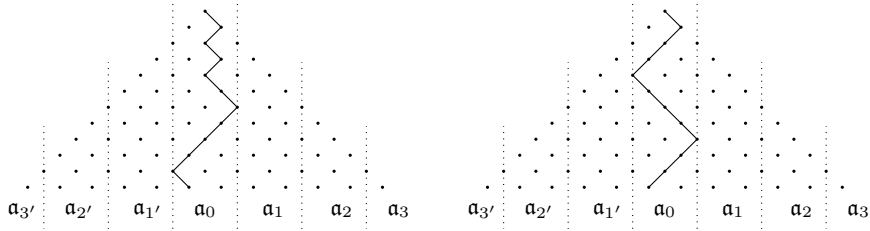

\begin{figure}[ht]\captionsetup{width=0.9\textwidth}
$$\scalefont{0.8}  \begin{tikzpicture}[scale=0.6]
{ \clip (-4.5,-4.8) rectangle ++(9,5);
\path [draw,name path=upward line] (-3,-5) -- (3,-5);
  \path 
      (0, 0)              coordinate (a1)
        (a1)    -- +(-45:0.5) coordinate (a2) 
        (a1)    -- +(-135:0.5) coordinate (a3) 
        (a2)    -- +(-45:0.5) coordinate (b1) 
        (a2)    -- +(-135:0.5) coordinate (b2) 
         (a3)    -- +(-135:0.5) coordinate (b3) 
  (b1)    -- +(-45:0.5) coordinate (c1) 
        (b1)    -- +(-135:0.5) coordinate (c2) 
  (b3)    -- +(-45:0.5) coordinate (c3) 
        (b3)    -- +(-135:0.5) coordinate (c4) 
  (c1)    -- +(-45:0.5) coordinate (d1) 
        (c1)    -- +(-135:0.5) coordinate (d2) 
  (c3)    -- +(-45:0.5) coordinate (d3) 
        (c3)    -- +(-135:0.5) coordinate (d4) 
                (c4)    -- +(-135:0.5) coordinate (d5) 
  (d1)    -- +(-45:0.5) coordinate (e1) 
        (d1)    -- +(-135:0.5) coordinate (e2) 
  (d3)    -- +(-45:0.5) coordinate (e3) 
        (d3)    -- +(-135:0.5) coordinate (e4) 
                (d5)    -- +(-135:0.5) coordinate (e6) 
                                (d5)    -- +(-45:0.5) coordinate (e5) 
  (e1)    -- +(-45:0.5) coordinate (f1) 
        (e1)    -- +(-135:0.5) coordinate (f2) 
  (e3)    -- +(-45:0.5) coordinate (f3) 
        (e3)    -- +(-135:0.5) coordinate (f4) 
                (e5)    -- +(-45:0.5) coordinate (f5) 
                (e5)    -- +(-135:0.5) coordinate (f6) 
                                (e6)    -- +(-135:0.5) coordinate (f7) 
  (f1)    -- +(-45:0.5) coordinate (g1) 
        (f1)    -- +(-135:0.5) coordinate (g2) 
  (f3)    -- +(-45:0.5) coordinate (g3) 
        (f3)    -- +(-135:0.5) coordinate (g4) 
                (f5)    -- +(-45:0.5) coordinate (g5) 
                (f5)    -- +(-135:0.5) coordinate (g6) 
                                (f6)    -- +(-135:0.5) coordinate (g7)          
                                                                (f7)    -- +(-135:0.5) coordinate (g8)        
          (g1)    -- +(-45:0.5) coordinate (h1) 
        (g1)    -- +(-135:0.5) coordinate (h2) 
  (g3)    -- +(-45:0.5) coordinate (h3) 
        (g3)    -- +(-135:0.5) coordinate (h4) 
                (g5)    -- +(-45:0.5) coordinate (h5) 
                (g5)    -- +(-135:0.5) coordinate (h6) 
                                (g6)    -- +(-135:0.5) coordinate (h7)          
                                                                (g7)    -- +(-135:0.5) coordinate (h8)           
                                                                       (h1)    -- +(-45:0.5) coordinate (i1) 
        (h1)    -- +(-135:0.5) coordinate (i2) 
  (h3)    -- +(-45:0.5) coordinate (i3) 
        (h3)    -- +(-135:0.5) coordinate (i4) 
                (h5)    -- +(-45:0.5) coordinate (i5) 
                (h5)    -- +(-135:0.5) coordinate (i6) 
                                (h6)    -- +(-135:0.5) coordinate (i7)          
                                                                (h7)    -- +(-135:0.5) coordinate (i8)           
                            (i1)    -- +(-45:0.5) coordinate (j1) 
        (i1)    -- +(-135:0.5) coordinate (j2) 
  (i3)    -- +(-45:0.5) coordinate (j3) 
        (i3)    -- +(-135:0.5) coordinate (j4) 
                (i5)    -- +(-45:0.5) coordinate (j5) 
                (i5)    -- +(-135:0.5) coordinate (j6) 
                                (i6)    -- +(-135:0.5) coordinate (j7)          
                                                                (i7)    -- +(-135:0.5) coordinate (j8)      
         (j1)    -- +(-45:0.5) coordinate (k1) 
        (j1)    -- +(-135:0.5) coordinate (k2) 
  (j3)    -- +(-45:0.5) coordinate (k3) 
        (j3)    -- +(-135:0.5) coordinate (k4) 
                (j5)    -- +(-45:0.5) coordinate (k5) 
                (j5)    -- +(-135:0.5) coordinate (k6) 
                                (j6)    -- +(-135:0.5) coordinate (k7)          
                                                                (j7)    -- +(-135:0.5) coordinate (k8)                                      
 (g8)    -- +(-135:0.5) coordinate (h9)
   (h9)    -- +(-45:0.5) coordinate (i9)  
   (h9)    -- +(-135:0.5) coordinate (i10)
   (i9)    -- +(-45:0.5) coordinate (j9)  
   (i9)    -- +(-135:0.5) coordinate (j10)
   (i10)    -- +(-135:0.5) coordinate (j11)   
   (j9)    -- +(-45:0.5) coordinate (k9)  
   (j9)    -- +(-135:0.5) coordinate (k10)
   (j10)    -- +(-135:0.5) coordinate (k11)  
      (j11)    -- +(-135:0.5) coordinate (k12)  
   ;
            \fill(a1) circle (1pt);             \fill(a2) circle (1pt); \fill(a3) circle (1pt); 
              \fill(b1) circle (1pt);                 \fill(b2) circle (1pt);                          \fill(b3) circle (1pt);   
               \fill(c1) circle (1pt);                 \fill(c2) circle (1pt);                          \fill(c3) circle (1pt);   \fill(c4) circle (1pt);   
                              \fill(d1) circle (1pt);                 \fill(d2) circle (1pt);                          \fill(d3) circle (1pt);   \fill(d4) circle (1pt);   \fill(d5) circle (1pt);   
                                    \fill(e1) circle (1pt);                 \fill(e2) circle (1pt);                          \fill(e3) circle (1pt);   \fill(e4) circle (1pt);   
\fill(e5) circle (1pt);     \fill(e6) circle (1pt);   
                                                                        \fill(f1) circle (1pt);                 \fill(f2) circle (1pt);                          \fill(f3) circle (1pt);   \fill(f4) circle (1pt);   \fill(f5) circle (1pt);     \fill(f6) circle (1pt);    \fill(f7) circle (1pt);   
    \fill(g1) circle (1pt);                 \fill(g2) circle (1pt);                          \fill(g3) circle (1pt);   \fill(g4) circle (1pt);   \fill(g5) circle (1pt);     \fill(g6) circle (1pt);    \fill(g7) circle (1pt);   
 \fill(g8) circle (1pt);   
  \fill(h1) circle (1pt);                 \fill(h2) circle (1pt);                          \fill(h3) circle (1pt);   \fill(h4) circle (1pt);   \fill(h5) circle (1pt);     \fill(h6) circle (1pt);    \fill(h7) circle (1pt);   
 \fill(h8) circle (1pt);   
  \fill(j1) circle (1pt);                 \fill(j2) circle (1pt);                          \fill(j3) circle (1pt);   \fill(j4) circle (1pt);   \fill(j5) circle (1pt);     \fill(j6) circle (1pt);    \fill(j7) circle (1pt);   
 \fill(j8) circle (1pt);   
  \fill(i1) circle (1pt);                 \fill(i2) circle (1pt);                          \fill(i3) circle (1pt);   \fill(i4) circle (1pt);   \fill(i5) circle (1pt);     \fill(i6) circle (1pt);    \fill(i7) circle (1pt);   
 \fill(i8) circle (1pt);  
   \fill(k1) circle (1pt);                 \fill(k2) circle (1pt);                          \fill(k3) circle (1pt);   \fill(k4) circle (1pt);   \fill(k5) circle (1pt);     \fill(k6) circle (1pt);    \fill(k7) circle (1pt);   
 \fill(k8) circle (1pt);   
        \draw[dotted](b3) -- +(270:10);       \draw[dotted](b3) -- +(90:1);
        \draw[dotted](b1) -- +(270:10);       \draw[dotted](b1) -- +(90:1);
        \draw[dotted](f7) -- +(270:10);        \draw[dotted](f7) -- +(90:1);
         \draw[dotted](f1) -- +(270:10);        \draw[dotted](f1) -- +(90:1);
 \draw[dotted](j1) -- +(270:10);        \draw[dotted](j1) -- +(90:1);
 \draw[dotted](j11) -- +(270:10);        \draw[dotted](j11) -- +(90:1);       ; 
          \fill(i10) circle (1pt);     \fill(i9) circle (1pt);   \fill(h9) circle (1pt);  
             \fill(j9) circle (1pt);   \fill(j10) circle (1pt);  \fill(j11) circle (1pt);  
             \fill(k9) circle (1pt);               \fill(k10) circle (1pt);      \fill(k11) circle (1pt);               \fill(k12) circle (1pt);  
\draw(a1) --  (a2); \draw(a2) --  (b2); 
\draw(b2) --  (c2);\draw(c2) --  (d3); \draw(d3) --  (e3); 
\draw(f4) --  (e3); \draw(f4) --  (g4);
\draw(g4) --  (h5);
\draw(h5) --  (j7); \draw(k7) --  (j7); 
\path (0,0) + (-90:4.4 cm) coordinate (origin); 
\path (origin) ++ (120:1.25)++(-120:1.25) coordinate (1'); 
\path (1') ++ (120:1.5)++(-120:1.5) coordinate (2'); 
\path (2') ++ (120:1.25)++(-120:1.25) coordinate (3'); 
\path (origin) ++ (60:1.5)++(-60:1.5) coordinate (1); 
\path (1) ++ (60:1.5)++(-60:1.5) coordinate (2); 
\path (2) ++ (60:1.25)++(-60:1.25) coordinate (3); 
\draw(origin) node {$\mathfrak{a}_0$};
\draw(1') node {$\mathfrak{a}_{1'}$};
\draw(2') node {$\mathfrak{a}_{2'}$};
\draw(3') node {$\mathfrak{a}_{3'}$};
\draw(1) node {$\mathfrak{a}_{1}$};
\draw(2) node {$\mathfrak{a}_{2}$};
\draw(3) node {$\mathfrak{a}_{3}$};
  }
    \end{tikzpicture} 
   \quad\quad
   \begin{tikzpicture}[scale=0.6]
{ \clip (-4.5,-4.8) rectangle ++(9,5);
\path [draw,name path=upward line] (-3,-5) -- (3,-5);
  \path 
      (0, 0)              coordinate (a1)
        (a1)    -- +(-45:0.5) coordinate (a2) 
        (a1)    -- +(-135:0.5) coordinate (a3) 
        (a2)    -- +(-45:0.5) coordinate (b1) 
        (a2)    -- +(-135:0.5) coordinate (b2) 
         (a3)    -- +(-135:0.5) coordinate (b3) 
  (b1)    -- +(-45:0.5) coordinate (c1) 
        (b1)    -- +(-135:0.5) coordinate (c2) 
  (b3)    -- +(-45:0.5) coordinate (c3) 
        (b3)    -- +(-135:0.5) coordinate (c4) 
  (c1)    -- +(-45:0.5) coordinate (d1) 
        (c1)    -- +(-135:0.5) coordinate (d2) 
  (c3)    -- +(-45:0.5) coordinate (d3) 
        (c3)    -- +(-135:0.5) coordinate (d4) 
                (c4)    -- +(-135:0.5) coordinate (d5) 
  (d1)    -- +(-45:0.5) coordinate (e1) 
        (d1)    -- +(-135:0.5) coordinate (e2) 
  (d3)    -- +(-45:0.5) coordinate (e3) 
        (d3)    -- +(-135:0.5) coordinate (e4) 
                (d5)    -- +(-135:0.5) coordinate (e6) 
                                (d5)    -- +(-45:0.5) coordinate (e5) 
  (e1)    -- +(-45:0.5) coordinate (f1) 
        (e1)    -- +(-135:0.5) coordinate (f2) 
  (e3)    -- +(-45:0.5) coordinate (f3) 
        (e3)    -- +(-135:0.5) coordinate (f4) 
                (e5)    -- +(-45:0.5) coordinate (f5) 
                (e5)    -- +(-135:0.5) coordinate (f6) 
                                (e6)    -- +(-135:0.5) coordinate (f7) 
  (f1)    -- +(-45:0.5) coordinate (g1) 
        (f1)    -- +(-135:0.5) coordinate (g2) 
  (f3)    -- +(-45:0.5) coordinate (g3) 
        (f3)    -- +(-135:0.5) coordinate (g4) 
                (f5)    -- +(-45:0.5) coordinate (g5) 
                (f5)    -- +(-135:0.5) coordinate (g6) 
                                (f6)    -- +(-135:0.5) coordinate (g7)          
                                                                (f7)    -- +(-135:0.5) coordinate (g8)        
          (g1)    -- +(-45:0.5) coordinate (h1) 
        (g1)    -- +(-135:0.5) coordinate (h2) 
  (g3)    -- +(-45:0.5) coordinate (h3) 
        (g3)    -- +(-135:0.5) coordinate (h4) 
                (g5)    -- +(-45:0.5) coordinate (h5) 
                (g5)    -- +(-135:0.5) coordinate (h6) 
                                (g6)    -- +(-135:0.5) coordinate (h7)          
                                                                (g7)    -- +(-135:0.5) coordinate (h8)           
                                                                       (h1)    -- +(-45:0.5) coordinate (i1) 
        (h1)    -- +(-135:0.5) coordinate (i2) 
  (h3)    -- +(-45:0.5) coordinate (i3) 
        (h3)    -- +(-135:0.5) coordinate (i4) 
                (h5)    -- +(-45:0.5) coordinate (i5) 
                (h5)    -- +(-135:0.5) coordinate (i6) 
                                (h6)    -- +(-135:0.5) coordinate (i7)          
                                                                (h7)    -- +(-135:0.5) coordinate (i8)           
                            (i1)    -- +(-45:0.5) coordinate (j1) 
        (i1)    -- +(-135:0.5) coordinate (j2) 
  (i3)    -- +(-45:0.5) coordinate (j3) 
        (i3)    -- +(-135:0.5) coordinate (j4) 
                (i5)    -- +(-45:0.5) coordinate (j5) 
                (i5)    -- +(-135:0.5) coordinate (j6) 
                                (i6)    -- +(-135:0.5) coordinate (j7)          
                                                                (i7)    -- +(-135:0.5) coordinate (j8)      
         (j1)    -- +(-45:0.5) coordinate (k1) 
        (j1)    -- +(-135:0.5) coordinate (k2) 
  (j3)    -- +(-45:0.5) coordinate (k3) 
        (j3)    -- +(-135:0.5) coordinate (k4) 
                (j5)    -- +(-45:0.5) coordinate (k5) 
                (j5)    -- +(-135:0.5) coordinate (k6) 
                                (j6)    -- +(-135:0.5) coordinate (k7)          
                                                                (j7)    -- +(-135:0.5) coordinate (k8)                                      
 (g8)    -- +(-135:0.5) coordinate (h9)
   (h9)    -- +(-45:0.5) coordinate (i9)  
   (h9)    -- +(-135:0.5) coordinate (i10)
   (i9)    -- +(-45:0.5) coordinate (j9)  
   (i9)    -- +(-135:0.5) coordinate (j10)
   (i10)    -- +(-135:0.5) coordinate (j11)   
   (j9)    -- +(-45:0.5) coordinate (k9)  
   (j9)    -- +(-135:0.5) coordinate (k10)
   (j10)    -- +(-135:0.5) coordinate (k11)  
      (j11)    -- +(-135:0.5) coordinate (k12)  
   ;
            \fill(a1) circle (1pt);             \fill(a2) circle (1pt); \fill(a3) circle (1pt); 
              \fill(b1) circle (1pt);                 \fill(b2) circle (1pt);                          \fill(b3) circle (1pt);   
               \fill(c1) circle (1pt);                 \fill(c2) circle (1pt);                          \fill(c3) circle (1pt);   \fill(c4) circle (1pt);   
                              \fill(d1) circle (1pt);                 \fill(d2) circle (1pt);                          \fill(d3) circle (1pt);   \fill(d4) circle (1pt);   \fill(d5) circle (1pt);   
                                    \fill(e1) circle (1pt);                 \fill(e2) circle (1pt);                          \fill(e3) circle (1pt);   \fill(e4) circle (1pt);   
\fill(e5) circle (1pt);     \fill(e6) circle (1pt);   
                                                                        \fill(f1) circle (1pt);                 \fill(f2) circle (1pt);                          \fill(f3) circle (1pt);   \fill(f4) circle (1pt);   \fill(f5) circle (1pt);     \fill(f6) circle (1pt);    \fill(f7) circle (1pt);   
    \fill(g1) circle (1pt);                 \fill(g2) circle (1pt);                          \fill(g3) circle (1pt);   \fill(g4) circle (1pt);   \fill(g5) circle (1pt);     \fill(g6) circle (1pt);    \fill(g7) circle (1pt);   
 \fill(g8) circle (1pt);   
  \fill(h1) circle (1pt);                 \fill(h2) circle (1pt);                          \fill(h3) circle (1pt);   \fill(h4) circle (1pt);   \fill(h5) circle (1pt);     \fill(h6) circle (1pt);    \fill(h7) circle (1pt);   
 \fill(h8) circle (1pt);   
  \fill(j1) circle (1pt);                 \fill(j2) circle (1pt);                          \fill(j3) circle (1pt);   \fill(j4) circle (1pt);   \fill(j5) circle (1pt);     \fill(j6) circle (1pt);    \fill(j7) circle (1pt);   
 \fill(j8) circle (1pt);   
  \fill(i1) circle (1pt);                 \fill(i2) circle (1pt);                          \fill(i3) circle (1pt);   \fill(i4) circle (1pt);   \fill(i5) circle (1pt);     \fill(i6) circle (1pt);    \fill(i7) circle (1pt);   
 \fill(i8) circle (1pt);  
   \fill(k1) circle (1pt);                 \fill(k2) circle (1pt);                          \fill(k3) circle (1pt);   \fill(k4) circle (1pt);   \fill(k5) circle (1pt);     \fill(k6) circle (1pt);    \fill(k7) circle (1pt);   
 \fill(k8) circle (1pt);   
        \draw[dotted](b3) -- +(270:10);       \draw[dotted](b3) -- +(90:1);
        \draw[dotted](b1) -- +(270:10);       \draw[dotted](b1) -- +(90:1);
        \draw[dotted](f7) -- +(270:10);        \draw[dotted](f7) -- +(90:1);
         \draw[dotted](f1) -- +(270:10);        \draw[dotted](f1) -- +(90:1);
 \draw[dotted](j1) -- +(270:10);        \draw[dotted](j1) -- +(90:1);
 \draw[dotted](j11) -- +(270:10);        \draw[dotted](j11) -- +(90:1);       ; 
          \fill(i10) circle (1pt);     \fill(i9) circle (1pt);   \fill(h9) circle (1pt);  
             \fill(j9) circle (1pt);   \fill(j10) circle (1pt);  \fill(j11) circle (1pt);  
             \fill(k9) circle (1pt);               \fill(k10) circle (1pt);      \fill(k11) circle (1pt);               \fill(k12) circle (1pt);  
\draw(a1) --  (a2); \draw(a2) --  (b2); 
\draw(b2) --  (c2);\draw(c2) --  (d3); \draw(d3) --  (e3); 
\draw(f4) --  (e3); \draw(f4) --  (g4);
\draw(g4) --  (h4);
\draw(h4)   --  (k7); 
\path (0,0) + (-90:4.4 cm) coordinate (origin); 
\path (origin) ++ (120:1.25)++(-120:1.25) coordinate (1'); 
\path (1') ++ (120:1.5)++(-120:1.5) coordinate (2'); 
\path (2') ++ (120:1.25)++(-120:1.25) coordinate (3'); 
\path (origin) ++ (60:1.5)++(-60:1.5) coordinate (1); 
\path (1) ++ (60:1.5)++(-60:1.5) coordinate (2); 
\path (2) ++ (60:1.25)++(-60:1.25) coordinate (3); 
\draw(origin) node {$\mathfrak{a}_0$};
\draw(1') node {$\mathfrak{a}_{1'}$};
\draw(2') node {$\mathfrak{a}_{2'}$};
\draw(3') node {$\mathfrak{a}_{3'}$};
\draw(1) node {$\mathfrak{a}_{1}$};
\draw(2) node {$\mathfrak{a}_{2}$};
\draw(3) node {$\mathfrak{a}_{3}$};
  }
    \end{tikzpicture}   $$
\caption{Maximal paths 
    in $\Path(\lambda_0,\lambda_{1'})$
     and $\Path(\lambda_0,\lambda_{1})$, respectively.    Both paths  have degree 1. }         \label{11law} \end{figure}
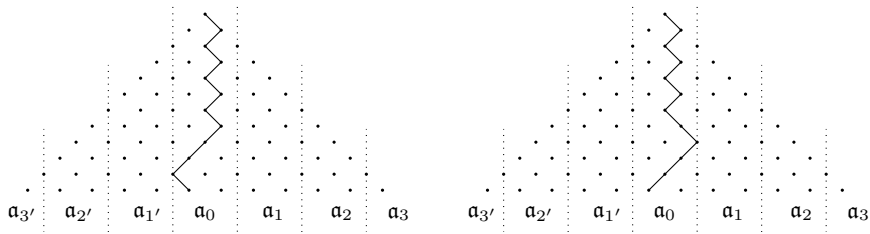

  The maximal paths in $\Path(\lambda_1,\lambda_{2'})$ and $\Path(\lambda_{1'},\lambda_{2'})$ (which index basis elements of $\Delta(\lambda_{1})$  and 
 $\Delta(\lambda_{1'})$ respectively) are depicted in Figure \ref{the other bits} below.
 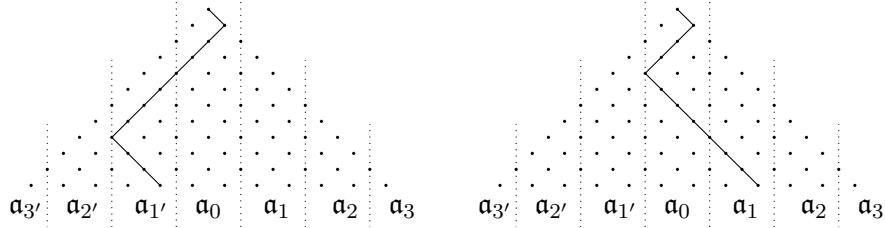
\begin{figure}[ht]\captionsetup{width=0.9\textwidth}
 $$
     \begin{tikzpicture}[scale=0.6]
{ \clip (-4.5,-4.8) rectangle ++(9,5);
\path [draw,name path=upward line] (-3,-5) -- (3,-5);
  \path 
      (0, 0)              coordinate (a1)
        (a1)    -- +(-45:0.5) coordinate (a2) 
        (a1)    -- +(-135:0.5) coordinate (a3) 
        (a2)    -- +(-45:0.5) coordinate (b1) 
        (a2)    -- +(-135:0.5) coordinate (b2) 
         (a3)    -- +(-135:0.5) coordinate (b3) 
  (b1)    -- +(-45:0.5) coordinate (c1) 
        (b1)    -- +(-135:0.5) coordinate (c2) 
  (b3)    -- +(-45:0.5) coordinate (c3) 
        (b3)    -- +(-135:0.5) coordinate (c4) 
  (c1)    -- +(-45:0.5) coordinate (d1) 
        (c1)    -- +(-135:0.5) coordinate (d2) 
  (c3)    -- +(-45:0.5) coordinate (d3) 
        (c3)    -- +(-135:0.5) coordinate (d4) 
                (c4)    -- +(-135:0.5) coordinate (d5) 
  (d1)    -- +(-45:0.5) coordinate (e1) 
        (d1)    -- +(-135:0.5) coordinate (e2) 
  (d3)    -- +(-45:0.5) coordinate (e3) 
        (d3)    -- +(-135:0.5) coordinate (e4) 
                (d5)    -- +(-135:0.5) coordinate (e6) 
                                (d5)    -- +(-45:0.5) coordinate (e5) 
  (e1)    -- +(-45:0.5) coordinate (f1) 
        (e1)    -- +(-135:0.5) coordinate (f2) 
  (e3)    -- +(-45:0.5) coordinate (f3) 
        (e3)    -- +(-135:0.5) coordinate (f4) 
                (e5)    -- +(-45:0.5) coordinate (f5) 
                (e5)    -- +(-135:0.5) coordinate (f6) 
                                (e6)    -- +(-135:0.5) coordinate (f7) 
  (f1)    -- +(-45:0.5) coordinate (g1) 
        (f1)    -- +(-135:0.5) coordinate (g2) 
  (f3)    -- +(-45:0.5) coordinate (g3) 
        (f3)    -- +(-135:0.5) coordinate (g4) 
                (f5)    -- +(-45:0.5) coordinate (g5) 
                (f5)    -- +(-135:0.5) coordinate (g6) 
                                (f6)    -- +(-135:0.5) coordinate (g7)          
                                                                (f7)    -- +(-135:0.5) coordinate (g8)        
          (g1)    -- +(-45:0.5) coordinate (h1) 
        (g1)    -- +(-135:0.5) coordinate (h2) 
  (g3)    -- +(-45:0.5) coordinate (h3) 
        (g3)    -- +(-135:0.5) coordinate (h4) 
                (g5)    -- +(-45:0.5) coordinate (h5) 
                (g5)    -- +(-135:0.5) coordinate (h6) 
                                (g6)    -- +(-135:0.5) coordinate (h7)          
                                                                (g7)    -- +(-135:0.5) coordinate (h8)           
                                                                       (h1)    -- +(-45:0.5) coordinate (i1) 
        (h1)    -- +(-135:0.5) coordinate (i2) 
  (h3)    -- +(-45:0.5) coordinate (i3) 
        (h3)    -- +(-135:0.5) coordinate (i4) 
                (h5)    -- +(-45:0.5) coordinate (i5) 
                (h5)    -- +(-135:0.5) coordinate (i6) 
                                (h6)    -- +(-135:0.5) coordinate (i7)          
                                                                (h7)    -- +(-135:0.5) coordinate (i8)           
                            (i1)    -- +(-45:0.5) coordinate (j1) 
        (i1)    -- +(-135:0.5) coordinate (j2) 
  (i3)    -- +(-45:0.5) coordinate (j3) 
        (i3)    -- +(-135:0.5) coordinate (j4) 
                (i5)    -- +(-45:0.5) coordinate (j5) 
                (i5)    -- +(-135:0.5) coordinate (j6) 
                                (i6)    -- +(-135:0.5) coordinate (j7)          
                                                                (i7)    -- +(-135:0.5) coordinate (j8)      
         (j1)    -- +(-45:0.5) coordinate (k1) 
        (j1)    -- +(-135:0.5) coordinate (k2) 
  (j3)    -- +(-45:0.5) coordinate (k3) 
        (j3)    -- +(-135:0.5) coordinate (k4) 
                (j5)    -- +(-45:0.5) coordinate (k5) 
                (j5)    -- +(-135:0.5) coordinate (k6) 
                                (j6)    -- +(-135:0.5) coordinate (k7)          
                                                                (j7)    -- +(-135:0.5) coordinate (k8)                                      
 (g8)    -- +(-135:0.5) coordinate (h9)
   (h9)    -- +(-45:0.5) coordinate (i9)  
   (h9)    -- +(-135:0.5) coordinate (i10)
   (i9)    -- +(-45:0.5) coordinate (j9)  
   (i9)    -- +(-135:0.5) coordinate (j10)
   (i10)    -- +(-135:0.5) coordinate (j11)   
   (j9)    -- +(-45:0.5) coordinate (k9)  
   (j9)    -- +(-135:0.5) coordinate (k10)
   (j10)    -- +(-135:0.5) coordinate (k11)  
      (j11)    -- +(-135:0.5) coordinate (k12)  
   ;
            \fill(a1) circle (1pt);             \fill(a2) circle (1pt); \fill(a3) circle (1pt); 
              \fill(b1) circle (1pt);                 \fill(b2) circle (1pt);                          \fill(b3) circle (1pt);   
               \fill(c1) circle (1pt);                 \fill(c2) circle (1pt);                          \fill(c3) circle (1pt);   \fill(c4) circle (1pt);   
                              \fill(d1) circle (1pt);                 \fill(d2) circle (1pt);                          \fill(d3) circle (1pt);   \fill(d4) circle (1pt);   \fill(d5) circle (1pt);   
                                    \fill(e1) circle (1pt);                 \fill(e2) circle (1pt);                          \fill(e3) circle (1pt);   \fill(e4) circle (1pt);   
\fill(e5) circle (1pt);     \fill(e6) circle (1pt);   
                                                                        \fill(f1) circle (1pt);                 \fill(f2) circle (1pt);                          \fill(f3) circle (1pt);   \fill(f4) circle (1pt);   \fill(f5) circle (1pt);     \fill(f6) circle (1pt);    \fill(f7) circle (1pt);   
    \fill(g1) circle (1pt);                 \fill(g2) circle (1pt);                          \fill(g3) circle (1pt);   \fill(g4) circle (1pt);   \fill(g5) circle (1pt);     \fill(g6) circle (1pt);    \fill(g7) circle (1pt);   
 \fill(g8) circle (1pt);   
  \fill(h1) circle (1pt);                 \fill(h2) circle (1pt);                          \fill(h3) circle (1pt);   \fill(h4) circle (1pt);   \fill(h5) circle (1pt);     \fill(h6) circle (1pt);    \fill(h7) circle (1pt);   
 \fill(h8) circle (1pt);   
  \fill(j1) circle (1pt);                 \fill(j2) circle (1pt);                          \fill(j3) circle (1pt);   \fill(j4) circle (1pt);   \fill(j5) circle (1pt);     \fill(j6) circle (1pt);    \fill(j7) circle (1pt);   
 \fill(j8) circle (1pt);   
  \fill(i1) circle (1pt);                 \fill(i2) circle (1pt);                          \fill(i3) circle (1pt);   \fill(i4) circle (1pt);   \fill(i5) circle (1pt);     \fill(i6) circle (1pt);    \fill(i7) circle (1pt);   
 \fill(i8) circle (1pt);  
   \fill(k1) circle (1pt);                 \fill(k2) circle (1pt);                          \fill(k3) circle (1pt);   \fill(k4) circle (1pt);   \fill(k5) circle (1pt);     \fill(k6) circle (1pt);    \fill(k7) circle (1pt);   
 \fill(k8) circle (1pt);   
        \draw[dotted](b3) -- +(270:10);       \draw[dotted](b3) -- +(90:1);
        \draw[dotted](b1) -- +(270:10);       \draw[dotted](b1) -- +(90:1);
        \draw[dotted](f7) -- +(270:10);        \draw[dotted](f7) -- +(90:1);
         \draw[dotted](f1) -- +(270:10);        \draw[dotted](f1) -- +(90:1);
 \draw[dotted](j1) -- +(270:10);        \draw[dotted](j1) -- +(90:1);
 \draw[dotted](j11) -- +(270:10);        \draw[dotted](j11) -- +(90:1);       ; 
          \fill(i10) circle (1pt);     \fill(i9) circle (1pt);   \fill(h9) circle (1pt);  
             \fill(j9) circle (1pt);   \fill(j10) circle (1pt);  \fill(j11) circle (1pt);  
             \fill(k9) circle (1pt);               \fill(k10) circle (1pt);      \fill(k11) circle (1pt);               \fill(k12) circle (1pt);  
\draw(a1) --  (a2); \draw(a2) --  (h8);  \draw(k8) --  (h8); 
\path (0,0) + (-90:4.4 cm) coordinate (origin); 
\path (origin) ++ (120:1.25)++(-120:1.25) coordinate (1'); 
\path (1') ++ (120:1.5)++(-120:1.5) coordinate (2'); 
\path (2') ++ (120:1.25)++(-120:1.25) coordinate (3'); 
\path (origin) ++ (60:1.5)++(-60:1.5) coordinate (1); 
\path (1) ++ (60:1.5)++(-60:1.5) coordinate (2); 
\path (2) ++ (60:1.25)++(-60:1.25) coordinate (3); 
\draw(origin) node {$\mathfrak{a}_0$};
\draw(1') node {$\mathfrak{a}_{1'}$};
\draw(2') node {$\mathfrak{a}_{2'}$};
\draw(3') node {$\mathfrak{a}_{3'}$};
\draw(1) node {$\mathfrak{a}_{1}$};
\draw(2) node {$\mathfrak{a}_{2}$};
\draw(3) node {$\mathfrak{a}_{3}$};
  }
    \end{tikzpicture}
    \quad
    \quad
        \begin{tikzpicture}[scale=0.6]
{ \clip (-4.5,-4.8) rectangle ++(9,5);
\path [draw,name path=upward line] (-3,-5) -- (3,-5);
  \path 
      (0, 0)              coordinate (a1)
        (a1)    -- +(-45:0.5) coordinate (a2) 
        (a1)    -- +(-135:0.5) coordinate (a3) 
        (a2)    -- +(-45:0.5) coordinate (b1) 
        (a2)    -- +(-135:0.5) coordinate (b2) 
         (a3)    -- +(-135:0.5) coordinate (b3) 
  (b1)    -- +(-45:0.5) coordinate (c1) 
        (b1)    -- +(-135:0.5) coordinate (c2) 
  (b3)    -- +(-45:0.5) coordinate (c3) 
        (b3)    -- +(-135:0.5) coordinate (c4) 
  (c1)    -- +(-45:0.5) coordinate (d1) 
        (c1)    -- +(-135:0.5) coordinate (d2) 
  (c3)    -- +(-45:0.5) coordinate (d3) 
        (c3)    -- +(-135:0.5) coordinate (d4) 
                (c4)    -- +(-135:0.5) coordinate (d5) 
  (d1)    -- +(-45:0.5) coordinate (e1) 
        (d1)    -- +(-135:0.5) coordinate (e2) 
  (d3)    -- +(-45:0.5) coordinate (e3) 
        (d3)    -- +(-135:0.5) coordinate (e4) 
                (d5)    -- +(-135:0.5) coordinate (e6) 
                                (d5)    -- +(-45:0.5) coordinate (e5) 
  (e1)    -- +(-45:0.5) coordinate (f1) 
        (e1)    -- +(-135:0.5) coordinate (f2) 
  (e3)    -- +(-45:0.5) coordinate (f3) 
        (e3)    -- +(-135:0.5) coordinate (f4) 
                (e5)    -- +(-45:0.5) coordinate (f5) 
                (e5)    -- +(-135:0.5) coordinate (f6) 
                                (e6)    -- +(-135:0.5) coordinate (f7) 
  (f1)    -- +(-45:0.5) coordinate (g1) 
        (f1)    -- +(-135:0.5) coordinate (g2) 
  (f3)    -- +(-45:0.5) coordinate (g3) 
        (f3)    -- +(-135:0.5) coordinate (g4) 
                (f5)    -- +(-45:0.5) coordinate (g5) 
                (f5)    -- +(-135:0.5) coordinate (g6) 
                                (f6)    -- +(-135:0.5) coordinate (g7)          
                                                                (f7)    -- +(-135:0.5) coordinate (g8)        
          (g1)    -- +(-45:0.5) coordinate (h1) 
        (g1)    -- +(-135:0.5) coordinate (h2) 
  (g3)    -- +(-45:0.5) coordinate (h3) 
        (g3)    -- +(-135:0.5) coordinate (h4) 
                (g5)    -- +(-45:0.5) coordinate (h5) 
                (g5)    -- +(-135:0.5) coordinate (h6) 
                                (g6)    -- +(-135:0.5) coordinate (h7)          
                                                                (g7)    -- +(-135:0.5) coordinate (h8)           
                                                                       (h1)    -- +(-45:0.5) coordinate (i1) 
        (h1)    -- +(-135:0.5) coordinate (i2) 
  (h3)    -- +(-45:0.5) coordinate (i3) 
        (h3)    -- +(-135:0.5) coordinate (i4) 
                (h5)    -- +(-45:0.5) coordinate (i5) 
                (h5)    -- +(-135:0.5) coordinate (i6) 
                                (h6)    -- +(-135:0.5) coordinate (i7)          
                                                                (h7)    -- +(-135:0.5) coordinate (i8)           
                            (i1)    -- +(-45:0.5) coordinate (j1) 
        (i1)    -- +(-135:0.5) coordinate (j2) 
  (i3)    -- +(-45:0.5) coordinate (j3) 
        (i3)    -- +(-135:0.5) coordinate (j4) 
                (i5)    -- +(-45:0.5) coordinate (j5) 
                (i5)    -- +(-135:0.5) coordinate (j6) 
                                (i6)    -- +(-135:0.5) coordinate (j7)          
                                                                (i7)    -- +(-135:0.5) coordinate (j8)      
         (j1)    -- +(-45:0.5) coordinate (k1) 
        (j1)    -- +(-135:0.5) coordinate (k2) 
  (j3)    -- +(-45:0.5) coordinate (k3) 
        (j3)    -- +(-135:0.5) coordinate (k4) 
                (j5)    -- +(-45:0.5) coordinate (k5) 
                (j5)    -- +(-135:0.5) coordinate (k6) 
                                (j6)    -- +(-135:0.5) coordinate (k7)          
                                                                (j7)    -- +(-135:0.5) coordinate (k8)                                      
 (g8)    -- +(-135:0.5) coordinate (h9)
   (h9)    -- +(-45:0.5) coordinate (i9)  
   (h9)    -- +(-135:0.5) coordinate (i10)
   (i9)    -- +(-45:0.5) coordinate (j9)  
   (i9)    -- +(-135:0.5) coordinate (j10)
   (i10)    -- +(-135:0.5) coordinate (j11)   
   (j9)    -- +(-45:0.5) coordinate (k9)  
   (j9)    -- +(-135:0.5) coordinate (k10)
   (j10)    -- +(-135:0.5) coordinate (k11)  
      (j11)    -- +(-135:0.5) coordinate (k12)  
   ;
            \fill(a1) circle (1pt);             \fill(a2) circle (1pt); \fill(a3) circle (1pt); 
              \fill(b1) circle (1pt);                 \fill(b2) circle (1pt);                          \fill(b3) circle (1pt);   
               \fill(c1) circle (1pt);                 \fill(c2) circle (1pt);                          \fill(c3) circle (1pt);   \fill(c4) circle (1pt);   
                              \fill(d1) circle (1pt);                 \fill(d2) circle (1pt);                          \fill(d3) circle (1pt);   \fill(d4) circle (1pt);   \fill(d5) circle (1pt);   
                                    \fill(e1) circle (1pt);                 \fill(e2) circle (1pt);                          \fill(e3) circle (1pt);   \fill(e4) circle (1pt);   
\fill(e5) circle (1pt);     \fill(e6) circle (1pt);   
                                                                        \fill(f1) circle (1pt);                 \fill(f2) circle (1pt);                          \fill(f3) circle (1pt);   \fill(f4) circle (1pt);   \fill(f5) circle (1pt);     \fill(f6) circle (1pt);    \fill(f7) circle (1pt);   
    \fill(g1) circle (1pt);                 \fill(g2) circle (1pt);                          \fill(g3) circle (1pt);   \fill(g4) circle (1pt);   \fill(g5) circle (1pt);     \fill(g6) circle (1pt);    \fill(g7) circle (1pt);   
 \fill(g8) circle (1pt);   
  \fill(h1) circle (1pt);                 \fill(h2) circle (1pt);                          \fill(h3) circle (1pt);   \fill(h4) circle (1pt);   \fill(h5) circle (1pt);     \fill(h6) circle (1pt);    \fill(h7) circle (1pt);   
 \fill(h8) circle (1pt);   
  \fill(j1) circle (1pt);                 \fill(j2) circle (1pt);                          \fill(j3) circle (1pt);   \fill(j4) circle (1pt);   \fill(j5) circle (1pt);     \fill(j6) circle (1pt);    \fill(j7) circle (1pt);   
 \fill(j8) circle (1pt);   
  \fill(i1) circle (1pt);                 \fill(i2) circle (1pt);                          \fill(i3) circle (1pt);   \fill(i4) circle (1pt);   \fill(i5) circle (1pt);     \fill(i6) circle (1pt);    \fill(i7) circle (1pt);   
 \fill(i8) circle (1pt);  
   \fill(k1) circle (1pt);                 \fill(k2) circle (1pt);                          \fill(k3) circle (1pt);   \fill(k4) circle (1pt);   \fill(k5) circle (1pt);     \fill(k6) circle (1pt);    \fill(k7) circle (1pt);   
 \fill(k8) circle (1pt);   
        \draw[dotted](b3) -- +(270:10);       \draw[dotted](b3) -- +(90:1);
        \draw[dotted](b1) -- +(270:10);       \draw[dotted](b1) -- +(90:1);
        \draw[dotted](f7) -- +(270:10);        \draw[dotted](f7) -- +(90:1);
         \draw[dotted](f1) -- +(270:10);        \draw[dotted](f1) -- +(90:1);
 \draw[dotted](j1) -- +(270:10);        \draw[dotted](j1) -- +(90:1);
 \draw[dotted](j11) -- +(270:10);        \draw[dotted](j11) -- +(90:1);       ; 
          \fill(i10) circle (1pt);     \fill(i9) circle (1pt);   \fill(h9) circle (1pt);  
             \fill(j9) circle (1pt);   \fill(j10) circle (1pt);  \fill(j11) circle (1pt);  
             \fill(k9) circle (1pt);               \fill(k10) circle (1pt);      \fill(k11) circle (1pt);               \fill(k12) circle (1pt);  
\draw(a1) --  (a2); \draw(a2) --  (d4);  \draw(d4) --  (k4); 
\path (0,0) + (-90:4.4 cm) coordinate (origin); 
\path (origin) ++ (120:1.25)++(-120:1.25) coordinate (1'); 
\path (1') ++ (120:1.5)++(-120:1.5) coordinate (2'); 
\path (2') ++ (120:1.25)++(-120:1.25) coordinate (3'); 
\path (origin) ++ (60:1.5)++(-60:1.5) coordinate (1); 
\path (1) ++ (60:1.5)++(-60:1.5) coordinate (2); 
\path (2) ++ (60:1.25)++(-60:1.25) coordinate (3); 
\draw(origin) node {$\mathfrak{a}_0$};
\draw(1') node {$\mathfrak{a}_{1'}$};
\draw(2') node {$\mathfrak{a}_{2'}$};
\draw(3') node {$\mathfrak{a}_{3'}$};
\draw(1) node {$\mathfrak{a}_{1}$};
\draw(2) node {$\mathfrak{a}_{2}$};
\draw(3) node {$\mathfrak{a}_{3}$};
  }
    \end{tikzpicture}
   $$ 
   \caption{Maximal paths 
    in $\Path(
    \lambda_{1'},
    \lambda_{2'})$
    and 
    $\Path(
    \lambda_{1},
    \lambda_{2'})$
  respectively.    The degree is equal to 1 in both cases.
}  
      \label{the other bits}  \end{figure}

\end{eg}

%

\begin{thm} If $l=2$,  
 the full submodule structure of the $\TL_n(\kappa)$-modules $\Delta(\lambda_i)$
 and $\Delta(\lambda_{i'})$ are given by the strong Alperin diagrams (in the sense of \cite{alperin}) below.
$$
 \scalefont{0.8} \begin{tikzpicture}[scale=0.7]
{  
   \path 
      (0, 0)              coordinate (a1)
        (a1)    -- +(-45:2) coordinate (a2) 
        (a1)    -- +(-135:2) coordinate (a3) 
        (a2)    -- +(-90:2) coordinate (b2) 
        (a3)    -- +(-90:2) coordinate (b3) 
        (b3)    -- +(-90:2) coordinate (c3) 
        (b2)    -- +(-90:2) coordinate (c2)  ;
\draw(a1)--(a2);
\draw(a1)--(a3); 
\draw(b2)--(a2);\draw(b3)--(a2);
\draw(b3)--(a3);  \draw(b2)--(a3);
\draw[dashed](b3)--(c3);  \draw[dashed](b2)--(c3); 
\draw[dashed](b3)--(c2);  \draw[dashed](b2)--(c2); 
 \fill[white] (a1) circle (0.7cm); 
 \draw(a1) node {$ \scalefont{0.8} L(\lambda_i)$};  
   \fill[white] (a2)circle (0.7cm); 
 \draw(a2) node {$ \scalefont{0.8} L(\lambda_{i+1'}	\langle  1 \rangle )$};  
  \fill[white] (a3) circle (0.7cm); 
 \draw(a3) node {$ L(\lambda_{i+1}\langle  1 \rangle)$};  
    \fill[white] (b2)circle (0.7cm); 
 \draw(b2) node {$ L(\lambda_{i+2}\langle  2\rangle)$};  
  \fill[white] (b3) circle (0.7cm); 
 \draw(b3) node {$ L(\lambda_{i+2'}\langle  2 \rangle)$};  
   \fill[white] (c3) circle (0.7cm); 
     \fill[white] (c2) circle (0.7cm); }
\end{tikzpicture} 
\quad \quad 
 \begin{tikzpicture}[scale=0.7]
{  
   \path 
      (0, 0)              coordinate (a1)
        (a1)    -- +(-45:2) coordinate (a2) 
        (a1)    -- +(-135:2) coordinate (a3) 
        (a2)    -- +(-90:2) coordinate (b2) 
        (a3)    -- +(-90:2) coordinate (b3) 
        (b3)    -- +(-90:2) coordinate (c3) 
        (b2)    -- +(-90:2) coordinate (c2)  ;
\draw(a1)--(a2);
\draw(a1)--(a3); 
\draw(b2)--(a2);\draw(b3)--(a2);
\draw(b3)--(a3);  \draw(b2)--(a3);
\draw[dashed](b3)--(c3);  \draw[dashed](b2)--(c3); 
\draw[dashed](b3)--(c2);  \draw[dashed](b2)--(c2); 
 \fill[white] (a1) circle (0.7cm); 
 \draw(a1) node {$ \scalefont{0.8} L(\lambda_{i'})$};  
   \fill[white] (a2)circle (0.7cm); 
 \draw(a2) node {$ \scalefont{0.8} L(\lambda_{i+1'}	\langle  1 \rangle )$};  
  \fill[white] (a3) circle (0.7cm); 
 \draw(a3) node {$ L(\lambda_{i+1}\langle  1 \rangle)$};  
    \fill[white] (b2)circle (0.7cm); 
 \draw(b2) node {$ L(\lambda_{i+2}\langle  2\rangle)$};  
  \fill[white] (b3) circle (0.7cm); 
 \draw(b3) node {$ L(\lambda_{i+2'}\langle  2 \rangle)$};  
   \fill[white] (c3) circle (0.7cm); 
     \fill[white] (c2) circle (0.7cm); }
\end{tikzpicture} 
$$
Therefore  $\Dim{(\Hom_{\TL_n(\kappa)}(\Delta(\lambda_j),\Delta(\lambda_i)))} = t^{j-i}$ for $i<j$ (in which case this homomorphism is injective) and the dimension is 0 otherwise.
 \end{thm}

\begin{proof}
Fix points $\lambda_{i^{(\prime)}},\lambda_{j^{(\prime)}}  \in E_l$ such  that 
$ i<j$.  
We have seen that  if $\omega(\SSTT)\in \Path(\lambda_{i^{(\prime)}},\lambda_{j^{(\prime)}})$ is maximal, then it labels a decomposition number $d_{i^{(\prime)}j^{(\prime)}}= t^{j-i}$.  Therefore
 $B_\SSTT$ generates a simple composition factor 
$L(\lambda_{j^{(\prime)}})\langle   j-i  \rangle$ of the standard module 
$\Delta(\lambda_{i^{(\prime)}})$. 

Given $\lambda_{i^{(\prime)}},\lambda_{j^{(\prime)}}\in E_l$, and $i<j$, 
we let $1^{{j^{(\prime)}}}_{{i^{(\prime)}}}$  denote the element 
 $B_\SSTT$ for $\omega(\SSTT)$ the 
 unique maximal path in $\Path(\lambda_{j^{(\prime)}},\lambda_{i^{(\prime)}})$.  
 We shall show that
 \[
1_{i+1}^{i+2^{(\prime)}}\circ 1_{i^{(\prime)}}^{i+1}=  \pm1_{i^{(\prime)}}^{i+2^{(\prime)}}= 
 1_{i+1'}^{i+2^{(\prime)}}\circ 1_{i^{(\prime)}}^{i+1'}
\]
and the result will follow.  
First, notice that 
$\deg(1_{i^{(\prime)}}^{j^{(\prime)}})=j-i$ and this is the unique basis element of 
$\Delta({i^{(\prime)}})$ of this degree.  
By comparing degrees, we deduce that 
\begin{equation}\label{jjjjjjjjj}
 1_{i+1^{(\prime)}}^{i+2^{(\prime)}}    \circ    1_{i^{(\prime)}}^{i+1^{(\prime)}} = 
  c   1_{i^{(\prime)}} ^{i+2^{(\prime)}} 
\end{equation}
for some $c \in \mathbb{C}$.  It remains to show that $c=\pm1$ (note that, for the result to hold, it is enough to show that $c \neq 0$).  
 It is clear that the lefthand side of Equation \ref{jjjjjjjjj} is a 
 diagram with distinguished black points on  northern and southern
 boundaries given by the loadings corresponding to the partitions $\lambda_{i+2^{(\prime)}}$ and $\lambda_{i^{(\prime)}}$, respectively.  
 If the bijection traced out by the strands (after concatenation) uses the minimal number of crossings, then we are done.

Suppose that we are not in the case above, then we must apply the relations 
to the product 
  to obtain a diagram of the form 
$      c   1_{i^{(\prime)}} ^{i+2^{(\prime)}} $ for some $c \in \mathbb{C}$.     
   This product 
   has a number of `extra crossings' of strands of the same residue (that is, crossings which do not appear in $ 1_{i^{(\prime)}} ^{i+2^{(\prime)}} $).
  The rightmost of these  crossings involves a pair of strands of 
  residue $r$, say.  This crossing is bypassed  by the ghost of the strand of residue $r-1$ immediately to its right (for an example, see Figure \ref{Lironsayingrelations}). 
   Applying relation (2.10), the product can be written as a sum of two terms: one is zero modulo  more dominant terms, the other differs from the  original diagram 
  only   where we have untied the distinguished  crossing (for an example, see Figure \ref{Lironsayingrelations}).   
   
 Now suppose that  the resulting diagram is not equal to
$ 1_{i^{(\prime)}} ^{i+2^{(\prime)}} $, in which case it 
 has a 
rightmost `extra crossing' of residue $r+1$.  Now consider the  ghost of the leftmost of the two strands we  untied in the previous step;
the ghost of    this strand  bypasses  the rightmost `extra crossing'.  Repeating the above argument for all the  crossings, we obtain the result.  
  \end{proof}

\begin{eg}The elements $1^\nu_\mu$ for $\nu,\mu \in \lambda_{0},\lambda_{1},\lambda_{1'},\lambda_{2'}$ are depicted in Figures \ref{plentyoffishdiagrams2} and \ref{plentyoffishdiagrams}, below.  Figure \ref{Lironsayingrelations}  depicts the first use of relation (2.10) on the  product $1_{1'}^{2'}\circ 1_0^{1'}$ to untie a crossing.  
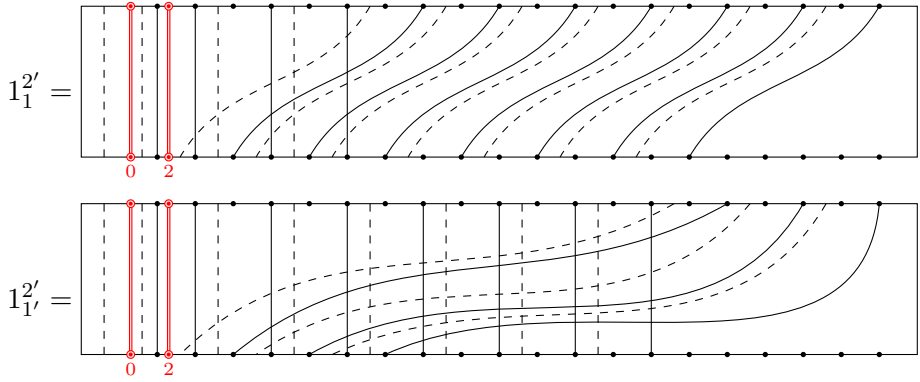
\begin{figure}[ht]\captionsetup{width=0.9\textwidth}
$$
   1_{1}^{2'}=\begin{minipage}{120mm} \scalefont{0.7}
    \begin{tikzpicture}[scale=1] 
   \draw (-2,0) rectangle (9,2);  
  \foreach \x in {-2,...,17}  
     {\fill (0.5*\x,2) circle (1pt); 
       \fill (0.5*\x,0) circle (1pt);}
  \draw[wei2] (-1.35,0)--(-1.35,2);
    \draw[wei2] (-1.35+0.5,0)--(-1.35+0.5,2);
      \node [wei2,below] at (-1.35,0) {\tiny $0$};
            \node [wei2,below] at (-1.35+0.5,0) {\tiny $2$}; 
            \draw[wei2,fill]  (-1.35,0)   circle(1pt);\draw[wei2,fill]  (-1.35,2)   circle(1pt);\draw[wei2,fill]  (-1.35+0.5,2)   circle(1pt);\draw[wei2,fill]  (-1.35+0.5,0)   circle(1pt);
                \foreach \x in {-2,-1,1,3}  
    {\draw (0.5*\x,2) -- (0.5*\x,0); ;}
  \foreach \x in {2,4,6,8,0,10,12}  
  {  \draw  (0.5*\x+2.5,2)  to [out=-120,in=60] (0.5*\x,0) ; }
            \foreach \x  in {-2,-1,1,3}  
    {\draw[dashed] (0.5*\x-0.7,2) -- (0.5*\x-0.7,0); ;}
  \foreach \x  in {2,4,6,8,0,10,12}  
  {  \draw[dashed]  (0.5*\x-0.7+2.5,2)  to [out=-120,in=60] (0.5*\x-0.7,0) ; }

 \end{tikzpicture}
 \end{minipage}
    $$ 
  $$
   1_{1'}^{2'}=\begin{minipage}{120mm} \scalefont{0.7}
    \begin{tikzpicture}[scale=1] 
   \draw (-2,0) rectangle (9,2);  
  \foreach \x in {-2,...,17}  
     {\fill (0.5*\x,2) circle (1pt); 
       \fill (0.5*\x,0) circle (1pt);}
  \draw[wei2] (-1.35,0)--(-1.35,2);
    \draw[wei2] (-1.35+0.5,0)--(-1.35+0.5,2);
      \node [wei2,below] at (-1.35,0) {\tiny $0$};
            \node [wei2,below] at (-1.35+0.5,0) {\tiny $2$}; 
            \draw[wei2,fill]  (-1.35,0)   circle(1pt);\draw[wei2,fill]  (-1.35,2)   circle(1pt);\draw[wei2,fill]  (-1.35+0.5,2)   circle(1pt);\draw[wei2,fill]  (-1.35+0.5,0)   circle(1pt);
               \foreach \x in {-2,-1,1,3,5,7,9,11}  
   {\draw (0.5*\x,2) -- (0.5*\x,0); ;}
       \foreach \x in {-2,-1,1,3,5,7,9,11}  
   {\draw[dashed] (0.5*\x-0.7,2) -- (0.5*\x-0.7,0); ;}

  \foreach \x in {0}  
  {  \draw  (0.5*\x+6.5,2)  to [out=-150,in=40] (0.5*\x,0) ; }
 \foreach \x in {2}  
  {  \draw  (0.5*\x+6.5,2)  to [out=-120,in=30] (0.5*\x,0) ; }
  \foreach \x in {4}  
  {  \draw  (0.5*\x+6.5,2)  to [out=-95,in=25] (0.5*\x,0) ; }
      \foreach \x in {0}  
  {  \draw[dashed]  (0.5*\x-0.7+6.5,2)  to [out=-150,in=45] (0.5*\x-0.7,0) ; }
 \foreach \x in {2}  
  {  \draw[dashed]  (0.5*\x-0.7+6.5,2)  to [out=-130,in=35] (0.5*\x-0.7,0) ; }
  \foreach \x in {4}  
  {  \draw[dashed]  (0.5*\x-0.7+6.5,2)  to [out=-120,in=25] (0.5*\x-0.7,0) ; }
  \end{tikzpicture}
 \end{minipage}    $$ 
 \caption{The elements $1^{2'}_{1'}$ and $1^{2'}_{1}$ corresponding to the maximal paths in Figure \ref{the other bits}.
 }
 \label{plentyoffishdiagrams2}
 \end{figure}

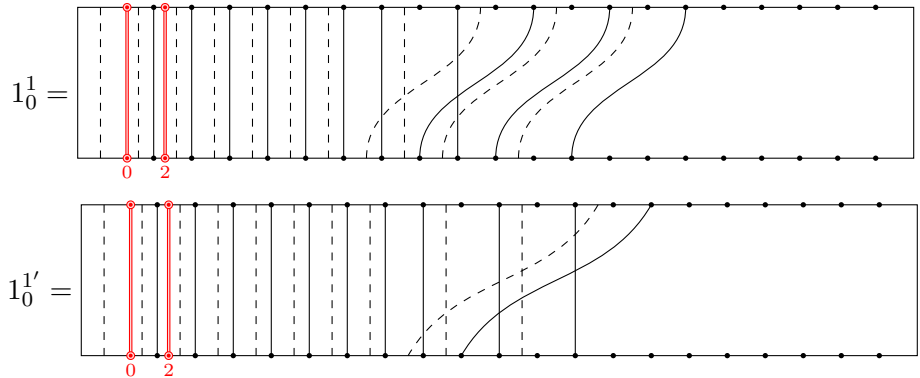
\begin{figure}[ht]\captionsetup{width=0.9\textwidth}
  $$
    \scalefont{0.7}
1_{0}^{1}=\begin{minipage}{120mm}    \begin{tikzpicture}[scale=1] 
   \draw (-2,0) rectangle (9,2);  
  \foreach \x in {-2,...,17}  
     {\fill (0.5*\x,2) circle (1pt); 
       \fill (0.5*\x,0) circle (1pt);}
  \draw[wei2] (-1.35,0)--(-1.35,2);
    \draw[wei2] (-1.35+0.5,0)--(-1.35+0.5,2);
      \node [wei2,below] at (-1.35,0) {\tiny $0$};
            \node [wei2,below] at (-1.35+0.5,0) {\tiny $2$}; 
            \draw[wei2,fill]  (-1.35,0)   circle(1pt);\draw[wei2,fill]  (-1.35,2)   circle(1pt);\draw[wei2,fill]  (-1.35+0.5,2)   circle(1pt);\draw[wei2,fill]  (-1.35+0.5,0)   circle(1pt);
      \foreach \x in {-2,-1,0,1,2,3,4,6}  
   {\draw (0.5*\x,2) -- (0.5*\x,0); ;}
      \foreach \x in {5,7,9}  
    {\draw (0.5*\x,0)  to [out=90,in=-90] (0.5*\x+1.5,2); ;}
        \foreach \x in {-2,-1,0,1,2,3,4,6}  
   {\draw[dashed] (0.5*\x-0.7,2) -- (0.5*\x-0.7,0); ;}
      \foreach \x in {5,7,9}  
    {\draw[dashed] (0.5*\x-0.7,0)  to [out=90,in=-90] (0.5*\x-0.7+1.5,2); ;}
 \end{tikzpicture}
 \end{minipage}
    $$  
    $$
    \scalefont{0.7}
1_{0}^{1'}=\begin{minipage}{120mm}    \begin{tikzpicture}[scale=1] 
   \draw (-2,0) rectangle (9,2);  
  \foreach \x in {-2,...,17}  
     {\fill (0.5*\x,2) circle (1pt); 
       \fill (0.5*\x,0) circle (1pt);}
  \draw[wei2] (-1.35,0)--(-1.35,2);
    \draw[wei2] (-1.35+0.5,0)--(-1.35+0.5,2);
      \node [wei2,below] at (-1.35,0) {\tiny $0$};
            \node [wei2,below] at (-1.35+0.5,0) {\tiny $2$}; 
            \draw[wei2,fill]  (-1.35,0)   circle(1pt);\draw[wei2,fill]  (-1.35,2)   circle(1pt);\draw[wei2,fill]  (-1.35+0.5,2)   circle(1pt);\draw[wei2,fill]  (-1.35+0.5,0)   circle(1pt);
      \foreach \x in {-2,-1,0,1,2,3,4,5}  
   {\draw (0.5*\x,2) -- (0.5*\x,0); ;}
     \foreach \x in {7,9}  
   {\draw (0.5*\x,2) -- (0.5*\x,0); ;}
      \draw  (0.5*11,2)  to [out=-120,in=60] (0.5*6,0) ; 
        \foreach \x in {-2,-1,0,1,2,3,4,5}  
   {\draw[dashed] (0.5*\x-0.7,2) -- (0.5*\x-0.7,0); ;}
     \foreach \x in {7,9}  
   {\draw[dashed] (0.5*\x-0.7,2) -- (0.5*\x-0.7,0); ;}
      \draw[dashed]  (0.5*11-0.7,2)  to [out=-120,in=60] (0.5*6-0.7,0) ; 
 \end{tikzpicture}
 \end{minipage}
    $$ 
 \caption{The elements $1^1_0$ and $1^{1'}_0$ corresponding to the maximal paths in Figure \ref{11law}.
 }
\label{plentyoffishdiagrams}
\end{figure}
\end{eg}

\begin{figure}[ht]\captionsetup{width=0.9\textwidth}
$$
   1_{0}^{2'}=\begin{minipage}{120mm} \scalefont{0.7}
    \begin{tikzpicture}[scale=1] 
   \draw (-2,0) rectangle (9,2);  
  \foreach \x in {-2,...,17}  
     {\fill (0.5*\x,2) circle (1pt); 
       \fill (0.5*\x,0) circle (1pt);}
  \draw[wei2] (-1.35,0)--(-1.35,2);
    \draw[wei2] (-1.35+0.5,0)--(-1.35+0.5,2);
      \node [wei2,below] at (-1.35,0) {\tiny $0$};
            \node [wei2,below] at (-1.35+0.5,0) {\tiny $2$}; 
            \draw[wei2,fill]  (-1.35,0)   circle(1pt);\draw[wei2,fill]  (-1.35,2)   circle(1pt);\draw[wei2,fill]  (-1.35+0.5,2)   circle(1pt);\draw[wei2,fill]  (-1.35+0.5,0)   circle(1pt);
                \foreach \x in {-2,-1,1,3}  
    {\draw (0.5*\x,2) -- (0.5*\x,0); ;}
  \foreach \x in {0,2,4,6}  
  {  \draw  (0.5*\x+2.5,2)  to [out=-130,in=90] (0.5*\x,0) ; }
 \foreach \x in {5}  
  {  \draw  (0.5*\x+4,2)  to [out=-110,in=50] (0.5*\x,0) ; }
   \foreach \x in {7}  
  {  \draw  (0.5*\x+4,2)  to [out=-110,in=40] (0.5*\x,0) ; }
     \foreach \x in {9}  
  {  \draw  (0.5*\x+4,2)  to [out=-110,in=30] (0.5*\x,0) ; }
                  \foreach \x in {-2,-1,1,3}  
    {\draw[dashed] (0.5*\x-0.7,2) -- (0.5*\x-0.7,0); ;}
  \foreach \x in {0,2,4,6}  
  {  \draw[dashed]  (0.5*\x-0.7+2.5,2)  to [out=-130,in=90] (0.5*\x-0.7,0) ; }
 \foreach \x in {5}  
  {  \draw[dashed]  (0.5*\x-0.7+4,2)  to [out=-110,in=50] (0.5*\x-0.7,0) ; }
   \foreach \x in {7}  
  {  \draw[dashed]  (0.5*\x-0.7+4,2)  to [out=-110,in=40] (0.5*\x-0.7,0) ; }
     \foreach \x in {9}  
  {  \draw[dashed]  (0.5*\x-0.7+4,2)  to [out=-110,in=30] (0.5*\x-0.7,0) ; }
  \end{tikzpicture}
 \end{minipage}
    $$ 
    \caption{The element $1_0^{2'}$.}
    \label{concatenateddiagram}
    \end{figure}
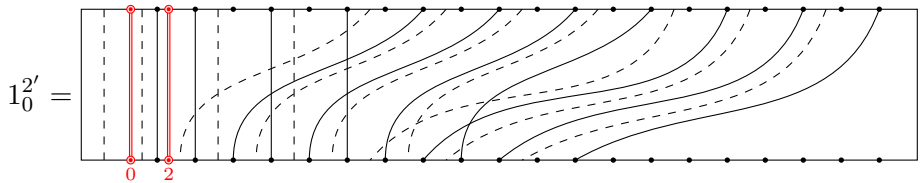

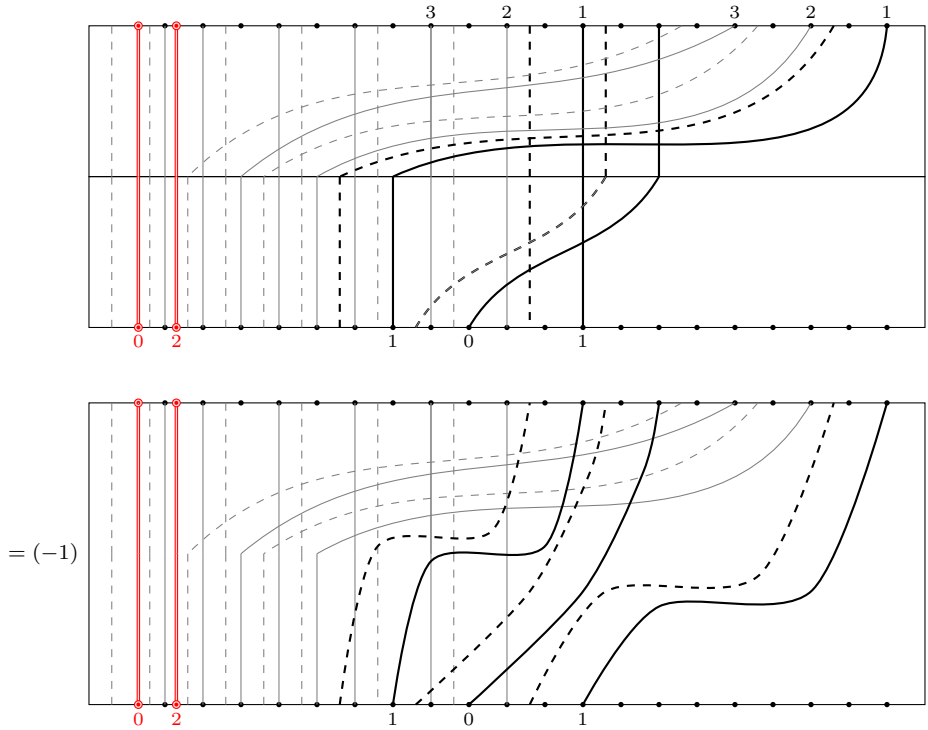
\begin{figure}[ht]\captionsetup{width=0.9\textwidth}
\captionsetup{width=0.9\textwidth}
\centering
$$
\scalefont{0.7}
    \begin{tikzpicture}[scale=1] 
   \draw (-2,0) rectangle (9,2);   \draw (-2,2) rectangle (9,4);  
  \foreach \x in {-2,...,17}  
     {\fill (0.5*\x,4) circle (1pt); 
 }
  \draw[wei2] (-1.35,2)--(-1.35,4);
    \draw[wei2] (-1.35+0.5,2)--(-1.35+0.5,4);
               \draw[wei2,fill]  (-1.35,4)   circle(1pt);\draw[wei2,fill]  (-1.35+0.5,4)   circle(1pt);                \foreach \x in {-2,-1,1,3,5,7}  
   {\draw[gray,] (0.5*\x,4) -- (0.5*\x,2); ;}  
    {\draw[thick] (0.5*11,4) -- (0.5*11,2); ;}
     {\draw[thick] (0.5*9,4) -- (0.5*9,0); ;}
          {\draw[thick,dashed] (0.5*9-0.7,4) -- (0.5*9-0.7,0); ;}
       \foreach \x in {-2,-1,1,3,7}  
   {\draw[gray,dashed] (0.5*\x-0.7,4) -- (0.5*\x-0.7,2); ;}
 {\draw[dashed, gray] (0.5*5-0.7,4) -- (0.5*5-0.7,-0); ;}
{\draw[dashed, thick] (0.5*11-0.7,4) -- (0.5*11-0.7,2); ;}
 {\draw[gray] (0.5*5,4) -- (0.5*5,0); ;}
  \foreach \x in {0}  
  {  \draw[gray]  (0.5*\x+6.5,4)  to [out=-150,in=40] (0.5*\x,2) ; }
 \foreach \x in {2}  
  {  \draw[gray]  (0.5*\x+6.5,4)  to [out=-120,in=30] (0.5*\x,2) ; }
  \foreach \x in {4}  
  {\draw[thick]  (0.5*\x+6.5,4)  to [out=-95,in=25] (0.5*\x,2) ; }
      \foreach \x in {0}  
  {  \draw[dashed,gray]  (0.5*\x-0.7+6.5,4)  to [out=-150,in=45] (0.5*\x-0.7,2) ; }
 \foreach \x in {2}  
  {  \draw[dashed,gray,]  (0.5*\x-0.7+6.5,4)  to [out=-130,in=35] (0.5*\x-0.7,2) ; }
  \foreach \x in {4}  
  {  \draw[dashed,thick]  (0.5*\x-0.7+6.5,4)  to [out=-120,in=25] (0.5*\x-0.7,2) 
  ; }
  {
   \foreach \x in {-2,...,17}  
     {
        \fill (0.5*\x,0) circle (1pt);}
  \draw[wei2] (-1.35,0)--(-1.35,2);
    \draw[wei2] (-1.35+0.5,0)--(-1.35+0.5,2);
      \node [wei2,below] at (-1.35,0) {\tiny $0$};
             \node [below] at (3,0) {\tiny $0$};
                          \node [below] at (4.5,0) {\tiny $1$};
\node [above] at (4.5,4) {\tiny $1$};\node [above] at (3.5,4) {\tiny $2$};
\node [above] at (2.5,4) {\tiny $3$};
\node [above] at (8.5,4) {\tiny $1$};\node [above] at (7.5,4) {\tiny $2$};
\node [above] at (6.5,4) {\tiny $3$};
                                       \node [below] at (2,0) {\tiny $1$};
            \node [wei2,below] at (-1.35+0.5,0) {\tiny $2$}; 
            \draw[wei2,fill]  (-1.35,0)   circle(1pt); \draw[wei2,fill]  (-1.35+0.5,0)   circle(1pt);
      \foreach \x in {-2,-1,0,1,2,3}  
   {\draw[gray,] (0.5*\x,2) -- (0.5*\x,0); ;}
    {\draw[thick] (0.5*4,2) -- (0.5*4,0); ;}
    {\draw[dashed,thick] (0.5*4-0.7,2) -- (0.5*4-0.7,0); ;}

     \foreach \x in {7}  
   {\draw[gray,] (0.5*\x,2) -- (0.5*\x,0); ;}
      \draw[dashed, thick]  (0.5*11-0.7,2)  to [out=-120,in=60] (0.5*6-0.7,0) ; 
      \draw[thick]  (0.5*11,2)  to [out=-120,in=60] (0.5*6,0) ; 
        \foreach \x in {-2,-1,0,1,2,3}  
   {\draw[dashed,gray,] (0.5*\x-0.7,2) -- (0.5*\x-0.7,0); ;}
     \foreach \x in {7}  
   {\draw[dashed,gray,] (0.5*\x-0.7,2) -- (0.5*\x-0.7,0); ;}
      \draw[dashed,gray]  (0.5*11-0.7,2)  to [out=-120,in=60] (0.5*6-0.7,0) ; 
  }
   \draw (-2,-5+0) rectangle (9,-5+4);  
  \foreach \x in {-2,...,17}  
     {\fill (0.5*\x,-5+4) circle (1pt); 
}
  \draw[wei2] (-1.35,-5+2)--(-1.35,-5+4);
    \draw[wei2] (-1.35+0.5,-5+2)--(-1.35+0.5,-5+4);
               \draw[wei2]  (-1.35,-5+4)   circle(1pt);\draw[wei2,fill]  (-1.35+0.5,-5+4)   circle(1pt);                \foreach \x in {-2, -1,1, 3, 5, 7}  
   {\draw[gray] (0.5*\x,-5+4) -- (0.5*\x,-5+2); ;}  
       \foreach \x in {-2, -1, +1, 3,7}  
   {\draw[gray,dashed] (0.5*\x-0.7,-5+4) -- (0.5*\x-0.7,-5+2); ;}
  {\draw[dashed,gray] (0.5*5-0.7,-5+4) -- (0.5*5-0.7,-5+-0); ;}
  {\draw[gray] (0.5*5,-5+4) -- (0.5*5,-5+0); ;}
  \foreach \x in {0}  
  {  \draw[gray]  (0.5*\x+6.5,-5+4)  to [out=-150,in=40] (0.5*\x,-5+2) ; }
 \foreach \x in {2}  
  {  \draw[gray]  (0.5*\x+6.5,-5+4)  to [out=-120,in=30] (0.5*\x,-5+2) ; }
  \foreach \x in {4}  
      \foreach \x in {0}  
  {  \draw[dashed,,gray]  (0.5*\x-0.7+6.5,-5+4)  to [out=-150,in=45] (0.5*\x-0.7,-5+2) ; }
 \foreach \x in {2}  
  {  \draw[dashed,,gray]  (0.5*\x-0.7+6.5,-5+4)  to [out=-130,in=35] (0.5*\x-0.7,-5+2) ; }
  {
   \foreach \x in {-2,...,17}  
     { 
        \fill (0.5*\x,-5+0) circle (1pt);}
  \draw[wei2] (-1.35,-5+0)--(-1.35,-5+2);
    \draw[wei2] (-1.35+0.5,-5+0)--(-1.35+0.5,-5+2);
      \node [wei2,below] at (-1.35,-5+0) {\tiny $0$};
             \node [below] at (3,-5+0) {\tiny $0$};
                          \node [below] at (4.5,-5+0) {\tiny $1$};
                                       \node [below] at (2,-5+0) {\tiny $1$};
            \node [wei2,below] at (-1.35+0.5,-5+0) {\tiny $2$}; 
            \draw[wei2,fill]  (-1.35,-5+0)   circle(1pt); \draw[wei2,fill]  (-1.35+0.5,-5+0)   circle(1pt);
      \foreach \x in {-2,-1,0,1,2,3}  
   {\draw[gray,] (0.5*\x,-5+2) -- (0.5*\x,-5+0); ;}
 
     \foreach \x in {7}  
   {\draw[gray,] (0.5*\x,-5+2) -- (0.5*\x,-5+0); ;}
        \foreach \x in {-2,-1, 0, 1, 2,3}  
   {\draw[dashed,,gray,] (0.5*\x-0.7,-5+2) -- (0.5*\x-0.7,-5+0); ;}
     \foreach \x in {7}  
   {\draw[dashed,,gray,] (0.5*\x-0.7,-5+2) -- (0.5*\x-0.7,-5+0); ;}

   \draw[black,thick] plot[smooth] coordinates 
     {(4.5,-5+0)(5.5,-5+1.3) (7.5,-5+1.5)(8.5,-5+4) };
   \draw[black,thick,dashed] plot[smooth] coordinates 
     {(4.5-0.7,-5+0)(5.5-0.7,-5+1.5) (7.5-0.7,-5+1.7)(8.5-0.7,-5+4) };

   \draw[black,thick] plot[smooth] coordinates 
     {(2,-5+0)(2.5,-5+1.9)(4,-5+2.1) (4.5,-5+4) };
  \draw[black,thick,dashed] plot[smooth] coordinates 
     {(2-0.7,-5+0)(2.5-0.7,-5+2.1)(4-0.7,-5+2.3) (4.5-0.7,-5+4) };

 \draw[black,thick] plot[smooth] coordinates 
     {(3,-5+0)(4.5,-5+1.5)(5.3,-5+3.1)  (5.5,-5+4) };
 \draw[black,dashed,thick] plot[smooth] coordinates 
     {(3-0.7,-5+0)(4.5-0.7,-5+1.5)(5.3-0.7,-5+3.1)  (5.5-0.7,-5+4) }; 
\draw(-2.6,-3) node {$= (-1) $};
 }
  \end{tikzpicture}
$$
 \caption{
The top diagram is obtained by concatenation of the diagram  $1_{1'}^{2'}$ above $1_0^{1'}$.
The lower diagram is obtained by applying relation (2.10) to the product $1_{1'}^{2'}\circ 1_0^{1'}$.
We move the ghost $0$ strand through the crossing pair of black strands of residue 1 (we do not record the diagram which is zero modulo more dominant terms). 
 \\
We have made emphasised the strands to which we are applying relation (2.10) and we have recorded their residues along
the southern edge of the frames.
Along the northern edge of the frame of the top diagram, we have recorded the residues of the 
3 extra crossings of like-labelled pairs.  
}
\label{Lironsayingrelations}
\end{figure}

 \clearpage
 
\newcommand{\etalchar}[1]{$^{#1}$}
\providecommand{\bysame}{\leavevmode\hbox to3em{\hrulefill}\thinspace}
\providecommand{\MR}{\relax\ifhmode\unskip\space\fi MR }
\providecommand{\MRhref}[2]{%
  \href{http://www.ams.org/mathscinet-getitem?mr=#1}{#2}
}
\providecommand{\href}[2]{#2}

 \end{document}